\documentclass{memo-l}

\usepackage{amssymb,amsmath,epsfig,graphics,mathrsfs}
\usepackage{graphicx}

\usepackage{fancyhdr}
\usepackage{hyperref}
\hypersetup{
 colorlinks   = true,
 urlcolor     = blue,
 linkcolor    = blue,
 citecolor   = red ,
 bookmarksopen=true
}

\theoremstyle{plain}

\newtheorem{theorem}{Theorem}[section]
\newtheorem{lemma}[theorem]{Lemma}
\newtheorem{prop}[theorem]{Proposition}
\newtheorem{corollary}[theorem]{Corollary}

\newtheorem{definition}[theorem]{Definition}

\newtheorem{remark}[theorem]{Remark}
\theoremstyle{definition}
\theoremstyle{remark}
\numberwithin{equation}{section} \numberwithin{figure}{section}

\begin{document}
\newcommand{\SL}{\mathcal L^{1,p}(\Om)}
\newcommand{\Lp}{L^p(\Omega)}
\newcommand{\CO}{C^\infty_0(\Omega)}
\newcommand{\Rn}{\mathbb R^n}
\newcommand{\Rm}{\mathbb R^m}
\newcommand{\R}{\mathbb R}
\newcommand{\Om}{\Omega}
\newcommand{\Hn}{\mathbb H^n}
\newcommand{\HH}{\mathbb H^1}
\newcommand{\eps}{\epsilon}
\newcommand{\BVX}{BV_H(\Omega)}
\newcommand{\IO}{\int_\Omega}
\newcommand{\bG}{\mathbb{G}}
\newcommand{\bg}{\mathfrak g}
\newcommand{\p}{\partial}
\newcommand{\Xnu}{\overset{\rightarrow}{ X_\nu}}
\newcommand{\nuX}{\boldsymbol{\nu}_H}
\newcommand{\Up}{\boldsymbol{\mathcal Y}_H}
\newcommand{\n}{\boldsymbol \nu}
\newcommand{\sigmau}{\boldsymbol{\sigma}^u_H}
\newcommand{\nui}{\nu_{H,i}}
\newcommand{\nuj}{\nu_{H,j}}
\newcommand{\dej}{\delta_{H,j}}
\newcommand{\cx}{\boldsymbol{c}_\mathscr S}
\newcommand{\sx}{\sigma_H}
\newcommand{\lx}{\mathcal L_H}
\newcommand{\pb}{\overline p}
\newcommand{\qb}{\overline q}
\newcommand{\ob}{\overline \omega}
\newcommand{\nuu}{\boldsymbol \nu_{H,u}}
\newcommand{\nuv}{\boldsymbol \nu_{H,v}}
\newcommand{\Bl}{\Bigl|_{\lambda = 0}}
\newcommand{\mS}{\mathscr S}
\newcommand{\delh}{\Delta_H}
\newcommand{\delinf}{\Delta_{H,\infty}}
\newcommand{\nabh}{\nabla^H}
\newcommand{\nabht}{\tilde{\nabla}^H}
\newcommand{\delp}{\Delta_{H,p}}
\newcommand{\mO}{\mathcal O}
\newcommand{\delhs}{\Delta_{H,S}}
\newcommand{\lhs}{\hat{\Delta}_{H,S}}
\newcommand{\bN}{\boldsymbol{N}}
\newcommand{\bnu}{\boldsymbol \nu}
\newcommand{\la}{\lambda}
\newcommand{\nup}{\boldsymbol{\nu}_H^\perp}
\newcommand{\fv}{\mathcal V^{H}_I(\mS;\mathcal X)}
\newcommand{\sv}{\mathcal V^{H}_{II}(\mS;\mathcal X)}
\newcommand{\di}{\nabla_{i}^{H,\mS}}
\newcommand{\one}{\nabla_{1}^{H,\mS}}
\newcommand{\two}{\nabla_{2}^{H,\mS}}
\newcommand{\del}{\nabla^{H,\mS}}
\newcommand{\delXY}{\nabla^{H,\mS}_X Y}
\newcommand{\oX}{\overline X}
\newcommand{\oY}{\overline Y}
\newcommand{\ou}{\overline u}
\newcommand{\duno}{\nabla^{H,\mS}_1}
\newcommand{\ddue}{\nabla^{H,\mS}_2}
\newcommand{\nh}{\nabla_H}
\newcommand{\jnu}{J_\nu}
\newcommand{\G}{\Gamma}
\newcommand{\vf}{\varphi}
\newcommand{\vt}{\vartheta}
\newcommand{\e}{\varepsilon}
\newcommand{\pt}{P^{(s)}_t} 
\newcommand{\mi}{\no^n}
\newcommand{\no}{\mathbb N_0}
\newcommand{\pa}{\p^\alpha}
\newcommand{\en}{e^{-2\pi i <\xi,x>}}
\newcommand{\eh}{e^{2\pi i <\xi,h>}}
\newcommand{\Ba}{\mathcal{P}_\alpha}
\newcommand{\us}{\R^{n+1}_+}
\newcommand{\din}{\dashint}
\newcommand{\Za}{Z_\alpha}
\newcommand{\ra}{\rho_\alpha}
\newcommand{\na}{\nabla_\alpha}
\newcommand{\Sa}{\mathbb{S}}
\newcommand{\Rnn}{\mathbb R^{n+1}}
\newcommand{\F}{\mathscr{F}}

\def\Xint#1{\mathchoice
{\XXint\displaystyle\textstyle{#1}}%
{\XXint\textstyle\scriptstyle{#1}}%
{\XXint\scriptstyle\scriptscriptstyle{#1}}%
{\XXint\scriptscriptstyle\scriptscriptstyle{#1}}%
\!\int}
\def\XXint#1#2#3{{\setbox0=\hbox{$#1{#2#3}{\int}$ }
\vcenter{\hbox{$#2#3$ }}\kern-.6\wd0}}
\def\ddashint{\Xint=}
\def\dashint{\Xint-}

\frontmatter

\title[Fractional thoughts]{Fractional thoughts}

\author{Nicola Garofalo\footnote{This work was supported in part by a grant ``Progetti d'Ateneo, 2013'', University of Padova}
}

\vskip 2.5in
 
\address{Dipartimento d'Ingegneria Civile e Ambientale (DICEA)\\ Universit\`a di Padova\\ Via Marzolo, 9 - 35131 Padova,  Italy}
\vskip 0.2in
\email{nicola.garofalo@unipd.it}

\maketitle

\begin{abstract}
In this note we present  some of the most basic aspects of the operator $(-\Delta)^s$ with a self-contained and purely didactic intent, and with a somewhat different slant from the existing excellent references. Given the interest that nonlocal operators have generated since the extension paper of Caffarelli and Silvestre \cite{CS07}, we feel it is appropriate offering to young researchers a quick additional guide to the subject which, we hope, will nicely complement the existing ones.
\end{abstract}

\tableofcontents

\mainmatter

\newpage

\dedicatory{\ \ \ \ \ \ \vskip 1.5in Dedicated to my family}

\newpage

\noindent ACKNOWLEDGMENT: 
I thank A. Banerjee, D. Danielli, D. M. Nhieu, A. Petrosyan, C. Pop and X. Ros-Oton for taking the time and effort to read the manuscript at different stages of its preparation and for so graciously providing me with much valuable feedback.

\newpage

\dedicatory{\ \ \ \ \ \ \vskip 1.0in \emph{``...Beyond this there is nothing but prodigies and fictions, the only inhabitants are the poets and inventors of fables; there is no credit, or certainty any farther"}
\vskip 0.2in Plutarch, Lives}

\newpage

\setcounter{page}{5}


\mainmatter



\section{Introduction}\label{S:intro}

In his visionary papers \cite{R} and \cite{R2} Marcel Riesz introduced the fractional powers of the Laplacean in Euclidean and Lorentzian space, developed the calculus of these nonlocal operators and studied the Dirichlet and Cauchy problems for respectively $(-\Delta)^s$ and $(\p_{tt} - \Delta)^s$. The introduction of \cite{R} reads:...``On peut en particulier consid\'erer certains proc\'ed\'es d'int\'egration de charact\`er elliptique, hyperbolique et parabolique respectivement. Dans tout ces proc\'ed\'es l'int\'egrale d'ordre deux joue un r\^ole particulier, elle constitue l'inverse des op\'erations qui figurent respectivement dans l'\'equation de Laplace, celle des ondes et celle de la chaleur. Nous nous sommes occup\'e en particulier des deux premiers proc\'ed\'es et nous avons l'intention de rassembler nos recherches dans un m\'emoire \'elabor\'e. En attendant, nous donnons dans le pr\'esent travail un r\'esum\'e assez d\'etaill\'e de nos r\'esultats concernant l'int\'egration elliptique et les potentiels qui y correspondant..."

 Pseudo-differential operators such as $(-\Delta)^s$, $(\p_{tt} - \Delta)^s$, $(\p_{t} - \Delta)^s$, and the very different operators $\p_{tt} + (-\Delta)^s$ and $\p_t + (-\Delta)^s$, play an important role in many branches of the applied sciences ranging from fluid dynamics, to elasticity and to quantum mechanics. For instance, a main protagonist of geophysical
fluid dynamics is the two-dimensional quasi-geostrophic equation (QGE)
\[
\begin{cases}
\theta_t + <u,\nabla \theta> = \kappa (-\Delta)^s \theta,
\\
u = \nabla^\perp \psi,\ \ \ \ - \theta = (-\Delta)^{1/2} \psi,
\end{cases}
\]
where:
\begin{itemize}
\item $\psi$ is the stream function;
\item $\theta$ is the potential temperature,
\item $u$ is the velocity.
\end{itemize}
The parameter $\kappa$ represents the viscosity, and $s\in (0,1)$. The QGE is one important instance in which the nonlocal operators $(-\Delta)^s$ and $\p_t + (-\Delta)^s$ appear, see \cite{CC2}, \cite{CV}, and the references therein. These nonlocal operators also present themselves in the convergence of nonlocal threshold dynamics approximations to front propagation, see \cite{CSo}. In \cite{FdlL} the authors study stability against collapse of a quantum mechanical system of $N$ electrons and $M$ nuclei interacting by pure Coulomb forces. For a single quantized electron attracted to a single nucleus having charge $Z$ the relevant operator at study is the Hamiltonian
\[
H = (-\Delta)^{1/2} - \frac{\alpha Z}{|x|},
\]
where $\alpha>0$ is the fine structure constant.
Another example comes from elasticity, where the famous Signorini problem has been shown to be equivalent to the obstacle problem for $(-\Delta)^{1/2}$, see  \cite{AC}, \cite{CS07}, \cite{ACS}, \cite{CSS}, \cite{GP} and \cite{PSU}. Yet another instance is the phenomenon of osmosis, whose description can be converted into an obstacle problem for the fractional heat equation $(\p_t - \Delta)^{1/2}$, see \cite{DL} and \cite{DGPT}. In the study of internal travelling solitary waves in a stable two-layer perfect fluid of infinite depth contained above a rigid horizontal bottom one, or in  soliton theory, one has the Benjamin-Ono equation
\[
(-\Delta)^{1/2} u + u - u^2 = 0
\]
on the line $\R$. A basic question is the uniqueness of solutions, see \cite{AT91}, and also the more recent works \cite{FLe13}, \cite{FLeS16} for important generalizations of the results in \cite{AT91}.

Besides these phenomena, the nonlocal operators listed above also arise prominently in other branches of mathematics, such as e.g. geometry, probability and financial mathematics. For some of these aspects we refer the reader to:
\begin{itemize} 
\item[1.] the classical volumes of E. Dynkin on Markov processes \cite{Dy};
\item[2.] the pioneering works of Silvestre \cite{Si}, and Caffarelli and Silvestre \cite{CS07};
\item[3.] the ``obstacle" book by Petrosyan, Shahgholian and Uraltseva \cite{PSU}; 
\item[4.] the hitchiker's guide by Di Nezza, Palatucci and Valdinoci \cite{DPV};
\item[5.] the lecture notes of Bucur and Valdinoci \cite{BV};
\item[6.] the survey paper \cite{dMG} by M. del Mar Gonz\'alez;
\item[7.] the variational book \cite{MRS16} by Molica Bisci, Radulescu and Servadei;
\item[8.] the survey paper \cite{DS} by Danielli and Salsa; 
\item[9.] the survey papers \cite{RO15}, \cite{RO17} and \cite{RO} by Ros-Oton;
\item[10.] the forthcoming volume edited by Kuusi and Palatucci \cite{KP};
\item[11.] the lecture notes \cite{DMV17} by Dipierro, Medina and Valdinoci;
\item[12.] the ``diffusion" lecture notes \cite{V17} by J. L. Vazquez; 
\item[13.] the recent lecture notes \cite{AV} by Abatangelo and Valdinoci.  
\end{itemize}
For an introduction to the subject of fractional differentiation and integration from the point of view of analysis the essential references are:
\begin{itemize}
\item[14.] M. Riesz' already cited original papers \cite{R} and \cite{R2};
\item[15.] E. Stein's landmark book on singular integrals \cite{St};
\item[16.] Landkov's book on potential theory \cite{La};
\item[17.] the volume on fractional differentiation by Samko, Kilbas and Marichev \cite{SKM}. 
\end{itemize}
An interesting account of the fractional calculus, its applications and historical development can be found in:
\begin{itemize}
\item[18.] the fractional book \cite{OS} by Oldham and Spanier;
\item[19.] the article \cite{Ross} by B. Ross.
\end{itemize}

Our objective in this note is presenting  some of the most basic aspects of the operator $(-\Delta)^s$ with a self-contained and purely didactic intent, and with a somewhat different slant from the above cited references which of course reflects the taste of the author. Given the interest that nonlocal operators have generated since the extension paper of Caffarelli and Silvestre \cite{CS07}, we feel it is appropriate offering to young researchers a quick additional guide to the subject which, we hope, will nicely complement the (often more advanced) existing ones. 

A list of the topics covered by this paper is provided by the table of content, but let us say something more in detail:
\begin{itemize}
\item In Section \ref{S:fl} we introduce the main pointwise definition of the nonlocal operator $(-\Delta)^s$, see \eqref{fls} below. This is the starting point of the whole note as all the material presented here is, in one way or the other, derived from it. In Proposition \ref{P:decay} we show that the definition \eqref{fls} implies a decay at infinity of the fractional Laplacean that plays an important role in its analysis.
\item Section \ref{S:mp} contains a brief discussion of the maximum principle, the Harnack inequality and the theorem of Liouville in the fractional setting. We do not make any attempt at discussing these aspects extensively, but we simply confine ourselves to make the (unfamiliar) reader acquainted of the differences with their local counterparts, and then refer to the existing sources.    
\item Section \ref{S:bi} constitutes a brief interlude on two important protagonists of classical analysis which also play a central role in this note: the Fourier transform and Bessel functions. These two classical subjects are inextricably connected. One the one hand, the Bessel functions are eigenfunctions of the Laplacean. On the other, they also appear (the curvature of the unit sphere in $\Rn$ is lurking in the shadows here) as the Fourier transform of the measure carried by the unit sphere. In this connection, and since it is a recurrent ingredient in this note, we recall the classical Fourier-Bessel integral formula due to Bochner, see Theorem \ref{T:Fourier-Bessel} below. 
\item Section \ref{S:ftb} opens with the proof of Proposition \ref{P:slapft}, which describes the action of $(-\Delta)^s$ on the Fourier transform side. This result proves an important fact: the fractional Laplacean is a \emph{pseudo-differential operator}, i.e., one of those nonlocal operators that can be written in the form
\[
Tu(x) = \int_{\Rn} e^{2\pi i<x,\xi>} p(x,\xi) \hat u(\xi) d\xi,
\]  
where the function $p(x,\xi)$, known as the \emph{symbol} of the operator, is required to belong to a certain class. A basic consequence of Proposition \ref{P:slapft} is the semigroup property in Corollary \ref{C:semi} and the ``integration by parts" Lemma \ref{L:ibp}, which shows that $(-\Delta)^s$ is a symmetric operator. We close the section with the computation in Proposition \ref{P:gns} of the normalization constant $\gamma(n,s)$ in the pointwise definition \eqref{fls}.
\item Section \ref{S:riesz} is devoted to discussing a basic question of interest in analysis and geometry which was asked by Strichartz in \cite{Str}, and which has generated a considerable amount of work. We introduce the vector-valued Riesz transform $\mathcal R = \nabla (-\Delta)^{-1/2}$, and we show that being able to answer in the affirmative such question hinges upon the $L^p$ mapping properties of $\mathcal R$. This is in turn intimately connected to the subject of  Section \ref{S:gamma} below.
\item Similarly to the classical Laplacean, $(-\Delta)^s$ preserves spherical symmetry. In Section \ref{S:flrad} we make this property more precise. Using Theorem \ref{T:Fourier-Bessel} we provide an ``explicit" formula for the fractional Laplacean of a spherically symmetric function.
\item The purpose of Section \ref{S:fs} is multifold. Our declared intent is computing the fundamental solution of $(-\Delta)^s$, i.e., proving Theorem \ref{T:fs}. This can be done in several ways. To the best of our knowledge, the approach we choose, although very classical, has not been tried before. We have introduced a regularization \eqref{Ee20} of the fundamental solution, and in Lemma \ref{L:fsreg0}  we compute its Fourier transform. In Lemma \ref{L:fsreg} below we use this result to calculate the fractional Laplacean of the regularized fundamental solution, and with such result we finally prove Theorem \ref{T:fs}. \item Using these results, in Section \ref{S:yamabe} we show that this approach leads in a natural way to the beautiful discovery of the functions \eqref{vyam}. Remarkably, such functions are solutions of the nonlocal nonlinear equation
\[
(-\Delta)^s u = u^{\frac{n+2s}{n-2s}},
\]
which generalizes to the fractional setting the celebrated \emph{Yamabe equation} from Riemannian geometry. The latter is obtained when $s=1$.
\item Section \ref{S:pk} presents in detail the central theme of the analysis of the fractional Laplacean: the extension problem of Caffarelli and Silvestre \eqref{ext2} below. We construct the Poisson kernel for the extension operator, and provide two proofs of \eqref{dn}, which characterizes $(-\Delta)^s$ as the weighted Dirichlet-to-Neumann map of the extension problem. The extension procedure is a very powerful tool which has been applied so far in many different directions, and it is hardly possible to accurately describe the impact of this paper in the field. A prominent one is the theory of \emph{free boundaries}, which was in fact the main motivation behind the work \cite{CS07} itself.  Another remarkable application has been given to geometry in the work \cite{CG11}, where the authors used the extension procedure to characterize the fractional powers of the so-called \emph{Paneitz operator}, a conformally covariant operator of order four, as the (weighted) Dirichlet-to-Neumann map on a conformally compact Einstein manifold. In the opening of Section \ref{S:pk} we also discuss briefly the beautiful 1965 paper \cite{MS} by Muchenhoupt and Stein which seems not known to the fractional community, but that deserves to be considered in connection with the extension procedure.     
\item In Section \ref{S:flso} we discuss one interesting aspect of the extension procedure which is perhaps not so well-known in the fractional community: the link between the nonlocal operator $(-\Delta)^s$ and the subelliptic operator $\Ba$, which we define in \eqref{baaa} below, that was  introduced by S. Baouendi in his 1967 Ph. D. Dissertation \cite{Ba67}. Proposition \ref{P:flp} below shows that $(-\Delta)^s$ arises as the true Dirichlet-to-Neumann map of the so-called Baouendi-Grushin operator. Furthermore, it is possible to relate in a one-to-one onto fashion solutions of the extension operator $L_a$ to those of $\Ba$. In Proposition \ref{P:rhod} we show that there is a direct link between the non-isotropic (sub-Riemannian) pseudo-balls naturally associated with $\Ba$, and the Euclidean balls which are instead the natural ones for the extension operator $L_a$. 
\item In Section \ref{S:how} we exploit this connection further to provide a proof of a fundamental property of the nonlocal operator $(-\Delta)^s$: its hypoellipticity. At first such property might appear surprising since now we are not dealing with solutions of a partial differential equation (pde). But, a moment's thought reveals that, in the end, what really makes harmonic functions infinitely smooth is the fact that they satisfy the crucial integral property \eqref{har} below, which then results into \eqref{k1}. The pde is only a vessel that takes harmonic functions into the blessed land of ``averaging". Since this aspect is shared by solutions of the nonlocal equation $(-\Delta)^s$, we should expect solutions of the latter to be infinitely smooth. The approach we take to the hypoellipticity of $(-\Delta)^s$ is ``elementary" and runs much along the lines of the Caccioppoli-Cimmino-Weyl lemma for the classical Laplacean, but we make use of the extension procedure. With the intent of advertising the link, discussed in Section \ref{S:flso}, with the theory of subelliptic equations, we start from Proposition \ref{P:gar}, which is a representation formula involving the Baouendi operator $\Ba$. We use it to establish a corresponding result for the extension operator $L_a$, see Proposition \ref{P:gar3} below. With such result we prove Theorem \ref{T:abg2} which provides an interesting mean-value formula for solutions to $(-\Delta)^s u = 0$. Finally, in Theorem \ref{T:fh} we establish our ``elementary" version of the Caccioppoli-Cimmino-Weyl lemma for $(-\Delta)^s$. We do not discuss at all the real-analytic hypoellipticity of $(-\Delta)^s$.   
\item Section \ref{S:reg} is devoted to the question of the regularity at the boundary for solutions of the Dirichlet problem \eqref{dpom}. Unlike the interior regularity, here the situation departs drastically from the local case, in the sense that there exist real-analytic domains and real-analytic ``boundary values"  for which the solution to \eqref{dpom} is not better then H\"older continuous at the boundary. One notable example of this negative phenomenon is the \emph{torsion function} for the ball for $(-\Delta)^s$, i.e., the solution to the Dirichlet problem $(-\Delta)^s u = 1$ in $B(0,R)$, $u = 0$ in $\Rn\setminus B(0,R)$. The relevance of such function, which we construct in Proposition \ref{P:tfs}, is multi-faceted. Remaining within the framework of the subject of interest of this section, the torsion function shows that standard Schauder theory fails for $(-\Delta)^s$, or at least such theory needs to be suitably reinterpreted. This negative phenomenon is akin, and not by chance, to the failure of Schauder theory which occurs at the so-called characteristic points in the theory of subelliptic equations. For the fractional Laplacean the correct boundary regularity is provided by Theorems \ref{T:xavi} and \ref{T:grubb} below: the former,  due to Ros-Oton and Serra, states that in the Dirichlet problem with zero ``boundary data" in a $C^{1,1}$ domain $\Om$, the function $\frac{u}{\operatorname{dist}(\cdot,\p \Om)^s}$ is H\"older continuous up to the boundary. The latter, due to Grubb, states that in a $C^\infty$ domain this same function is in fact $C^\infty$ up to the boundary.
\item In partial differential equations, the most fundamental property of interest is the so-called \emph{strong unique continuation property}. It states that if a solution to a certain differential operator $P(x,\p_x)$ vanishes to infinite order at a point of a connected open set, then it must vanish identically. This property is true when $P(x,\p_x) = -\Delta$, but it is shared by large classes of second order partial differential equations, even with very rough coefficients. Section \ref{S:almgren} is devoted to establishing the strong unique continuation property of the fractional Laplacean, see Theorem \ref{T:sucpfl} below. We prove such result using monotonicity formulas of Almgren type. We resort again to the extension procedure, and use the monotonicity formula in Theorem \ref{T:almgrenEO} from \cite{CS07}. The difficulty in using the extension procedure is that the information that a solution of $(-\Delta)^s u = 0$ vanishes to infinite order at a point does not transfer to the solution of the extension problem. One needs to further implement a delicate blowup analysis of a special family of rescalings first introduced in \cite{ACS} in the study of the Signorini problem.      
\item In Section \ref{S:smean} we discuss the nonlocal Poisson kernel for the ball and one of its direct consequences, the mean-value formula for $(-\Delta)^s$. These tools were introduced by M. Riesz in \cite{R} and are by now part of the fractional folklore. We consider the nonlocal mean-value operator $\mathscr A^{(s)}_r u(x)$ defined by \eqref{ar}, \eqref{smvo}, and in Proposition \ref{P:vague} we show that 
\[
\underset{s\to 1}{\lim} \mathscr A^{(s)}_r u(x) = \mathscr M_r u(x),
\]
where $\mathscr M_r u(x)$ is the spherical mean-value operator of the classical potential theory, see \eqref{MA0} below.
In Proposition \ref{P:bps} we show that the nonlocal mean-value operator can be used to provide yet another expression of the fractional Laplacean, much like the Blaschke-Privalov Laplacean is used in classical potential theory to define the Laplacean on nonsmooth functions. In Corollary \ref{C:nlkoebe} we establish a nonlocal analogue of the classical theorem of K\"oebe.
\item Section \ref{S:heat} departs from the previous ones in that we start discussing the heat flow associated with $(-\Delta)^s$. There is of course more than one nonlocal heat equation, but here we focus on $\p_t u + (-\Delta)^s u = 0$. We introduce the fractional heat semigroup \eqref{hsg} below, and we spend most part of the section proving a basic property of the fractional heat kernel $G_s(x,t)$, namely its positivity, see Propositions \ref{P:posG} and \ref{P:sp}.
\item In Section \ref{S:sub} we use the principle of subordination introduced by Bochner to establish a pointwise representation of $(-\Delta)^s$ in terms of the classical heat semigroup, see Theorem \ref{T:flheat} below. 
\item Section \ref{S:moresub}  contains yet another important instance of subordination. In \eqref{ftau} below we introduce Bochner's subordination function and in Theorem \ref{T:sub} show the important fact that the nonlocal heat semigroup is obtained through subordination with the standard heat semigroup. 
\item The chain rule is one of the most basic and useful tools in the theory of partial differential equations. In Section \ref{S:cr} we discuss a simple, yet quite remarkable nonlocal analogue of the chain rule which was first found in \cite{CC}. An important consequence of it is that when $u$ is a solution to $(-\Delta)^s u =0$,  then $u^2$ is also a subsolution. More in general, if $u$ is a nonnegative subsolution, then $u^p$ is a nonnegative subsolution for every $p>1$. We recall that, in the local case, such property is at the heart, for instance, of Moser's proof of the Harnack inequality for divergence form equations with bounded measurable coefficients. 
\item Over the recent years there has been an explosion of activity surrounding the so-called Gamma calculus of Bakry-Emery and various powerful generalizations of the latter. Section \ref{S:gamma} is devoted to providing the reader with a bird's eye-view of the basics of such calculus. Our primary motivation is proposing the development of a nonlocal Gamma calculus. In this perspective the reader might consider this section just as a glimpse into a possibly rich theory to come. In Definition \ref{D:nlcdc} we introduce the notion of nonlocal \emph{carr\'e du champ}. Such object defines a Dirichlet form whose associated energy $\mathscr E_{(s)}(u)$ is given in Definition \ref{D:energy}. Proposition \ref{P:el} shows that $(-\Delta)^s$ has a variational nature, since it arises as the Euler-Lagrange equation of the nonlocal energy $\mathscr E_{(s)}(u)$. The experienced fractional reader will immediately recognize familiar objects here. The section ends with a discussion of the famous Bakry-Emery \emph{curvature-dimension inequality}, at the heart of which there is the celebrated Bochner identity, and with a challenging open question. 
\item In Section \ref{S:LY} we continue the discussion from the previous one. Our intent is to provide the reader with an elementary motivation for undertaking a new bigger effort. Namely, understanding the beautiful Li-Yau theory. We introduce a special case of the celebrated Li-Yau inequality and in Theorem \ref{T:LYH} we provide an elementary proof of such inequality for the classical  heat semigroup in flat $\Rn$. We use this result to give a simple, yet elegant proof of the well-known scale invariant Harnack inequality for the standard heat equation independently proved by Pini and Hadamard in the 50's. We close the section with two equivalent interesting  conjectures regarding the nonlocal heat semigroup.   
\item Section \ref{S:bessel} is devoted to discussing a Li-Yau inequality for the Bessel semigroup \eqref{Ba} on the half-line. The Bessel process $\mathscr B_a$ is ubiquitous in the fractional world, especially in view of his role in the extension procedure. Since this topic is perhaps more frequented by workers in probability than analysts and geometers (with the exception of people in harmonic analysis), we provide a purely analytical construction of the fundamental solution with Neumann boundary conditions of the heat semigroup associated with $\mathscr B_a$, see Proposition \ref{P:fsbessel} below. We close the section with Proposition \ref{P:lybessel}, in which we show that the heat semigroup associated with the Bessel process satisfies an inequality of Li-Yau type. 
\item We end this note with the very brief Section \ref{S:fpl} in which we discuss a fractional \emph{nonlinear} operator which constitutes the nonlocal counterpart of the well-known $p$-\emph{Laplacean} defined by $-\Delta_p u = \operatorname{div}(|\nabla u|^{p-2} \nabla u)$. Since its introduction in \cite{AMRT} and independently in \cite{IN}, the nonlocal $p$-Laplacean $(-\Delta_p)^s$ has generated a great deal of interest in the fractional community and thus we could not close without a brief mention of the fundamental open question in the area: the optimal interior regularity of its variational solutions.  
\end{itemize}    
 
 A notable omission in this note is the beautiful developing theory of nonlocal minimal surfaces. For this we refer the reader to the seminal works \cite{CRS}, \cite{SV}, \cite{CSV}, \cite{CDS}, \cite{DSV}, \cite{DSV2}, and the references therein, as well as to the survey paper \cite{BV} which contains a nice introduction to the subject.

\medskip

The reader understands that, for obvious considerations of space, it is not possible to formally introduce every definition or tool used in this paper. Thus, for instance, we will not discuss the Schwartz space $\mathscr S(\Rn)$ of rapidly decreasing functions in $\Rn$, and its topological dual, the space $\mathscr S'(\Rn)$ of tempered distributions. Similarly, we will not explicitly introduce the topology of the spaces $C^\infty(\Rn)$, or $C^\infty_0(\Rn)$, and their duals, the spaces $\mathscr E'(\Rn)$ of compactly supported distributions, and the larger space of all distributions $\mathscr D'(\Rn)$ on $\Rn$. Nor we will discuss in detail the Fourier transform in $\Rn$. For these topics there exist several excellent classical books, such as for instance: \cite{BC}, \cite{Bo}, \cite{GeS}, \cite{Sc}, \cite{T}, \cite{St}, \cite{SW} and \cite{Y}. One additional source that the reader is encouraged to peruse is the monograph \cite{La}, in which the author provides an extensive discussion of the potential theoretic aspects of the nonlocal Laplacean, based on M. Riesz' paper \cite{R}.

\medskip

Finally, the present note has been written within the constraints imposed by timeliness. Many important and/or relevant references have been left out simply because it has been impossible, within the short amount of time available, to consult the ample existing literature on nonlocal equations. The author sincerely apologizes with all those people whose work is not properly acknowledged here.

\section{The fractional Laplacean}\label{S:fl}

In this section we introduce the protagonist of this note, M. Riesz' fractional Laplacean $(-\Delta)^s$, with $0<s<1$. At the onset we seek to define the action of such nonlocal  operator on a suitable function in the pointwise sense. With this objective in mind, it will be convenient to work with the space $\mathscr S(\Rn)$ of L. Schwartz' rapidly decreasing functions (whose dual $\mathscr S'(\Rn)$ is the space of tempered distributions), although larger classes can be allowed, see Remark \ref{R:more} and Proposition \ref{P:silv} below.
We recall that $\mathscr S(\Rn)$ is the space $C^\infty(\Rn)$ endowed with the metric topology
\[
d(f,g) = \sum_{p=0}^\infty 2^{-p} \frac{||f-g||_p}{1+||f-g||_p},
\]
 generated by the countable family of norms
\begin{equation}\label{cfn}
||f||_p = \underset{|\alpha|\le p}{\sup}\ \underset{x\in \Rn}{\sup}  (1+|x|^2)^{\frac p2} |\p^\alpha f(x)|,\ \ \ \ \ p\in \mathbb N \cup\{0\}.
\end{equation}
As it is customary, if $\alpha = (\alpha_1,...,\alpha_n)$, then $|\alpha| = \alpha_1+...+\alpha_n$, and we have indicated with $\p^\alpha$ the partial derivative $\frac{\p^{|\alpha|}}{\p x_1^{\alpha_1}...\p x_n^{\alpha_n}}$.
 
Our initial observation is the following simple calculus lemma which could be used to provide a probabilistic interpretation of the classical Laplacean on the real line.

\begin{lemma}\label{L:taylor1}
Let $f\in C^2(a,b)$, then for every $x\in (a,b)$ one has
\[
- f''(x) = \underset{y\to 0}{\lim} \frac{2 f(x) - f(x+y) - f(x-y)}{y^2}.
\]
\end{lemma}

The expression in the right-hand side in the equation in Lemma \ref{L:taylor1} is known as the symmetric difference quotient of order two. If we introduce the ``spherical" surface and ``solid" averaging operators 
\[
\mathscr M_y f(x) = \frac{f(x+y) + f(x-y)}{2},\ \ \ \ \ \mathscr A_y f(x) = \frac{1}{2y} \int_{x-y}^{x+y} f(t) dt,
\]
then we can reformulate the conclusion in Lemma \ref{L:taylor1} as follows:
\[
- f''(x) = 2 \underset{y\to 0}{\lim} \frac{f(x) - \mathscr M_y f(x)}{y^2} = 6  \underset{y\to 0}{\lim} \frac{f(x) - \mathscr A_y f(x)}{y^2},
\]
where it is easily seen that the second equality follows from the first one and L'Hopital's rule. The result that follows generalizes this observation to $n\ge 2$.

\begin{prop}\label{P:genlap}
Let $\Om\subset \Rn$ be an open set. For any $f\in C^2(\Om)$ and $x\in \Om$ we have 
\begin{equation}\label{BP0}
- \Delta f (x) = 2n\ \underset{r\to 0}{\lim} \frac{f(x) - \mathscr M_r f(x)}{r^2} = 2(n+2) \underset{r\to 0}{\lim} \frac{f(x) - \mathscr A_r f(x)}{r^2},
\end{equation}
where $\Delta f = \sum_{k=1}^n \frac{\p^2 f}{\p x_k^2}$ is the operator of Laplace.
\end{prop}
In the equation \eqref{BP0} we have indicated with
\begin{equation}\label{MA0}
\mathscr M_r u(x) = \frac{1}{\sigma_{n-1} r^{n-1}} \int_{S(x,r)} u(y) d\sigma(y),\ \ \ \ \ \ \ \ \ \  \mathscr A_r u(x) =  \frac{1}{\omega_{n} r^{n}} \int_{B(x,r)} u(y) dy, 
\end{equation}
the spherical surface and solid mean-value operators. Here, $B(x,r) = \{y\in \Rn\mid |y-x| < r\}$, $S(x,r) = \p B(x,r)$, $d\sigma$ is the $(n-1)$-Lebesgue measure on $S(x,r)$, and the numbers $\sigma_{n-1}$ and $\omega_n$ respectively represent the measure of the unit sphere and that of the unit ball in $\Rn$, see \eqref{sn1} below.
Either one of the limits in the right-hand side of \eqref{BP0} is known as the \emph{Blaschke-Privalov Laplacean}, and its relevance in potential theory is that, unlike the standard Laplacean, one can define such operator on functions which are not smooth, see for instance \cite{He} and \cite{DP}. For the probabilistic interpretation of the Laplacean one should see \cite{Dy}. 

Before proceeding, and in preparation for the central definition of this section, let us observe that using \eqref{sn1} 
it is easy to recognize that we can write the second identity in \eqref{BP0} in the more suggestive  
fashion:
\begin{equation}\label{MA00}
- \Delta u(x) = \frac{(n+2)\G(\frac{n}2 +1)}{\pi^{\frac{n}2}} \underset{r\to 0^+}{\lim} \int_{\Rn} [2u(x) - u(x+y) - u(x-y)] \frac{1}{r^{n+2}}\mathbf 1_{B(0,r)}(y) dy,
\end{equation}
where we have denoted by $\mathbf 1_E$ the indicator function of a set $E\subset \Rn$.

In the applied sciences it is of great importance to be able to consider fractional derivatives of functions. There exist many different definitions of such operations, see \cite{OS}, \cite{SKM}, and the recent paper \cite{BMST}, but perhaps the most prominent one is based on the notion of (Marcel) Riesz' potential of a function. To motivate such operation let us assume that $n\ge 3$, and recall that in mathematical physics the \emph{Newtonian potential} of a function $f\in \mathcal
S(\Rn)$ is given by
\[
I_2(f)(x) = \frac{1}{4 \pi^{\frac{n}{2}}}
\Gamma\left(\frac{n-2}{2}\right) \int_{\Rn} \frac{f(y)}{|x-y|^{n-2}}
dy,
\]
where we have denoted by $\G(z)$ Euler's gamma function (for its definition and basic properties see Section \ref{S:bi} below). Now, using \eqref{sn1} and the properties of the gamma function, one recognizes that the convolution kernel  $\frac{1}{4 \pi^{\frac{n}{2}}}
\Gamma\left(\frac{n-2}{2}\right) \frac{1}{|x|^{n-2}}$ in the definition of $I_2(f)$ is just the fundamental solution 
\[
E(x) = \frac{1}{(n-2)\sigma_{n-1}} \frac{1}{|x|^{n-2}}
\]
of $-\Delta$. With this observation in mind, we recall the well-known identity of Gauss-Green that says that for any $f\in \mathcal
S(\Rn)$ one has
\[
I_2(-\Delta f) = f.
\]
Recall M. Riesz' words in the opening of this note: ``...l'int\'egrale d'ordre deux joue un r\^ole particulier, elle constitue l'inverse des op\'erations qui figurent respectivement dans l'\'equation de Laplace..." In other words, the Newtonian potential is the inverse of $-\Delta$, i.e., $I_2 = (-\Delta)^{-1}$. This important observation leads to the introduction of M. Riesz' generalization of the Newtonian potential.

\begin{definition}[Riesz' potentials]\label{D:riesz} For any $n\in \mathbb N$, let $0<\alpha<n$. The \emph{Riesz potential} of order $\alpha$ is the operator whose action on a function $f\in \mathscr S(\Rn)$ is given by
\[
I_\alpha(f)(x) = \frac{\Gamma\left(\frac{n-\alpha}{2}\right)}{\pi^{\frac{n}{2}} 2^\alpha
\Gamma\left(\frac{\alpha}{2}\right)} \int_{\Rn}
\frac{f(y)}{|x-y|^{n-\alpha}} dy.
\]
\end{definition}
It is not difficult to prove that $I_\alpha(f)\in C^\infty(\Rn)$ for any $f\in \mathcal
S(\Rn)$. Concerning the definition of $I_\alpha$, we note that the normalization constant in it matches that of $I_2$ when $\alpha = 2$. The important reason behind it is that such constant is chosen to guarantee the 
validity of the following crucial result, a kind of fractional fundamental theorem of calculus, stating that for any $f\in \mathscr S(\Rn)$ one has in $\mathscr S'(\Rn)$
\begin{equation}\label{alphafi}
I_\alpha (-\Delta)^{\alpha/2}
f = (-\Delta)^{\alpha/2} I_\alpha f = f.
\end{equation}
Of course \eqref{alphafi} makes no sense unless we say what we mean by the fractional operator $(-\Delta)^{\alpha/2}$. The most natural way to introduce it (suggested in fact by the spectral theorem) is by defining the action of $(-\Delta)^{\alpha/2}$  on the Fourier
transform side by the equation
\begin{equation}\label{ftfl}
\F((-\Delta)^{\alpha/2}u) = (2\pi |\cdot|)^{\alpha} \F(u),\ \ \
u\in \mathscr S'(\Rn).
\end{equation}
The equation \eqref{alphafi} shows that $I_\alpha$ inverts the fractional powers of the Laplacean, i.e., 
\begin{equation}\label{fi}
I_\alpha = (-\Delta)^{-\alpha/2},\ \ \ \ 0<\alpha <n.
\end{equation}
For this reason $I_\alpha$ is also called the \emph{fractional integration operator} of order $\alpha$, see \cite{R}, but also \cite{St}, \cite{La}, \cite{SKM}. For a nice account of M. Riesz' work one should read the commemorative note \cite{Gar}. An interesting historical overview of the development of fractional calculus is provided by the book \cite{OS} and the article \cite{Ross}.

Since our focus in this note is the fractional Laplacean $(-\Delta)^s$ in the range $0<s<1$, we will henceforth let $s = \alpha/2$ in the above formulas, or equivalently $\alpha = 2s$. Although we have formally introduced such operator in the equation \eqref{ftfl} above, such definition has a major drawback: it is not easy to understand a given function (or a distribution) by prescribing its Fourier transform. It is for this reason that we begin our story by introducing a different pointwise definition of the fractional Laplacean that is more directly connected to the symmetric difference quotient of order two in the opening calculus Lemma \ref{L:taylor1}, and with \eqref{MA00}, and thus has the advantage of underscoring the probabilistic interpretation of the operator $(-\Delta)^s$ as a symmetric random process with jumps, see \cite{MO69}, \cite{Si}, \cite{RO15}  and \cite{BV}. Later in Proposition \ref{P:slapft} we will reconcile the two definitions.

\begin{definition}\label{D:fl}
Let $0<s<1$. The \emph{fractional Laplacean} of a function $u\in \mathscr S(\Rn)$ is the nonlocal operator in $\Rn$ defined by the expression
\begin{equation}\label{fls}
(-\Delta)^s u(x) = \frac{\gamma(n,s)}{2} \int_{\Rn} \frac{2 u(x) - u(x+y) - u(x-y)}{|y|^{n+2s}} dy,
\end{equation}
where $\gamma(n,s)>0$ is a suitable normalization constant that is given implicitly in \eqref{gns}, and explicitly in Proposition \ref{P:gns} below.
\end{definition}

Before proceeding we remark that when dealing with the nonlocal operator $(-\Delta)^s$ one often needs to distinguish between the three possible cases:
\begin{itemize}
\item  $0< s <\frac 12$;
\item $s = \frac 12$;
\item $\frac 12 < s<1$.
\end{itemize}

Since as $s\to 1^-$ the fractional Laplacean tends (at least, formally right now) to $-\Delta$, one might surmise that in the regime $\frac 12 < s<1$ the operator $(-\Delta)^s$ should display properties closer to those of the classical Laplacean, whereas since $(-\Delta)^s\to I$ as $s\to 0^+$, the stronger discrepancies might present themselves in the range $0< s <\frac 12$.

Having said this, it will be good for the reader who is for the first time confronted with definition \eqref{fls} above to have in mind the following quote from p. 51 in \cite{La}:...``\emph{In the theory of M. Riesz kernels, the role of the Laplace operator, which has a local character, is taken...by a non-local integral operator...This circumstance often substantially complicates the theory...}"   

It is obvious that \eqref{fls} defines a linear operator since for any $u, v\in \mathscr S(\Rn)$ and $c\in \R$ one has
\[
(-\Delta)^s(u + v) = (-\Delta)^s u + (-\Delta)^s v,\ \ \ \ \ \ \ \ (-\Delta)^s (cu) = c (-\Delta)^s u.
\]

It is also important to observe that the integral in the right-hand side of \eqref{fls} is convergent. To see this, it suffices to write
\begin{align*}
& \int_{\Rn} \frac{2 u(x) - u(x+y) - u(x-y)}{|y|^{n+2s}} dy = \int_{|y|\le 1} \frac{2 u(x) - u(x+y) - u(x-y)}{|y|^{n+2s}} dy
\\
& + \int_{|y|> 1} \frac{2 u(x) - u(x+y) - u(x-y)}{|y|^{n+2s}} dy.
\end{align*}
Taylor's formula for $C^2$ functions gives for $|y|\le 1$
\[
2 u(x) - u(x+y) - u(x-y) = - <\nabla^2 u(x) y,y> + o(|y|^2),
\]
where we have indicated with $\nabla^2 u$ the Hessian matrix of $u$.
Therefore, 
\[
\left|\int_{|y|\le 1} \frac{2 u(x) - u(x+y) - u(x-y)}{|y|^{n+2s}} dy\right| \le C \int_{|y|\le 1} \frac{dy}{|y|^{n-2(1-s)}} < \infty,
\]
since $0<s<1$. On the other hand, keeping in mind that $u\in \mathscr S(\Rn)$ implies in particular that $u\in L^\infty(\Rn)$, we have
\[
\left|\int_{|y|> 1} \frac{2 u(x) - u(x+y) - u(x-y)}{|y|^{n+2s}} dy\right| \le 4 ||u||_{L^\infty(\Rn)} \int_{|y|> 1} \frac{dy}{|y|^{n+2s}} < \infty.
\]

We have thus seen that for every $u\in \mathscr S(\Rn)$ definition \eqref{fls} provides a well-defined function on $\Rn$.

\begin{remark}\label{R:more}
The reader should note that we have in fact just proved that $(-\Delta)^s u(x)$ is well-defined for every $u\in C^2(\Rn)\cap L^\infty(\Rn)$. For instance, one seemingly trivial, yet useful, situation to which this remark applies is when $u \equiv c\in \R$, for which we have 
\[
(-\Delta)^s c \equiv 0.
\]
Note that such $u$ is not in $\mathscr S(\Rn)$, unless $c = 0$, but of course for such $u$ we have $u\in C^2(\Rn)\cap L^\infty(\Rn)$. 
\end{remark}

Two basic operations in analysis are the Euclidean translations and dilations 
\[
\tau_h f(x) = f(x+h),\ \ \ \  h\in \Rn,\ \ \ \ \ \ \ \  \delta_\la f(x) = f(\la x),\ \ \ \  \la >0.
\]
The next result clarifies the interplay of $(-\Delta)^s$ with them. Its simple proof based on \eqref{fls} is left to the reader.

\begin{lemma}\label{L:diltr}
For every function $u\in \mathscr S(\Rn)$ we have for every $h\in \Rn$
\begin{equation}\label{ti}
(-\Delta)^s (\tau_h u) = \tau_h((-\Delta)^s u),
\end{equation}
and every $\la>0$
\begin{equation}\label{di}
(-\Delta)^s (\delta_\la u) = \la^{2s} \delta_\la((-\Delta)^s u).
\end{equation}
\end{lemma}  
\noindent We note explicitly that the equation \eqref{di} says in particular that $(-\Delta)^s$ is a homogeneous operator of order $2s$. Since obviously $0<2s<2$, one may surmise that this ``decreased" homogeneity is bound to create problems near the boundary in the Dirichlet problem. This aspect will be discussed more precisely in Section \ref{S:reg} below. 

A fundamental property of the Laplacean $\Delta$ is its invariance with respect to the action of the orthogonal group $\mathbb O(n)$ on $\Rn$. This means that if $u$ is a function in $\Rn$, then for every $T\in \mathbb O(n)$ one has $\Delta(u\circ T) = \Delta u \circ T$. The following lemma shows that $(-\Delta)^s$ enjoys the same property.

\begin{lemma}\label{L:sym}
Let $u(x) = f(|x|)$ be a function with spherical symmetry in $C^2(\Rn)\cap L^\infty(\Rn)$. Then, also $(-\Delta)^s u$ has spherical symmetry.
\end{lemma}

\begin{proof}
This follows in a simple way from \eqref{fls}. In order to prove that $(-\Delta)^s u$ is spherically symmetric we need to show that for every $T\in \mathbb O(n)$ and every $x\in \Rn$ one has
\[
(-\Delta)^s u(Tx) = (-\Delta)^s u(x).
\]
We have
\begin{align*}
(-\Delta)^s u(Tx) & = \frac{\gamma(n,s)}{2} \int_{\Rn} \frac{2 f(|Tx|) - f(|Tx+y|) - f(|Tx-y|)}{|y|^{n+2s}} dy
\\
& = \frac{\gamma(n,s)}{2} \int_{\Rn} \frac{2 f(|x|) - f(|x+T^t y|) - f(|x-T^t y|)}{|y|^{n+2s}} dy.
\end{align*}
If we make the change of variable $z = T^t y$, we conclude
\begin{align*}
(-\Delta)^s u(Tx) & = \frac{\gamma(n,s)}{2} \int_{\Rn} \frac{2 f(|x|) - f(|x+z|) - f(|x-z|)}{|Tz|^{n+2s}} dy
\\
& = \frac{\gamma(n,s)}{2} \int_{\Rn} \frac{2 f(|x|) - f(|x+z|) - f(|x-z|)}{|z|^{n+2s}} dy = (-\Delta)^s u(x),
\end{align*}
and we are done.

\end{proof} 

Before proceeding we note the following alternative expression for $(-\Delta)^s$ that is at times quite useful in the computations.

\begin{prop}\label{P:equivalent}
For any $u\in \mathscr S(\Rn)$ one has
\begin{equation}\label{fl2}
(-\Delta)^s u(x) = \gamma(n,s) \operatorname{PV} \int_{\Rn} \frac{u(x) - u(y)}{|x-y|^{n+2s}} dy,
\end{equation}
where now the integral is taken according to Cauchy's principal value sense
\[
\operatorname{PV} \int_{\Rn} \frac{u(x) - u(y)}{|x-y|^{n+2s}} dy = \underset{\e\to 0^+}{\lim} \int_{|y-x|>\e} \frac{u(x) - u(y)}{|x-y|^{n+2s}} dy.
\]
\end{prop}

\begin{proof}
 The expression \eqref{fl2} follows directly from \eqref{fls} above as follows
\begin{align*}
& \frac 12 \int_{\Rn} \frac{2u(x) - u(x+y) - u(x-y)}{|y|^{n+2s}} dy =  \frac 12 \underset{\e\to 0}{\lim} \int_{|y|>\e} \frac{2u(x) - u(x+y) - u(x-y)}{|y|^{n+2s}} dy
\\
& =  \frac 12 \underset{\e\to 0}{\lim} \int_{|y|>\e} \frac{u(x) - u(x+y)}{|y|^{n+2s}} dy + \frac 12 \underset{\e\to 0}{\lim} \int_{|y|>\e} \frac{u(x) - u(x-y)}{|y|^{n+2s}} dy
\\
& = \underset{\e\to 0}{\lim} \int_{|x-y|>\e} \frac{u(x) - u(y)}{|x-y|^{n+2s}} dy.
\end{align*} 
However, it is now necessary to take the principal value of the integral since we have eliminated the cancellation of the linear terms in the symmetric difference of order two, and $u(x) - u(y)$ is only $O(|x-y|)$. Thus, the smoothness of $u$ no longer guarantees the local integrability, unless we are in the regime $0<s<1/2$.  

\end{proof}

One can see from \eqref{fls3} in Proposition \ref{P:slapft} below that for $u\in \mathscr S(\Rn)$ it is not true in general that $(-\Delta)^s u\in \mathscr S(\Rn)$. However, one can verify that $(-\Delta)^s u\in C^\infty(\Rn)$. But $(-\Delta)^s u$ is not only smooth, it also suitably decays at infinity according to the following result.

\begin{prop}\label{P:decay}
Let  $u\in \mathscr S(\Rn)$. Then, for every $x\in \Rn$ with $|x|>1$, we have
\[
|(-\Delta)^s u(x)| \le C_{u,n,s} |x|^{-(n+2s)},
\]
where with $||u||_p$ as in \eqref{cfn}, we have let
\[
C_{u,n,s} = C_{n,s}\left(||u||_{n+2} + ||u||_n + ||u||_{L^1(\Rn)}\right).
\]
\end{prop}

\begin{proof}
 
To see this we write 
\begin{align*}
(\Delta)^s u(x) & = \frac{\gamma(n,s)}{2} \bigg(\int_{|y|<\frac{|x|}{2}} \frac{2 u(x) - u(x+y) - u(x-y)}{|y|^{n+2s}} dy 
\\
& + \int_{|y| \ge \frac{|x|}{2}} \frac{2 u(x) - u(x+y) - u(x-y)}{|y|^{n+2s}} dy\bigg).
\end{align*}  
Taylor's formula gives
\[
2 u(x) - u(x+y) - u(x-y) = - \frac 12 <\nabla^2 u(y^\star) y,y> - \frac 12 <\nabla^2 u(y^{\star \star}) y,y>,
\]
where $y^\star = x + t^\star y$, $y^{\star \star} = x + t^{\star \star} y$, for $t^\star, t^{\star \star}\in [0,1]$. We now observe that on the set where $|y|<\frac{|x|}{2}$ we have by the triangle inequality
\begin{equation}\label{unastar}
|x|  < 2 |y^\star|,\ \ \ \ \ \ \ \ \ \ \ |x| < 2 |y^{\star\star}|.
\end{equation}
Using \eqref{unastar} and the definition \eqref{cfn} of the norm $||u||_{n+2}$ in $\mathscr S(\Rn)$, we find
\begin{align*}
& \left|\int_{|y|<\frac{|x|}{2}} \frac{2 u(x) - u(x+y) - u(x-y)}{|y|^{n+2s}} dy\right| \le \frac 12 \int_{|y|<\frac{|x|}{2}} \frac{|\nabla^2 u(y^\star)| + |\nabla^2 u(y^{\star \star})|}{|y|^{n+2s}} |y|^2 dy 
\\
& \le C  ||u||_{n+2} \left(\int_{|y|<\frac{|x|}{2}} \frac{|y|^2}{(1+|y^\star|^2)^{\frac{n+2}{2}} |y|^{n+2s}} dy + \int_{|y|<\frac{|x|}{2}} \frac{|y|^2}{(1+|y^{\star\star}|^2)^{\frac{n+2}{2}} |y|^{n+2s}}  dy\right)
\\
& \le C |x|^{-n-2} ||u||_{n+2} \int_{|y|<\frac{|x|}{2}}  \frac{dy}{|y|^{n+2s-2}}   = C |x|^{-n-2} ||u||_{n+2} |x|^{2-2s} = C  ||u||_{n+2} |x|^{-(n+2s)},
\end{align*}   
where $C = C_{n,s}>0$. 

Next, we estimate
\begin{align*}
& \left|\int_{|y| \ge \frac{|x|}{2}} \frac{2 u(x) - u(x+y) - u(x-y)}{|y|^{n+2s}} dy\right| \le 2  \int_{|y| \ge \frac{|x|}{2}}  \frac{|u(x+y) - u(x)|}{|y|^{n+2s}} dy 
\\
& \le 2  \int_{|y| \ge \frac{|x|}{2}}  \frac{|u(x+y)| + |u(x)|}{|y|^{n+2s}} dy
\end{align*}
We have
\begin{align*}
& \int_{|y| \ge \frac{|x|}{2}}  \frac{|u(x)|}{|y|^{n+2s}} dy \le \underset{x\in\Rn}{\sup} \left((1+|x|^2)^{\frac n2} |u(x)|\right)  \int_{|y| \ge \frac{|x|}{2}}  \frac{dy}{(1+|x|^2)^{\frac n2} |y|^{n+2s}}
\\
& \le \underset{x\in\Rn}{\sup} \left((1+|x|^2)^{\frac n2} |u(x)|\right) |x|^{-n} \int_{|y| \ge \frac{|x|}{2}}  \frac{dy}{|y|^{n+2s}} = \frac{C ||u||_n}{|x|^{n+2s}},
\end{align*}
where  $C = C_{n,s}>0$. Finally, we have trivially
\begin{align*}
& \int_{|y| \ge \frac{|x|}{2}}  \frac{|u(x+y)|}{|y|^{n+2s}} dy \le \frac{2^{n+2s}}{|x|^{n+2s}} \int_{|y| \ge \frac{|x|}{2}} |u(x+y)| dy = \frac{2^{n+2s}||u||_{L^1(\Rn)}}{|x|^{n+2s}}.
\end{align*}
This completes the proof.

\end{proof}

Proposition \ref{P:decay} has the following nontrivial consequence.

\begin{corollary}\label{C:decay}
Let $u\in \mathscr S(\Rn)$. Then, $(-\Delta)^s u \in C^\infty(\Rn)\cap L^1(\Rn)$.
\end{corollary}

The estimate in Proposition \ref{P:decay} can be written
\[
- C_{u,n,s} |x|^{-(n+2s)} \le -(-\Delta)^s u(x) \le C_{u,n,s} |x|^{-(n+2s)}.
\]
Let us notice that on a nonnegative  bump function the estimate from below can be made stronger, a fact that reflects the nonlocal character of $(-\Delta)^s$. Suppose for instance that $u\in C^\infty_0(\Rn)$, with $0\le u \le 1$, $u \equiv 1$ on $B(0,1)$ and supp$\ u\subset \overline B(0,2)$. Then, for $x\in \Rn\setminus B(0,3)$ one has from \eqref{fls}
\begin{align*}
-(-\Delta)^s u(x) & = \frac{\gamma(n,s)}{2} \int_{\Rn} \frac{u(x+y) + u(x-y)}{|y|^{n+2s}} dy 
\\
& \ge \gamma(n,s) \int_{B(0,1)} \frac{dz}{|x-z|^{n+2s}} dz.
\end{align*}
Since $|x-z|\ge 2$, for $|z|\le 1$ we infer $|x| \ge |x-z| - |z| \ge |x-z| -1\ge |x-z|/2$. This gives for some $C(n,s)>0$
\begin{equation}\label{nln}
-(-\Delta)^s u(x) \ge C(n,s) |x|^{-(n+2s)} > 0,
\end{equation}
which shows that $(-\Delta)^s u$ needs not to vanish even far away from the support of $u$. This is clearly impossible for local operators $P(x,\p_x)$, for which one has the obvious property supp$\ P(x,\p_x)u\subset$ supp$\ u$. 

In the next result we provide a useful expression of $(-\Delta)^s u(x)$ in terms of an integral involving the spherical mean-value operator $\mathscr M_r u(x)$. We will use this result later in the proof of Theorem \ref{T:flheat}.

\begin{prop}\label{P:flave}
Let $u \in \mathscr S(\Rn)$. For every $0<s<1$ one has
\begin{equation}\label{saslap}
(-\Delta)^s u(x) = - \sigma_{n-1} \gamma(n,s) \int_{0}^\infty r^{-1-2s} \big[\mathscr M_r u(x) - u(x)] dr,
\end{equation}
where $\gamma(n,s)$ is the constant in \eqref{fls} (and whose value is given in \eqref{gnsfin} in Proposition \ref{P:gns}).
\end{prop}  

\begin{proof}
We observe that \eqref{BP0} in Proposition \ref{P:genlap} shows that $\mathscr M_r u(x) - u(x) = O(r^2)$ as $r\to 0^+$. Therefore, the integrand in the right-hand side of \eqref{saslap} behaves like $r^{1-2s}$ as $r\to 0^+$. Since at infinity it behaves like $r^{-1-2s}$, we conclude that the integral in the right-hand side of \eqref{saslap} is convergent. Next, we note that Cavalieri's principle allows to write \eqref{fls} in the following way
\begin{align*}
(-\Delta)^s u(x) & = \gamma(n,s) \int_0^\infty \int_{S(x,r)} \frac{u(x) - u(y)}{|x-y|^{n+2s}} d\sigma(y) dr 
\\
& = \gamma(n,s) \int_0^\infty \frac 1{r^{n+2s}} \int_{S(x,r)} [u(x) - u(y)] d\sigma(y) dr
\\
& = - \sigma_{n-1} \gamma(n,s)  \int_0^\infty \frac {r^{n-1}}{r^{n+2s}} [\mathscr M_r u(x) - u(x)] dr.
\end{align*}
This gives the desired conclusion.

\end{proof}

We close this section by introducing another functional class that is relevant in connection with the nonlocal operator $(-\Delta)^s$ and whose motivation will become clearer later in this note. Up to now, in order to define $(-\Delta)^s u(x)$ pointwise as in formula \eqref{fls} above we have considered functions $u$ in $\mathscr S(\Rn)$ or in $C^2(\Rn)\cap L^\infty(\Rn)$. There exist however  larger spaces in which it is still possible to define the nonlocal Laplacean either pointwise or as a tempered distribution. 

\begin{definition}\label{D:sobs}
Let $0<s<1$. We denote by $\mathscr L_s(\Rn)$ the space of measurable functions $u:\Rn\to \overline \R$ for which the norm
\[
||u||_{\mathscr L_s(\Rn)} = \int_{\Rn} \frac{|u(x)|}{1+|x|^{n+2s}} dx < \infty.
\]
\end{definition}
Notice that we trivially have $u\in \mathscr L_s(\Rn)\Longrightarrow u\in L^1_{loc}(\Rn)$. This inclusion implies in particular that 
\[
\mathscr S(\Rn) \subset \mathscr L_s(\Rn)\subset \mathscr S'(\Rn).
\]
We also note that for any $1\le p\le \infty$, we have $L^p(\Rn) \subset \mathscr L_s(\Rn)$.
 
\begin{remark}\label{R:space}
Definition \ref{D:sobs} can be found on p. 73 in \cite{Si}, but the class $\mathscr L_s(\Rn)$ was implicitly first introduced in (1.6.4) in \cite{La} in connection with the nonlocal mean-value operator. For the latter one should see Definition \ref{D:nlpk} below.
\end{remark}

From the definition \eqref{fls} it should be clear to the reader that if $u\in C^2(\Rn)\cap \mathscr L_s(\Rn)$, then we can define $(-\Delta)^s u(x)$ at every $x\in \Rn$. 
This conclusion continues to be true provided that $u$ has a minimal smoothness in terms of the following H\"older class.

\begin{definition}\label{D:holder}
Given $0<s<1$ and $\e>0$ sufficiently small, we define the class $C^{2s+\varepsilon}_{loc}$ according to the following convention:
\[
C_{loc}^{2s+\varepsilon} = \begin{cases}
C_{loc}^{0,2s+\e},\ \ \ \ \ \ \ \ \ \text{if}\ \ 0<s<\frac 12,
\\
C_{loc}^{1,2s+\e-1},\ \ \ \ \ \  \text{if}\ \ \frac 12 \le s < 1.
\end{cases}
\]
\end{definition}

The following result is a special case of Proposition 2.4 in \cite{Si}. For its simple proof we refer the reader to that source.

\begin{prop}\label{P:silv}
Let $u\in \mathscr L_s(\Rn)\cap C_{loc}^{2s+\varepsilon}$, for some $\e\in (0,1)$. Then, for every $x \in \Rn$ the integral in \eqref{fls} is convergent, and $(-\Delta)^s u\in C(\Rn)$.
\end{prop}

Although the next result follows directly from the inclusion $\mathscr L_s(\Rn)\subset \mathscr S'(\Rn)$, we nonetheless present a different simple proof.

\begin{corollary}\label{C:sob}
 Let $u\in \mathscr L_s(\Rn)$, then $(-\Delta)^s u\in  \mathscr S'(\Rn)$.
 \end{corollary}
 
\begin{proof}
We have in fact for every $\vf\in \mathscr S(\Rn)$
\begin{align*}
& \left|<(-\Delta)^s u,\vf>\right| = \left|<u,(-\Delta)^s \vf>\right| 
\\
& =  \left|\int_{\Rn} u(x) (-\Delta)^s \vf(x) dx\right| \le C_{n,s}   \int_{\Rn} \frac{|u(x)|}{1+|x|^{n+2s}} dx < \infty,
\end{align*}
where in the second to the last inequality we have used Proposition \ref{P:decay} above.

\end{proof}


\section{Maximum principle, Harnack inequality and Liouville theorem}\label{S:mp}

In this section we briefly discuss the three properties in the title with the intent of bringing to the reader's attention those changes that are imposed by the nonlocal nature of $(-\Delta)^s$. One of the most fundamental properties of the theory of second order elliptic and parabolic equations is the so called maximum principle. Let $\Om\subset \Rn$ be an open set and let $u\in C^2(\Om)$. A standard fact from calculus states that if $u$ attains a local maximum at a point $x_0\in \Om$, then for the Hessian matrix of $u$ we must have $- \nabla^2 u(x_0) \ge 0$. In particular, this gives $- \Delta u(x_0) = \operatorname{trace}(-\nabla^2 u(x_0))\ge 0$. The (weak) maximum principle for $-\Delta$ states that a subharmonic function in a bounded open set (or, for that matter, any convex function) must attain it supremum on the boundary. The strong maximum principle says much more. For a subharmonic function $u$ in a connected open set we must have $u(x) < \underset{\Om}{\sup}\ u$ for every $x\in \Om$, unless $u \equiv \underset{\Om}{\sup}\ u$.  

What happens with $s$-subharmonic functions, i.e., solutions of $(-\Delta)^s u \le 0$? Suppose $u$ has a \emph{global maximum} at $x = x_0$, i.e., for instance $u(x) \le u(x_0)$ for every $x\in \Rn$. Then, a property analogous to the local case trivially holds. Under such assumption we have in fact from \eqref{fls} 
\begin{equation}\label{mp}
(-\Delta)^s u(x_0) = \frac{\gamma(n,s)}{2} \int_{\Rn} \frac{2 u(x_0) - u(x_0+y) - u(x_0-y)}{|y|^{n+2s}} dy \ge 0,
\end{equation}
with strict inequality if $x_0$ is a strict global maximum. 

However, the inequality \eqref{mp} fails to hold, in general, when $x_0$ is only a \emph{local maximum} for $u$! Thus, in the nonlocal case the maximum principle, weak or strong, does not admit a formulation similar to the local one. This is caused by the nonlocal nature of $(-\Delta)^s$ which makes a $s$-harmonic (or, more in general, a $s$-subharmonic function) feel the effect of far away data. For an interesting example of this negative aspect one should see Theorem 2.2 in \cite{Ka2}, and also Theorem 3.3.1 in \cite{BV}.  The appropriate nonlocal maximum principle for $(-\Delta)^s$ is as follows.

\begin{prop}\label{P:mp}
Let $\Om\subset \Rn$ be a bounded open set and let $(-\Delta)^s u \ge 0$ in $\Om$. If $u\ge 0$ in $\Rn \setminus \Om$, then $u\ge 0$ in $\Om$.
\end{prop}

We stress that in Proposition \ref{P:mp} the ``boundary values" of the $s$-subharmonic function $u$ are prescribed not just on $\p \Om$, but globally in $\Rn\setminus \Om$. Interpreted in the sense of Proposition \ref{P:mp} also the strong maximum principle holds. For a proof of Proposition \ref{P:mp} we refer the reader to \cite{Si} and \cite{BV}. One should also see the discussion in the survey paper \cite{RO15}, which contains a version of the maximum principle for weak subsolutions of $(-\Delta)^s$ or more general nonlocal equations.

A direct consequence  of the weak maximum principle for $-\Delta$ is the uniqueness in the Dirichlet problem: given $f\in C(\Om)$ and $\vf\in C(\p \Om)$, find a function $u\in C^2(\Om)\cap C(\overline \Om)$ such that 
\[
\begin{cases}
- \Delta u = f\ \ \ \ \ \ \text{in}\ \Om,
\\
u = \vf\ \ \ \ \ \ \ \ \text{on}\ \p \Om.
\end{cases}
\]

Using Proposition \ref{P:mp} it is possible to obtain a similar result of uniqueness in the nonlocal Dirichlet problem which is formulated as follows: 
 given a function $f\in C(\Om)$ and $\vf\in C(\Rn\setminus \Om)$, find a function $u\in \mathscr L_s(\Rn)\cap C_{loc}^{2s+\varepsilon}$, for some $\e\in (0,1)$, such that 
\[
\begin{cases}
(-\Delta)^s u = f\ \ \ \ \ \ \text{in}\ \Om,
\\
u = \vf\ \ \ \ \ \ \ \ \text{in}\ \Rn\setminus \Om.
\end{cases}
\]  

Since in the nonlocal setting the strong maximum principle fails in its classical formulation, it is not surprising that also the Harnack inequality needs to be suitably stated. 

For these important aspects of the theory we refer the reader to the paper \cite{Ka3} and the nice introductory presentations in \cite{BV}, \cite{RO15}, but also to the original paper by Riesz \cite{R} and the book \cite{La}, as well as the  works \cite{BK}, \cite{Bo99}, \cite{BB}, \cite{BB2}, \cite{Si} and \cite{FKV}. 

The Liouville theorem for harmonic functions states that there is no such function on the whole $\Rn$ which is one-side bounded other than the constants. A similar result holds in the nonlocal case.

\begin{theorem}[see Lemma 3.2 in \cite{BKN}]\label{T:bkn}
Let $n\ge 2$ and $0<s<1$. Suppose that $u\ge 0$ satisfies $u\in \mathscr L_s(\Rn)$. If $(-\Delta)^s u = 0$ in $\mathscr D'(\Rn)$, then $u$ must be constant.
\end{theorem}

Interesting extensions of Theorem \ref{T:bkn} have been recently independently obtained in \cite{CDL} and \cite{Fa16}. In the former paper the condition $u\ge 0$ is relaxed to 
\[
\underset{|x|\to \infty}{\liminf} \frac{u(x)}{|x|^\gamma} \ge 0,
\]
for some $\gamma\in [0,1]$ with $\gamma<2s$. In the latter, the author shows that if $u\in \mathscr L_s(\Rn)$ and $(-\Delta)^s u = 0$ in $\mathscr D'(\Rn)$, then $u$ must be affine, and constant if $0<s\le 1/2$. Liouville theorems for some anisotropic nonlocal operators are contained in \cite{FV1}, \cite{FV2}. 

The natural space of distributions for the operator $(-\Delta)^s$ is not the space $\mathscr S'(\Rn)$, but the smaller space $\mathscr S'_s(\Rn)$ introduced in the opening of Section \ref{S:fs} below. The motivation comes from Lemma \ref{L:Ss} there. A natural question is whether the assumption $u\in \mathscr L_s(\Rn)$ in Theorem \ref{T:bkn} above can be relaxed to $u\in \mathscr S'_s(\Rn)$. Nothing seems to be known about this intriguing aspect. In this connection we quote the interesting work \cite{DSV16} which develops an extension of the fractional Laplacean.  

In closing, we mention the remarkable recent work \cite{DSVJEMS} in which the authors show the surprising result that, given any positive integer $k$, and a function $f\in C^k(\overline B_1)$, there exists a solution of $(-\Delta)^s u = 0$ in $\Rn$, vanishing outside a suitably large ball $B_R$, which approximates $f$ in $C^k$ norm on $B_1$. This purely nonlocal result is in stunning contrast with what happens in the local case. It is not possible to approximate in $C^k$ norm a function on a ball with harmonic functions since such functions are very rigid. For instance, if they have a local maximum in the ball, they must be constant. 


\section{A brief interlude about very classical stuff}\label{S:bi}

To proceed with the analysis of the nonlocal operator $(-\Delta)^s$ we will need some basic properties of two important, and deeply interconnected, protagonists of classical analysis: the Fourier transform and Bessel functions. Since they both play a pervasive role in these notes, as a help to the reader in this section we recall their definition along with some elementary (and also not so elementary) facts. Before we do that, however, we introduce the ever present Euler's gamma function (see e.g. chapter 1 in \cite{Le}):
\[
\Gamma(x) = \int_0^\infty t^{x-1} e^{-t}\ dt,\ \ \ \ \ \ \ \ \ \ \ \ x>0.
\]
The well-known identity $\G(1/2) = \sqrt \pi$ is simply a reformulation of the famous integral 
\[
\int_\R e^{-x^2} dx = \sqrt \pi.
\]
 Of course, $\G(z)$ can be equally defined as a holomorphic function for every $z\in \mathbb C$ with $\Re z >0$. It is easy to check that for such $z$, one has
\begin{equation}\label{fact}
\G(z+1) = z \G(z).
\end{equation}
This formula, and its iterations, can be used to meromorphically extend $\G(z)$ to the whole complex plane having simple poles at $z = -k$, $k\in \mathbb N \cup\{0\}$, with residues $(-1)^k$. In particular, when $0<s<1$ one obtains from \eqref{fact}
\begin{equation}\label{gammas}
\G(1-s) = - s \G(-s).
\end{equation}
Furthermore, one has the following basic relations:
\begin{equation}\label{sine}
\G(z) \G(1-z) = \frac{\pi}{\sin \pi z}.
\end{equation}
and 
\begin{equation}\label{prod}
2^{2z-1} \G(z) \G(z+\frac 12) = \sqrt \pi \G(2z).
\end{equation}
\emph{Stirling's formula} provides the asymptotic behavior of the gamma function for large positive values of its argument
\begin{equation}\label{stirling}
\G(x) = \sqrt{2\pi}\ x^{x-\frac 12} e^{-x}\left(1+ O\left(\frac 1x\right)\right),\ \ \ \ \ \ \ \ \ \text{as}\ x\to +\infty.
\end{equation}

We close this brief prelude with a very classical formula which connects the gamma function to the $(n-1)$-dimensional Hausdorff measure of the unit sphere $\mathbb S^{n-1}\subset \Rn$, and the $n$-dimensional volume of the unit ball 
\begin{equation}\label{sn1}
\sigma_{n-1} = \frac{2 \pi^{\frac n2}}{\G(\frac n2)},\ \ \ \ \ \ \ \ \ \ \ \ \ \ \omega_{n} = \frac{\sigma_{n-1}}{n} = \frac{\pi^{\frac n2}}{\G(\frac n2 + 1)}.
\end{equation}
Here, Gaussians are lurking in the shadows! In fact, from Bonaventura Cavalieri's principle we easily see that for any $f\in L^1(\Rn)$ and spherically symmetric, i.e., such that $f(x) = f^\star(|x|)$, we have 
\begin{equation}\label{gaussians}
\sigma_{n-1} = \frac{\int_{\Rn} f(x) dx}{\int_0^\infty f^\star(r) r^{n-1} dr}.
\end{equation}
To compute $\sigma_{n-1}$ it thus suffices to produce one spherically symmetric $f\in L^1(\Rn)$ for which we know how to compute both numerator and denominator in \eqref{gaussians}. As it turns out, Gaussians are the first prize winners. If in fact we take $f(x) = e^{-|x|^2}$, then we know that 
\[
\int_{\Rn} f(x) dx = \pi^{\frac n2},
\]
 whereas
 \[
\int_0^\infty f^\star(r) r^{n-1} dr = \int_0^\infty e^{-r^2} r^{n} \frac{dr}{r} = \frac 12  \int_0^\infty e^{-t} t^{\frac n2} \frac{dt}{t} = \frac 12 \G(\frac n2).
\]
Substituting the latter two formulas in \eqref{gaussians} proves the first part of \eqref{sn1}.   

One identity that we will use is the following
\begin{equation}\label{hsg0}
\int_0^\infty u^{-s -1} (1-e^{-u}) du = \frac{\G(1-s)}{s},\ \ \ \ \ \ \ \  \ 0<s<1.
\end{equation}
It is easy to recognize that the integral converges absolutely. The proof of \eqref{hsg0} then easily follows writing $u^{-s-1} = \frac{d}{du} (\frac{u^{-s}}{-s})$ and integrating by parts as follows
\[
\int_0^\infty u^{-s-1} \left(1 -  e^{- u}\right) du = \frac 1s \int_0^\infty u^{-s} e^{-u} du = \frac{\G(1-s)}{s}.
\] 
 
Deeply connected with the gamma function is Euler's \emph{beta function} which for $x,y>0$  is defined as follows
\begin{equation}\label{beta}
B(x,y) = 2 \int_0^{\frac{\pi}{2}} \big(\cos \vt\big)^{2x-1}
\big(\sin \vt\big)^{2y-1} d\vt.
\end{equation}
It is an easy exercise to recognize that
\begin{equation}\label{beta2}
B(x,y) = 2 \int_0^1 (1-\tau^2)^{x-1} \tau^{2y-1} d\tau = \int_0^1
(1-s)^{x-1} s^{y-1} ds.
\end{equation}
The link between the beta  and the gamma function is expressed by the following equation
\begin{equation}\label{bg}
B(x,y) = \frac{\Gamma(x) \Gamma(y)}{\Gamma(x+y)},
\end{equation}
see e.g. (1.5.6) on p. 14 in \cite{Le}. A useful integral which is expressed in terms of the beta, or gamma function is contained in the following proposition.

\begin{prop}\label{P:poisson}
Let $b> - n $ and $a>n+b$, then
\begin{equation}\label{poissongen}
\int_{\Rn} \frac{|x|^b}{(1+|x|^2)^{\frac{a}{2}}}dx =
\frac{\pi^{\frac{n}{2}}}{\Gamma\left(\frac{n}{2}\right)}
\frac{\Gamma\left(\frac{b+n}{2}\right)
\Gamma\left(\frac{a-b-n}{2}\right)}{\Gamma\left(\frac{a+b}{2}\right)}.
\end{equation}
In particular, if $b=0$ and $a = n+1$, then
\begin{equation}\label{poisson0}
\int_{\Rn} \frac{dx}{(1+|x|^2)^{\frac{n+1}{2}}} =
\frac{\pi^{\frac{n+1}{2}}}{\Gamma\left(\frac{n+1}{2}\right)}\ .
\end{equation}
\end{prop}

\begin{proof}
Let us observe preliminarily that the assumption $b>-n$ serves to guarantee that the integrand belongs to $L^1_{loc}(\Rn)$, whereas it is in $L^1(\Rn)$ if and only if $a-b>n$. Under these hypothesis 
we have 
\begin{align*}
& \int_{\Rn} \frac{|x|^b}{(1+|x|^2)^{\frac{a}{2}}}dx = \sigma_{n-1}
\int_0^{\infty} \frac{r^{b+n-1}}{(1+r^2)^{\frac{a}{2}}} dr
\\
& (\text{change of variable}\ r = \tan \xi)
\\
& = \sigma_{n-1} \int_0^{\frac{\pi}{2}} \frac{(\tan
\xi)^{b+n-1}}{(1+\tan^2\xi)^{\frac{a-2}{2}}} d\xi = \sigma_{n-1}
\int_0^{\frac{\pi}{2}} (\sin \xi)^{b+n-1} (\cos \xi)^{a-b-n-1} d\xi
\\
& = \frac{\sigma_{n-1}}{2}
B\left(\frac{b+n}{2},\frac{a-b-n}{2}\right),
\end{align*}
where the last inequality follows by a comparison with \eqref{beta}.
If we now apply formulas \eqref{sn1} and \eqref{bg} we obtain
\[
\frac{\sigma_{n-1}}{2} B\left(\frac{b+n}{2},\frac{a-b-n}{2}\right) =
\frac{\pi^{\frac{n}{2}}}{\Gamma\left(\frac{n}{2}\right)}
\frac{\Gamma\left(\frac{b+n}{2}\right)
\Gamma\left(\frac{a-b-n}{2}\right)}{\Gamma\left(\frac{a+b}{2}\right)},
\]
which gives the desired conclusion \eqref{poissongen}. To obtain
\eqref{poisson0} it suffices to keep in mind that $\Gamma(1/2) = \sqrt \pi$.

\end{proof}

We are ready to introduce the queen of classical analysis: given a function $u\in L^1(\Rn)$, we define its Fourier transform as 
\[
\F(u)(\xi) = \hat u(\xi) = \int_{\Rn} e^{-2\pi i <\xi,x>} u(x) dx.
\]
We notice that the normalization that we have adopted in the above definition is the one which makes $\F$ an isometry of $L^2(\Rn)$ onto itself, see \cite{SW}. We recall next some of the basic properties of $\F$. If $\tau_h u(x) = u(x+h)$ and $\delta_\la u(x) = u(\la x)$ are the translation and dilation operators in $\Rn$, then we have
\begin{equation}\label{shift}
\widehat{\tau_y u}(\xi) = e^{2\pi i<\xi,y>} \hat u(\xi),
\end{equation}
and
\begin{equation}\label{dil}
\widehat{\delta_\la u}(\xi) = \la^{-n} \hat u\left(\frac{\xi}{\la}\right).
\end{equation}

The Fourier transform is also invariant under the action of the orthogonal group $\mathbb O(n)$. We have in fact for every $T\in \mathbb O(n)$
\begin{equation}\label{on}
\widehat{u\circ T} = \hat u \circ T.
\end{equation}
Formula \eqref{on} says that the Fourier transform of a spherically symmetric function is spherically symmetric as well, see also Theorem \ref{T:Fourier-Bessel} below for a deeper formulation of this fact.

Another crucial property is the Riemann-Lebesgue lemma:
 \begin{equation}\label{rl}
u\in L^1(\Rn)\ \Longrightarrow\  |\hat u(\xi)| \to 0\ \  \text{as}\ \  |\xi|\to \infty.
\end{equation}
This result has important consequences when combined with the following two formulas. Let $u\in L^1(\Rn)$ be such that for $\alpha\in \mi$ also $\pa u\in
L^1(\Rn)$. Then,
\begin{equation}\label{ftder}
\widehat{(\pa u)}(\xi) = (2\pi i)^{|\alpha|} \xi^\alpha \hat u(\xi).
\end{equation}
In particular, \eqref{ftder} and \eqref{rl} give: $|\xi^\alpha| |\hat u(\xi)|\to 0$ as $|\xi|\to \infty$. Furthermore, if $u\in L^1(\Rn)$ is such that for $\alpha\in \mi$ one has $x\to
x^\alpha u(x) \in L^1(\Rn)$, then,
\begin{equation}\label{derft}
\p^\alpha \hat u(\xi) = (-2\pi i)^{|\alpha|} \widehat{(\cdot)^\alpha
u}(\xi).
\end{equation}
In particular, \eqref{derft} and \eqref{rl} imply that: $\p^\alpha \hat u\in C(\Rn)$ and $|\p^\alpha \hat u(\xi)|\to 0$ as $|\xi|\to \infty$. 

Combining these observations one derives one of the central properties of $\F$: it maps continuously $\mathscr S(\Rn)$ onto
itself and is an isomorphism. Its inverse is also continuous, and is given by the Fourier inversion formula
\[
\F^{-1}(u)(x) = \int_{\Rn} e^{2\pi i <\xi,x>} \hat u(\xi) d\xi.
\]

We next introduce the second main character of this section: the Bessel functions. The  book \cite{Le} provides a rewarding account of this beautiful classical subject. For a comprehensive study the reader can also consult G. N. Watson's classical treatise \cite{W}, as well as the first two volumes of the Bateman manuscript project \cite{EMOT}. 

\begin{definition}\label{D:extbessel}
For every $\nu\in \mathbb C$ such that $\Re\nu>-\dfrac12$ we define
the  \emph{Bessel function of the first kind} and of complex order $\nu$ by the formula
\begin{equation}\label{extbessel}
\jnu(z)
=\frac{1}{\G(\frac12)\G(\nu+\frac{1}{2})}\left(\frac{z}2\right)^\nu\int^1_{-1}
e^{izt}(1-t^2)^{\frac{2\nu-1}2}dt,
\end{equation}
where $\G(x)$ denotes the Euler gamma function.
\end{definition}

The function $J_\nu(z)$ in \eqref{extbessel} derives its name from the fact that it solves the linear ordinary differential equation known as \emph{Bessel equation of order $\nu$} 
\begin{equation}\label{besseleq}
z^2 \frac{d^2 J}{dz^2} + z \frac{dJ}{dz} + (z^2 - \nu^2)J = 0.
\end{equation}
An expression of $J_\nu$ as a power series for an arbitrary value of $\nu\in \mathbb C$ is provided by
\begin{equation}\label{psbessel}
J_\nu(z) = \sum_{k=0}^\infty (-1)^k \frac{(z/2)^{\nu+2k}}{\G(k+1)\G(k+\nu+1)},\ \ \ \ \ |z|<\infty,\ \ |\arg(z)|<\pi,
\end{equation} 
see e.g. (5.3.2) on p. 102 in \cite{Le}.
When $\nu\not\in \mathbb Z$, another linearly independent solution of \eqref{besseleq} is provided by the function $J_{-\nu}(z)$. When $\nu\in \mathbb Z$ the two functions $J_{\nu}$ and $J_{-\nu}$ are linearly dependent, and in order to find a second solution linearly independent from $J_{\nu}$ one has to proceed differently.

The observation that follows is very important in most concrete applications of the theory. Suppose that $\Phi(z)$ be a solution to the Bessel equation
\eqref{besseleq}, and consider the function defined by the
transformation
\begin{equation}\label{besselcv}
u(y) = y^\alpha \Phi(\beta y^\gamma).
\end{equation}
Then, one easily verifies that $u(y)$ satisfies the
\emph{generalized Bessel equation}
\begin{equation}\label{genbessel0}
y^2 u''(y) + (1 - 2\alpha) y u'(y) + \left[\beta^2 \gamma^2
y^{2\gamma} + (\alpha^2 - \nu^2 \gamma^2)\right] u(y) = 0.
\end{equation}
We show with an example how this observation is applied. Consider the ball $B_R = \{x\in \Rn\mid |x|<R\}$ and the cylinder $\mathcal C = B_R\times \R\subset \R^{n+1}$. Denote the variable in $\mathcal C$ as $(x,t)$, with $x\in \Rn, t\in \R$. Suppose we look for nontrivial functions $u(x,t)$ which are harmonic in $\mathcal C$. A classical tool is using separation of variables, i.e., look for $u$ in the form $u(x,t) = \vf(x) h(t)$. Imposing that $\Delta u = 0$ in $\mathcal C$ leads to the equation
\[
h(t) \Delta \vf(x) + h''(t) \vf(x) = 0.
\]
If we divide by $\vf(x)h(t)$, we conclude that we must have
\[
- \frac{\Delta \vf(x)}{\vf(x)} = \frac{h''(t)}{h(t)}.
\]
This is possible only if there exists a number $\la\ge 0$ such that
\[
- \Delta \vf = \la \vf,\ \ \ \ \ \ \ \ h'' = \la h.
\]
The second equation clearly gives $h(t) = A e^{\sqrt \la t} + B e^{-\sqrt \la t}$, whereas for the first one we seek a solution which is spherically symmetric, i.e., $\vf(x) = F(|x|)$. We easily see that $F(r)$ must satisfy the equation
\begin{equation}\label{lapeigen}
r^2 F''(r) + (n-1) r F'(r) + \la  r^2 F(r) = 0.
\end{equation}
It is not obvious at first glance that this is a Bessel equation. However, if we compare \eqref{lapeigen} with  \eqref{genbessel0}, we conclude that the former is a special instance of the latter if we take
\[
1-2\alpha = n-1,\ \ \ \ \ \beta^2 = \la,\ \ \ \ \ \gamma = 1, \ \ \ \ \ \nu^2 = \alpha^2.
\]
We must thus have
\[
\alpha = - \frac{n-2}{2},\ \ \ \ \ \beta = \sqrt \la,\ \ \ \ \ \gamma = 1,\ \ \ \ \ \nu = \pm \frac{n-2}{2}.
\]
But then \eqref{besselcv} gives us, at least for $n\not= 2k+2$, the two linearly independent solutions 
\[
\vf_1(x) =  |x|^{- \frac{n-2}{2}} J_{\frac{n-2}{2}}(\sqrt \la |x|),\ \ \ \ \ \ \ \ \vf_2(x) = |x|^{- \frac{n-2}{2}} J_{- \frac{n-2}{2}}(\sqrt \la |x|).
\]
Since from \eqref{bfbehzero} below we see that $\vf_2\not\in C^2$, we must discard such solution and keep only $\vf_1$. In conclusion, the function
\[
u(x,t) = \left(A e^{\sqrt \la t} + B e^{-\sqrt \la t}\right) |x|^{- \frac{n-2}{2}} J_{\frac{n-2}{2}}(\sqrt \la |x|),
\]
provides a harmonic function in the cylinder $\mathcal C$, and the appearance of the Bessel function $J_{\frac{n-2}{2}}$ in such formula is the reason for which solutions of \eqref{besseleq}  are also called in the literature \emph{cylinder functions}.

Returning to Definition \ref{D:extbessel}, from \eqref{extbessel} and \eqref{beta2} we immediately find 
\[
z^{-\nu} J_\nu(z)\ \underset{z\to 0}{\longrightarrow}\ \frac{2^{-\nu +1}}{\G(\frac12)\G(\nu+\frac{1}{2})}  \int^1_{0} (1-s^2)^{\frac{2\nu-1}2}ds =  \frac{2^{-\nu +1}}{\G(\frac12)\G(\nu+\frac{1}{2})} B(\nu+\frac 12,\frac 12).
\]
From this asymptotic relation and \eqref{bg} one obtains
\begin{equation}\label{bfbehzero}
J_\nu(z)\cong\frac{2^{-\nu}}{\Gamma(\nu+1)}z^\nu,\quad\text{as }z\to 0.
\end{equation}

Unlike the simple expression of the asymptotic of $J_\nu(z)$ as $z\to 0$, the behavior at infinity of $J_\nu(z)$ is  more delicate to come by. We have the following result, see for instance Lemma 3.11 in \cite{SW}, or also (5.11.6) on p. 122 in \cite{Le}.
\begin{lemma}\label{L:jnuinfty}
Let $\Re\nu>-\dfrac12$. One has
\begin{align}\label{jnuinfty} J_\nu(z)&=\sqrt{\frac2{\pi
z}}\cos\left(z-\frac{\pi\nu}2-\frac\pi4\right)+
O(z^{-\frac32})\\
&\quad\text{as }|z|\to\infty,\quad-\pi+\delta<\arg z<\pi-\delta.
\notag
\end{align}
In particular, 
\begin{equation}\label{jnuinfty2}
J_\nu(z) = O(z^{-1/2}),\ \ \ \ \ \text{as}\ z\to \infty,\  z\ge 0.
\end{equation}
\end{lemma}
 
Along with the Bessel equation \eqref{besseleq}, in Sections \ref{S:fs} and \ref{S:pk} below  we will need the \emph{modified Bessel equation} of order $\nu\in \mathbb C$,
\begin{equation}\label{modbesseleq}
z^2 \frac{d^2\Phi}{dz^2} + z \frac{d\Phi}{dz} - (z^2 + \nu^2)\Phi = 0.
\end{equation}
We stress that the (substantial) difference between \eqref{modbesseleq} and \eqref{besseleq} is the sign of the coefficient $z^2$ in the zero order term.

Two linearly independent solutions of \eqref{modbesseleq} are the modified Bessel function of the first kind, 
\begin{equation}\label{Inu}
I_\nu(z) = \sum_{k=0}^\infty \frac{(z/2)^{\nu+2k}}{\G(k+1) \G(k+\nu+1)},\ \ \ \ \ \ \  |z|<\infty,\ \ |\arg(z)| < \pi,
\end{equation}
and the modified Bessel function of the third kind, or Macdonald function, which for order $\nu\not= 0, \pm 1, \pm 2, ... $, is given by
\begin{equation}\label{Knu}
K_\nu(z) = \frac \pi{2} \frac{I_{-\nu}(z) - I_\nu(z)}{\sin \pi \nu},\ \ \ \ \ \ \ \ \ \ \  \ |\arg(z)| < \pi.
\end{equation}
Notice that $K_\nu(z) = K_{-\nu}(z)$. 

Similarly to what was observed for \eqref{genbessel} above, it is easy to verify that if $\Phi(z)$ is a solution to the modified Bessel equation
\eqref{modbesseleq}, then the function defined by the
transformation \eqref{besselcv}
satisfies the
\emph{generalized modified Bessel equation}
\begin{equation}\label{genbessel}
y^2 u''(y) + (1 - 2\alpha) y u'(y) + \left[(\alpha^2 - \nu^2 \gamma^2) - \beta^2 \gamma^2
y^{2\gamma}\right] u(y) = 0.
\end{equation}

As we have stated in the opening of this section the Fourier transform and the Bessel functions are deeply connected. One important instance of this link is the following result which provides a deeper meaning to the invariance of
the Fourier transform with respect to the
action of the orthogonal group $\mathbb O(n)$. We emphasize that the presence of Bessel functions in Theorem \ref{T:Fourier-Bessel} below underscores the interplay between curvature (that of the unit sphere $\mathbb S^{n-1}\subset \Rn$) and Fourier analysis. The reader who wants to become familiar with this deep aspect can look at the beautiful expository paper of E. Stein \cite{SHA} and also \cite{SW77}. For the following result we refer to Theorem 40 on p. 69 in \cite{BC}. 

\begin{theorem}[Fourier-Bessel representation]\label{T:Fourier-Bessel}
Let $u(x)=f(|x|)$, and suppose that 
\[
t \to t^{\frac{n}2} f(t) J_{\frac{n}2-1}
(t)\in L^1(\R^+),
\]
where we have denoted by $J_{\frac{n}2-1}$ the Bessel function of order $\nu = \frac n2 -1$ defined by \eqref{extbessel}.
Then,
$$
\hat u(\xi)=2\pi|\xi|^{-\frac{n}2 +1}\int^\infty_0 t^{\frac{n}2} f(t)J_{\frac{n}2-1}
(2\pi|\xi|t) dt.
$$
\end{theorem}

To check the integrability assumption in Theorem \ref{T:Fourier-Bessel} we can use the above given asymptotic \eqref{bfbehzero} and \eqref{jnuinfty2}  for the Bessel function $J_\nu$. 

One interesting application of Theorem \ref{T:Fourier-Bessel} that will be needed in Section \ref{S:reg} is the following result about the \emph{Bochner-Riesz kernels} in $\Rn$:
\begin{equation}\label{br}
B_z(x) = \frac{\left(1-|x|^2\right)^z_+}{\Gamma(z+1)},\ \ \ \ \Re
z>-1,
\end{equation}
where we have denoted $a_+ = \max\{a,0\}$. Notice that, thanks to
the assumption $\Re z>-1$, we have $B_z\in L^1(\Rn)$.

\begin{lemma}\label{L:BR}
For every $z\in \mathbb C$, such that $\Re z > -1$, one has
\begin{equation}\label{FTBR}
\hat
B_z(\xi)=\pi^{-z}|\xi|^{-\left(\frac{n}2+z\right)}J_{\frac{n}2+z}(2\pi|\xi|),\
\ \ \ \xi\in \Rn.
\end{equation}
\end{lemma}

\begin{proof}
Since $B_z(x) = B^\star(|x|)$, with $B^\star(r) = \frac{\left(1-r^2\right)^z_+}{\Gamma(z+1)}$,
applying Theorem \ref{T:Fourier-Bessel} we find
$$
\hat B_z(\xi)=\frac{2\pi|\xi|^{-\frac{n-2}2}}{\Gamma(z+1)}\int^1_0
\left(1-r^2\right)^z J_{\frac{n-2}2} (2\pi|\xi|r)r^{\frac{n}2}dr.
$$
Next, we use the following formula 6.567 in \cite{GR} 
\begin{equation}\label{Besselintegral}
\int^1_0(1-r^2)^z
s^{\nu+1}J_\nu(ar)dr=2^z\Gamma(z+1)a^{-(z+1)}J_{\nu+z+1}(a),
\end{equation}
which is valid for any $a>0, \Re z>-1$, and $\Re \nu >-1$. Applying
\eqref{Besselintegral} with $\nu=\dfrac{n-2}2$, $a=2\pi|\xi|$, we
finally obtain \eqref{FTBR}.

\end{proof}

Another important instance of Theorem \ref{T:Fourier-Bessel} is the following. Consider the spherical mean-value operator 
\[
\mathscr M_r(u,x) = \frac{1}{\sigma_{n-1} r^{n-1}} \int_{S(x,r)} u(y) d\sigma(y), 
\]
introduced in \eqref{MA0}. The normalized surface measure on the sphere $S(0,r)$, i.e., 
\begin{equation}\label{nsm}
d\sigma_r = \frac{1}{\sigma_{n-1} r^{n-1}} d\sigma(y),
\end{equation}
is the compactly supported distribution whose action on a test function $\vf$ is defined by
\[
<d\sigma_r,\vf> = \int_{S(0,r)} \vf(y) d\sigma_r(y).
\]
Clearly, supp$\ d\sigma_r = S(0,r)$. Since $d\sigma_r\in \mathscr E'(\Rn)$, by the Theorem of Paley-Wiener its Fourier transform $\F(d\sigma_r)(\xi)$ is not just in $\mathscr S'(\Rn)$, but it is in fact a $C^\infty$ function in $\Rn$. An easy exchange of order of integration argument shows that 
\[
\F(d\sigma_r)(\xi) = \frac{1}{\sigma_{n-1} r^{n-1}} \int_{S(0,r)} e^{-2\pi i <\xi,x>} d\sigma(x).
\]

The following beautiful classical result, which is a special case of Theorem \ref{T:Fourier-Bessel}, provides us with an explicit expression of such function, see p. 154 in \cite{SW}.

\begin{prop}\label{P:FTspheren}
For any $r>0$ one has
\[
\F(d\sigma_r)(\xi) = \frac{2\pi}{\sigma_{n-1}} \left(r |\xi|\right)^{-\frac{n}2 +1}
J_{\frac n2 -1}(2\pi r |\xi|).\]
\end{prop}

Now, the operator $\mathscr M_r u(\cdot)$ can be seen as the convolution of $d\sigma_r$ with the function $u$. In fact, for any $u\in \mathscr S(\Rn)$, if we let $\check{u}(x) = u(-x)$, then we have 
\begin{align*}
<d\sigma_r \star u,\vf> & = <d\sigma_r, \check{u}\star \vf> = \frac{1}{\sigma_{n-1} r^{n-1}} \int_{S(0,r)} \check{u}\star \vf(y) d\sigma(y)
\\
& = \frac{1}{\sigma_{n-1} r^{n-1}} \int_{S(0,r)} \int_{\Rn} \check{u}(y-x)\vf(x) dx d\sigma(y) 
\\
& = \int_{\Rn} \left(\frac{1}{\sigma_{n-1} r^{n-1}} \int_{S(0,r)} u(x-y)d\sigma(y)\right) \vf(x) dx
\\
& = \int_{\Rn} \mathscr M_r u(x) \vf(x) dx = <\mathscr M_r u,\vf>.  
\end{align*}
This shows that in $\mathscr S'(\Rn)$ we have
\begin{equation}\label{nsm2}
\mathscr M_r u  = d\sigma_r \star u.
\end{equation}
If we take the Fourier transform of both sides of \eqref{nsm2} we obtain
\begin{equation}\label{nsmFT}
\F(\mathscr M_r u)(\xi) = \F(d\sigma_r)(\xi) \F(u)(\xi).
\end{equation}
Combining \eqref{nsmFT} with Proposition \ref{P:FTspheren}, we conclude that the Fourier transform of the spherical mean-value operator is given by 
\begin{equation}\label{beau}
\F(\mathscr M_r u)(\xi) = \frac{2\pi}{\sigma_{n-1}} \left(r |\xi|\right)^{-\frac{n}2 +1}
J_{\frac n2 -1}(2\pi r |\xi|) \F(u)(\xi),
\end{equation}
yet one more instance of the fascinating interplay between a classical operation of analysis, such as taking the spherical mean-value of a function, and the Bessel functions. Again, the presence of these special functions is a manifestation of curvature, see \cite{SHA} and \cite{SW77}.

Another family of special functions that will be needed in this paper are the so-called \emph{hypergeometric functions}. In order to introduce them we recall the definition of the Pochammer's symbols
\[
\alpha_0 = 1,\ \ \ \ \alpha_k \overset{def}{=} \frac{\G(\alpha + k)}{\G(\alpha)} = \alpha(\alpha+1)...(\alpha + k -1),\ \ \ \ \ \ \ \ \ \ k\in \mathbb N. 
\] 
Notice that since, as we have said, the gamma function has a pole in $z=0$, we have
\[
0_k = \begin{cases}
1\ \ \ \ \ \ \text{if}\ k = 0
\\
0\ \ \ \ \ \ \text{for}\ k\ge 1.
\end{cases}
\]
\begin{definition}\label{D:hyper}
Let $p, q\in \no$ be such that $p\le q+1$, and let $\alpha_1,...,\alpha_p$ and $\beta_1,...,\beta_q$ be given parameters such that $-\beta_j\not\in \no$ for $j=1,...,q$. Given a number $z\in \mathbb C$, the power series
\[
_p F_q(\alpha_1,...,\alpha_p;\beta_1,...,\beta_q;z) = \sum_{k=0}^\infty \frac{(\alpha_1)_k . . . (\alpha_p)_k}{(\beta_1)_k . . . (\beta_q)_k} \frac{z^k}{k!}
\]
is called the \emph{generalized hypergeometric function}. When $p = 2$ and $q=1$, then the function $_2 F_1(\alpha_1,\alpha_2;\beta_1;z)$ is the \emph{Gauss' hypergeometric function}, and it is usually denoted by $F(\alpha_1,\alpha_2;\beta_1;z)$.
\end{definition}

Using the ratio test one easily verifies that the radius of convergence of the above hypergeometric series is $\infty$ when $p\le q$, whereas it equals $1$ when $p = q+1$. For later reference we record the following facts that follow easily from Definition \ref{D:hyper}:
\begin{equation}\label{zeroF}
F(\alpha,0;\beta;z) = F(0,\alpha;\beta;z) = 1,
\end{equation}
and (see also p. 275 in \cite{Le})
\begin{equation}\label{fs6}
F(\alpha,\beta;\beta;-z) =\ _1F_0(\alpha;-z) = (1+z)^{-\alpha}.
\end{equation}
It is also interesting to observe that the hypergeometric function $_0F_1$ is in essence a Bessel function, up to powers and rescaling. One has in fact from \eqref{Inu} and Definition \ref{D:hyper},
\begin{equation}\label{hypbessel}
I_\nu(z) = \frac{1}{\G(\nu+1)} \left(\frac{z}{2}\right)^{\nu}\  _0F_1(\nu+1;(z/2)^2).
\end{equation}
Formula \eqref{hypbessel}, Definition \ref{D:hyper} and a change of variable give 
\begin{align*}
\int_0^z t^a I_\nu(t) dt & = \frac{2^{\frac{a-\nu-1}{2}}}{\G(\nu+1)} \int_0^{(z/2)^2}  t^{\frac{a+\nu-1}{2}}\  _0F_1(\nu+1;t) dt
\\
& = \frac{2^{-2\nu}}{(a+\nu+1)\G(\nu+1)}\ _1F_2(\frac{a+\nu+1}{2};\nu+1, \frac{a+\nu+1}{2}+1;(z/2)^2). 
\end{align*}

We will also need the following beautiful result that connects Bessel functions to the Gauss' hypergeometric function $F(\alpha_1,\alpha_2;\beta_1;z)$. For its proof see p.51 and forward in Vol.2 of \cite{EMOT}. If one is interested only in the result, see $6.574$, formula $3.$ on p.692 in \cite{GR}.

\begin{lemma}[The discontinuous integral of Weber and Schafheitlin]\label{L:GR}
Let $\Re(\nu + \mu - \la + 1)>0$, $\Re\la > -1$, $0<b<a$. Then,
\begin{align*}
 \int_0^\infty t^{-\la} J_\nu(at) J_\mu(bt) dt & = \frac{b^\mu \Gamma(\frac{\nu + \mu - \la + 1}{2})}{2^\la a^{\mu-\la+1} \Gamma(\frac{\nu - \mu + \la + 1}{2}) \Gamma(\mu + 1)}
\\
& \times F\left(\frac{\nu + \mu - \la + 1}{2},\frac{-\nu + \mu - \la + 1}{2};\mu+1,\frac{b^2}{a^2}\right). 
\end{align*}
\end{lemma}

\begin{remark} Let us note explicitly that, thanks to \eqref{bfbehzero}, the assumption $\Re(\nu + \mu - \la + 1)>0$ guarantees that $t\to t^{-\la} J_\nu(at) J_\mu(bt)\in L^1(0,\delta)$ for every $\delta>0$. On the other hand, the hypothesis $\Re\la > -1$ guarantees, in view of Lemma \ref{L:jnuinfty}, that $t\to t^{-\la} J_\nu(at) J_\mu(bt)\in L^1(\delta,\infty))$. Therefore, under the given assumptions we do have $t\to t^{-\la} J_\nu(at) J_\mu(bt)\in L^1(\R^+)$.
\end{remark}


\section{Fourier transform, Bessel functions and $(-\Delta)^s$}\label{S:ftb}

After our brief interlude on the Fourier transform and Bessel functions, we now return to the main protagonist of these notes. The two main objectives of this section are: 
\begin{itemize}
\item[(i)] to establish an alternative way of computing $(-\Delta)^s$ based on the Fourier transform, see Proposition \ref{P:slapft} below, and show that the fractional Laplacean is an elliptic pseudodifferential operator in the Kohn-Nirenberg's class $\Psi^{2s}_{1,0}$, see Proposition \ref{P:pseudo}; 
\item[(ii)] to compute explicitly the constant $\gamma(n,s)$ in \eqref{fls}, see Proposition \ref{P:gns} below.
\end{itemize}
\begin{prop}[Pseudodifferential nature of $(-\Delta)^s$]\label{P:slapft}
Let $\gamma(n,s)>0$ be the number identified by the following formula
\begin{equation}\label{gns}
\gamma(n,s)  \int_{\Rn} \frac{1 - \cos(z_n)}{|z|^{n+2s}} dz = 1.
\end{equation}
Then, for any $u\in \mathscr S(\Rn)$ we have
\begin{equation}\label{fls3}
\widehat{(-\Delta)^s u}(\xi) = (2\pi |\xi|)^{2s} \hat u(\xi).
\end{equation}
\end{prop}

\begin{proof}
Let us observe that in view of Corollary \ref{C:decay} we know that $(-\Delta)^s u\in L^1(\Rn)$ and thus we can take its Fourier transform in the sense of $L^1$. Having said this, if we denote by $\tau_h u(x) = u(x+h)$ the translation operator in $\Rn$, we can rewrite \eqref{fls} in the following way
\begin{equation}\label{fls2}
(-\Delta)^s u(x) = \frac{\gamma(n,s)}{2} \int_{\Rn} \frac{2 u(x) - \tau_y u(x) - \tau_{-y}u(x)}{|y|^{n+2s}} dy.
\end{equation}

Using \eqref{shift} we easily find
\begin{equation}\label{fls22}
\widehat{(-\Delta)^s u}(\xi) = \gamma(n,s) \left(\int_{\Rn} \frac{1 - \cos(2\pi<\xi,y>)}{|y|^{n+2s}} dy\right) \hat u(\xi) = J(\xi)\hat u(\xi),
\end{equation}
where we have let
\[
J(\xi) = \gamma(n,s) \int_{\Rn} \frac{1 - \cos(2\pi<\xi,y>)}{|y|^{n+2s}} dy.
\]
We notice that the integral defining $J(\xi)$ only depends on $|\xi|$. For every $T\in \mathbb O(n)$ one in fact easily verifies that $J(T\xi) = J(\xi)$. 
For $\xi\not= 0$ we can thus write 
\[
J(\xi) = \gamma(n,s) \int_{\Rn} \frac{1 - \cos(<\frac{\xi}{|\xi|},2\pi |\xi| y>)}{|y|^{n+2s}} dy.
\]
The change of variable $z = 2\pi |\xi| y$ now gives
\begin{align}\label{JJ}
J(\xi) & = (2\pi |\xi|)^{2s} \gamma(n,s)  \int_{\Rn} \frac{1 - \cos(<\frac{\xi}{|\xi|},z>)}{|z|^{n+2s}} dz
\\
& = (2\pi |\xi|)^{2s} \gamma(n,s)  \int_{\Rn} \frac{1 - \cos(<e_n,z>)}{|z|^{n+2s}} dz = (2\pi |\xi|)^{2s} \gamma(n,s)  \int_{\Rn} \frac{1 - \cos (z_n)}{|z|^{n+2s}} dz.
\notag
\end{align}
Notice that the integrand in the right-hand side of the latter equation is nonnegative, and that the integral is convergent. We have in fact
\begin{align*}
& \int_{\Rn} \frac{1 - \cos(z_n)}{|z|^{n+2s}} dz = \int_{|z|\le 1} \frac{1 - \cos(z_n)}{|z|^{n+2s}} dz + \int_{|z|>1} \frac{1 - \cos (z_n)}{|z|^{n+2s}} dz
\\
& \le C \int_{|z|\le 1} \frac{dz}{|z|^{n-2(1-s)}}  + 2 \int_{|z|>1} \frac{dz}{|z|^{n+2s}} < \infty.
\end{align*}
Finally, if we substitute in \eqref{fls22} the expression given by \eqref{JJ}, it becomes clear that if we choose $\gamma(n,s)>0$ as in \eqref{gns}, then \eqref{fls3} holds.

\end{proof}

Formula \eqref{fls3} in Proposition \ref{P:slapft} shows that the fractional Laplacean $(-\Delta)^s$ belongs to a class of operators known as \emph{pseudodifferential operators}. Their action on functions is specified by the following formula
\begin{equation}\label{Ppseudo}
Pf(x) = \int_{\Rn} e^{2\pi i <\xi,x>} p(x,\xi) \hat f(\xi) d\xi,
\end{equation}
where the function $p(x,\xi)$, known as a \emph{symbol}, is requested to fulfill suitable hypothesis, see e.g. \cite{Ta}. For instance, if $p(x,\xi)$ is a $C^\infty$ function on $\Rn\times\Rn$ with the property that there exist $m\in \R$ such that for every $\alpha, \beta \in \mathbb N \cup\{0\}$ and every $x, \xi\in \Rn$ one has
\[
|\partial _{\xi }^{\alpha }\partial _{x}^{\beta } p(x,\xi)| \leq C_{{\alpha ,\beta }}\,(1+|\xi |)^{{m-|\alpha |}},
\]
for some constant $C_{{\alpha ,\beta }}$, then we say that $p(x,\xi)$ belongs to the symbol class $S_{{1,0}}^{m}$ introduced by Kohn and Nirenberg in their seminal works \cite{KN1}, \cite{KN2}. A more general class of symbols, denoted by $S^m_{\rho,\delta}$, was introduced by H\"ormander, see \cite{Ho66}, and also \cite{Ta}. If $p(x,\xi)\in S^m_{1,0}$, then the corresponding operator $P$ defined by \eqref{Ppseudo} is called a pseudodifferential operator of order $m$ and it is said to belong to the class 
$\Psi _{{1,0}}^{m}$. A pseudodifferential operator $P\in \Psi^m_{1,0}$ is called \emph{elliptic} if there exists $r>0$ such that its symbol $p(x,\xi)$ satisfies the following condition
\[
|p(x,\xi)^{-1}| \le \frac{C}{1+|\xi|^m},\ \ \ \ \ |x|\ge r.
\]
From the equation \eqref{fls3} above, we see that the symbol of $(-\Delta)^s$ is $p(x,\xi) = (2\pi |\xi|)^{2s}$, and therefore one easily sees that $p\in S^{2s}_{1,0}$ and that $(-\Delta)^s$ is elliptic. We state this observation in a proposition since it is a basic aspect of $(-\Delta)^s$ which has important repercussions. For a notable one we refer the reader to  Theorem \ref{T:hypo} below.

\begin{prop}\label{P:pseudo}
The operator $(-\Delta)^s$ is an elliptic pseudodifferential operator in the class $\Psi^{2s}_{1,0}$.
\end{prop}

The pseudodifferential character of the operator $(-\Delta)^s$ has been exploited in the recent work by Epstein and Pop \cite{EP} to study the regularity theory for the fractional Laplacean with a drift in the supercritical range $0<s<1/2$. More general pseudodifferential operators which include $(-\Delta)^s$ as a special case have been treated in the works of G. Grubb \cite{Gr1}, \cite{Gr2}, \cite{Gr3}.

Equation \eqref{fls3} in Proposition \ref{P:slapft} has the following immediate consequence.

\begin{corollary}[Semigroup property]\label{C:semi}
Let $0<s, s' <1$, with $s+s'\le 1$. Then, for any $u\in \mathscr S(\Rn)$ we have
\[
(-\Delta)^{s+s'} u = (-\Delta)^s (-\Delta)^{s'} u = (-\Delta)^{s'} (-\Delta)^{s} u.
\]
\end{corollary}

\begin{proof}
It is enough to verify the desired inequality on the Fourier transform side. Using \eqref{fls3} we find
\begin{align*}
\F\left((-\Delta)^{s+s'} u\right) & = (2\pi|\xi|)^{2(s+s')} \hat u = (2\pi|\xi|)^{2s} (2\pi|\xi|)^{2s'} \hat u 
\\
& = \F((-\Delta)^s (-\Delta)^{s'} u) = \F((-\Delta)^{s'} (-\Delta)^{s} u).
\end{align*}

\end{proof}

With Proposition \ref{P:slapft} in hands we can now prove the following important ``integration by parts" formula.

\begin{lemma}[Symmetry]\label{L:ibp}
Let $0<s\le 1$. Then, for any $u ,v\in \mathscr S(\Rn)$ we have
\begin{equation}\label{ibp}
\int_{\Rn} u(x) (-\Delta)^s v(x) dx = \int_{\Rn} (-\Delta)^s u(x) v(x) dx.
\end{equation}
\end{lemma}

\begin{proof}
The case $s=1$ is well-known, and it is just integration by parts, so let us focus on $0<s<1$. Since by Corollary \ref{C:decay} we know $\widehat{(-\Delta)^s u}, \widehat{(-\Delta)^s v}\in L^1(\Rn)$, we can use the following formula, valid for any $f, g \in L^1(\Rn)$,
\begin{equation}\label{sft}
\int_{\Rn} \hat f(\xi) g(\xi) d\xi = \int_{\Rn} f(\xi) \hat g(\xi) d\xi.
\end{equation}
Applying \eqref{sft} and \eqref{fls3} in Proposition \ref{P:slapft}, we find
\begin{align*}
& \int_{\Rn} (-\Delta)^s u(x) v(x) dx = \int_{\Rn} (-\Delta)^s u(x) \F(\F^{-1}v)(x) 
dx = \int_{\Rn} \F((-\Delta)^s u)(\xi) \F^{-1}v(\xi) 
d\xi
\\
& = \int_{\Rn} (2\pi |\xi|)^{2s} \hat u(\xi)  \F^{-1}v(\xi) = \int_{\Rn}  \hat u(\xi)  (2\pi |\xi|)^{2s} \F^{-1}v(\xi) 
d\xi.
\end{align*}
Using \eqref{fls3} again we have
\begin{equation}\label{ifls3}
\F^{-1}((-\Delta)^s v)(\xi) = (2\pi |\xi|)^{2s} \F^{-1} v(\xi).
\end{equation}
Inserting this information in the above equation, and applying \eqref{sft} again, we find
\begin{align*}
& \int_{\Rn} (-\Delta)^s u(x) v(x) dx = \int_{\Rn}  \hat u(\xi) \F^{-1}((-\Delta)^s v)(\xi) 
d\xi
\\
& = \int_{\Rn} \F^{-1}(\hat u)(x) (-\Delta)^s v(x) 
dx = \int_{\Rn} u(x) (-\Delta)^s v(x) dx.
\end{align*}

\end{proof}

\begin{remark}\label{R:bogdan} 
For an extension of Lemma \ref{L:ibp} one should see \cite{BB2}.
\end{remark}

We next turn to computing explicitly the constant $\gamma(n,s)$ in \eqref{gns}.

\begin{prop}\label{P:gns} Let $0<s<1$. Then, we have
\begin{equation}\label{gnsfin}
\gamma(n,s) = \frac{s 2^{2s} \G\left(\frac{n+ 2s}{2}\right)}{\pi^{\frac n2} \G(1-s)}.
\end{equation}
\end{prop}

\begin{proof}
We use the beautiful idea of Bochner in his proof of Theorem \ref{T:Fourier-Bessel} above. If we denote by $\theta\in [0,\pi]$ the angle that the vector $z\in \Rn\setminus \{0\}$ forms with the positive direction of the $z_n$-axis, then Cavalieri's principle, and Fubini's theorem, give
\begin{align*}
\int_{\Rn} \frac{1 - \cos z_n}{|z|^{n+2s}} dz & = \int_0^\infty \int_{\mathbb S^{n-1}} \frac{1 - \cos(r\cos \theta)}{r^{n+2s}} d\sigma r^{n-1} dr
\\
& = \int_0^\infty \frac{1}{r^{1+2s}} \int_0^{\pi} \left[1 - \cos(r\cos \theta)\right] \int_{L_{\theta}} d\sigma' d\theta dr,
\end{align*} 
where we have indicated by $L_\theta = \{y\in \mathbb S^{n-1}\mid <y,e_n> = \cos \theta\}$ the $(n-2)$-dimensional sphere in $\Rn$ with radius $\sin \theta$ obtained by intersecting $\mathbb S^{n-1}$ with the hyperplane $y_n = \cos \theta$. Since with $\sigma_{n-2}$ given by \eqref{sn1} above we have
\[
\int_{L_{\theta}} d\sigma' = \sigma_{n-2} (\sin \theta)^{n-2},
\] 
we obtain
\begin{align}\label{good}
& \int_{\Rn} \frac{1 - \cos z_n}{|z|^{n+2s}} dz  =  \sigma_{n-2}  \int_0^\infty \frac{1}{r^{1+2s}} \int_0^{\pi} \left[1 - \cos(r\cos \theta)\right] (\sin \theta)^{n-2} d\theta dr
\\
& = \sigma_{n-2} \int_0^\infty \frac{1}{r^{1+2s}} \int_0^{\pi} \left[1 - \cos(r\cos \theta)\right] (1 -\cos^2 \theta)^{\frac{n-3}{2}} \sin \theta d\theta dr\ \ \ (\text{set}\ u = \cos \theta)
\notag\\
& =  \sigma_{n-2}  \int_0^\infty \frac{1}{r^{1+2s}} \int_{-1}^{1} \left[1 - \cos(r u)\right] (1 -u^2)^{\frac{n-3}{2}} du dr
\notag\\
& = \sigma_{n-2} \int_0^\infty \frac{1}{r^{1+2s}} \left[\int_{-1}^{1}(1 -u^2)^{\frac{n-3}{2}} du - \int_{-1}^{1} \cos(r u) (1 -u^2)^{\frac{n-3}{2}} du\right] dr.
\notag
\end{align} 

From \eqref{beta} and \eqref{beta2} we thus find
 \[ \int^1_{-1} (1-s^2)^{\frac{2\nu-1}2}ds = 2 \int_0^1 (\cos
\theta)^{2\nu} d\theta = B\left(\nu+\frac{1}{2},\frac{1}{2}\right) =
\frac{\Gamma\left(\nu+\frac{1}{2}\right)\Gamma\left(\frac{1}{2}\right)}{\Gamma(\nu
+ 1)}.
\] 
This gives
\[
\int_{-1}^{1}(1 -u^2)^{\frac{n-3}{2}} du = \frac{\G(\frac{n-1}{2})\Gamma(\frac 12)}{\G(\frac n2)}.
\]
On the other hand, we have
\[
\int_{-1}^{1} \cos(r u) (1 -u^2)^{\frac{n-3}{2}} du = \int_{-1}^{1} e^{i r u} (1 -u^2)^{\frac{n-3}{2}} du.
\]
From this equation and \eqref{extbessel} in Definition \ref{D:extbessel} we obtain with $\nu = \frac{n-2}{2}$ and $z = r$, 
\[
\int_{-1}^{1} \cos(r u) (1 -u^2)^{\frac{n-3}{2}} du = \G(\frac{n-1}{2}) \Gamma(\frac 12)  \left(\frac{2}{r }\right)^{\frac{n-2}{2}} J_{\frac{n-2}{2}}(r)
\]
Substituting in \eqref{good} above, we find
\begin{align*}
\int_{\Rn} \frac{1 - \cos(z_n)}{|z|^{n+2s}} dz & = \sigma_{n-2} \frac{\G(\frac{n-1}{2})\Gamma(\frac 12)}{\G(\frac n2)} \int_0^\infty \frac{1}{r^{1+2s}} \left[1 - \G(\frac n2)\left(\frac{2}{r }\right)^{\frac{n-2}{2}} J_{\frac{n-2}{2}}(r)\right] dr.
\end{align*}
Keeping \eqref{sn1} in mind, which gives
\[
\sigma_{n-2} = \frac{2 \pi^{\frac{n-1}{2}}}{\G(\frac{n-1}{2})},
\]
and that $\sqrt{\pi} = \G(1/2)$, we conclude that
\[
\int_{\Rn} \frac{1 - \cos(z_n)}{|z|^{n+2s}} dz = \sigma_{n-1}  \int_0^\infty \frac{1}{r^{1+2s}} \left[1 - \G(\frac n2)\left(\frac{2}{r }\right)^{\frac{n-2}{2}} J_{\frac{n-2}{2}}(r)\right] dr.
\]
From this equation and \eqref{gns} in Proposition \ref{P:slapft} above, it is clear that the constant $\gamma(n,s)$ must be chosen so that
\begin{equation}\label{gammans}
\gamma(n,s) \sigma_{n-1}  \int_0^\infty \frac{1}{r^{1+2s}} \left[1 - \G(\frac n2)\left(\frac{2}{r }\right)^{\frac{n-2}{2}} J_{\frac{n-2}{2}}(r)\right] dr = 1.
\end{equation}
In order to complete the proof, we are thus left with computing explicitly the integral in the right-hand side of \eqref{gammans}. 

With $\nu = \frac n2 - 1$, consider now the function
\[
\Psi_\nu(r) = 1 - \G(\nu+1)\left(\frac{2}{r }\right)^{\nu} J_{\nu}(r).
\]
From the series expansion of $J_\nu(r)$, see \eqref{psbessel} above, we have
\[
J_\nu(r) = \frac{\left(\frac r2\right)^\nu}{\G(\nu+1)} - \frac{\left(\frac r2\right)^{\nu+2}}{\G(\nu+2)} + \frac{\left(\frac r2\right)^{\nu+4}}{\G(\nu+3)} -...
\]
This expansion gives for some function $h(r) = O(r^2)$ as $r\to 0$,
\begin{equation}\label{lim}
\Psi_\nu(r) = \left(1+h(r)\right)\left(\frac r2\right)^{2}.
\end{equation}
On the other hand, \eqref{jnuinfty2} implies that as $r\to \infty$
\begin{equation}\label{lim2}
\Psi_\nu(r) = 1 + O(r^{-(\nu+\frac 12)}), 
\end{equation}
and thus, in particular, $\Psi_\nu\in L^\infty[0,\infty)$.
We thus find
\begin{align*}
& \int_0^\infty \frac{1}{r^{1+2s}} \left[1 - \G(\frac n2)\left(\frac{2}{r }\right)^{\frac{n-2}{2}} J_{\frac{n-2}{2}}(r)\right] dr = \int_0^\infty \left(\frac{r^{-2s}}{-2s}\right)' \Psi_\nu(r) dr
\\
& = \underset{R\to \infty, \varepsilon \to 0^+}{\lim} \int_\varepsilon^R \left(\frac{r^{-2s}}{-2s}\right)' \Psi_\nu(r) dr = - \underset{R\to \infty}{\lim}  \frac{R^{-2s}}{2s} \Psi_\nu(R) + \underset{\varepsilon\to 0^+}{\lim}  \frac{\varepsilon^{-2s}}{2s} \Psi_\nu(\varepsilon)
\\
& +  \int_0^\infty \frac{r^{-2s}}{2s} \Psi_\nu'(r) dr.
\end{align*}
Since as we have observed $\Psi_\nu\in L^\infty[0,\infty)$, we clearly have
\[
\underset{R\to \infty}{\lim}  \frac{R^{-2s}}{2s} \Psi_\nu(R) = 0.
\]
From \eqref{lim} and the fact that $0<s<1$, we obtain
\[
\underset{\varepsilon\to 0^+}{\lim}  \frac{\varepsilon^{-2s}}{2s} \Psi_\nu(\varepsilon) = 0.
\]
We thus infer that
\[
\int_0^\infty \frac{1}{r^{1+2s}} \left[1 - \G(\frac n2)\left(\frac{2}{r }\right)^{\frac{n-2}{2}} J_{\frac{n-2}{2}}(r)\right] dr =  \frac{1}{2s} \int_0^\infty \frac{1}{r^{2s}} \Psi'_\nu(r) dr.
\]
On the other hand, the recursion formula for $J_\nu$, see e.g. (5.3.5) on p. 103 in \cite{Le},
\[
(z^{-\nu} J_\nu(z))' = - z^{-\nu} J_{\nu+1}(z),
\]
 gives
\[
\Psi'(r) = - 2^\nu \G(\nu+1)\left(r^{- \nu} J_{\nu}(r)\right)' = 2^\nu \G(\nu+1)r^{- \nu} J_{\nu+1}(r).
\]
We thus find
\[
\int_0^\infty \frac{1}{r^{1+2s}} \left[1 - \G(\frac n2)\left(\frac{2}{r }\right)^{\frac{n-2}{2}} J_{\frac{n-2}{2}}(r)\right] dr =  \frac{2^\nu \G(\nu+1)}{2s} \int_0^\infty \frac{1}{r^{\frac n2 - 1 +2s}} J_{\frac n2}(r) dr.
\]
Recalling that $\nu = \frac n2 - 1$ we can write the right-hand side as follows
\[
\frac{2^\nu \G(\nu+1)}{2s} \int_0^\infty \frac{1}{r^{\nu +2s}} J_{\nu +1}(r) dr = \frac{2^\nu \G(\nu+1)}{2s} \int_0^\infty \frac{1}{r^{\mu - q}} J_{\mu}(r) dr,
\]
where $\mu = \nu+1 = \frac n2$, and $q = 1-2s$. We now invoke the following result,  which is formula (17) on p. 684 in \cite{GR}:
\begin{equation}\label{17}
\int_0^\infty \frac{1}{r^{\mu - q}} J_{\mu}(ar) dr = \frac{\G(\frac{q+1}{2})}{2^{\mu-q} a^{q-\mu+1} \G(\mu-\frac q2 + \frac 12)},
\end{equation}
provided that
\[
- 1 < \Re q < \Re \mu - \frac 12.
\]
With the above values of the parameters $\mu$ and $q$ this condition becomes
\[
- 1 < 1-2s< \frac n2 - \frac 12.
\]
Now, the former inequality is satisfied since it is equivalent to $s<1$, and the second is also satisfied since it is equivalent to $s>\frac{1-n}{4}$, which is of course true since $s>0$, whereas $\frac{1-n}{4}\le 0$. In conclusion, we obtain from \eqref{17}
\[
\frac{2^\nu \G(\nu+1)}{2s}   \int_0^\infty \frac{1}{r^{\nu +2s}} J_{\nu +1}(r) dr = \frac{2^{\frac n2 - 1} \G(\frac n2)}{2s}  \frac{\G(1-s)}{2^{\frac n2 - 1 + 2s} \G(\frac n2 + s)} = \frac{\G(\frac n2)}{2s}  \frac{\G(1-s)}{2^{2s} \G(\frac n2 + s)}.
\]
Returning to \eqref{gammans}, and keeping the first identity in \eqref{sn1} in mind, we reach the conclusion that the constant $\gamma(n,s)$ is given by the equation 
\[
\gamma(n,s)  \frac{2 \pi^{\frac n2}}{\G(\frac n2)} \frac{\G(\frac n2)}{2s}  \frac{\G(1-s)}{2^{2s} \G(\frac n2 + s)} = 1,
\]
which finally gives 
\[
\gamma(n,s) = \frac{s 2^{2s} \G(\frac n2 + s)}{\pi^{\frac n2} \G(1-s)}.
\]
This proves \eqref{gnsfin}, thus completing the lemma.

\end{proof}


\section{The fractional Laplacean and Riesz transforms}\label{S:riesz}

In this section we pause for discussing some interesting consequences of Corollary \ref{C:semi} and Lemma \ref{L:ibp}. In analysis and geometry one is interested in the following basic question. Consider a $n$-dimensional Riemannian manifold $M$, with gradient $\nabla$ and Laplacean $\Delta$. Given $1<p<\infty$, when is it true that the two Sobolev spaces of order one obtained by completion of $C^\infty_0(M)$ with respect to the seminorms $ \|\nabla f \|_{L^p(M)}$ and $\|\Delta^{1/2} f \|_{L^p(M)}$ coincide? This question is important for the purpose of developing analysis on the manifold $M$ and was raised in 1983 by R. Strichartz  in \cite{Str}.  

In order to understand it, let us remain within the familiar surroundings of flat Euclidean space, i.e., when $M = \Rn$. 
If we denote by $(\cdot,\cdot)$ the inner product in $L^2(\Rn)$, a simple integration by parts shows that for any $u\in \mathscr S(\Rn)$ we have
\[
(-\Delta u,u) = (\nabla u,\nabla u) = ||\nabla u||^2_{L^2(\Rn)}.
\]
However, applying Corollary \ref{C:semi} and Lemma \ref{L:ibp} we find 
\[
(-\Delta u,u) = ((-\Delta)^{1/2} (-\Delta)^{1/2}u,u) = ((-\Delta)^{1/2}u,(-\Delta)^{1/2}u) = ||(-\Delta)^{1/2}u||^2_{L^2(\Rn)}. 
\]
The reader should not underestimate the apparent simplicity of the latter conclusion. In a way, it is rather unintuitive that composing two nonlocal operators, such as $(-\Delta)^{1/2}$, we obtain a local operator.
 
Combining the latter two equations we obtain the remarkable conclusion
\begin{equation}\label{nice}
||\nabla u||^2_{L^2(\Rn)} = ||(-\Delta)^{1/2}u||^2_{L^2(\Rn)}.
\end{equation}
We emphasize at this point that \eqref{nice} allows to identify the first-order Sobolev subspaces of $L^2(\Rn)$ obtained by completion of $C^\infty_0(\Rn)$ with respect to the seminorms $||\nabla u||^2_{L^2(\Rn)}$ and $||(-\Delta)^{1/2}u||^2_{L^2(\Rn)}$. For the definition of the latter we refer the reader to \eqref{fss} below.

However, when $1<p<\infty$ and $p \neq 2$, a similar identification with respect to the seminorms $ \|\nabla f \|_{L^p(\Rn)}$ and $\|(-\Delta)^{1/2} f \|_{L^p(\Rn)}$ is no longer such a simple matter. 
It is a easy to recognize that an estimate such as 
\begin{align}\label{RZ}
A_p \|(-\Delta)^{1/2} f \|_{L^p(\Rn)} \le \| \nabla f \|_{L^p(\Rn)} \le B_p \| (-\Delta)^{1/2} f \|_{L^p(\Rn)},\ \ \ \ \ f \in C_0^\infty(\Rn),
\end{align}
would suffice for such identification. It is also easy to see (by a duality argument) that the validity of the right-hand inequality in \eqref{RZ} for a certain $1<p<\infty$ implies that of  the left-hand inequality in $L^{p'}(\Rn)$, where $\frac 1p + \frac 1p' = 1$. 

It is at this point that the Riesz transform enters the stage. One operator that occupies a central position in analysis is the $k$-th \emph{Riesz transform} $R_k$, which, on the Fourier transform side, is defined by the formula
\[
\widehat{R_k u}(\xi) = i \frac{\xi_k}{|\xi|} \hat u(\xi),\ \ \ \ \ \ \ \ k = 1,...,n.
\]
The vector-valued Riesz transform $\mathcal R = (R_1,...,R_n)$ are the first basic examples of singular integrals, as they generalize to dimension $n\ge 2$ the classical Hilbert transform, see \cite{St}. Using the Fourier transform it is immediate to verify that 
\begin{equation}\label{rieszt}
R_k = \frac{\p}{\p x_k} (-\Delta)^{-1/2},\ \ \ \ \ \ \ \ \ \ \ k=1,...,n,
\end{equation} 
which in vector-valued form can be compactly written as $\mathcal R = \nabla (-\Delta)^{-1/2}$. If we now apply \eqref{nice} to the function $u = (-\Delta)^{-1/2} f$ (this is fine, if $f\in \mathscr S(\Rn)$), we have proved the following result.

\begin{prop}\label{P:rieszt}
The vector-valued Riesz transform $R$ maps $L^2(\Rn)$ to $L^2(\Rn)$, and one has for any $f\in L^2(\Rn)$ 
\[
||\mathcal R f||^2_{L^2(\Rn)} = ||f||^2_{L^2(\Rn)}.
\]
\end{prop}

Thus \eqref{nice} is equivalent to the $L^2(\Rn)$ continuity of the Riesz operator $\mathcal R$ (obviously, in $\Rn$ we could have proved Proposition \ref{P:rieszt} using the Fourier transform as well, but the above proof works as well in situations in which such tool is not available). In a similar way, the right-hand inequality in \eqref{RZ} is equivalent to the $L^p$ continuity of the Riesz operator $\mathcal R$. In conclusion, the inequality \eqref{RZ} is true  for all $1<p<\infty$ if  
\begin{equation}\label{RZ2}
||\mathcal R f||_{L^p(\Rn)} \le C_p ||f||_{L^p(\Rn)},\ \ \ \ \ f\in C^\infty_0(\Rn).
\end{equation}

These arguments show that, at least when $M = \Rn$, the above question whether the two Sobolev spaces of order one obtained by completion of $C^\infty_0(M)$ with respect to the seminorms $ \|\nabla f \|_{L^p(M)}$ and $\|(-\Delta)^{1/2} f \|_{L^p(M)}$ coincide can be answered affirmatively if one knows that \eqref{RZ2} holds. On the other hand, one of the major accomplishments of the theorem of singular integrals is precisely their continuity in $L^p(\Rn)$, $1<p<\infty$, and thus the opening question of this section admits an affirmative answer in $\Rn$. 

It was because of the above considerations that in the above cited paper \cite{Str} R. Strichartz asked what hypothesis on a Riemannian manifold $M$ would ensure the continuity of the Riesz operator $\mathcal R = \nabla (-\Delta)^{-1/2}$ in $L^p(M)$ for $1<p<\infty$. An interesting answer was given in 1987 by D. Bakry who proved in \cite{Bak0} that if the Ricci tensor of $M$ is bounded from below by a non negative constant then \eqref{RZ2}, and therefore \eqref{RZ} hold for every $1<p<\infty$. The reader should also consult the subsequent developments in the papers \cite{ACDH}, \cite{CD}, and the more recent generalization to sub-Riemannian geometry in \cite{BG2}. These developments are intimately connected to the framework introduced in Section \ref{S:gamma} below.


\section{The fractional Laplacean of a radial function}\label{S:flrad}

We have seen in Lemma \ref{L:sym} that when $u(x) = f(|x|)$, then $(-\Delta)^s u(x)$ also has spherical symmetry. The next result provides a useful recipe for actually computing such function. It constitutes the non-local replacement of the well-known formula $\Delta u(x) = f''(|x|) + \frac{n-1}{|x|} f'(|x|)$ of the Laplacean of a spherically symmetric function.

\begin{lemma}\label{L:ftss}
Let $u(x) = f(|x|)$. Then, 
\begin{align}\label{ftss}
(-\Delta)^s u(x) & = \frac{(2\pi)^{2s+2}}{|x|^{\frac n2 - 1}} \int_0^\infty t^{2s+1} J_{\frac n2 -1}(2\pi |x| t) \left(\int_0^\infty \tau^{\frac n2} f(\tau) J_{\frac n2 - 1}(2\pi t \tau) d\tau\right) dt.
\\
& = |x|^{- \frac n2 - 2s - 1} \int_0^\infty t^{2s+1} J_{\frac n2 -1}(t) \left(\int_0^\infty \tau^{\frac n2} f(\tau) J_{\frac n2 - 1}(t |x|^{-1} \tau) d\tau\right) dt,
\notag
\end{align}
provided that the integrals exist and are convergent.
\end{lemma}

\begin{proof}
Let $U(x) = (-\Delta)^s u(x)$. We know from Lemma \ref{L:sym} that $U(x) = F(|x|)$.
We also know by \eqref{fls3} in Proposition \ref{P:slapft} that $\hat U(\xi) = (2\pi |\xi|)^{2s} \hat u(\xi)$. Combining this with Theorem \ref{T:Fourier-Bessel}, we find
\[
\hat U(\xi) = (2\pi)^{2s+1} |\xi|^{-\frac{n}2 +1 + 2s} \int^\infty_0 \tau^{\frac{n}2} f(\tau) J_{\frac{n}2 - 1}
(2\pi|\xi|\tau) d\tau.
\]
Applying again Theorem \ref{T:Fourier-Bessel} we obtain
\[
U(x) = (2\pi)^{2s+2} |x|^{-\frac n2 + 1} \int_0^\infty t^{\frac n2} J_{\frac n2 - 1}(2\pi|x| t) t^{-\frac{n}2 +1 + 2s} \left(\int_0^\infty \tau^{\frac n2} f(\tau) J_{\frac n2 - 1}(2\pi t \tau) d\tau\right) dt.
\]
This gives the desired conclusion \eqref{ftss}.

\end{proof}

For more elaborated representations related to Lemma \ref{L:ftss} the reader should see the paper \cite{FV}.


\section{The fundamental solution of $(-\Delta)^s$}\label{S:fs}

In this section we compute the fundamental solution of the fractional Laplacean operator. 
In most parts of this paper we will be implicitly assuming that the dimension of the ambient space is $n\ge 2$. Since $0<s<1$, this assumption obviously forces $s< \frac n2$. However, unlike the local case of the Laplacean, the situation when $n=1$ has its own interest when dealing with $(-\Delta)^s$ and at times it needs to be discussed separately. Theorem \ref{T:fs} below is one instance of this situation. The main reason is that, when $n=1$, then the case $s = \frac n2$ does occur when $s = \frac 12$. A remarkable study of the nondegeneracy and uniqueness for the nonlocal nonlinear equation 
\[
(-\Delta)^s u + u - |u|^\alpha u = 0,\ \ \ \ \ \ \ \ \ \ \alpha>0,
\]
entirely in the case $n=1$ is \cite{FLe13}. One should also see the sequel paper \cite{FLeS16} in which the authors obtain a generalization to any dimension $n\ge 1$.

Before we turn to the proof of the main results we pause for a moment to recall that there exist spaces larger than $\mathscr S(\Rn)$, or $L^\infty(\Rn)\cap C^2(\Rn)$, in which it is still possible to define the nonlocal Laplacean either pointwise or as a tempered distribution. We have seen an instance of this in Proposition \ref{P:silv} above. Following Definition 2.3 in \cite{Si}, given $0<s<1$ we can also consider the linear space of the functions $u\in C^\infty(\Rn)$ such that for every multi-index $\alpha\in \mi$
\[
[u]_{\alpha} = \underset{x\in \Rn}{\sup} \left(1+|x|^{n+2s}\right) |\p^\alpha u(x)| < \infty. 
\]
We denote by $\mathscr S_s(\Rn)$ the space $C^\infty(\Rn)$ endowed with the countable family of seminorms $[\cdot]_\alpha$, and by $\mathscr S'_s(\Rn)$ its  topological dual. We clearly have the inclusions
\begin{equation}\label{inclu1}
C^\infty_0(\Rn) \hookrightarrow \mathscr S(\Rn) \hookrightarrow \mathscr S_s(\Rn) \hookrightarrow
C^\infty(\Rn),
\end{equation}
with the dual inclusions given by
\begin{equation}\label{inclu2}
\mathscr E'(\Rn) \hookrightarrow \mathscr S'_s(\Rn) \hookrightarrow \mathscr S'(\Rn) \hookrightarrow
\mathscr D'(\Rn),
\end{equation}
where we recall that $\mathscr E'(\Rn)$ indicates the space of distributions with compact support. 
The next lemma justifies the introduction of the space $\mathscr S_s(\Rn)$.

\begin{lemma}\label{L:Ss}
Let $u\in \mathscr S(\Rn)$. Then,  $(-\Delta)^s u\in \mathscr S_s(\Rn)$.
\end{lemma}

\begin{proof}
We have already observed that $(-\Delta)^s u\in C^\infty(\Rn)$, and that it is not true in general that $(-\Delta)^s u\in \mathscr S(\Rn)$. From Proposition \ref{P:decay} we know however that 
\[
[(-\Delta)^s u]_{0} = \underset{x\in \Rn}{\sup} \left(1+|x|^{n+2s}\right) |(-\Delta)^s u(x)| < \infty. 
\]
Suppose now that $\alpha \in \mi$ and $|\alpha| = 1$. We can write $\alpha = e_k$, where $e_k$ indicate one of the vectors of the standard basis of $\Rn$. Applying \eqref{fls3} in Proposition \ref{P:slapft} and \eqref{derft}, we have
\begin{align*}
\p^\alpha (-\Delta)^s u(x) & = \p_k \F^{-1}\widehat{(-\Delta)^s u}(x) = (-2\pi i) \F^{-1}\left(\xi_k \widehat{(-\Delta)^s u}\right)(x) 
\\
& = (-2\pi i) \F^{-1} \left(\xi_k (2\pi |\xi|)^{2s} \hat u(\xi)\right) \ \ \ \  (\text{by}\ \eqref{ftder})
\\
& =  \F^{-1} \left((2\pi |\xi|)^{2s} \widehat{\p_k u}(\xi)\right) \ \ \ \ (\text{by}\ \eqref{fls3} \ \text{again})
\\
& =  \F^{-1} \F((-\Delta)^s \p_k u) = (-\Delta)^s \p_k u.
\end{align*}
Since $\p_k u\in \mathscr S(\Rn)$, again by Proposition \ref{P:decay} we conclude that 
\[
[u]_{e_k} = \underset{x\in \Rn}{\sup} \left(1+|x|^{n+2s}\right) |\p_k u(x)| < \infty. 
\]
Proceeding by induction on $|\alpha|$, for all $\alpha \in \mi$, we reach the desired conclusion.

\end{proof} 

With Lemma \ref{L:Ss} in hands we can now extend the notion of solution to distributional ones.

\begin{definition}\label{D:ds}
Let $T\in \mathscr S'(\Rn)$. We say that a distribution $u\in \mathscr S'_s(\Rn)$ solves $(-\Delta)^s u = T$ if for every test function $\vf\in \mathscr S(\Rn)$ one has
\[
<u,(-\Delta)^s \vf> = <T,\vf>.
\]
\end{definition}

In the special case in which $T = \delta$, the Dirac delta, then Definition \ref{D:ds} leads to the following.

\begin{definition}[Fundamental solution]\label{D:fs}
We say that a distribution $E_s \in \mathscr S'_s(\Rn)$ is a \emph{fundamental solution} of $(-\Delta)^s$ if $(-\Delta)^s E_s = \delta$. This means that for every $\vf\in \mathscr S(\Rn)$ one has
\[
<E_s,(-\Delta)^s \vf> = \vf(0).
\]
\end{definition}
It is clear from Definition \ref{D:fs} that if $E_s \in \mathscr S'_s(\Rn)$ is a fundamental solution of $(-\Delta)^s$, then  one has $(-\Delta)^s E_s = 0$ in $\mathscr D'(\Rn\setminus \{0\})$. The following result establishes the existence of an explicit fundamental solution $E_s\in C^\infty(\Rn\setminus\{0\})$ of $(-\Delta)^s$. As we will see in the important Theorem \ref{T:hypo} below, the smoothness of such $E_s$ in $\Rn\setminus\{0\}$ is in agreement with the above observed fact that $(-\Delta)^s E_s = 0$ in $\mathscr D'(\Rn\setminus \{0\})$.

\begin{theorem}\label{T:fs}
Let $n\ge 2$ and $0<s<1$. Denote by 
\begin{equation}\label{fs}
E_s(x) = 
\alpha(n,s) |x|^{-(n-2s)},
\end{equation}
where the normalizing constant in \eqref{fs} is given by
\begin{equation}\label{ans}
\alpha(n,s) = \frac{\G(\frac{n}2 -s)}{2^{2s} \pi^{\frac n2} \G(s)}.
\end{equation}
Then, $E_s$ is a fundamental solution for $(-\Delta)^s$. 
\end{theorem}

The proof of Theorem \ref{T:fs} will be given after Lemma \ref{L:fsreg} below. For such proof we have chosen one approach that, although very classical, to the best of our knowledge has not been pursued elsewhere. We have done so since, in the course of proving Theorem \ref{T:fs}, we establish some auxiliary results that have an interest in their own right (but also play an important role later in this note, see \eqref{pk5} in the proof of Theorem \ref{T:cs} below). In particular, we are led to discover in a natural way some remarkable solutions of the following semilinear nonlocal equation which generalizes the celebrated \emph{Yamabe equation} from Riemannian geometry
\begin{equation}\label{yamyam}
(-\Delta)^s u = u^{\frac{n+2s}{n-2s}}.
\end{equation}
Of course, there exist proofs of Theorem \ref{T:fs} different from the one presented here. Besides the original work of M. Riesz \cite{R}, the reader should also consult Stein's landmark book \cite{St} and Landkof's cited monograph \cite{La}.  
We begin with the following preparatory result. 

\begin{lemma}\label{L:fsreg0}
Suppose that either $n\ge 2$, or $n=1$ and $0<s<1/2$. For every $y>0$ consider the regularized fundamental solution 
\begin{equation}\label{Ee20}
E_{s,y}(x) = \alpha(n,s) (y^2 + |x|^2)^{-\frac{n-2s}{2}}.
\end{equation}
Then, 
\begin{equation}\label{ft130}
\widehat{E_{s,y}}(\xi) = \frac{y^s}{2^{2s-1} \pi^{s} \G(s)} |\xi|^{-s} K_s(2 \pi y |\xi|),
\end{equation}
where we have denoted by $K_\nu$ the modified Bessel function of the third kind, see \eqref{Knu} above.
From \eqref{ft130} we obtain for every $\xi\not= 0$
\begin{equation}\label{ft1302}
\widehat{E_{s}}(\xi) = \underset{y\to 0^+}{\lim} \widehat{E_{s,y}}(\xi) = (2\pi|\xi|)^{-2s}.
\end{equation}
\end{lemma}

\begin{proof}
To prove \eqref{ft130} it suffices to show that for every $f\in \mathscr S(\Rn)$ we have
\begin{align}\label{sub}
<\widehat{E_{s,y}},f> & = \frac{y^s}{2^{2s-1} \pi^{s} \G(s)}  \int_{\Rn}  |\xi|^{-s} K_s(2 \pi y |\xi|)\ f(\xi) d\xi.
\end{align}
To establish \eqref{sub} we use the heat semigroup and Bochner's subordination. 
The idea is to start from the observation that for every $L>0$ and $\alpha > 0$ one has
\begin{equation}\label{L}
\int_0^\infty e^{-t L} t^\alpha \frac{dt}{t} = \frac{\G(\alpha)}{L^{\alpha}}.
\end{equation}
Using Fubini and \eqref{L} with $L = |\xi|^2 + y^2$, we obtain for any $\alpha>0$
\begin{align*}
& \int_0^\infty t^\alpha \left(\int_{\Rn} e^{-t(|\xi|^2 + y^2)} \hat f(\xi) d\xi\right) \frac{dt}{t}
\\
& = \int_{\Rn} \hat f(\xi) \left(\int_0^\infty t^\alpha e^{-t(|\xi|^2 + y^2)}  \frac{dt}{t}\right) d\xi
\\
& = \G(\alpha) \int_{\Rn} \hat f(\xi) (|\xi|^2 + y^2)^{-\alpha} d\xi.
\end{align*}
The above assumptions $n\ge 2$, or $n=1$ and $0<s<1/2$, imply that $\alpha = \frac n2 -s>0$. If we thus let $\alpha = \frac n2 -s$ in the latter formula we find
\begin{equation}\label{ft10}
\int_0^\infty t^{\frac n2 -s} \left(\int_{\Rn} e^{-t(|\xi|^2 + y^2)} \hat f(\xi) d\xi\right) \frac{dt}{t} = \G\left(\frac{n -2s}{2}\right) \int_{\Rn} \hat f(\xi) (|\xi|^2 + y^2)^{-(\frac{n-2s}{2})} d\xi.
\end{equation}
On the other hand, \eqref{sft} above gives for any $f\in \mathscr S(\Rn)$ and $y>0$  
\begin{align*}
& \int_{\Rn} \F_{x\to \xi}\left(e^{-t(|x|^2 + y^2)}\right) f(\xi) d\xi = \int_{\Rn} e^{-t(|\xi|^2 + y^2)} \hat f(\xi) d\xi.
\end{align*}
Multiplying both sides of this equation by $t^{\frac n2 -s}$ and integrating between $0$ and $\infty$ with respect to the dilation invariant measure $\frac{dt}{t}$ we obtain
\begin{align*}
& \int_0^\infty t^{\frac n2 -s} \int_{\Rn} \F_{x\to \xi}\left(e^{-t(|x|^2 + y^2)}\right) f(\xi) d\xi \frac{dt}{t}  = \int_0^\infty t^{\frac n2 -s} e^{-y^2 t} \int_{\Rn} \widehat{e^{-t |\cdot|^2}}(\xi)\ f(\xi) d\xi \frac{dt}{t}  
\end{align*}
We next recall the following notable Fourier transform in $\Rn$: for every $t>0$, and every $\xi\in \Rn$, one has 
\begin{equation}\label{gaussdil0}
\widehat{(e^{- t |\cdot|^2})}(\xi) =
\frac{\pi^{\frac{n}{2}}}{t^{\frac n2}} \exp\left(-\pi^2
\frac{|\xi|^2}{t}\right).
\end{equation}
Substituting \eqref{gaussdil0} in the preceeding formula, we find
\begin{align*}
& \int_0^\infty t^{\frac n2 -s} \int_{\Rn} \F_{x\to \xi}\left(e^{-t(|x|^2 + y^2)}\right)\ f(\xi) d\xi \frac{dt}{t} 
\\
& = \pi^{\frac{n}{2}} \int_0^\infty t^{-s} e^{-y^2 t} \int_{\Rn} \exp\left(-\pi^2
\frac{|\xi|^2}{t}\right) f(\xi) d\xi \frac{dt}{t} 
\\
& =  \pi^{\frac{n}{2}} \int_{\Rn} f(\xi)\left(\int_0^\infty t^{-s} e^{-y^2 t}  \exp\left(-\pi^2
\frac{|\xi|^2}{t}\right) \frac{dt}{t}\right) d\xi. 
\end{align*}

We now use the following formula that can be found in 9. on p. 340 of \cite{GR}
\begin{equation}\label{f9GR}
\int_0^\infty t^{\nu-1} e^{-(\frac{\beta}{t} + \gamma t)} dt = 2 \left(\frac \beta{\gamma}\right)^{\frac \nu{2}} K_\nu(2 \sqrt{\beta \gamma}),
\end{equation}
provided $\Re \beta, \Re \gamma >0$. Applying \eqref{f9GR} with
\[
\nu = - s,\ \ \beta = \pi^2 |\xi|^2,\  \ \gamma = y^2,
\]
and keeping in mind that, as we have already observed, $K_\nu = K_{-\nu}$ (see 5.7.10 in \cite{Le}), we find
\begin{equation}\label{f92}
\int_0^\infty t^{-s} e^{-y^2 t}  \exp\left(-\pi^2
\frac{|\xi|^2}{t}\right) \frac{dt}{t} = 2 \left(\frac{y}{\pi |\xi|}\right)^{s} K_{s}(2 \pi y |\xi|).
\end{equation}
Substituting \eqref{f92} in the above integral, we conclude
\begin{align}\label{lhs}
& \int_0^\infty t^{\frac n2 -s} \int_{\Rn} \widehat{e^{-t(|\cdot|^2 + y^2)}}(\xi)\ f(\xi) d\xi \frac{dt}{t}  = 2 \pi^{\frac{n}{2} -s} y^s  \int_{\Rn}  |\xi|^{-s} K_s(2 \pi y |\xi|)\ f(\xi) d\xi. 
\end{align}
Since the integral in the left-hand side of \eqref{lhs} equals that in the left-hand side of \eqref{ft10}, we finally have
\begin{equation}\label{ft11}
\alpha(n,s)  \int_{\Rn} \hat f(\xi) (|\xi|^2 + y^2)^{-(\frac{n-2s}{2})} d\xi = \alpha(n,s)  \frac{2 \pi^{\frac{n}{2} -s}y^s}{ \G(\frac{n -2s}{2})}  \int_{\Rn}  |\xi|^{-s} K_s(2 \pi y |\xi|)\ f(\xi) d\xi. 
\end{equation}
Recalling \eqref{ans}, which gives $\alpha(n,s) = \frac{\G(\frac{n}2 - s)}{2^{2s} \pi^{\frac n2} \G(s)}$, we infer from \eqref{ft11} that
\begin{equation}\label{ft12}
\alpha(n,s)  \int_{\Rn} \hat f(\xi) (|\xi|^2 + y^2)^{-(\frac{n-2s}{2})} d\xi = \frac{y^s}{2^{2s-1} \pi^{s} \G(s)}  \int_{\Rn}  |\xi|^{-s} K_s(2 \pi y |\xi|)\ f(\xi) d\xi. 
\end{equation}
Keeping \eqref{Ee20} in mind, we can rewrite \eqref{ft12} as follows
\[
<E_{s,y},\hat f> = \frac{y^s}{2^{2s-1} \pi^{s} \G(s)}  \int_{\Rn}  |\xi|^{-s} K_s(2 \pi y |\xi|)\ f(\xi) d\xi.
\]
Since by definition $<\widehat{E_{s,y}},f> = <E_{s,y},\hat f>$, we conclude that \eqref{sub} holds, thus completing the proof.

\end{proof}

We next prove a remarkable result concerning the function $E_{s,y}$ defined by \eqref{fs} and \eqref{ans} above.

\begin{lemma}\label{L:fsreg}
For every $y>0$ the function $E_{s,y}$ satisfies the equation
\begin{align}\label{yam10}
& (-\Delta)^s E_{s,y}(x) = y^{2s} \frac{\G(\frac n2 + s)}{\pi^{\frac n2} \G(s)} \left(y^2 + |x|^2\right)^{-(\frac n2 + s)}.
\end{align}
\end{lemma}

\begin{proof}
In order to establish \eqref{yam10} we begin by computing the function
\[
F_{s,y}(x) \overset{def}{=}  (-\Delta)^s E_{s,y}(x).
\]
With this objective in mind we appeal to \eqref{fls3}, which gives
\begin{equation}\label{fs2}
\widehat{F_{s,y}}(\xi) = \widehat{(-\Delta)^s E_{s,y}}(\xi) = (2\pi |\xi|)^{2s} \widehat{E_{s,y}}(\xi).
\end{equation} 
We now use \eqref{ft130} in Lemma \ref{L:fsreg0}. Inserting such equation in \eqref{fs2} we obtain
 \begin{equation}\label{fs4}
\widehat{F_{s,y}}(\xi) = (2\pi |\xi|)^{2s}  \frac{y^s}{2^{2s-1} \pi^{s} \G(s)} |\xi|^{-s} K_s(2 \pi y |\xi|) = \frac{2 y^{s} \pi^{s}}{\G(s)} |\xi|^{s} K_{s}(2\pi y |\xi|).
\end{equation} 
Using Theorem \ref{T:Fourier-Bessel} we find from \eqref{fs4}
\begin{align}\label{pk1}
& F_{s,y}(x) =\frac{4 y^{s} \pi^{s+1}}{\G(s)} \frac{1}{|x|^{\frac n2 - 1}} \int_0^\infty t^{\frac n2 +s} K_{s}(2\pi y t) J_{\frac n2 - 1}(2\pi |x| t) dt.
\end{align}
If we now let
\[
\la = - \frac n2 - s,\ \ \mu = s,\ \ \nu = \frac n2 - 1,
\]
then we can write the integral in the right-hand side of \eqref{pk1} in the form
\[
\int_0^\infty t^{-\la} K_\mu(at) J_\nu(bt) dt,
\]
with
\[
a =  2\pi y,\ \ \ b= 2\pi |x|.
\]
Under the assumption 
$\nu - \la + 1 > |\mu|$, that is presently equivalent to $n + s > s$,
which is obviously true, we can appeal to formula 3. in 6.576 on p. 693 in \cite{GR}. Such formula states that
\begin{align}\label{fs5}
& \int_0^\infty t^{-\la} K_\mu(at) J_\nu(bt) dt = \frac{b^\nu \G(\frac{\nu-\la +\mu +1}{2}) \G(\frac{\nu - \la - \mu + 1}{2})}{2^{\la +1} a^{\nu - \la + 1} \G(1+\nu)}
\\
& \times F\left(\frac{\nu-\la +\mu +1}{2},\frac{\nu - \la - \mu + 1}{2};\nu+1; - \frac{b^2}{a^2}\right),
\notag
\end{align}
where, we recall, $F(\alpha,\beta;\gamma;z)$ indicates the hypergeometric function $_2F_1(\alpha,\beta;\gamma;z)$, see Definition \ref{D:hyper} above. Since 
\[
\frac{\nu-\la +\mu +1}{2} = \frac n2 + s,\ \ \frac{\nu-\la -\mu +1}{2} = \frac n2,
\]
from \eqref{pk1} and \eqref{fs5} we obtain
\begin{align}\label{pk}
& \int_0^\infty t^{\frac n2 +s} K_{s}(2\pi y t) J_{\frac n2 - 1}(2\pi |x| t) dt = \frac{(2\pi |x|)^{\frac n2 - 1} \G(\frac n2 + s)}{2^{-\frac n2 - s + 1} (2\pi y)^{n+s}} 
\\
& \times F\left(\frac n2 + s,\frac{n}{2};\frac n2; - \frac{|x|^2}{y^2}\right).
\notag\end{align}
We now apply \eqref{fs6} to find
\[
F\left(\frac n2 + s,\frac{n}{2};\frac n2; - \frac{|x|^2}{y^2}\right) = \left(1 + \frac{|x|^2}{y^2}\right)^{-(\frac n2 + s)}.
\]
Inserting this information into \eqref{pk} we have
\begin{align}\label{pk2}
& \int_0^\infty t^{\frac n2 +s} K_{s}(2\pi y t) J_{\frac n2 - 1}(2\pi |x| t) dt = \frac{(2\pi |x|)^{\frac n2 - 1} \G(\frac n2 + s)}{2^{-\frac n2 - s + 1} (2\pi y)^{n+s}} \left(1 + \frac{|x|^2}{y^2}\right)^{-(\frac n2 + s)}.
\end{align}
From \eqref{pk1} and \eqref{pk2} we finally conclude
\begin{align*}
& F_{s,y}(x) = \frac{\G(\frac n2 + s)}{y^{n} \pi^{\frac n2} \G(s)} \left(1 + \frac{|x|^2}{y^2}\right)^{-(\frac n2 + s)} = \frac{y^{2s} \G(\frac n2 + s)}{\pi^{\frac n2} \G(s)} \left(y^2 + |x|^2\right)^{-(\frac n2 + s)}.
\end{align*}
This establishes \eqref{yam10}, thus completing the proof.

\end{proof}

We are now ready to provide the

\begin{proof}[Proof of Theorem \ref{T:fs}]
Our objective is establishing 
\begin{equation}\label{E}
\int_{\Rn} E_s(x) (-\Delta)^s \vf(x) dx = \vf(0).
\end{equation}
 for every test function $\vf\in \mathscr S(\Rn)$. 
We begin by observing that, since we are assuming that $n\ge 2$, we automatically have that $0< s < \frac n2$. For $y>0$ we now consider the regularization $E_{s,y}$ of the distribution $E_s$ defined by \eqref{fs} and \eqref{ans} above.
Notice that $E_{s,y} \in C^\infty(\Rn)$ and that it decays at $\infty$ like $|x|^{-(n-2s)}$. Since for $\vf\in \mathscr S(\Rn)$ we know from Lemma \ref{L:Ss} that $(\Delta)^s \vf\in \mathscr S_s(\Rn)$, it should be clear that Lebesgue dominated convergence theorem gives
\[
\int_{\Rn} E_{s,y}(x)\ (-\Delta)^s \vf(x) dx\ \longrightarrow\  \int_{\Rn} E_{s}(x)\ (-\Delta)^s \vf(x) dx
\]
as $y\to 0^+$. On the other hand, Lemma \ref{L:ibp} (which continues to be valid in the present situation) gives
\begin{equation}\label{exchange}
\int_{\Rn} E_{s,y}(x)\ (-\Delta)^s \vf(x) dx = \int_{\Rn} (-\Delta)^s E_{s,y}(x)\ \vf(x) dx.
\end{equation}
Therefore, in view of \eqref{exchange}, in order to complete the proof it will suffice to show that as $y\to 0^+$
\begin{equation}\label{fs1}
\int_{\Rn} (-\Delta)^s E_{s,y}(x)\ \vf(x) dx\ \longrightarrow\ \vf(0).
\end{equation}
To establish \eqref{fs1} we use \eqref{yam10} in Lemma \ref{L:fsreg} which gives
\begin{align*}
& \int_{\Rn} (-\Delta)^s E_{s,y}(x)\ \vf(x) dx = \frac{\G(\frac n2 + s)}{y^{n} \pi^{\frac n2} \G(s)}  \int_{\Rn} \left(1 + \frac{|x|^2}{y^2}\right)^{-(\frac n2 + s)} \vf(x) dx
\\
& = \frac{\G(\frac n2 + s)}{\pi^{\frac n2} \G(s)}  \int_{\Rn} \left(1 + |x'|^2\right)^{-(\frac n2 + s)} \vf(y x') dx'
\\
& \longrightarrow\ \vf(0)\ \frac{\G(\frac n2 + s)}{\pi^{\frac n2} \G(s)}  \int_{\Rn} \left(1 + |x'|^2\right)^{-(\frac n2 + s)}  dx',
\end{align*}
where in the last equality we have used Lebesgue dominated convergence theorem. To complete the proof of \eqref{fs1} it would be sufficient to prove that
\begin{equation}\label{fs8}
\frac{\G(\frac n2 + s)}{\pi^{\frac n2} \G(s)}   \int_{\Rn} \left(1 + |x'|^2\right)^{-(\frac n2 + s)}  dx' = 1.
\end{equation}
Now, the validity of \eqref{fs8} follows from a straightforward application of Proposition \ref{P:poisson} with the choice $a = n+2s, b = 0$. 

\end{proof}

\begin{remark}\label{R:cases}
Our proof of Theorem \ref{T:fs} uses the fact that $n-2s>0$. Therefore, besides the situation $n\ge 2$, for which this is automatically true, our proof continues to work when $n=1$ and $0<s<\frac 12$ since in such case we still have $n-2s>0$. It does not cover instead the following two cases:
\begin{itemize}
\item[(i)] $n = 1$ and $\frac 12 < s <1$;
\item[(ii)] $n=1$ and $s = \frac 12$.
\end{itemize} 
In case (i) formulas \eqref{fs}, \eqref{ans} continue to be valid unchanged, whereas in the case (ii) one has to replace them with the following  
\[
E_{s}(x) = - \frac{1}\pi \log |x|.
\]
The interested reader can find a discussion of such cases in the paper \cite{Bu}. 
\end{remark}


\section{The nonlocal Yamabe equation}\label{S:yamabe}

In the previous section we have proved that the function
\begin{equation*}
E_{s,y}(x) = \alpha(n,s) (y^2 + |x|^2)^{-\frac{n-2s}{2}}.
\end{equation*}
satisfies the equation
\begin{align*}
& (-\Delta)^s E_{s,y}(x) = y^{2s} \frac{\G(\frac n2 + s)}{\pi^{\frac n2} \G(s)} \left(y^2 + |x|^2\right)^{-(\frac n2 + s)},
\end{align*}
see \eqref{yam10}. If we consider 
\begin{equation}\label{vv}
v(x) = \la y^\gamma E_{s,y},
\end{equation}
where $\la>0$ and $\gamma\in \R$ are to be chosen in a moment, then it is clear that $v$ solves the equation
\begin{align*}
(-\Delta)^s v(x) & =  \la y^{\gamma + 2s} \frac{\G(\frac n2 + s)}{\pi^{\frac n2} \G(s)} \left(y^2 + |x|^2\right)^{-(\frac n2 + s)}
\\
& = \la y^{\gamma + 2s} \frac{\G(\frac n2 + s)}{\pi^{\frac n2} \G(s)} \left(\frac{1}{(y^2 + |x|^2)^{\frac n2 - s}}\right)^{\frac{\frac n2 + s}{\frac n2 - s}} 
\\
& =  \la y^{\gamma + 2s} \frac{\G(\frac n2 + s)}{\pi^{\frac n2} \G(s)} \left(\la y^{\gamma} \alpha(n,s) \frac{1}{(y^2 + |x|^2)^{\frac n2 - s}}\right)^{\frac{\frac n2 + s}{\frac n2 - s}} \left(\la y^{\gamma} \alpha(n,s)\right)^{- \frac{\frac n2 + s}{\frac n2 - s}} 
\\
& = \la y^{\gamma + 2s} \frac{\G(\frac n2 + s)}{\pi^{\frac n2} \G(s)} \left(\la y^{\gamma} \alpha(n,s)\right)^{- \frac{\frac n2 + s}{\frac n2 - s}} v(x)^{\frac{\frac n2 + s}{\frac n2 - s}}.
\end{align*}
We now choose $\gamma>0$ in such a way that the powers of $y$ add up to $0$, and we also choose $\la$ so that
\[
\la \frac{\G(\frac n2 + s)}{\pi^{\frac n2} \G(s)} \left(\la \alpha(n,s)\right)^{- \frac{\frac n2 + s}{\frac n2 - s}} = 1.
\]
For the first condition to be true, we must have
\[
\gamma = \frac{n-2s}{2},
\]
whereas the second condition will be valid if 
\[
\la^{\frac{4s}{n-2s}} = \frac{\frac{\G(\frac n2 + s)}{\pi^{\frac n2} \G(s)}}{\alpha(n,s)^{ \frac{n + 2s}{n - 2s}}}.
\]
Inserting this choices of $\gamma$ and $\la$ in \eqref{vv}, after some elementary computations we obtain 
\begin{equation}\label{vyam}
v_y(x) = \kappa(n,s) \left(\frac{y}{y^2 + |x|^2}\right)^{\frac{n-2s}{2}},\ \ \ \ \ \ y>0,
\end{equation}
with 
\begin{equation}\label{kyam}
\kappa(n,s) = 2^{\frac{n-2s}{2}} \left(\frac{\G(\frac{n+2s}2)}{\G(\frac{n-2s}2)}\right)^{\frac{n-2s}{4s}}.
\end{equation}

In conclusion, we have proved the following remarkable fact.

\begin{theorem}\label{T:yammy}
Let $n\ge 2$ and $0<s<1$. Then, for every $y>0$ the function $v = v_y$ defined by \eqref{vyam}, with $\kappa(n,s)$ given by \eqref{kyam}, solves the \emph{nonlocal Yamabe equation} 
\begin{equation}\label{nly}
(-\Delta)^s v = v^{\frac{n-2s}{n+2s}}.
\end{equation}
Every translation in $x$ of such function is also a solution. 
\end{theorem}

We close this section by recalling that in \cite{CLO} the authors proved that, given $n\ge 1$ and $0<s<\frac n2$, then every positive solution of the integral equation 
\[
u(x) = \int_{\Rn} \frac{u(y)^{\frac{n-2s}{n+2s}}}{|x-y|^{n-2s}} dy,
\]
such that $u\in L^{\frac{2n}{n-2s}}_{loc}(\Rn)$, is a translation of one of the functions in \eqref{vyam}. They also showed that, under the same hypothesis, an analogous conclusion holds for all positive solutions of the nonlocal Yamabe equation \eqref{nly}. Since when $n\ge 2$ and $0<s<1$ the condition $0<s<\frac n2$ is automatically fulfilled, Theorem 1.2 in \cite{CLO} provides a deep converse to Theorem \ref{T:yammy} above. We also mention the paper \cite{CT} in which the authors compute the sharp constant in the Sobolev embedding $H^{s,2}(\Rn) \hookrightarrow L^q(\Rn)$, where for any $0<s<n/2$, the Sobolev exponent $q$ is determined by the equation
\[
\frac 12 - \frac 1q = \frac sn,\ \ \ \ \ \ \text{or, equivalently,}\ \ \ \ q = \frac{2n}{n-2s}.
\] 
In their Theorem 1.1 they prove that the minimizers in the Sobolev embedding are of the form \eqref{vyam}, or translations of it. 
Here, the Sobolev space is the standard one
\begin{align}\label{sobs}
H^{s,2}(\Rn) & = \{u\in L^2(\Rn)\mid (-\Delta)^{s/2} u\in L^2(\Rn)\}
\\
& = \{u\in L^2(\Rn)\mid (1+|\xi|^2)^{s/2} \hat{u}\in L^2(\Rn)\}, 
\notag
\end{align}
see e.g. \cite{LM}, and also \cite{DPV}. Finally, for a beautiful introduction to the role of nonlocal operators in geometry the reader should see the paper \cite{dMG}. Also, for works on nonlocal equations and geometry one should see \cite{GZ}, \cite{CG11}, \cite{DSS}, \cite{BDS}, \cite{FF}, \cite{FGMT},  \cite{CLZ}, \cite{DDGW}. For related works on nonlocal nonlinear equation at interface of analysis and geometry one should see  \cite{Tan}, \cite{FLe13}, \cite{BCDS}, \cite{CR}, \cite{Sec13}, \cite{CS14}, \cite{CS15}, \cite{SV15}, \cite{Ab15}, \cite{FLeS16}, \cite{DMPS}, \cite{PS16}, \cite{Ab17}, \cite{JKS17}. The recent paper \cite{CLL} provides an interesting account of the method of moving planes applied directly to nonlocal equations, rather than going through the extension (discussed in the next Section \ref{S:pk}), as it was done for instance in \cite{BCDS}. This is an important aspect since it allows to cover the full range of fractional powers $s\in (0,1)$.


\section{Traces of Bessel processes: the extension problem}\label{S:pk}

When dealing with nonlocal operators such as $(-\Delta)^s$ a major difficulty is represented by the fact that they do not act on functions like differential operators do, but instead through nonlocal integral formulas such as \eqref{fls}. As a consequence, the rules of differentiation are not readily available, and in these notes we have already seen several instances of this obstruction. In this perspective it would be highly desirable to have some kind of procedure that allows to connect nonlocal problems to ones for which the rules of differential calculus are available. Exploring this connection is the principal objective of this section.

During the past decade there has been an explosion of interest in the analysis of nonlocal operators such as \eqref{fls} in connection with various problems from the applied sciences, analysis and geometry. The majority of these developments has been motivated by the remarkable 2007 ``extension paper" \cite{CS07} by Caffarelli and Silvestre. In that paper the authors introduced a method that allows to convert nonlocal problems in $\Rn$ into ones that involve a certain (degenerate) differential operator in $\R^{n+1}_+$.
Precisely, it was shown in \cite{CS07} that if for a given $0<s<1$ and $u\in \mathscr S(\Rn)$ one considers the function $U(x,y)$ that solves the following Dirichlet problem in the half-space $\R^{n+1}_+$:
\begin{equation}\label{ext2}
\begin{cases}
L_a U(x,y) = \operatorname{div}_{x,y}(y^{a} \nabla_{x,y} U) = 0\ \ \ \ \ \ \ x\in \Rn, y>0,
\\
U(x,0) = u(x),
\end{cases}
\end{equation}
where now $a = 1-2s$,
then one can recover $(-\Delta)^s u(x)$ by the following ``trace" relation
\begin{equation}\label{dn}
- \frac{2^{2s-1} \G(s)}{\G(1-s)} \underset{y\to 0^+}{\lim} y^{1-2s} \frac{\p U}{\p y}(x,y) = (-\Delta)^s u(x).
\end{equation}
Thus, remarkably, \eqref{dn} provides yet another way of characterizing $(-\Delta)^s u(x)$ as the weighted  \emph{Dirichlet-to-Neumann}  map of the extension problem \eqref{ext2}. 

Before turning to solving \eqref{ext2} and proving \eqref{dn}, we mention that, in connection with the extension problem, there is one reference that should be cited since, with a somewhat different perspective, it contains closely related circle of ideas. The 1965 paper \cite{MS} by Muckenhoupt and Stein does not seem well-known to people in the fractional community, or to workers in geometry. In that paper the authors developed a detailed analysis of the equation
\begin{equation}\label{ms}
\operatorname{div}(y^{2\lambda} \nabla u) = y^{2\lambda}\left(\frac{\p^2 u}{\p x^2} + \frac{\p^2 u}{\p y^2} + \frac{2\la}{y} \frac{\p u}{\p y}\right) = 0,\ \ \ \ \ \ \ \la>0,
\end{equation}
in the upper half-plane $\R_x\times\R^+_y$ (notice that \eqref{ms} is precisely the extension equation \eqref{ext2} above). They made substantial use of the conjugate equation 
\[
\frac{\p^2v}{\p x^2} + \frac{\p^2v}{\p y^2} - \frac{2\la}{y} \frac{\p v}{\p y} = 0,
\]
 and of the fact that if $u$ solves \eqref{ms}, then $v = y^{2\lambda} u_y$ solves the conjugate equation. Remarkably, they also proved a strong maximum principle in regions across the singular line $\{y=0\}$ for solutions to \eqref{ms} such that $u(x,-y) = u(x,y)$, see Theorem 1 in \cite{MS}. Many of these aspects have presently become common knowledge to users of the extension procedure.

We also mention that in probability the extension procedure was introduced by Molchanov and Ostrovskii in \cite{MO69}, see also the earlier related work by Spitzer \cite{S58} and the more recent paper by Kolsrud \cite{K89}. 
Although it is fair to say that the contribution of \cite{MO69} to the development of nonlocal operators in analysis and geometry is not nearly comparable to that of \cite{CS07}, it should be said that in the probabilistic literature there is a wealth of works that have developed thanks to \cite{MO69}.

We also want to emphasize another important aspect of the extension operator $L_a$. If we let $X = (x,y) \in \R^{n+1}$, then $L_a$ is a special example of the class of differential equations
\begin{equation}\label{fks}
\operatorname{div}_X(A(X) \nabla_X f) = 0,
\end{equation}
first studied by Fabes, Kenig and Serapioni in \cite{FKS}. For such equations the authors assumed that $X\to A(X)$ be a symmetric matrix-valued function with bounded measurable coefficients verifying the following degenerate ellipticity assumption for a.e. $X\in \R^{n+1}$ and every $\xi\in \R^{n+1}$:
\[
\la \omega(X) |\xi|^2 \le <A(X)\xi,\xi> \le \la^{-1} \omega(X) |\xi|^2,
\]
for some $\la >0$. Here, $\omega(X)$ is a so-called Muckenhoupt $A_2$-weight. This means that there exists a constant $A>0$ such that for any ball $B\subset \R^{n+1}$ one has
\[
\dashint_{B}  \omega(X) dX \din_B \omega(X)^{-1} dX \le A.
\]
 Under such hypothesis they established a strong Harnack inequality, and the local H\"older continuity of the weak solutions of \eqref{fks}. Now, the extension equation in \eqref{ext2} is a special case of \eqref{fks} since, given that $a = 1-2s\in (-1,1)$, the function $\omega(X) = \omega(X) = |y|^a$ is an $A_2$-weight in $\R^{n+1}$ (for a very nice introduction to Muckenhoupt $A_p$-weights the reader is referred to the classical paper \cite{CFe}). As a consequence, one can obtain quantitative information on solutions of $(-\Delta)^s u = 0$, say, from corresponding properties of solutions of the extension problem \eqref{ext2}. 

Suppose for instance we want to establish the scale invariant Harnack inequality on balls in $\Rn$ for solutions of $(-\Delta)^s u = 0$ that are globally nonnegative (this is an important hypothesis when dealing with nonlocal operators). We extend such a $u$ to a nonnegative function $U$ in $\R^{n+1}_+$ that solves \eqref{ext2} above. In view of \eqref{dn} and of the fact that $(-\Delta)^s u = 0$, we obtain for every $x\in \Rn$
\[
\underset{y\to 0^+}{\lim} y^{1-2s} \frac{\p U}{\p y}(x,y) = 0.
\]
This condition implies (after some work!) that if we reflect $U$ evenly in $y>0$, the resulting function is a nonnegative local solution in a ball in the thick space $\R^{n+1}$ of $L_aU = 0$. Therefore, the Harnack inequality established in \cite{FKS} holds in such ball for $U$. If in such inequality we set $y = 0$, using the fact that $U(x,0) = u(x)$, we obtain a corresponding Harnack inequality for $u$. This is one example of how the extension procedure is used to turn nonlocal problems into local ones, see Theorem 5.1 in \cite{CS07}. For a different approach based on probability the reader should see the papers \cite{BK} and \cite{BGR}. In connection with the extension procedure one should also see the works by Stinga and Torrea and by Nystr\"om and Sande. In \cite{ST10} using the theory of semigroups this method has been generalized to define $(-L)^s$, where  $L = \operatorname{div}(A(x)\nabla )$ is a variable coefficient elliptic operator in divergence form, whereas in \cite{NS} and \cite{ST} it has been generalized to the nonlocal heat operator $(\p_t -\Delta)^s$. We also mention the work \cite{CS} in which the authors develop a Schauder type regularity theory, both interior and at the boundary, in the Dirichlet and Neumann problems for the nonlocal operator $(-L)^s$, with $L$ as above.   

We now turn to the task of actually solving the extension problem. One key observation is that the second order degenerate elliptic equation in \eqref{ext2} can also be written in nondivergence form in the following way
\begin{equation}\label{ext2'}
\begin{cases}
 - \Delta_x U = \mathscr B_a U,\ \ \ \ \ \ \ \ \ \ \ \ \ \ \ \ \ (x,y)\in \R^{n+1}_+
\\
U(x,0) = u(x),\ \ \ \ \ \ \ \ \ \ \  \ \ \ \ \ \ x \in \Rn,
\\
U(x,y) \to 0,\ \text{as}\ y \to \infty,\ \ \ \ \ \ x\in \Rn,
\end{cases}
\end{equation}  
where we have denoted by 
\begin{equation}\label{Ba}
\mathscr B_a = \frac{\p^2}{\p y^2} + \frac ay \frac{\p }{\p y}
\end{equation}
 the generator of the \emph{Bessel semigroup} on $(\R^+,y^a dy)$. We will return to the discussion of this semigroup in Section \ref{S:bessel} below.

\begin{theorem}\label{T:cs}
Let $u\in \mathscr S(\Rn)$. Then, the solution $U$ to the extension problem \eqref{ext2} is given by
\begin{equation}\label{U}
U(x,y) = P_s(\cdot,y) \star u(x) = \int_{\Rn} P_s(x-z,y) u(z) dz,
\end{equation}
where 
\begin{equation}\label{Pfinal}
P_s(x,y) = \frac{\G(\frac n2 + s)}{\pi^{\frac n2}\G(s)} \frac{y^{2s}}{(y^2 + |x|^2)^{\frac{n+2s}{2}}}
\end{equation}
is the \emph{Poisson kernel} for the extension problem in the half-space $\R^{n+1}_+$.
For $U$ as in \eqref{U} one has 
\begin{equation}\label{dtn}
(-\Delta u)^s u(x) = - \frac{2^{2s-1} \G(s)}{\G(1-s)} \underset{y\to 0^+}{\lim} y^a \frac{\p U}{\p y}(x,y).
\end{equation}
\end{theorem}

\begin{proof}
Consider the extension problem \eqref{ext2}, written in the form \eqref{ext2'}. If we take a partial Fourier transform of the latter with respect to the variable $x\in \Rn$, we find
\begin{equation}\label{ep2}
\begin{cases} 
\frac{\p^2 \hat U}{\p y^2}(\xi,y) + \frac ay \frac{\p \hat U}{\p y}(\xi,y) - 4 \pi^2 |\xi|^2 \hat U(\xi,y)  = 0 \ \ \ \ \text{in}\ \R^{n+1}_+,
\\
\hat U(\xi,0) = \hat u(\xi),\ \ \ \ \hat U(\xi,y) \to 0,\ \text{as}\ y \to \infty,\ \ \ x\in \Rn,
\end{cases}
\end{equation}
where we have denoted
\[
\hat U(\xi,y) = \int_{\Rn} e^{-2\pi i <\xi,x>} U(x,y) dx.
\]

In order to solve \eqref{ep2} we fix $\xi\in \Rn\setminus\{0\}$, and with $Y(y) = Y_\xi(y) = \hat U(\xi,y)$, we write \eqref{ep2} as
\begin{equation}\label{ep30}
\begin{cases}
y^2 Y''(y) + a y Y'(y) -  4 \pi^2 |\xi|^2  y^2 Y(y) = 0,\ \ \ \ y\in \R^+,
\\
Y(0) = \hat u(\xi),
\\
Y(y) \to 0,\ \text{as}\ y \to \infty.
\end{cases}
\end{equation}  
Comparing \eqref{ep30} with the generalized modified Bessel equation in \eqref{genbessel} above we see that the former fits into the general form of the latter provided that
\[
\alpha = s,\ \ \ \ \gamma = 1,\ \ \ \ \nu = s,\ \ \ \ \beta = 2\pi |\xi|.
\]  
Thus, according to \eqref{besselcv}, two linearly independent solutions of \eqref{ep30} are given by
\[
u_1(y) = y^s I_s(2\pi |\xi| y),\ \ \ \ \ \ \ u_2(y) = y^s K_s(2\pi |\xi| y).
\]
It ensues that, for every $\xi \not= 0$, the general solution of \eqref{ep2}  is given by
\[
\hat U(\xi,y) = A y^s I_s(2\pi |\xi| y) + B y^s K_s(2\pi |\xi| y).
\]
The condition $\hat U(\xi,y) \to 0$ as $y \to \infty$ forces $A = 0$ (see e.g. formulas (5.11.9) and (5.11.10) on p. 123 of \cite{Le} for the asymptotic behavior at $\infty$ of $K_s$ and $I_s$), and thus 
\begin{equation}\label{hatU}
\hat U(\xi,y) = B y^s K_s(2\pi |\xi| y).
\end{equation}
Next, we use the condition $\hat U(\xi,0) = \hat u(\xi)$ to fix the constant $B$. When $y\to 0^+$ we have
\[
\hat U(\xi,y)  = B y^s K_s(2\pi |\xi| y) = B \frac \pi{2} \frac{y^s I_{-s}(2\pi |\xi| y) - y^s I_\nu(2\pi |\xi| y)}{\sin \pi s} \ \longrightarrow\ \frac{B \pi}{2 \G(1-s) \sin \pi s} (2\pi |\xi|)^{-s},
\]
Now from formula (5.7.1) on p. 108 of \cite{Le}, we have as $z\to 0$
\[
I_s(z) \cong \frac 1{\G(s+1)} \left(\frac{z}{2}\right)^s,\ \ \ \ I_{-s}(z) \cong \frac 1{\G(1-s)} \left(\frac{z}{2}\right)^{-s}.
\]
Using this asymptotic, along with the formula \eqref{sine} above, we find that as $y\to 0^+$,
\[
B y^s K_s(2\pi |\xi| y) \ \longrightarrow\ \frac{B \pi 2^{s-1}}{\G(1-s) \sin \pi s} (2\pi |\xi|)^{-s} = B 2^{s-1} \G(s) (2\pi |\xi|)^{-s}.
\]
In order to fulfill the condition $\hat U(\xi,0) = \hat u(\xi)$ we impose that the right-hand side of the latter equation equal $\hat u(\xi)$. For this to happen we must have
\[
B = \frac{(2\pi |\xi|)^{s} \hat u(\xi)}{2^{s-1} \G(s)}.
\]
Substituting such value of $B$ in \eqref{hatU}, we finally obtain
\begin{equation}\label{ep3}
\hat U(\xi,y) = \frac{(2\pi |\xi|)^{s} \hat u(\xi)}{2^{s-1} \G(s)} y^s K_s(2\pi |\xi| y).
\end{equation}
At this point we want to invert the Fourier transform in \eqref{ep3}. In fact, it is clear from the latter equation that the function $U(x,y)$ will be given by \eqref{U}, with $P_s(x,y)$ as in \eqref{Pfinal}, if we can show that
\begin{equation}\label{pk4}
\F^{-1}_{\xi\to x}\left(\frac{(2\pi |\xi|)^{s}}{2^{s-1} \G(s)} y^s K_s(2\pi |\xi| y)\right)= \frac{\G(\frac n2 + s)}{\pi^{\frac n2}\G(s)} \frac{y^{2s}}{(y^2 + |x|^2)^{\frac{n+2s}{2}}}.
\end{equation}
Since the function between parenthesis in the left-hand side of  \eqref{pk4} is spherically symmetric, proving \eqref{pk4} is equivalent to establishing  the following identity
\[
\F_{\xi\to x}\left(2 \pi^s |\xi|^{s} y^s K_s(2\pi |\xi| y)\right)= \frac{\G(\frac n2 + s)}{\pi^{\frac n2}} \frac{y^{2s}}{(y^2 + |x|^2)^{\frac{n+2s}{2}}}.
\]
In view of Theorem \ref{T:Fourier-Bessel}, the latter identity is equivalent to
\begin{equation}\label{pk5}
\frac{2^2 \pi^{s+1} y^s}{|x|^{\frac n2 - 1}} \int_0^\infty t^{\frac n2 +s} K_{s}(2\pi y t) J_{\frac n2 - 1}(2\pi |x| t) dt = \frac{\G(\frac n2 + s)}{\pi^{\frac n2}} \frac{y^{2s}}{(y^2 + |x|^2)^{\frac{n+2s}{2}}}.
\end{equation}
We are thus left with proving \eqref{pk5}. Remarkably, this identity has already been established in \eqref{pk2} above. Therefore, \eqref{pk5} does hold and, with it, \eqref{U} and \eqref{Pfinal} as well. 

In order to complete the proof of the theorem we are thus left with establishing \eqref{dtn}. With this objective in mind we note that in view of \eqref{fls3} in Proposition \ref{P:slapft}, proving 
\eqref{dtn} is equivalent to showing 
\begin{equation}\label{ep4}
(2\pi |\xi|)^{2s} \hat u(\xi) = - \frac{2^{2s-1} \G(s)}{\G(1-s)} \underset{y\to 0^+}{\lim} y^a \frac{\p \hat U}{\p y}(\xi,y).
\end{equation}
Keeping in mind that $a = 1-2s$, and using the formula
 \[
 K_s'(z) = \frac sz K_s(z) - K_{s+1}(z)
 \]
 (see (5.7.9) on p. 110 of \cite{Le}), we obtain
 \[
y^a \frac{\p \hat U}{\p y}(\xi,y) = \frac{(2\pi |\xi|)^{s+1} \hat u(\xi)}{2^{s-1} \G(s)} y^{1-s} \left\{\frac{2s}{(2\pi|\xi|) y} K_s(2\pi |\xi| y) - K_{s+1}(2\pi |\xi| y)\right\}.
\]
Since
\[
\frac{2s}{z} K_s(z) - K_{s+1}(z) = - K_{s-1}(z) = - K_{1-s}(z)
\]
 (again, by (5.7.9) on p. 110 of \cite{Le}), we finally have
\[
y^a \frac{\p \hat U}{\p y}(\xi,y) = -  \frac{(2\pi |\xi|)^{s+1} \hat u(\xi)}{2^{s-1} \G(s)} y^{1-s} K_{1-s}(2\pi |\xi| y).
\]
Now, as before, we have as $y\to 0^+$,
\[
y^{1-s} K_{1-s}(2\pi |\xi| y)\ \longrightarrow\ 2^{-s} \G(1-s)(2\pi |\xi|)^{s-1} .
\]
We finally reach the conclusion that
\[
y^a \frac{\p \hat U}{\p y}(\xi,y)\ \underset{y\to 0^+}{\longrightarrow}\ -  \frac{\G(1-s)}{2^{2s-1} \G(s)}  (2\pi |\xi|)^{2s} \hat u(\xi).
\]
This proves \eqref{ep4}, thus completing the proof. For an alternative proof of \eqref{dtn} see Remark \ref{R:otherproof} below.

\end{proof}

\begin{remark}\label{R:poissoncione}
Using Proposition \ref{P:poisson} with the choice $b=0$ and $a = n+2s$, it is easy to recognize from \eqref{Pfinal} that
\begin{equation}\label{poissoneone}
||P_s(\cdot,y)||_{L^1(\Rn)} = \int_{\Rn} P_s(x,y) dx = 1,\ \ \ \ \ \ \ \ \ \ \ \ \ \ \text{for every}\ y>0.
\end{equation}
\end{remark}

\begin{remark}\label{R:12}
Notice that when $s = \frac 12$ we have $a = 1 - 2s = 0$, and the extension operator $L_a$ becomes the standard Laplacean $L_a = \Delta_x + \frac{\p^2}{\p y^2}$ in $\R^{n+1}$. From formula \eqref{Pfinal} we obtain in such case 
\[
P_{1/2}(x,y) = \frac{\G(\frac{n+1}{2})}{\pi^{\frac{n+1}{2}}} \frac{y}{(y^2 + |x|^2)^{\frac{n+1}{2}}},
\]
which is in fact the standard Poisson kernel for the upper half-space $\R^{n+1}_+$, see e.g. \cite{SW}.
\end{remark}

\begin{remark}\label{R:yam}
If we compare the expression of the Poisson kernel in \eqref{Pfinal} with \eqref{yam10} in Lemma \ref{L:fsreg}, we conclude that, remarkably, we have shown that 
\begin{equation}\label{pfs}
P_s(x,y) = (-\Delta)^s E_{s,y}(x),
\end{equation}
where for $y>0$ the function $E_{s,y} = c(n,s) (y^2 + |x|^2)^{-\frac{n-2s}{2}}$ is the $y$-regularization of the fundamental solution of $(-\Delta)^s$. If we combine \eqref{pfs} with \eqref{fs1} above, we see that we can reformulate \eqref{fs1} as follows
\[
\underset{y\to 0^+}{\lim} P_s(\cdot,y) = \delta\ \ \ \ \ \ \ \ \text{in}\ \ \mathscr S'(\Rn),
\]
or, equivalently, for any $\vf\in \mathscr S(\Rn)$
\[
\underset{y\to 0^+}{\lim} \int_{\Rn} P_s(x,y)\vf(x) dx =  \vf(0).
\]
If we let $\check{\vf}(x) = \vf(-x)$, then we obtain from the latter limit relation
\begin{equation}\label{dirac}
P_s(\cdot,y) \star \vf(x) = \int_{\Rn} P_s(z,y)\tau_{-x}\check{\vf}(z) dz\ \longrightarrow\ \tau_{-x}\check{\vf}(0) = \vf(x).
\end{equation}
\end{remark}

\begin{remark}[Alternative proof of \eqref{dtn}]\label{R:otherproof}
Using the property \eqref{dirac} of the Poisson  kernel $P_s(x,y)$ we can provide another ``short" proof of \eqref{dtn} along the following lines, see Section 3.1 in \cite{CS07}. Let $u\in \mathscr S(\Rn)$ and consider the solution $U(x,y) = P_s(\cdot,y)\star u(x)$ to the extension problem \eqref{ext2}, see \eqref{U}. Using \eqref{poissoneone} we can write
\[
U(x,y) = \frac{\G(\frac n2 + s)}{\pi^{\frac n2}\G(s)} \int_{\Rn} \frac{y^{2s}}{(y^2 + |x-z|^2)^{\frac{n+2s}{2}}} (u(z) - u(x)) dz + u(x).
\]
Differentiating both sides of this formula with respect to $y$ and keeping in mind that $a = 1-2s$, we obtain that as $y\to 0^+$
\begin{align*}
y^a \frac{\p U}{\p y}(x,y) = 2s \frac{\G(\frac n2 + s)}{\pi^{\frac n2}\G(s)} \int_{\Rn}\frac{u(z) - u(x)}{(y^2 + |z-x|^2)^{\frac{n+2s}{2}}} dz + O(y^2).
\end{align*}
Letting $y\to 0^+$ and using Lebesgue dominated convergence theorem, we thus find
\begin{align*}
\underset{y\to 0^+}{\lim} y^a \frac{\p U}{\p y}(x,y) & = 2s \frac{\G(\frac n2 + s)}{\pi^{\frac n2}\G(s)} \operatorname{PV} \int_{\Rn}\frac{u(z) - u(x)}{|z-x|^{n+2s}} dz
\\
& = - 2s \frac{\G(\frac n2 + s)}{\pi^{\frac n2}\G(s)} \gamma(n,s)^{-1} (-\Delta u)^s u(x),
\end{align*}
where in the second equality we have used \eqref{fl2} above. If in the latter equation we now replace the expression \eqref{gnsfin} of the constant $\gamma(n,s)$, we reach the conclusion that \eqref{dtn} is valid.
\end{remark}

The Poisson kernel $P_s(x,y)$ is of course a solution of $L_a P_s = 0$ in $\R^{n+1}_+$. What is instead not obvious is that the $y$-regularization $E_{s,y}$ of the fundamental solution $E_s$ of $(-\Delta)^s$ introduced in  \eqref{Ee20} in Lemma \ref{L:fsreg0} is also a solution of the extension operator $L_a$. It was shown in \cite{CS07} that, up to a constant, such function is in fact the fundamental solution of $L_a$. The heuristic motivation behind this is that, with $x\in \Rn$,  and $\eta\in \R^{a+1}$, if $y = |\eta|$ then the operator
\begin{equation}\label{op}
y^{-a} L_a = \Delta_x + \frac{\p^2 }{\p y^2} + \frac{a}{y} \frac{\p}{\p y}
\end{equation}
can be thought of as the Laplacean in the fractional dimension $N = n+a+1$ acting on functions $U(x,|\eta|)$. Such heuristic is confirmed by the following result.

\begin{prop}
For $y\in \R$ consider the function $G(x,y) = (|x|^2 + y^2)^{-\frac{n-2s}{2}}$, see \eqref{Ee20}. Then, for every $(x,y)\in \R^{n+1}_+$, with $a = 1 - 2s$ we have
\[
L_a G(x,y) = 0.
\]
\end{prop}  

\begin{proof}
It is convenient to use the expression of \eqref{op} on functions depending on $r = |x|$ and $y$
\[
y^{-a} L_a = \frac{\p^2 }{\p r^2} + \frac{n-1}{r} \frac{\p}{\p r} + \frac{\p^2 }{\p y^2} + \frac{a}{y} \frac{\p}{\p y}.
\]
Then, the proof becomes a simple computation. Abusing the notation we write $G(x,y) = G(r,y) = (r^2 + y^2)^{-\frac{n-2s}{2}}$. We have
\[
G_r = -(n+a-1) (r^2 + y^2)^{-\frac{n+a-1}{2} - 1} r,
\]
\[
G_{rr} = (n+a-1) (r^2 + y^2)^{-\frac{n+a-1}{2} - 2} ((n+a)r^2 - y^2).
\]
This gives
\[
G_{rr} + \frac{n-1}{r} G_r =  (n+a-1) (r^2 + y^2)^{-\frac{n+a-1}{2} - 2} ((1+a)r^2 - n y^2).
\]
On the other hand, a similar computation gives
\[
G_{yy} + \frac{a}{y} G_y = - (n+a-1) (r^2 + y^2)^{-\frac{n+a-1}{2} - 2} ((1+a)r^2 - n y^2).
\]
Adding the latter two equations gives the desired conclusion $L_a G = 0$.

\end{proof}


\section{Fractional Laplacean and subelliptic equations}\label{S:flso}

In Section \ref{S:pk} we have analyzed the important fact \eqref{dn} that $s$-harmonic functions arise as weighted Dirichlet-to-Neumann traces of the solutions of the extension problem \eqref{ext2}. This aspect underscores the deep connection between the fractional Laplacean and the class of second order partial differential equations of degenerate type introduced in \cite{FKS}.

In this section we want to advertise another aspect of nonlocal equations, namely the link between the nonlocal operator $(-\Delta)^s$ and the theory of the so-called \emph{subelliptic equations}. This name comes from the fact that, although the relevant differential operator $L$ fails to satisfy the a priori estimates of the elliptic theory, it does satisfy the following replacement estimate \emph{below} the elliptic index, hence the name subelliptic:
\[
||u||_{H^{2\e}}\le C\left(||u||_{L^2} + ||Lu||_{L^2}\right),
\]
for all $C^\infty_0$ functions $u$, and for some $0<\e<1$. Subelliptic operators typically display loss of control of derivatives in a set of directions. 

The aspect that we have in mind originates with the following particular subelliptic operator
\[
\Ba = \frac{\p^2}{\p z^2} + |z|^{2\alpha} \frac{\p^2}{\p x^2},\ \ \ \ \ \ \ \alpha>0,
\]
 that was first introduced by S. Baouendi in 1967 in his Ph. D. Dissertation under the supervision of B. Malgrange, see \cite{Ba67}. At that time M. Vishik was visiting Malgrange, who discussed with him the thesis project of Baouendi. Vishik subsequently asked Malgrange permission to suggest to his own Ph. D. student, Grushin, to work on some questions related to the hypoellipticity of $\Ba$ when $\alpha \in \mathbb N$, see \cite{Gru1} and \cite{Gru2}. This is how the operator $\Ba$ became known as the \emph{Baouendi-Grushin operator}. A decade later, in the early 80's, Franchi and Lanconelli introduced a class of operators which include $\Ba$, and they pioneered the study of the fine properties of their weak solutions, such as the Harnack type inequality and the H\"older continuity, by studying a control distance associated with the relevant operators, see \cite{FL1}-\cite{FL5}, and also the subsequent work \cite{FS}. As we will see in this and the subsequent section, there is an underlying strong connection between these works, the paper \cite{FKS} of Fabes, Kenig and Serapioni mentioned in the previous section, and the fractional Laplacean $(-\Delta)^s$. We will further the discussion of the interconnection between these operators in Sections \ref{S:how} and \ref{S:almgren} below.

On one hand, we will see from Proposition \ref{P:flp} below that, at least in the range $0<s\le 1/2$, the fractional Laplacean arises as the true Dirichlet-to-Neumann map of the Baouendi-Grushin operator $\Ba$ defined in \eqref{bg0}. On the other hand, the  recent work of Koch, Petrosyan and Shi \cite{KPS} has underscored an even deeper link between nonlocal and subelliptic equations. The central tool in their study of the real-analytic smoothness of the regular free boundary in the obstacle problem for $(-\Delta)^{1/2}$ is a partial hodograph transformation. After such change of variables, they obtain a fully nonlinear partial differential equation which has a subelliptic structure. In the sense that the linearization of such fully nonlinear equation is precisely a Baouendi operator such as \eqref{ba} below with $\alpha = 1$, see Section 5 in \cite{KPS}. Before proceeding we mention that the $C^\infty$ smoothness of the regular free boundary in the obstacle problem for $(-\Delta)^{1/2}$ has also been proved with a completely different approach by De Silva and Savin in \cite{DS16}, see also \cite{DS15} for a related result in the Bernoulli problem. Their approach has been subsequently generalized to all $s\in (0,1)$ in \cite{JN17}.

These facts represent an interesting opportunity for interaction between two seemingly disjoint communities: that of workers in subelliptic equations and the closely connected field of sub-Riemannian geometry, and that of workers in nonlocal equations. We hope that the present discussion, as well as the content of Sections \ref{S:how} and \ref{S:almgren}, will encourage such exchange. Most of the material that follows is borrowed from the papers \cite{Ga}, \cite{CS07} and \cite{GRO}. 

To introduce our discussion let us consider the fractional Laplacean $(-\Delta)^s$ with $0<s<1$. With $a = 1-2s$, we have $-1<a<1$, and we have seen in \eqref{op} that the extension operator $L_a$ can be written in the form
\begin{equation}\label{la}
L_a = y^a \left(\Delta_x + \frac{\p^2}{\p y^2} + \frac ay \frac{\p }{\p y}\right).
\end{equation}
 
We now want to connect the operator $L_a$ to another degenerate elliptic operator. Following \cite{CS07} we introduce the change of variable $\Phi: \R^{n+1}_+ \to \R^{n+1}_+$ 
\begin{equation}\label{phi}
(x,z) = \Phi(x,y) = (x,h(y)),
\end{equation}
where the function $h(y)$ is chosen so to eliminate the drift term in \eqref{la}. To do this, given a function $U(x,z)$ defined for $(x,z)\in \us$, we define a function $\tilde U(x,y)$, with $(x,y) \in \us$, by the formula
\begin{equation}\label{tu}
\tilde U(x,y) := U(\Phi(x,y)) = U(x,h(y)).
\end{equation}
A simple computation gives
\[
L_a \tilde U(x,y) = y^a\left[\Delta_x U(x,h(y)) + \left(h''(y) + \frac ay h'(y)\right) D_z U(x,h(y)) + h'(y)^2 D_{zz} U(x,h(y))\right].
\]
From this equation it is clear that if $h(y)$ satisfies the differential equation
\begin{equation}\label{hde}
h''(y) + \frac ay h'(y)  \equiv 0,
\end{equation}
then we obtain
\begin{equation}\label{lab}
L_a \tilde U(x,y) =  (h^{-1}(z))^a\left[\Delta_x U(x,z) + h'(h^{-1}(z))^2 D_{zz} U(x,z)\right].
\end{equation}
Solving \eqref{hde} we find $h(y) = A y^{1-a}$ for some $A\in \R\setminus \{0\}$. We now choose $A$ so that $h'(h^{-1}(z)) = z^{-\frac{a}{1-a}}$, which gives $A = (1-a)^{-(1-a)}$. Summarizing, $h:(0,\infty)\to (0,\infty)$ is
the strictly increasing function given by
\begin{equation}\label{h}
z = h(y) = \left(\frac{y}{1-a}\right)^{1-a},\ \ \ \text{with inverse}\ \ \ \ y = h^{-1}(z) = (1-a) z^{\frac{1}{1-a}}.
\end{equation}
With this choice we have $h'(y) = (1-a)^a y^{-a}$, and thus we conclude from \eqref{lab} that
\begin{align}\label{lab2}
L_a \tilde U(x,y) & =  (1-a)^a z^{\frac{a}{1-a}} \left[\Delta_x U(x,z) + z^{-\frac{2a}{1-a}} D_{zz} U(x,z)\right].
\end{align}
The next proposition summarizes the content of \eqref{lab2}.

\begin{prop}\label{P:laba}
Let $-1<a<1$, and $\alpha = \frac{a}{1-a}\in (-1/2,\infty)$.
The mapping $U\leftrightarrow \tilde U$ defined by \eqref{tu}, with $h(y)$ given by \eqref{h},
converts in a one-to-one, onto fashion, solutions of the equation $L_a \tilde U = 0$ with respect to the variables $(x,y)\in \us$ into solutions with respect to the variables $(x,z)\in \us$ of the equation
\begin{equation}\label{bg0}
\mathcal P_\alpha U = D_{zz} U(x,z) + z^{2\alpha}  \Delta_x U(x,z)  = 0.
\end{equation}
\end{prop}

We note explicitly that, in the correspondence $U\leftrightarrow \tilde U$, the equation \eqref{lab2} can be more suggestively expressed as
\begin{equation}\label{labb2}
y^a L_a \tilde U(x,y) = \Ba u(x,z).
\end{equation}

Using Proposition \ref{P:laba} one obtains the following interesting result.

\begin{prop}\label{P:flp}
Given any $s\in (0,1)$, the fractional Laplacean $(-\Delta)^s$ in $\Rn$ can  be interpreted as the Dirichlet-to-Neumann map of the operator $\Ba$ in $\us$ defined in \eqref{bg0}, where $\alpha = \frac{1}{2s} - 1$. By this we mean that if for any $u\in \mathscr S(\Rn)$ one considers the solution $U(x,z)$ to the Dirichlet problem
\begin{equation}\label{bg00}
\begin{cases}
\mathcal P_\alpha U(x,z) = D_{zz} U(x,z) + z^{2\alpha}  \Delta_x U(x,z)  = 0,\ \ \ \ (x,z)\in \us,
\\
U(x,0) = u(x),
\end{cases}
\end{equation}
then one has
\begin{equation}\label{dtnbg}
(-\Delta)^s u(x) = - \frac{\G(1+s)}{\G(1-s)} \underset{z\to 0^+}{\lim} \frac{\p U}{\p z}(x,z).
\end{equation}
\end{prop}

\begin{proof}
Consider the solution to the Dirichlet problem \eqref{bg00}.
Denote by $\tilde U(x,y)$ the function associated to $U$ in the variables $(x,y)\in \us$ by the correspondence \eqref{tu}, where $h(y)$ is given by \eqref{h}, and $a$ and $\alpha$ are related by the equation $\alpha = \frac{a}{1-a}$, or equivalently $a = \frac{\alpha}{\alpha+1}$. Since $\alpha = \frac{1}{2s} - 1$, we see that  $a = 1-2s$. By Proposition \ref{P:laba} we know that $\tilde U$ satisfies $L_a \tilde U = 0$ in $\us$. Furthermore, we have $\tilde U(x,0) = U(x,0) = u(x)$. By \eqref{dn} we have
\begin{equation*}
(-\Delta)^s u(x) = - \frac{2^{2s-1} \G(s)}{\G(1-s)} \underset{y\to 0^+}{\lim} y^{1-2s} \frac{\p \tilde U}{\p y}(x,y).
\end{equation*}
On the other, \eqref{tu} and the chain rule give
\[
y^{1-2s} \frac{\p \tilde U}{\p y}(x,y) = y^{1-2s} \frac{\p U}{\p z}(x,h(y)) h'(y) = (2s)^{-2s+1} \frac{\p U}{\p z}(x,(2s)^{-2s} y^{2s}).
\]
From the latter two equations we obtain
\[
(-\Delta)^s u(x) = - \frac{2^{2s-1} \G(s)}{\G(1-s)} \underset{y\to 0^+}{\lim} y^{1-2s} \frac{\p \tilde U}{\p y}(x,y) = - \frac{2^{2s-1} \G(s)}{\G(1-s)}(2s)^{-2s+1} \underset{z\to 0^+}{\lim} \frac{\p U}{\p z}(x,z).
\]
Keeping \eqref{fact} in mind, we have thus proved \eqref{dtnbg}.

\end{proof}

We pause for a moment to emphasize that Proposition \ref{P:flp} shows that the true Dirichlet-to-Neumann map that defines the nonlocal operator $(-\Delta)^s$ is the one associated with the degenerate elliptic operator $\Ba$ in \eqref{bg0}. 
If we now restrict the attention to the regime $0<s\le 1/2$, then $\alpha = \frac{1}{2s} - 1 \ge 0$ and the operator in \eqref{bg0} is a model of the \emph{Baouendi-Grushin operators}
in $\Rn_x\times \Rm_z$ given by
 \begin{equation}\label{ba}
 \Ba=\Delta_z+ |z|^{2\alpha}\Delta_x,\ \ \ \ \ \ \alpha \ge 0.
\end{equation}
When $\alpha>0$ such operators are degenerate elliptic along the $n$-dimensional subspace $M = \Rn \times  \{0\}_{\R^m} $. They are the prototype of a class of equations that continues to be much studied nowadays. One of the reasons for such continuing interest is that, as we next illustrate, \eqref{ba} is closely connected with an object of fundamental relevance in harmonic analysis, partial differential equations and geometry, the Heisenberg group $\mathbb H^n$. 
This is the stratified nilpotent Lie group whose underlying manifold is $\mathbb C^n \times \R \cong \R^{2n+1}$ with noncommutative group law given in real coordinates by
\begin{equation}\label{affine}
(x,y,t) \circ (x',y',t') = (x+x',y+y',t+t' + \frac 12 (<x,y'> - <x',y>)).
\end{equation}
If we let $p = (x,y,t), p' = (x',y',t')\in \mathbb H^n$, and define the operator of left-translation by $L_p(p') = p\circ p'$, then denoting by $dL_p$ the differential of the map \eqref{affine}, a basis for the real Lie algebra of left-invariant differential operators is given by applying $dL_p$ to the standard basis in $\R^{2n+1}$. We thus find
\begin{align*}
& X_j = X_j(p) = dL_p(e_j) = \p_{x_j} - \frac{y_j}{2} \p_t,\ \ \ \ j=1,...,n,
\\
& X_{n+j} = X_{n+j}(p) = dL_p(e_{n+j}) = \p_{y_j} + \frac{x_j}{2} \p_t,\ \ \ \ j=1,...,n,
\\
& T = T(p) = dL_p(e_{2n+1}) = \p_t.
\end{align*}
These vector fields satisfy the commutation relations
\begin{equation}\label{comm}
[X_j,X_{n+k}] = \delta_{jk} T,\ \ \ j, k = 1,....,n,
\end{equation}
all other commutators being trivial (remember that the commutator of two vector fields $X$ and $Y$ is defined by $[X,Y] = XY - YX$). The name \emph{Heisenberg group} comes from the fact that, when $n=1$, then in $\mathbb H^1 = (\R^3,\circ)$ the resulting equation \eqref{comm} represents an abstract version of Heisenberg's canonical commutation relations in quantum mechanics for position and momentum of a relativistic particle, see \cite{Fo}. In fact, much before than mathematicians christened it with such name, the three-dimensional Heisenberg group had long been known to physicist as the \emph{Weyl's group}, and it was identified with the following group of $3\times3$ matrices:
\[
\left(
\begin{array}
[c]{ccc}
~1~ & ~x~   & ~z ~\\
~0~ & ~1~   & ~y ~\\
~0~ & ~0~   & ~1 ~
\end{array}
\right)  ,\text{ \ }x,y,z\in\mathbb{R}.
\]
The Lie algebra is clearly spanned by the matrices
\[
X=\left(
\begin{array}
[c]{ccc}
~0~ & ~1~ & ~0~\\
~0~ & ~0~ & ~0~\\
~0~ & ~0~ & ~0~
\end{array}
\right),\ Y=\left(
\begin{array}
[c]{ccc}%
~0~ & ~0~ & ~0~\\
~0~ & ~0~ & ~1~\\
~0~ & ~0~ & ~0~
\end{array}
\right),\  Z=\left(
\begin{array}
[c]{ccc}%
~0~ & ~0~ & ~1~\\
~0~ & ~0~ & ~0~\\
~0~ & ~0~ & ~0~
\end{array}
\right)  ,
\]
for which the following commutation relations hold
\[
\lbrack X,Y]=Z,\text{ }[X,Z]=[Y,Z]=0.
\]

Now, much like the Laplacean in $\Rn$, in the Heisenberg group there is a second order partial differential operator which plays a fundamental role in the analysis of such group. Since according to \eqref{comm} the vector fields $X_1,..., X_{2n}$ generate the whole Lie algebra, it is natural to consider the following operator 
\begin{equation}\label{sl}
\Delta_H = \sum_{j=1}^{2n} X_j^2,
\end{equation}
which is known as the real part of the \emph{Kohn-Spencer sub-Laplacean}. We will simply call it the \emph{sub-Laplacean} on $\Hn$.   
Although it plays in the analysis of $\Hn$ a role quite similar to that played by the standard Laplacean in classical analysis, the differences between these two objects are stunning since the geometry of $\Hn$ is not Riemannian and it is not easy to grasp. 

In the real coordinates $p = (x,y,t)\in \mathbb H^n$, if we indicate $z = (x,y)\in \R^{2n}$, then the operator \eqref{sl} takes the form
\begin{equation}\label{sl2}
\Delta_H = \Delta_z + \frac{|z|^2}{4} \p_{tt} + \p_t \sum_{j=1}^n (x_j \p_{y_j} - y_j \p_{x_j}).
\end{equation}
One remarkable feature of \eqref{sl2} is that this operator fails to be elliptic at every point $p\in \mathbb H^n$. It is in fact an easy exercise to verify that the matrix of the quadratic form associated with \eqref{sl2} has a vanishing eigenvalue. However, since by \eqref{comm} we know that the vector fields $\{X_1,...,X_n,X_{n+1},...,X_{2n}\}$ generate the whole Lie algebra of left-invariant vector fields, then thanks to a celebrated theorem of H\"ormander we know that solutions of $\Delta_H u = 0$ are $C^\infty$, see \cite{Ho} (in fact, they are real-analytic, but that does not follow from H\"ormander's theorem) and also the lecture notes \cite{Ga2}. For an introduction to the Heisenberg group one should see \cite{CDPT}. Second order partial differential equations such as the sub-Laplacean \eqref{sl} on $\Hn$ and the Baouendi operator \eqref{ba} are called \emph{subelliptic}. The reason for this name is that, despite the fact that they may fail to be \emph{elliptic} (at every point, in the case of \eqref{sl}, along a submanifold for \eqref{ba}), they satisfy a so-called a priori subelliptic estimate. But the discussion of this deep aspect would take us too far, and thus we must leave it to the interested reader to possibly further it on his/her own. 

Returning to the Baouendi operator \eqref{ba}, suppose that $n=1$ and $\alpha = 1$, in which case \eqref{ba} becomes
 \begin{equation}\label{baa}
 \mathcal P=\Delta_z+ |z|^{2}\p_{xx}.
\end{equation}
If we now consider \eqref{sl2}, we see that in $\mathbb H^n$ every solution of $\Delta_H u = 0$ which is invariant under the action of the vector field
\[
\Theta = \sum_{j=1}^n (x_j \p_{y_j} - y_j \p_{x_j}),
\]
is a solution of the equation
\begin{equation}\label{slb}
\Delta_z u + \frac{|z|^2}{4} \p_{tt} u = 0,
\end{equation}
and, up to a rescaling factor, this is precisely \eqref{baa}. For instance, in the three-dimensional Heisenberg group every function $u:\mathbb H^1 \to \R$ that has cylindrical symmetry, i.e., $u(z,t) = u^\star(|z|,t)$, is such that $\Theta u = 0$, and vice-versa. Therefore, it solves $\Delta_H u = 0$ if and only if it solves \eqref{slb}.  This explains the connection of the operator of Baouendi  with the sub-Laplacean on the Heisenberg group $\Hn$.

We next explore further the connection between to Baouendi operator \eqref{ba} and $(-\Delta)^s$. With this objective in mind we assume that $m=1$ in \eqref{ba}, so that the resulting operator  in $\Rn_x\times \R_z$ is
\begin{equation}\label{baaa}
\mathcal P_\alpha U = \frac{\p^2 U}{\p z^2} + |z|^{2\alpha}  \Delta_x U.
\end{equation}
As stated in \eqref{bg0} we want $\alpha = \frac{1}{2s} - 1$, and since $0<s<1$, this means that in \eqref{baaa} we must have $- 1/2 <\alpha <\infty$. We have three possibilities:
\begin{itemize}
\item $0<s<1/2$\ \  $\Longrightarrow\ \alpha > 0$ ($\Ba$ is of Baouendi type);
\item $s = 1/2$\ \ \ \ \ \ \  $\Longrightarrow\ \alpha = 0$ ($\Ba$ is the standard Laplacean in $\Rn_x\times \R_z$);
\item $1/2<s<1$\ \    $\Longrightarrow\ - 1/2< \alpha < 0$ ($\Ba$ is not of Baouendi type).
\end{itemize}

\begin{remark}\label{R:alpha}
Although $\Ba$ is not a Baouendi operator when $-1/2<\alpha<0$, all the subsequent discussion covers such case as well. This is true in particular of the results from \cite{Ga} that we are going to use, and which were in that paper obtained under the hypothesis that $\alpha\ge 0$.
\end{remark}

First, we equip $\Rn_x\times \R_z$ with the following non-isotropic dilations
\begin{align}\label{dilB}
 \delta_\lambda(x,z)&=(\lambda^{\alpha+1}x, \lambda z),\ \ \ \ \ \ \lambda>0.
\end{align}
A function $u$ is said $\delta_\la$-homogeneous of degree $\kappa$ if
\[
u(\delta_\la(x,z)) = \la^\kappa u(x,z),\ \ \ \ \ \la>0.
\]
It is straightforward to verify that the partial differential operator $\Ba$ is $\delta_\la$-homogeneous of degree two, i.e.,
\begin{equation*}
 \Ba(\delta_\lambda \circ u)=\lambda^2\delta_\lambda\circ (\Ba u).
\end{equation*}
We note that Lebesgue measure in $\Rn_x\times \R_z$ changes according to the equation
\begin{equation}\label{leb}
 d(\delta_\lambda(x,z)) =\lambda^{(\alpha+1)n+ 1} dx dz,
\end{equation}
which motivates the definition of the \emph{homogeneous dimension} for the number
\begin{align}\label{Qa}
 Q=Q_\alpha=(\alpha+1)n + 1.
\end{align}

In the analysis of \eqref{baaa} the following pseudo-gauge introduced in \cite{Ga} plays an important role
\begin{align}\label{ra}
 \rho_\alpha(x,z) & = \left((\alpha+1)^2 |x|^2 + |z|^{2(\alpha+1)}\right)^{\frac1{2(\alpha+1)}}.
\end{align}
We clearly have
\begin{equation}\label{gaugehom}
\rho_\alpha(\delta_\la(x,z)) = \lambda \rho_\alpha(x,z),
\end{equation}
i.e., the pseudo-gauge is homogeneous of degree one. The pseudo-ball and sphere centered at the origin with radius $r>0$ are respectively defined as
\begin{align}\label{BSr}
 B_{\rho_\alpha}(r)&=\{(x,z)\in \Rn_x\times \R_z \mid \rho_\alpha(x,z)<r\},\ \ \ \ \ S_{\rho_\alpha}(r) = \p B_{\rho_\alpha}(r).
\end{align}
In \cite{Ga} it was proved that, with $Q$ as in \eqref{Qa}, and $C_\alpha>0$ given by
\[
C_\alpha^{-1} = (Q + 2\alpha)(Q-2) \int_{\Rn_x\times \R_z} \frac{|z|^{\alpha} dx dz}{\left[((\alpha +1)^2 |x|^2 + |z|^{\alpha+1} + 1)\right]^{1 +\frac{Q+2\alpha}{2(\alpha+1)}}},
\]
 the function
\begin{equation}\label{Ga}
 \G_\alpha(x,z)=\frac{C_{\alpha}}{\ra(x,z)^{Q-2}}
\end{equation}
is a fundamental solution for $-\Ba$ with singularity at $(0,0)$. Since the operator is invariant with respect to translations along $M = \Rn \times  \{0\}$, from \eqref{Ga} we immediately obtain the fundamental solution for $\Ba$ with singularity at any point of the subspace $M$.

In what follows we indicate with $d_e(x,y) = (|x|^2 + y^2)^{1/2}$
the standard Euclidean distance in $\R^{n+1}$, and to emphasize certain differences we will indicate with $B_e(r) = \{(x,y)\in \R^{n+1}\mid d_e(x,y)<r\}$
the Euclidean ball in $\R^{n+1}$ of radius $r$ centered at the origin. Balls centered at a different point $X = (x,y)$ will be indicated with $B_e(X,r)$. When $X = (x,0)$, with a slight abuse of notation we will write $B_e(x,r)$ instead of $B_e((x,0),r)$. We will use analogous notations for the spheres $S_e(r), S_e(X,r), S_e(x,r)$ in $\R^{n+1}$.  

The function $h(y)$ is that given by \eqref{h} above. We have the following simple yet important fact.

\begin{prop}\label{P:rhod}
Given $a\in (-1,1)$, let $\alpha = \frac{a}{1-a}$. Then, for any $(x,y)\in \us$ we have
\begin{equation}\label{rd}
\rho_\alpha(x,h(|y|)) = h(d_e(x,y)).
\end{equation}
 The equation \eqref{rd} implies in particular that
\begin{equation}\label{balls}
B_{\rho_\alpha}(h(r)) = \Phi(B_e(r)),\ \ \ \ \ \ \ \ r>0,
\end{equation}
where $\Phi(x,y) = (x,h(|y|))$.
\end{prop}

In view of \eqref{lab2}, or Proposition \ref{P:laba}, it is clear that if we consider the function in $\us$ given by
\begin{equation}\label{2g}
\tilde \Gamma(x,y) = (1-a)^a \Gamma_\alpha(x,h(y)),
\end{equation}
then we have $L_a \tilde \Gamma = 0$ in $\us$. Notice that from \eqref{rd}, \eqref{Ga} we have
\begin{equation}\label{tildeG}
\tilde \Gamma(x,y) = \frac{(1-a)^a C_{\alpha}}{\ra(x,h(y))^{Q-2}} = \frac{(1-a)^a C_{\alpha}}{h(d_e(x,y))^{Q-2}} = \frac{(1-a)^{n+2a-1}C_{\alpha}}{d_e(x,y)^{n+a-1}} = \frac{\tilde C_a}{d_e(x,y)^{n+a-1}},
\end{equation}
where in the secondo to the last equality we have used the above expression \eqref{Qa} of the homogeneous dimension associated with the dilations \eqref{dilB} in $\R^{n+1}$. Now, the function $\tilde \Gamma(x,y)$ is precisely the fundamental solution of the Laplacean
\[
\Delta = \Delta_x + D_{yy} + \frac{a}{y} D_y
\]
in the fractional dimension
\begin{equation}\label{tQ}
\tilde Q = n+a+1,
\end{equation}
 found by Caffarelli and Silvestre in formula (2.1) in \cite{CS07}. 
Furthermore, if we keep \eqref{tQ} in mind, we see that the exponent $n+a-1$ in \eqref{tildeG} is nothing but $\tilde Q -2$, whereas $\tilde C_a= C_\alpha (1-a)^{n+2a-1}$.

We close this section with the following result from \cite{GRO} that allows to connect integrals on the pseudo-balls and spheres  $B_{\rho_\alpha}(r)$ and $S_{\rho_\alpha}(r)$ in the space of the variables $(x,z)$, to corresponding integrals on the Euclidean balls and spheres in the variables $(x,y)$.

\begin{prop}\label{P:integrals}
Let $U$ be a continuous function in the space $\R^{n+1}$ with the variables $(x,z)$, even in $z$, and let $\tilde U(x,y) =  U(x,h(|y|))$. Then, we have for every $r>0$
\begin{align}\label{si}
\int_{B_{\rho_\alpha}(h(r))} U(x,z) dx dz =   (1-a)^a \int_{B_e(r)} \tilde U(x,y) |y|^{-a} dx dy,
\end{align}
and also
\begin{equation}\label{si2}
h'(r) \int_{S_{\rho_\alpha}(h(r))} \frac{U(x,z)}{|\nabla \rho_\alpha(x,z)|} d H_n(x,z) = (1-a)^a\int_{S_e(r)} \tilde U(x,y) |y|^{-a} dH_n(x,y).
\end{equation}
\end{prop}


\section{Hypoellipticity of $(-\Delta)^s$}\label{S:how}

One of the most important properties of Laplace equation is its \emph{hypoellipticity}. This means that distributional solutions of $\Delta u = f$ are $C^\infty$ wherever $f$ is. Here is the formal definition, see e.g. \cite{T2}. 

\begin{definition}\label{D:treves}
A linear partial differential operator $P(x,\p_x)$ in an open set $\Om \subset \Rn$ is said to be \emph{hypoelliptic} if, given any open subset $U\subset \Om$ and any distribution $u$ in $U$, $u$ is a $C^\infty$ function in $U$ if this is true of $P(x,\p_x) u$.
\end{definition}

To understand this aspect let us recall a classical fact. Let $\Om\subset \Rn$ be an open set and for a function $u\in C(\Om)$ consider the spherical mean-value $\mathscr M_r u(x)$ defined in \eqref{MA0}.

\begin{prop}\label{P:epd}
Let $\Om\subset \Rn$ be an open set. For $x\in \Om$ let $R>0$ be such that $\overline B(x,R)\subset \Om$. One has:
\begin{itemize}
\item[(i)] if $u\in C^1(\Om)$, then for every $0<r<R$ one has
\[
\frac{\p \mathscr M_r u}{\p r}(x) = \frac{1}{\sigma_{n-1} r^{n-1}} \int_{S(x,r)} \frac{\p u}{\p \nu}(y) d\sigma(y);
\]
\item[(ii)] if $u\in C^2(\Om)$, then for every $0<r<R$ one has
\[
\frac{\p \mathscr M_r u}{\p r}(x) = \frac{1}{\sigma_{n-1} r^{n-1}} \int_{B(x,r)} \Delta u(y) dy.
\]
\end{itemize}
\end{prop}
A basic consequence of (ii) is that if $u$ is harmonic in $\Om$, i.e., $u\in C^2(\Om)$ and $\Delta u = 0$ in $\Om$, then for any $x\in \Om$ and any $0<r<\operatorname{dist}(x,\p \Om)$, one has 
\begin{equation}\label{har}
u(x) = \mathscr M_r u(x).
\end{equation}
Now, the fact that a function satisfies the mean-value formula \eqref{har} has truly remarkable consequences. One of them, is the following well-known converse to Gauss' mean value theorem. 

\begin{theorem}[of Ko\"ebe]\label{T:koebe}
Let $u\in C(\Om)$ and suppose that for every $x\in \Om$ and $0<r<\operatorname{dist}(x,\p \Om)$ formula \eqref{har} hold. Then, $u\in C^\infty(\Om)$ and in fact $\Delta u = 0$ in $\Om$.
\end{theorem}

\begin{proof}
To prove that $u\in C^\infty(\Om)$ it is obviously enough to show that $u\in C^\infty(\Om_\e)$ for every $\e>0$, where $\Om_\e = \{x\in \Om\mid \operatorname{dist}(x,\p \Om)>\e\}$. With this objective in mind let $K$ be a spherically symmetric Friedrichs' mollifier, i.e., $K\in C^\infty_0(\Rn)$, with $\int_{\Rn} K(y) dy = 1$, supp$\ K\subset \overline B(0,1)$, and $K(y) = K^\star(|y|)$, and denote by $K_\e(y) = \e^{-n} K(y/\e)$ the corresponding approximation to the identity. We claim that as consequence of \eqref{har} the following must be true in $\Om_\e$:
\begin{equation}\label{k1}
u = K_\e \star u.
\end{equation}
To verify this claim we use Cavalieri's principle to write for every $x\in \Om_\e$
\begin{align*}
K_\e \star u(x) & = \int_0^\infty \int_{|y| = r} K(y) u(x-y) d\sigma(y) dr = \int_0^\infty K^\star(r)\int_{|y| = r} u(x-y) d\sigma(y) dr
\\
& = \sigma_{n-1} \int_0^\infty K^\star(r) r^{n-1} \mathscr M_r u(x) dr = u(x) \sigma_{n-1} \int_0^\infty K^\star(r) r^{n-1} dr
\\
& = u(x) \int_{\Rn} K(y) dy = u(x),
\end{align*}
where in the second to the last equality we have used the spherical symmetry of $K$ and Cavalieri's principle again.

From \eqref{k1} and a well-known property of the convolution, we conclude that \eqref{har} implies that $u\in C^\infty(\Om_\e)$, and therefore $u\in C^\infty(\Om)$. We pause for a moment to emphasize this remarkable conclusion: \emph{a continuous function which locally satisfies the mean value property \eqref{har} must in fact be infinitely smooth!} Once we know this we can appeal
to (ii) in Proposition \ref{P:epd} above to infer that for every $x\in \Om$ and $0<r<\operatorname{dist}(x,\p \Om)$ we can differentiate at that point $r$ the function $t\to \mathscr M_t u(x)$, and we have
\[
\frac{\p \mathscr M_r u}{\p r}(x) = \frac{1}{\sigma_{n-1} r^{n-1}} \int_{B(x,r)} \Delta u(y) dy.
\]
On the other hand, the constancy of $r\to \mathscr M_r u(x,r)$ that follows from by \eqref{har} implies that $\frac{\p \mathscr M_r u}{\p r}(x) = 0$ for every $0<r<\operatorname{dist}(x,\p \Om)$, and thus in particular we have for every such value of $r$
\[
\frac{1}{\omega_{n} r^{n}} \int_{B(x,r)} \Delta u(y) dy = \frac{n}{\sigma_{n-1} r^{n}} \int_{B(x,r)} \Delta u(y) dy = 0,
\]
where in the first equality we have used the second identity in \eqref{sn1} above. Since $\Delta u\in C(\Om)$, letting $r\to 0^+$ in the above equation we infer that it must be $\Delta u(x) = 0$. By the arbitrariness of $x\in \Om$, we conclude $\Delta u = 0$ in $\Om$.

\end{proof}

From Theorem \ref{T:koebe} and the fact that harmonic functions satisfy \eqref{har}, we immediately obtain the following important result.
 
\begin{corollary}[Smoothing property of $\Delta$]\label{C:hypo}
Let $u\in C^2(\Om)$ be such that $\Delta u = 0$ in $\Om$. Then, $u\in C^\infty(\Om)$.
\end{corollary}

Hypoellipticity means that the conclusion of Corollary \ref{C:hypo} continues to be true if we replace the hypothesis that $u\in C^2(\Om)$ and $\Delta u = 0$, with the much weaker assumption that $u$ be harmonic in the distributional sense, i.e., $u\in \mathscr D'(\Om)$ and $\Delta u = 0$ in $\mathscr D'(\Om)$. 
Historically, this result is known as \emph{Weyl's lemma}, after the famous 1940 paper by H. Weyl \cite{Weyl}. Although less known, R. Caccioppoli in \cite{C} had already established such result for $n=2$ in a more general form in 1937, and subsequently G. Cimmino extended Caccioppoli's theorem to all elliptic operators with smooth coefficients in the plane, see \cite{Ci1}, \cite{Ci2}. Since their results preceded Weyl's paper, it should be called the \emph{Caccioppoli-Cimmino-Weyl lemma}.
 
After this prelude on the Laplacean, we return to the main focus on this note and ask the natural question: \emph{how smooth are solutions of $(-\Delta)^s u = 0$}? The  answer to this question is that the nonlocal operator $(-\Delta)^s$ behaves much like the standard Laplacean, and thus \emph{distributional solutions of $(-\Delta)^s u = 0$ are $C^\infty$}. This is contained in Theorem \ref{T:hypo} below, whose proof however relies on important results from the theory of pseudo-differential operators. Since the declared intent of these notes is didactic and being self-contained, we next present a somewhat less general result whose proof has the advantage of being conceptually much simpler, and closer in spirit to the opening discussion on the Laplacean. It also keeps up with the spirit of Section \ref{S:flso} of connecting $(-\Delta)^s$ to the class of degenerate elliptic operators such as $\Ba$ and $L_a$ via the extension procedure.

With this comment in mind we return to the Baouendi operator $\Ba$ in \eqref{baaa} in the space $\Rn_x\times \R_z$, and recall some results from \cite{Ga}. For any $\alpha\ge 0$ we define the $\alpha$-gradient of a function $U$ that lives in an open set of such space as follows
\begin{equation*}
 \nabla_\alpha U= \left(|z|^\alpha \nabla_x U, D_z U\right).
\end{equation*}
Given two functions $U$ and $V$ we set
\[
\langle\na U,\na V\rangle =  |z|^{2\alpha} \langle\nabla_x U,\nabla_x V\rangle + D_z U D_z V.
\]
The square of the length of $\nabla_\alpha U$ is
\begin{equation}\label{carre}
 |\nabla_\alpha U|^2= |z|^{2\alpha} |\nabla_x U|^2 + (D_z U)^2.
\end{equation}
The following lemma, collects the identities (2.12)-(2.14) in \cite{Ga}.

\begin{lemma}\label{l:gradalpharho}
Let $\ra$ be the pseudo-gauge in \eqref{ra} above. One has in $\R^{n+1}\setminus\{0\}$,
 \begin{equation}\label{nara}
  \psi_{\alpha} \overset{def}{=} |\na \ra|^2 = \frac{|z|^{2\alpha}}{\ra^{2\alpha}}.
 \end{equation}
Moreover, given a function $u$ one has
 \begin{equation}\label{nara2}
 \langle\na u,\na \ra\rangle = \frac{\Za u}{\ra} \psi_\alpha,
 \end{equation}
 where $Z_\alpha$ is the infinitesimal generator of the dilations \eqref{dilB}, i.e.,
\begin{equation}\label{Za}
 Z_\alpha= (\alpha+1)\sum_{i=1}^n x_i \partial_{x_i} +  z\partial_{z}.
\end{equation}
\end{lemma}

The next result provides a generalization to the Baouendi operator $\Ba$ of classical representation formulas. It combines Theorem 2.1 and Corollary 2.1 in \cite{Ga}. The pseudo-ball $B_{\rho_\alpha}(r)$ and the pseudo-sphere $S_{\rho_\alpha}(r)$ centered at the origin with radius $r>0$ are those defined in \eqref{BSr} above.

\begin{prop}\label{P:gar}
Let $\alpha \ge 0$ and consider a sufficiently smooth function $U$ in $\Rn_x\times\R_z$. For every $r>0$ one has
\begin{align}\label{garmv}
&\frac{1}{|S_{\rho_\alpha}(r)|_\alpha} \int_{S_{\rho_\alpha}(r)} U(x,z) \frac{\psi_\alpha(x,z)}{|\nabla \rho_\alpha(x,z)|} dH_n(x,z) = U(0,0) 
\\
& + \int_{B_{\rho_\alpha}(r)} \Ba U(x,z) \left[\G_\alpha(x,z) - \frac{C_\alpha}{r^{Q-2}}\right] dx dz,
\notag
\end{align}
where $\G_\alpha$ and $C_\alpha$ are as in \eqref{Ga}. In particular, if $\Ba U = 0$, then we have for every $r>0$
\begin{equation}\label{garmv2}
U(0,0) = \frac{1}{|S_{\rho_\alpha}(r)|_\alpha} \int_{S_{\rho_\alpha}(r)} U(x,z) \frac{\psi_\alpha(x,z)}{|\nabla \rho_\alpha(x,z)|} dH_n(x,z).  
\end{equation}
\end{prop}

In the statement of Proposition \ref{P:gar} by slightly abusing the notation we have indicated with
\[
|S_{\rho_\alpha}(r)|_\alpha = \int_{S_{\rho_\alpha}(r)} \frac{\psi_\alpha(x,z)}{|\nabla \rho_\alpha(x,z)|} dH_n(x,z).
\]
If, by a similar abuse of notation, we set 
\[
|B_{\rho_\alpha}(r)|_\alpha = \int_{B_{\rho_\alpha}(r)} \psi_\alpha(x,z) dx dz,
\]
then keeping in mind that from \eqref{nara} we easily see that $\psi_\alpha$ is homogeneous of degree zero with respect to the anisotropic dilations \eqref{dilB}, by a rescaling and \eqref{leb} we obtain
\[
|B_{\rho_\alpha}(r)|_\alpha = \omega_\alpha r^Q,
\]
where 
\[
\omega_\alpha = |B_{\rho_\alpha}(1)|_\alpha =  \int_{B_{\rho_\alpha}(1)} \psi_\alpha(x,z) dx dz,
\]
 and $Q$ is as in \eqref{Qa}. Differentiating this formula and using Federer's coarea formula (aka Bonaventura Cavalieri's principle), see for instance \cite{EG}, we find
\begin{equation}\label{sigmaalpha}
Q \omega_\alpha r^{Q-1} = \frac{d}{dr} \int_{B_\alpha(r)} \psi_\alpha(x,z) dx dz = \int_{S_{\rho_\alpha}(r)} \frac{\psi_\alpha(x,z)}{|\nabla \rho_\alpha(x,z)|} dH_n(x,z) = |S_{\rho_\alpha}(r)|_\alpha.
\end{equation}

We now apply the formula \eqref{si2} in Proposition \ref{P:integrals} to the function $U\psi_\alpha$, instead of just $U$, obtaining 
\begin{align*}
& h'(r) \int_{S_{\rho_\alpha}(h(r))} \frac{U(x,z)\psi_\alpha(x,z)}{|\nabla \rho_\alpha(x,z)|} d H_n(x,z) = (1-a)^a\int_{S_e(r)} \tilde U(x,y) \tilde \psi_\alpha(x,y) |y|^{-a} dH_n(x,y).
\end{align*}
If we observe that \eqref{h} and \eqref{rd} give
\[
\tilde \psi_\alpha(x,y) = \frac{h(|y|)^{2\alpha}}{\ra^{2\alpha}(x,h(|y|))} =  \frac{h(|y|)^{2\alpha}}{h(d_e(x,y))^{2\alpha}} = \frac{h(|y|)^{\frac{2a}{1-a}}}{h(d_e(x,y))^{\frac{2a}{1-a}}} = \frac{|y|^{2a}}{d_e(x,y)^{2a}},
\]
then keeping in mind that $h'(r) = \frac{(1-a)^a}{r^a}$,  we have 
\begin{align}\label{gar3}
&  \int_{S_{\rho_\alpha}(h(r))} \frac{U(x,z)\psi_\alpha(x,z)}{|\nabla \rho_\alpha(x,z)|} d H_n(x,z) = \frac{1}{r^{a}}\int_{S_e(r)} \tilde U(x,y) |y|^{a} dH_n(x,y).
\end{align}
Since by \eqref{sigmaalpha} we have 
\[
|S_{\rho_\alpha}(h(r))|_\alpha =  \frac{Q \omega_\alpha}{(1-a)^n} r^{n},
\]
we conclude from \eqref{gar3} that
\begin{align}\label{gar4}
&  \frac{1}{|S_{\rho_\alpha}(h(r))|}\int_{S_{\rho_\alpha}(h(r))} \frac{U(x,z)\psi_\alpha(x,z)}{|\nabla \rho_\alpha(x,z)|} d H_n(x,z) = \frac{(1-a)^n}{Q \omega_\alpha r^{n+a}} \int_{S_e(r)} \tilde U(x,y) |y|^{a} dH_n(x,y).
\end{align} 
Since when $U\equiv 1$ the left-hand side of \eqref{gar4} is $1$, we obtain 
\begin{equation}\label{coarea}
|S_e(r)|_a \overset{def}{=} \int_{S_e(r)} |y|^{a} dH_n(x,y) = \frac{Q \omega_\alpha}{(1-a)^n} r^{n+a}. 
\end{equation}

Furthermore, if we combine \eqref{si} above with \eqref{labb2}, \eqref{2g} and \eqref{tildeG}, we find
\begin{align}\label{si3}
& \int_{B_{\rho_\alpha}(h(r))} \Ba U(x,z) \left[\G_\alpha(x,z) - \frac{C_\alpha}{h(r)^{Q-2}}\right] dx dz 
\\
& =  (1-a)^a \int_{B_e(r)} L_a \tilde U(x,y)\left[\G_\alpha(x,h(y)) - \frac{(1-a)^{n+a-1}C_\alpha}{r^{n+a-1}}\right]  dx dy,
\notag
\end{align}

Combining \eqref{gar3}, \eqref{si3} with \eqref{garmv2} above, and using \eqref{sigmaalpha} and the translation invariance of either $L_a$ or $\Ba$ in the $x$-variable, we obtain the following.

\begin{prop}\label{P:gar3}
Let $\Tilde U$ be a sufficiently regular function in $\Rn_x\times \R_y$. Then, for every $x\in \Rn$ and  $r>0$ we have
\begin{align}\label{garmv3}
\tilde U(x,0) & + \int_{B_e(r)} L_a \tilde U(x-x',y)\left[\tilde \G(x',y) - \frac{\tilde C_a}{r^{n+a-1}}\right]  dx' dy
\\ 
& = \frac{(1-a)^n}{Q \omega_\alpha r^{n+a}} \int_{S_e(r)} \tilde U(x-x',y) |y|^{a} dH_n(x',y).  
\notag
\end{align}
\end{prop}

\begin{remark}\label{R:alla}
We note that, since we have used Proposition \ref{P:gar}, and since $a = \frac{\alpha}{\alpha+1}$, strictly speaking we have proved Proposition \ref{P:gar3} only for the range $0\le a<1$. However, it is easy to recognize that \eqref{garmv3} does in fact hold also in the range $-1<a <0$. 
\end{remark}

If we now define the $L_a$-spherical average of a function $\tilde U$ as
\begin{equation}\label{aave}
\mathscr M_{\tilde U,a}(x,r) =  \frac{1}{|S_e(r)|_a } \int_{S_e(r)} \tilde U(x-x',y) |y|^a dH_n(x',y),
\end{equation}
then we can reformulate \eqref{garmv3} as follows
\begin{align}\label{garmv33}
\mathscr M_{\tilde U,a}(x,r) = \tilde U(x,0)  + \int_{B_e(r)} L_a \tilde U(x-x',y)\left[\tilde \G(x',y) - \frac{\tilde C_a}{r^{n+a-1}}\right]  dx' dy.
\end{align} 

Differentiating \eqref{garmv33} with respect to $r$ produces the following interesting consequence.

\begin{prop}\label{P:der}
Let $\Tilde U$ be a sufficiently regular function in $\Rn_x\times \R_y$. Then, for every $x\in \Rn$ and  $r>0$ we have
\begin{equation}\label{deraave}
\frac{\p \mathscr M_{\tilde U,a}}{\p r}(x,r) = \frac{(n+a-1)\tilde C_a}{r^{n+a}} \int_{B_e(r)} L_a \tilde U(x-x',y) dx' dy.
\end{equation}
\end{prop}

We note explicitly that, since $-1<a<1$ and $n\ge 2$, we have $n+a-1>0$.
 
\begin{corollary}\label{C:der}
Let $\tilde U$ be as in Proposition \ref{P:der}. If either $L_a \tilde U = 0$, or $\tilde U \ge 0$ and $L_a \tilde U\ge 0$, one has for every $x\in \Rn$ and  $r>0$
\begin{equation}\label{deraave2}
\frac{\p \mathscr \mathscr M_{\tilde U^2,a}}{\p r}(x,r) \ge \frac{(n+a-1)\tilde C_a}{r^{n+a}} \int_{B_e(x,r)} |\nabla \tilde U(x',y)|^2 |y|^a dx' dy.
\end{equation}
Equality holds in \eqref{deraave2} when $L_a \tilde U = 0$, and in either case the function $r\to \mathscr M_{\tilde U^2,a}(x,r)$ is monotonically increasing.
\end{corollary}

\begin{proof}
It suffices to apply \eqref{deraave} to the function $\tilde U^2$ and observe that
\[
L_a \tilde U^2 = 2 \tilde U L_a \tilde U + 2 |y|^a |\nabla \tilde U|^2.
\]
If either $L_a \tilde U = 0$, or $\tilde U \ge 0$ and $L_a \tilde U\ge 0$, one has
\[
L_a \tilde U^2 \ge 2 |y|^a |\nabla \tilde U|^2,
\]
with equality when $L_a \tilde U = 0$.

\end{proof}

We will use the following inequality in the proof of Theorem \ref{T:sucpfl}
 below.

\begin{prop}[Caccioppoli inequality]\label{P:cacc}
Let $\Om\subset \Rn$ be open, and suppose that either $\tilde U$ be a solution to $L_a \tilde U = 0$ in the open set $\Om\times (-d,d) \subset \Rn\times \R$, where $d = \operatorname{diam}(\Om)$, or $\tilde U\ge 0$ and $L_a \tilde U \ge 0$ there. Then, there exists $C(n,a)>0$ such that for every $x\in \Om$ and every $0<s<t<\operatorname{dist}(x,\p \Om)$, one has
\begin{equation}\label{cacc}
\int_{B_e(x,s)} |\nabla \tilde U(x',y)|^2 |y|^a dx' dy \le \frac{C(n,a)}{t-s} \int_{S_e(x,t)} \tilde U(x',y)^2 |y|^a dH_n(x',y).
\end{equation}
\end{prop}

\begin{proof}
Integrating \eqref{deraave} in $r$ on the interval $(s,t)$ we find
\begin{align*}
& (n+a-1)\tilde C_a \int_s^t \frac{1}{r^{n+a+1}} \left(\int_{B_e(x,r)} |\nabla \tilde U(x',y)|^2 |y|^a dx' dy\right) dr
\le \int_s^t \frac 1r \frac{\p \mathscr M_{\tilde U^2,a}}{\p r}(x,r) dr.
\end{align*}
Since we trivially have
\[
\int_s^t \frac{1}{r^{n+a+1}} \left(\int_{B_e(x,r)} |\nabla \tilde U(x',y)|^2 |y|^a dx' dy\right) dr \ge \int_s^t \frac{1}{r^{n+a+1}} dr \left(\int_{B_e(x,s)} |\nabla \tilde U(x',y)|^2 |y|^a dx' dy\right),
\]
and since integrating and applying the mean value theorem we find
\[
\int_s^t \frac{1}{r^{n+a+1}} dr \ge \frac{t-s}{s t^{n+a}},
\]
we obtain
\[
(n+a-1)\tilde C_a \frac{t-s}{s t^{n+a}} \int_{B_e(x,s)} |\nabla \tilde U(x',y)|^2 |y|^a dx' dy \le \int_s^t \frac 1r \frac{\p \mathscr M_{\tilde U^2,a}}{\p r}(x,r) dr.
\]
On the other hand, an integration by parts gives
\begin{align*}
& \int_s^t \frac 1r \frac{\p \mathscr M_{\tilde U^2,a}}{\p r}(x,r) dr = \frac 1t \mathscr M_{\tilde U^2,a}(x,t) - \frac 1s \mathscr M_{\tilde U^2,a}(x,s) + \int_s^t \frac 1{r^2} \mathscr M_{\tilde U^2,a}(x,r) dr
\\
& \le \frac 1t \mathscr M_{\tilde U^2,a}(x,t) - \frac 1s \mathscr M_{\tilde U^2,a}(x,s) + \mathscr M_{\tilde U^2,a}(x,t) \left(\frac 1s - \frac 1t\right) = \frac 1s \left(\mathscr M_{\tilde U^2,a}(x,t) - \mathscr M_{\tilde U^2,a}(x,s)\right),
\end{align*}
where the inequality is justified by the monotonicity of the function $r\to \mathscr M_{\tilde U^2,a}(x,r)$, see the second part of Corollary \ref{C:der}.
In conclusion, we have found
\[
(n+a-1)\tilde C_a \frac{t-s}{s t^{n+a}} \int_{B_e(x,s)} |\nabla \tilde U(x',y)|^2 |y|^a dx' dy \le \frac 1s \left(\mathscr M_{\tilde U^2,a}(x,t) - \mathscr M_{\tilde U^2,a}(x,s)\right).
\]
Keeping \eqref{coarea} and \eqref{aave} in mind, it is clear that this inequality trivially implies the desired conclusion \eqref{cacc}.

\end{proof}

\begin{remark}\label{R:cacc}
The above proof of Proposition \ref{P:cacc} is self-contained and provides a slightly stronger version of the usual Caccioppoli inequality since in the right-hand side one has the spherical $L^2$ norm of the function, instead of the solid one. Such proof needs to be modified when dealing with equations with rough coefficients. The reader can find the Caccioppoli inequality for more general degenerate elliptic equations in the paper \cite{FKS}.
\end{remark}

Suppose now that $\vf\in C_0^\infty(\Rn_x\times \R_y)$ is a spherically symmetric bump function, i.e., $\vf(X) = \vf^\star(|X|)$, such that, say, supp$\ \vf^\star \subset [1/4,3/4]$, and 
\[
\int_{\Rn_x\times \R_y} \vf(X) |y|^a dX = 1.
\]
Using the coarea formula  and \eqref{coarea} we can reformulate this latter assumption as follows
\[
1 = \int_0^\infty \vf^\star(r) \int_{S_e(r)} |y|^{a} dH_n(x,y) dr = \frac{Q \omega_\alpha}{(1-a)^n} \int_0^\infty \vf^\star(r) r^{n+a} dr.
\]
Now, if we multiply both sides of \eqref{garmv3} by $\vf^\star(r) r^{n+a}$, and we integrate with respect to $r\in (0,\infty)$, keeping in mind that
\[
\int_{\Rn\times \R} \tilde U(x-x',y)|y|^a dx'dy = \int_0^\infty \int_{S_e(r)} \tilde U(x-x',y) |y|^{a} dH_n(x',y) dr,
\]
we obtain the following result.

\begin{prop}\label{P:gar4}
Let $\Tilde U$ be a solution to $L_a \tilde U = 0$ in $\Rn\times \R$. Then, for every $x\in \Rn$ we have
\begin{equation}\label{garmv333}
\tilde U(x,0) =  \int_{\Rn\times \R} \tilde U(x',y) \vf(x-x',y) |y|^{a} dx' dy.  
\end{equation}
If instead $\Om\subset \Rn$ is an open set, suppose that $\tilde U$ be a solution to $L_a \tilde U = 0$ in the open set $\Om\times (-d,d) \subset \Rn\times \R$, where $d = \operatorname{diam}(\Om)$. Then, for every $x\in \Om$ and every $0<r<\operatorname{dist}(x,\p \Om)$, one has
\begin{equation}\label{garmv34}
\tilde U(x,0) =  \int_{\Rn\times \R} \tilde U(x',y) \vf_r(x-x',y) |y|^{a} dx' dy,  
\end{equation}
where we have let $\vf_r(X) = \frac{1}{r^{n+a+1}}\vf(\frac{X}{r})$.
\end{prop}

The reader should be aware that, while in keeping up with the spirit of the present section and of the previous one we have derived Proposition \ref{P:gar4} from Proposition \ref{P:gar}, the mean value formulas \eqref{garmv33}, \eqref{garmv34} were independently obtained by a direct computation in Theorem 1 of the paper \cite{ABG}. In Theorem 2  of the same paper, combining \eqref{garmv3} with Theorem \ref{T:cs} above, the authors established an interesting representation formula for solutions of $(-\Delta)^s$ that can be seen as the nonlocal counterpart of formula \eqref{k1} in Theorem \ref{T:koebe} above. Before we state such formula in Theorem \ref{T:abg2} below, we introduce a useful lemma.

\begin{lemma}\label{L:abg}
Let $\vf(X) = \vf^\star(|X|)$ be a $C^\infty_0(\R^{n+1})$ function as in Proposition \ref{P:gar4}, and, with $P_s$ as in \eqref{Pfinal} above, and $a = 1-2s$, define
\begin{equation}\label{Fi}
\Phi(x)  = \int_{\Rn\times \R} \vf(x-x',y) P_s(x',|y|) |y|^a dx' dy. 
\end{equation}
Then, one has
\begin{itemize}
\item[(i)] $\Phi$ is spherically symmetric, i.e., there exists a function $\Phi^\star:[0,\infty)\to \R$ such that $\Phi(x) = \Phi^\star(|x|)$.
\item[(ii)] $\Phi\in C^\infty(\Rn)$, and for every multi-index $\alpha$ such that $|\alpha| = k$, there exists $C(n,s,k) >0$ such that one has for every $x\in \Rn$
\[
|\p^\alpha \Phi(x)| \le \frac{C(n,s,k)}{(1+|x|)^{n+1-a}}.
\] 
\end{itemize}
\end{lemma}

\begin{proof}
(i). Since $z\to \vf(z,y)$ and $z\to P_s(z,y)$ are spherically symmetric functions, and since the convolution of two spherically symmetric functions is spherically symmetric, it is clear that there exists a function $\Phi^\star:[0,\infty)\to \R$ such that $\Phi(x) = \Phi^\star(|x|)$. 

(ii) Observe that, from the definition \eqref{Pfinal} and the fact that $a=1-2s$, we can alternatively write the Poisson kernel for $L_a$ as
\[
P_a(x,y) = \frac{\G(\frac{n+1-a}{2})}{\pi^{\frac n2}\G(\frac{1-a}{2})} \frac{|y|^{1-a}}{(y^2 + |x|^2)^{\frac{n+1-a}{2}}}
 = c(n,a) \frac{|y|^{1-a}}{(y^2 + |x|^2)^{\frac{n+1-a}{2}}}.
\]
We thus have
\[
\Phi(x) = c(n,a) \int_{\R}\int_{\Rn} \vf(x-x',y) \frac{|y|}{(y^2 + |x'|^2)^{\frac{n+1-a}{2}}}  dx' dy.
\]
Differentiating under the integral sign in the definition of $\Phi(x)$, and keeping in mind that supp\ $\vf\subset \overline B_e(0,1)$, for every $\alpha\in (\mathbb N \cup\{0\})^n$ such that $|\alpha|=k$ we have
 \[
|\p^\alpha \Phi(x)| \le C(n,a,k) \int_{|y|\le1}\int_{|x'-x|\le 1} |\p^\alpha \vf(x-x',y)| \frac{|y|}{(y^2 + |x'|^2)^{\frac{n+1-a}{2}}} dx' dy.
\]
If now $|x|>2$, then when $|x'-x|\le 1$ we have $|x'| = |x - (x-x')|\ge |x|-|x'-x| \ge |x|-1\ge |x|/2$. Therefore, when $|x|>2$ we have
\begin{align*}
|\p^\alpha \Phi(x)| &  \le \frac{C(n,a,k)}{|x|^{n+1-a}}.
\end{align*}
On the other hand, when $|x|\le 2$, then the triangle inequality gives $B(x,1)\subset B(0,3)$, and therefore
\begin{align*}
|\p^\alpha \Phi(x)| & \le C(n,a,k) \int_{|y|\le1}\int_{|x'|\le 3} \frac{|y|}{(y^2 + |x'|^2)^{\frac{n+1-a}{2}}} dx' dy
\\
& \le C(n,a,k) \int_{|y|\le 1} |y|^a \int_{|z|\le \frac{3}{y}} \frac{dz}{(1 + |z|^2)^{\frac{n+1-a}{2}}} dy 
\\
& \le C(n,a,k) \int_{|y|\le 1} |y|^a dy \int_{\Rn} \frac{dz}{(1 + |z|^2)^{\frac{n+1-a}{2}}} = \overline{C}(n,a,k) < \infty,
\end{align*}
since $|a|<1$. These estimates prove (ii).

\end{proof}

\begin{theorem}[\cite{ABG}]\label{T:abg2}
Let $u\in \mathscr L_s(\Rn)$ be such that $(-\Delta)^s u = 0$ in $\mathscr D'(\Om)$  for a given open set $\Om\subset \Rn$. Then, for a.e. $x\in \Om$ and $0<r<\operatorname{dist}(x,\p \Om)$, one has
\begin{equation}\label{mvffl}
u(x) = \Phi_r \star u(x) = \int_{\Rn} u(x') \Phi_r(x-x') dx',
\end{equation}
where $\Phi$ is the function in \eqref{Fi}, and we have let $\Phi_r(x) = r^{-n} \Phi(x/r)$. 
\end{theorem}

\begin{proof}
 Consider the extension problem \eqref{ext2} above with Dirichlet datum $u$, and denote by 
 \[
U(x,y) = P_s(\cdot,y)\star u(x)  = \int_{\Rn} P_s(x-x',y) u(x') dx'
\]
 its solution in $\Rn\times [0,\infty)$. Consider the symmetric extension of $U$ to the whole $\Rn\times \R$ defined by 
 \[
 \tilde U(x,y) = U(x,|y|) = \int_{\Rn} P_s(x-x',|y|) u(x') dx'. 
 \]
 By the Dirichlet-to-Neumann condition \eqref{dn} we know that at a.e. $x\in \Om$ we have
\begin{equation}\label{dn2}
- \frac{2^{2s-1} \G(s)}{\G(1-s)} \underset{y\to 0^+}{\lim} y^{1-2s} \frac{\p U}{\p y}(x,y) = (-\Delta)^s u(x) = 0.
\end{equation}
We can thus invoke Lemma 4.1 in \cite{CS07} to infer that $\tilde U$ is a weak solution to the equation $L_a \tilde U = 0$  in $D = \Om\times (-d,d)$, where $d = \operatorname{diam}(\Om)$. By the Harnack inequality for weak solutions in \cite{FKS} we can redefine $\tilde U$ on a set of $(n+1)$-dimensional measure zero in $D$ so that it be locally H\"older continuous in $D$. Applying \eqref{garmv34} in Proposition \ref {P:gar4} we find for every $x\in \Om$ and $0<r<\operatorname{dist}(x,\p\Om)$
\begin{align*}
u(x) = \tilde U(x,0) & =  \int_{\Rn\times \R} \tilde U(x',y) \vf_r(x-x',y) |y|^{a} dx' dy
\\
& = \int_{\Rn\times \R} \left(\int_{\Rn} P_s(x'-x'',|y|) u(x'') dx''\right) \vf_r(x-x',y) |y|^{a} dx' dy
\\
& = \int_{\Rn}u(x'')\left(\int_{\Rn\times \R} \vf_r(x-x',y) P_s(x'-x'',|y|)  |y|^{a} dx' dy\right) dx''
\\
& = \int_{\Rn} \Psi_r(x,x'') u(x'') dx'',    
\end{align*}
where we have let
\[
\Psi_r(x,x'') = \int_{\Rn\times \R} \vf_r(x-x',y) P_s(x'-x'',|y|)  |y|^{a} dx' dy.
\]
The exchange of order of integration (Fubini's theorem) is justified by the fact that, since $a = 1-2s$, the assumption $u\in \mathscr L_s(\Rn)$ can be reformulated as $\int_{\Rn} \frac{|u(x'')|}{(1+|x''|^2)^{\frac{n+1-a}{2}}} dx'' <\infty$. On the other, we claim that for any fixed $x\in \Rn$ we have
\begin{equation}\label{claimfi}
\left|\Psi_r(x,x'')\right|\le \frac{C_r}{(1+|x''|^2)^{\frac{n+1-a}{2}}}.
\end{equation}
Assuming the claim, we thus see that $\int_{\Rn} |\Psi_r(x,x'')| |u(x'')| dx'' < \infty$, and thus the application of Fubini's theorem would be justified in view of Tonelli's theorem. We thus need to prove \eqref{claimfi}. For this, recall that from our choice of $\vf(X)$ we know that supp$\ \vf\subset \overline B_e(0,3/4)\setminus B_e(0,1/4)\subset \overline B_e(0,1) \subset \Rn_x\times \R_y$, and thus we have 
\begin{align*}
|\Psi_r(x,x'')| & \le \int_{-1}^1\left(\int_{\Rn} |\vf_r(x-x',y)| P_s(x'-x'',|y|)dx'\right)  |y|^{a}  dy
\\
& \le ||\vf_r||_{L^\infty(\Rn\times \R)} \int_{-1}^1||P_s(\cdot-x'',|y|)||_{L^1(\Rn)}  |y|^{a}  dy
\\
& =  ||\vf_r||_{L^\infty(\Rn\times \R)} \int_{-1}^1 |y|^{a}  dy = C_r,
\end{align*}
where in the last equality we have used \eqref{poissoneone}. On the other hand, if $x''\not\in \overline B(x,2)$ and $|x'-x|<1$, we have $|x''-x'| \ge |x''-x| - |x-x'| \ge |x''-x| - 1>  |x''-x|/2$. We thus obtain from \eqref{Pfinal}
\begin{align*}
|\Psi_r(x,x'')| & \le C(n,s) ||\vf_r||_{L^\infty(\Rn\times \R)} \int_{B_e((x,0),1)} \frac{|y|^{1-a}}{(y^2 + |x''-x'|^2)^{\frac{n+1-a}{2}}} |y|^a dx' dy
\\
& \le  C(n,s) ||\vf_r||_{L^\infty(\Rn\times \R)} \int_{B_e((x,0),1)} \frac{1}{|x''-x'|^{n+1-a}}  dx' dy 
\\
& \le \frac{C(n,s) ||\vf_r||_{L^\infty(\Rn\times \R)}}{1+|x''-x|^{n+1-a}}
\\
& \le \frac{C_r}{1+|x''|^{n+1-a}}.
\end{align*}
This proves \eqref{claimfi}. 

In order to complete the proof of \eqref{mvffl}, we are thus left with showing that
\[
\Psi_r(x,x'') = \frac{1}{r^n} \Phi^\star\left(\frac{|x-x''|}{r}\right).
\]
This easily follows from a change of variable. We have in fact
\begin{align*}
\frac{1}{r^n} \Phi^\star\left(\frac{|x-x''|}{r}\right) & = \frac{1}{r^n}\int_{\Rn\times \R} \vf(\frac{x-x'' - rx'}{r},y) P_s(x',|y|) |y|^a dx' dy\ \ \ \ (z = x -r x', t = r y)
\\
& = \frac{1}{r^{2n+1+a}} \int_{\Rn\times \R} \vf(\frac{z-x''}{r},\frac{t}{r}) P_s(\frac{z-x}{r},\frac{|t|}{r})|t|^a dz dt
\\
& = \int_{\Rn\times \R} \vf_r(z-x'',t) P_s(z-x,|t|)|t|^a dz dt
\\
& = \int_{\Rn\times \R} \vf_r(x''-z,t) P_s(z-x,|t|)|t|^a dz dt = \Psi_r(x'',x),
\end{align*}
where in the second to the last equality we have used the fact that $P_s(\frac{z-x}{r},\frac{|t|}{r}) = r^n P_s(z-x,|t|)$. Since $\Phi^\star\left(\frac{|x-x''|}{r}\right) = \Phi^\star\left(\frac{|x-x''|}{r}\right)$, the proof is completed.

\end{proof}

\begin{remark}\label{R:hypo}
Formula \eqref{mvffl} can be seen as a nonlocal counterpart of \eqref{har} for classical harmonic functions.  
\end{remark}

We stress that from the validity of \eqref{mvffl} we also indirectly infer that 
\[
\int_{\Rn} \Phi(x) dx = 1.
\]
We are now ready to state our first result about the hypoellipticity of $(-\Delta)^s$.

\begin{theorem}[Nonlocal Caccioppoli-Cimmino-Weyl lemma]\label{T:fh}
Suppose that $u\in \mathscr L_s(\Rn)$, and that on a given open set $\Om \subset \Rn$ one has $(-\Delta)^s u = 0$ in $\mathscr D'(\Om)$. Then, $u$ can be modified on a set of measure zero in $\Om$ so to coincide with a function in $C^\infty(\Om)$.
\end{theorem}

\begin{proof}
By Theorem \ref{T:abg2} we know that for a.e. $x\in \Om$ and $0<r<\operatorname{dist}(x,\p \Om)$, one has
\[
u(x) = \Phi_r \star u(x).
\]
To complete the proof it suffices to observe that the right-hand side $\Phi_r \star u$ defines a function in $C^\infty(\Om)$. But this follows combining part (ii) of Lemma \ref{L:abg} with the assumption $u\in \mathscr L_s(\Rn)$ which can be reformulated as 
\[
\int_{\Rn} \frac{|u(x)|}{1+|x|^{n+1-a}} dx <\infty.
\]
These two facts allow us to differentiate under the integral sign in the expression
\[
\Phi_r \star u(x) = \int_{\Rn} u(x') \Phi_r(x-x') dx',
\]
to conclude that $\Phi_r \star u\in C^\infty(\Om)$.

\end{proof}

\begin{remark}\label{R:bog2}
A similar approach based on the nonlocal Poisson kernel in Section \ref{S:smean} below can be found in Theorem 3.12 in \cite{BB2}.
\end{remark}

As we have already mentioned, Theorem \ref{T:fh} admits a stronger form which is contained in the next result, whose use has been kindly suggested to us by Camelia Pop. We also acknowledge a nice conversation with Fausto Segala, who pointed us to the classical reference \cite{Ho66} of the ``founding father".  

\begin{theorem}\label{T:hypo}
Let $m\in \R$ and let $p(x,\xi)\in S^{m}_{1,0}$. If the pseudodifferential operator $Pu(x) = \F^{-1}(p(\cdot,\xi))(x)$ is elliptic, then $P$ is hypoelliptic. In particular, $(-\Delta)^s$ is hypoelliptic. Therefore, if for given $f\in \mathscr S'(\Rn)$ the distribution $u\in \mathscr S'(\Rn)$ solves the equation
\[
(-\Delta)^s u = f, 
\]
then, $u\in C^\infty(\Om)$ on every open set $\Om\subset \Rn$ where $f\in C^\infty(\Om)$.
\end{theorem}

\begin{proof}
By \eqref{fls3} in Proposition \ref{P:pseudo} we know that $(-\Delta)^s$ is an elliptic pseudodifferential operator with symbol in the Kohn-Nirenberg class $S^{2s}_{1,0}$. Therefore, the second part of the theorem is an immediate consequence of the former. To prove the first part, we recall that the \emph{singular support} of a distribution is defined as the complement of the open set where the distribution is smooth. The \emph{pseudolocal property} of pseudodifferential operators states that if $Q$ is such an operator, then:
\begin{equation}\label{pseudo}
\operatorname{singsupp}\ Q(u)\ \subset \ \operatorname{singsupp}\ u.
\end{equation}
On the other hand, a basic theorem in microlocal analysis states that every elliptic pseudodifferential operator $P$ is invertible, in the sense that there exists a pseudodifferential operator $Q$, and a $C^\infty$ smoothing operator $R$ such that 
\[
Q \circ P = I+R,
\]
where $I$ is the identity operator, see section 4 in Chapter 7 in \cite{Ta2}. From this identity it follows that
\begin{equation}\label{pseudo2}
\operatorname{singsupp}\ (Q\circ P)(u) = \operatorname{singsupp}\ u. 
\end{equation}
Now, by a repeated application of \eqref{pseudo} we find
\begin{equation}\label{pseudo3}
\operatorname{singsupp}\ (Q\circ P)(u)\ \subset\ \operatorname{singsupp}\ P(u)\ \subset\ \operatorname{singsupp}\ u.
\end{equation}
From \eqref{pseudo2} and \eqref{pseudo3} we conclude that
for any elliptic pseudodifferential operator $P$ we must have
\begin{equation}\label{pseudo33}
\operatorname{singsupp}\ P(u) = \operatorname{singsupp}\ u.
\end{equation}
This equation proves that the set where $P(u)\in C^\infty$ coincides with that where $u\in C^\infty$, or, equivalently, that the operator $P$ is hypoelliptic.                     
For \eqref{pseudo33} one should see Proposition 4.1 of \cite{Ta2}, or also H\"ormander's seminal work \cite{Ho66}. We mention that the above proof also works for the more general class of H\"ormander symbols $S^m_{\rho,\delta}$, with $0\le \delta<\rho\le 1$.

\end{proof}


\section{Regularity at the boundary}\label{S:reg}

The results in the previous section show that, as far as interior regularity is concerned, 
the nonlocal operator $(-\Delta)^s$ behaves much as its local predecessor $-\Delta$. But what about global regularity in boundary value problems? We will see in the present section that here the situation departs dramatically from the local case. 

A central focus of investigation in the classical theory of elliptic and parabolic equations is the Dirichlet problem. For instance, for the Laplacean, given a bounded open set $\Om\subset \Rn$ and a function $\vf\in C(\p \Om)$, one seeks a function $u\in C^2(\Om)\cap C(\overline \Om)$ such that
\begin{equation}\label{dirpb}
\begin{cases}
\Delta u = 0\ \ \ \ \ \ \text{in}\ \Om,
\\
u = \vf\ \ \ \ \ \ \ \ \text{on}\ \p \Om.
\end{cases}
\end{equation}
From Corollary \ref{C:hypo} we know that, when such a $u$ exists, we have in fact $u\in C^\infty(\Om)$. One of the culminating achievements of classical potential theory is the Wiener-Perron-Brelot method which, for any bounded open set $\Om$, allows to construct a generalized solution $u$ to \eqref{dirpb}. Watch out! Here generalized means that, although $u\in C^\infty(\Om)$ and $\Delta u = 0$ in $\Om$, it is not guaranteed that $u\in C(\overline \Om)$. A famous example is provided by the so-called \emph{Lebesgue spine}, see for instance \cite{He}. 

However, the celebrated Wiener's criterion guarantees that when, for instance, $\Om$ is a Lipschitz domain (much less regularity does in fact suffice), then the generalized solution Wiener-Perron-Brelot solution to the problem \eqref{dirpb} is in fact H\"older continuous up to the boundary, if such is $\vf$ on $\p \Om$. If $\Om$ is better than Lipschitz, then more regularity  is expected for $u$. For instance, a well-known fact is that when $\Om$ is a $C^{1,1}$ domain and $\vf = 0$ on a portion of $\p \Om$, if $u\ge 0$ in $\Om$ then nearby a point $x_0\in \p \Om$ where $\vf$ vanishes one has for some universal constant $C>0$,
\begin{equation}\label{le}
C  \frac{d(x,\p \Om)}{r} \le \frac{u(x)}{u(A_r(x_0))} \leq C^{-1}  \frac{d(x,\p \Om)}{r},
\end{equation}
where $d(x,\p \Om) = \operatorname{dist}(x,\p \Om)$, and $A_r(x_0)\in \Om$ is a so-called nontangential point attached to $x_0$, see for instance \cite{Ga0}. The estimate \eqref{le} in essence says that $u$ has a linear growth near the boundary, and thus it is Lipschitz continuous up to $\p \Om$. We observe the important fact that, since the constant $C$ in \eqref{le} is independent of the particular harmonic function $u$, we conclude that for any two nonnegative harmonic functions $u, v$, which continuously vanish on a portion of the boundary, the so-called \emph{Boundary Harnack Principle} holds:
 \begin{equation}\label{cpi}
 C \frac{u(A_r(x_0))}{v(A_r(x_0))} \le \frac{u(x)}{v(x)} \leq C^{-1} \frac{u(A_r(x_0))}{v(A_r(x_0))}.
 \end{equation}
Thus, \emph{all nonnegative harmonic functions which vanish on a portion of the boundary in a $C^{1,1}$ domain, must do so at the same linear rate}.

We note that for nonnegative solutions of $(-\Delta)^s$ in a Lipschitz domain the estimate \eqref{cpi} has been proved in Theorem 1 in \cite{BK}, thus in a Lipschitz domain the Boundary Harnack Principle does hold for $(-\Delta)^s$ (however, for the correct formulation of the Dirichlet problem in the nonlocal case, see \eqref{dpom} below).

However, even in the classical local case of the Laplacean the Boundary Harnack Principle \eqref{cpi} \underline{does not} imply the linear decay estimate! The reason for this is that, whereas \eqref{cpi} only depends on general \emph{metric properties} of the relevant domain,  and therefore does hold
for large classes of domains with rough boundaries (for instance, in Lipschitz
or even NTA domains, see \cite{CFMS}, \cite{JK}), the linear decay estimate
\eqref{le} breaks down if the domain fails to satisfy a uniform bound on its
curvatures. 
For instance, if $0<\theta_0 < \pi/2$ and we consider in $\R^2$
the convex circular sector $\Om = \{(r,\theta)\mid 0<r<1 ,
|\theta|<\theta_0\}$, where $\theta$ indicates the angle formed by
the directional vector of the point $(x,y)$ with the positive
direction of the $y$-axis, the function $u(r,\theta) = r^\lambda
\cos (\lambda\theta)$ is a nonnegative harmonic function in $\Om$
vanishing on that portion of $\p \Om$ corresponding to $|\theta| =
\theta_0$ provided that\[ \lambda  = \frac{\pi}{2\theta_0}.
\]
From our choice, we have $\lambda >1$ and therefore this example
shows that for domains without an interior tangent ball the estimate
from below in \eqref{le} cannot possibly hold in general. Using the
same type of domain and function, but this time with
$\pi/2<\theta_0<\pi$ (a non-convex cone) we see that if the tangent
outer ball condition fails, then there exist harmonic functions
which vanish at the boundary at best with a H\"older rate $<1$.
Therefore, the estimate from above in \eqref{le} cannot possibly
hold in general. Now, it is well-known that the uniform tangent ball condition both from inside and outside characterizes $C^{1,1}$ domains, and from the above examples it is clear that this degree of smoothness is essentially optimal if one looks for a linear decay such as that in \eqref{le}.

In the nonlocal case the Dirichlet problem is formulated as follows, see 21. on p. 267 in \cite{La}. Let $\Om\subset \Rn$ be a bounded open set. Given $\vf\in C(\Rn\setminus \Om)$, and such that
\[
\int_{\Rn\setminus \Om} \frac{|\vf(x)|}{1+|x|^{n+2s}} dx < \infty,
\]
one seeks $u\in \mathscr L_s(\Rn)\cap C(\Rn)$ such that
\begin{equation}\label{dpom}
\begin{cases}
(-\Delta)^s u = 0\ \ \  \text{in}\ \Om, 
\\
u = \vf\ \ \ \ \ \ \ \  \text{in}\ \Rn\setminus \Om.
\end{cases}
\end{equation}
As we have seen, there exists a unique solution to \eqref{dpom}. This follows from the maximum principle in Proposition \ref{P:mp} above.

From Theorem \ref{T:fh} we know that such $u$ must be in $C^\infty(\Om)$. But what about its regularity at the boundary of $\Om$? For instance, if we have a $C^{1,1}$ domain $\Om\subset \Rn$, then is an estimate such as \eqref{le} true? 
The answer to this question is negative! Proposition \ref{P:tfs} below shows that there exist $C^\infty$ domains (in fact, real analytic), and $C^\infty$ boundary data $\vf$ (in fact, real analytic)
such that the solution to the Dirichlet problem \eqref{dpom} is not any better than $C^{0,s}$ H\"older continuous near $\p \Om$. Perhaps this negative phenomenon at the boundary can be understood with the fact that standard smoothness is badly altered if we look at it with the spectacles of the intrinsic scalings \eqref{di} of $(-\Delta)^s$, which instead preserve only some degree of H\"older regularity.

This is yet another example of the fact, highlighted in Section \ref{S:flso}, that the nonlocal operator $(-\Delta)^s$ displays a behavior that is typical of subelliptic operators such as the Baouendi operator $\Ba$, and/or the sub-Laplacean on the Heisenberg group $\Hn$. We recall in this respect that it was shown by D. Jerison in his Ph.D. Dissertation \cite{Jer} that there exists real analytic domains $\Om\subset \Hn$ and real analytic boundary data $\vf$ for which the solution to the Dirichlet problem for the sub-Laplacean $\Delta_H$ in \eqref{sl2} is not any better than H\"older continuous up to the boundary. However, Theorems \ref{T:xavi} and \ref{T:grubb} below show that, despite the failure of the estimate \eqref{le} in $C^{1,1}$ domains, the latter does hold if one replaces $d(x,\p \Om)$ with $d(x,\p \Om)^s$!

In Euclidean analysis and partial differential equations the function $u(x) = \frac{|x|^2}{2n}$ plays a special role. First of all, it has the remarkable property that $\Delta u \equiv 1$, which means that the global vector field $\nabla u$ has constant divergence (this is strongly connected to the \emph{flatness} of Euclidean space). Secondly, and because of this, it serves as a barrier in many boundary value problems. A first beautiful and simple, yet enormously important, instance of this is the weak Maximum Principle for a function $v$ such that $\Delta v\ge 0$.  By considering the function $w = v + \e u$ one reduces matters to the situation when $\Delta w >0$, in which case the maximum principle trivially follows from calculus. 

In connection with the globally defined function $u(x) = \frac{|x|^2}{2n}$, it is important to consider the so-called \emph{torsion function} of a connected, bounded open set $\Om\subset \Rn$, i.e., the unique solution of the Dirichlet problem
\begin{equation}\label{iso}
\begin{cases}
- \Delta u = 1\ \ \ \ \ \text{in}\ \Om,
\\
u = 0\ \ \ \ \ \ \ \ \ \ \text{on}\ \p \Om.
\end{cases}
\end{equation}
A celebrated theorem of Serrin in \cite{Se} states that $\Om$ is a ball if and only if $\frac{\p u}{\p \nu} = - \kappa$ on $\p \Om$. In fact, when $\Om = B(x_0,R)$, then one knows that
\begin{equation}\label{tf}
u(x) = \frac{R^2 - |x-x_0|^2}{2n},
\end{equation}
for which we easily recognize that $\frac{\p u}{\p \nu} = - \frac Rn$. An  alternative approach, based on the strong maximum principle and a beautiful integral identity of Rellich, was proposed by Weinberger in \cite{We} simultaneously to Serrin's paper. One should also see the later work \cite{GLe}, in which Weinberger's ideas were generalized to nonlinear degenerate elliptic equations such as the $p$-Laplacean.

It is remarkable that the mere existence of a smooth solution to \eqref{iso}, or to the related Neumann problem
\begin{equation}\label{iso22}
\begin{cases}
\Delta u = 1\ \ \ \ \ \ \ \ \ \ \ \text{in}\ \Om,
\\
\frac{\p u}{\p \nu} = \frac{|\Om|}{|\p \Om|}\ \ \ \ \ \ \ \  \text{on}\ \p \Om,
\end{cases}
\end{equation}
have deep implications in geometry. For instance, in his note \cite{Re} R. Reilly using \eqref{iso} and an adaption of Weinberger's ideas in \cite{We} provided a beautiful proof of A.D. Alexandrov's \emph{soap bubble theorem} stating that the only (sufficiently) smooth compact hypersurface in $\Rn$ having constant mean curvature must be a ball. Another beautiful result is Cabr\'e's proof of the isoperimetric inequality which combines the Alexandrov-Bakelman-Pucci maximum principle with the solution of problem \eqref{iso22}, see \cite{Ca} and \cite{Ca2}. Section 2.1.2 of \cite{Ca} also contains an interesting account of the probabilistic interpretation of the torsion function $u(x)$ in \eqref{iso} above as the expected time for a particle located at $x\in \Om$ to hit the boundary. For this one should also see Getoor's original paper on the subject \cite{G}.  

In connection with the fractional Laplacean we mention that an interesting nonlocal version of Serrin's theorem was obtained in the work \cite{FJ}, based on an adaption of the moving plane method. Whereas such method lends itself well to the fractional setting, we strongly feel that Weinberger's approach should be equally investigated because of its potential geometric implications. In this respect one should see the interesting work \cite{ROS2}, in which the authors establish a nonlocal version of the Rellich-Pohozaev identity as a consequence of their regularity results in \cite{ROS}. It is an intriguing and challenging question whether such result can be employed to provide a proof alternative to that in \cite{FJ}. In fact, this aspect is deeply connected to the questions we raise at the end of Section \ref{S:gamma} below. 

In light of the above considerations it is natural to seek a function analogous to \eqref{tf} for the non-local operator $(-\Delta)^s$, i.e., a solution to the equation $(-\Delta)^s u = 1$ in ball, with zero boundary data. In what follows we construct such a function. Our presentation is based on Lemma \ref{L:ftss} above, and it differs in part from the original one in the above cited paper \cite{G}. The reader should also see \cite{D} and \cite{DKK17} for more general results and the discussion in \cite{BV} of the one-dimensional case.

\begin{prop}[Torsion function for the ball]\label{P:tfs}
For $0<s<1$ consider  the non-homogeneous Dirichlet problem
\begin{equation}\label{problem}
\begin{cases}
(-\Delta)^s u = 1 \ \ \ \ \text{in}\ B(x_0,R),
\\
 u = 0\ \ \ \ \ \ \ \ \ \ \ \ \text{in}\ \Rn\setminus B(x_0,R).
\end{cases}
\end{equation}
Then, the unique solution to \eqref{problem} is provided by 
\begin{equation}\label{torsion}
u(x) = \frac{\Gamma(\frac n2)}{4^s \Gamma(\frac{n+2s}{2})\Gamma(s+1)} \left(R^2-|x-x_0|^2\right)^s_+.
\end{equation}
\end{prop}

\begin{proof}
By dilation and translation it suffices to consider the case $x_0 = 0$ and $R = 1$. For $0<s<1$ consider the equation
\begin{equation}\label{ntor}
(- \Delta)^s B_s = f,
\end{equation}
where 
\[
B_s(x) = \frac{\left(1-|x|^2\right)^s_+}{\Gamma(s+1)}
\]
is the Bochner-Riesz kernel introduced in \eqref{br} above. Applying the Fourier transform to both sides of \eqref{ntor}, and using \eqref{fls3} in Proposition \ref{P:slapft}, we find
\[
(4\pi^2 |\xi|^2)^s \hat B_s(\xi) = \hat f(\xi).
\]
If in this formula we use \eqref{FTBR} in Lemma \ref{L:BR} with $z = s$ to compute $\hat B_s(\xi)$, we obtain
\begin{align*}
\hat f(\xi) & = (4\pi^2 |\xi|^2)^s \pi^{-s}|\xi|^{-\left(\frac{n}2+s\right)}J_{\frac{n}2+s}(2\pi|\xi|)
\\
& = 4^s \pi^s |\xi|^{-\frac n2 + s} J_{\frac{n}2+s}(2\pi|\xi|).
\end{align*}

We next use Theorem \ref{T:Fourier-Bessel} again to recover $f(x)$ from the latter formula
\begin{align*}
f(x) & = 2\pi 4^s \pi^s |x|^{-\frac{n-2}2}\int^\infty_0 t^{\frac{n}2} t^{-\frac n2 + s} J_{\frac{n}2+s}(2\pi t)J_{\frac{n-2}2}
(2\pi|\xi|t) dt.
\\
& = 2\pi 4^s \pi^s |x|^{-\frac{n-2}2}\int^\infty_0 t^{s} J_{\frac{n}2+s}(2\pi t)J_{\frac{n-2}2}(2\pi|\xi|t) dt.
\end{align*}
Comparing this formula with Lemma \ref{L:GR} we see that we presently need to have
\[
\la = - s,\ \ \ \nu = \frac{n}2+s,\ \ \ \mu =  \frac{n-2}2,\ \  \ \ a = 2\pi,\ \ \ b = 2\pi |x|.
\]
Since $0<s<1$, we clearly have $\la > - 1$. Furthermore, since we are interested in values of $x$ such that $0<|x|<1$, we have $0<b<a$. With the above choices of $\la, \nu, \mu$ and $a$ we obtain
\[
\nu + \mu - \la + 1= \frac{n}2+s + \frac{n}2 -1 + s + 1 = n+2s,
\]
and
\[
-\nu + \mu - \la + 1 = - \frac{n}2 - s + \frac{n}2 -1 + s + 1 = 0,
\]
\[
\mu-\la+1 = \frac{n}2 -1 + s + 1 = \frac{n+2s}{2},
\]
\[
\nu - \mu + \la + 1 = \frac{n}2+s - \frac{n}2 +1 -  s + 1 = 2.
\]

Applying Lemma \ref{L:GR} above we thus find for $0<|x|<1$
\begin{align*}
f(x) & = 2\pi 4^s \pi^s |x|^{-\frac{n-2}2}   (2\pi |x|)^{\frac{n-2}2} \frac{\Gamma(\frac{n+2s}{2})}{2^{-s} (2\pi)^{\frac n2 + s} \Gamma(1) \Gamma(\frac n2)}\ F\left(\frac{\nu + \mu - \la + 1}{2},0;\frac n2;|x|^2\right), 
\end{align*}
where $F(a,b;c;z)$ is Gauss' hypergeometric function in Definition \ref{D:hyper} above.
Since thanks to \eqref{zeroF} we have
\[
F\left(\frac{\nu + \mu - \la + 1}{2},0;\frac n2;|x|^2\right) = 1,
\]
we finally conclude that for any $|x|<1$
\[
f(x) \equiv \frac{4^s \Gamma(\frac{n+2s}{2})}{\Gamma(\frac n2)}.
\]
If we substitute such $f$ in \eqref{ntor} above, we see that we have proved that for $|x|<1$
\[
(- \Delta)^s B_s(x) \equiv  \frac{4^s \Gamma(\frac{n+2s}{2})}{\Gamma(\frac n2)}.
\]
Keeping in mind the above definition of $B_s$ it is thus clear that the function
\[
u(x) = \frac{\Gamma(\frac n2)}{4^s \Gamma(\frac{n+2s}{2})\Gamma(s+1)} \left(1-|x|^2\right)^s_+.
\]
provides the desired solution \eqref{torsion} to the problem \eqref{problem}.

\end{proof}

\begin{remark}\label{R:tf}
Let us note explicitly that as $s\to 1$ the constant
\[
\frac{\Gamma(\frac n2)}{4^s \Gamma(\frac{n+2s}{2})\Gamma(s+1)}\ \longrightarrow\ \frac{1}{2n}.
\] 
Thus, in the local case when $s = 1$ we recover the standard torsion function in \eqref{tf} above.
\end{remark}

\begin{remark}\label{R:imbert}
We want to bring to the reader's attention an interesting connection between the nonlocal torsion function $u$ in Proposition \ref{P:tfs}, and the recent paper \cite{BIK} in which the authors construct an analogue of the self-similar Barenblatt-Kompaneets-Pattle-ZelÕdovich solution for the following nonlocal porous medium equation
\[
\p_t U = \operatorname{div}(|U|\nabla^{2s-1}(|U|^{m-2} U)),
\]
where $\nabla^{2s-1} \overset{def}{=} \nabla (-\Delta)^{s-1}$. 
In their Theorem 2.2 they prove that such self-similar solution is obtained in the form  
\[
U(x,t) = \frac{1}{t^{n\la}} \Phi\left(\frac{x}{t^\la}\right),\ \ \ \ \ \la = \frac{1}{n(m-1) + 2s},
\]
where $\Phi(x) = C(n,m,s) u(x)^{\frac1{m-1}}$, $u(x)$ is the function in Proposition \ref{P:tfs}, and $C(n,m,s)$ is an explicit constant.  
\end{remark}

Concerning Proposition  \ref{P:tfs}  the reader should notice that, since for $B(0,1)$ we have $d(x) = \operatorname{dist}(x,\p B(0,1)) = 1 - |x|$, the torsion function $u$ in \eqref{torsion}  satisfies in an obvious way the property that 
\[
\frac{u}{d^s}\ \in C^{\infty}(\overline B(0,1)\setminus\{0\}). 
\]
In the framework of $C^{1,1}$ domains, in Theorem 1.2 of their cited paper \cite{ROS} the authors prove the following general result.

\begin{theorem}\label{T:xavi}
Let $\Om$ be a bounded $C^{1,1}$ domain, $f\in L^\infty(\Om)$, and $u$ be a weak solution to the non-homogeneous Dirichlet problem
\begin{equation}\label{nhdp}
\begin{cases}
(-\Delta)^s u = f\ \ \ \      \text{in}\ \Om,
\\
u = 0\ \ \ \ \ \ \ \ \ \text{in}\ \Rn\setminus \Om.
\end{cases}
\end{equation}
Then, $u/d^s\in C^{0,\alpha}(\overline \Om)$ for some $0<\alpha <\min\{s,1-s\}$, and one has
\[
||u/d^s||_{C^{0,\alpha}(\overline \Om)} \le C ||f||_{L^\infty(\Om)},
\]
with $\alpha, C$ depending only on $\Om, n, s$. 
\end{theorem}

The reader is also encouraged to see the beautiful survey paper \cite{RO} on boundary regularity and the nonlocal Pohozaev identity. We also thank X. Ros-Oton for pointing to our attention the following interesting result of Grubb, see \cite{Gr1}, \cite{Gr2}.

\begin{theorem}\label{T:grubb}
Let $\Om\subset \Rn$ be a $C^\infty$ domain, and let $f\in C^\infty(\overline \Om)$. If $u$ solves \eqref{nhdp}, then $u/d^s\in C^\infty(\overline \Om)$.
\end{theorem}

In Theorem \ref{T:grubb} the notation $d = d (x)$ represents a positive function which coincides with the distance to the boundary of $\Om$ in a neighborhood of $\p \Om$.  


\section{Monotonicity formulas and unique continuation for $(-\Delta)^s$}\label{S:almgren}

We open this section with a quote from the paper \cite{CS07}: ``Monotonicity formulas are a very powerful tool in the study of the regularity
properties of elliptic PDEs. They have been used in a number of problems to
exploit the local properties of the equations by giving information about the blowup configurations". In what follows we discuss some monotonicity formulas related to nonlocal operators. 

In order to provide some motivation and historical perspective we recall that the most fundamental property of harmonic functions is the so-called \emph{strong unique continuation property}, which we will abbreviate in (sucp) henceforth. It states that a solution of $\Delta u = 0$ in a connected open set cannot vanish to infinite order at one point, unless it vanishes identically. Here, vanishing to infinite order means, for instance, that for any $N\in \mathbb N$ one has as $r\to 0^+$
\[
\underset{B(x_0,r)}{\sup} |u| = O(r^N).
\]
The weaker \emph{unique continuation property} (ucp) states that if $u$ vanishes in an open subset, then $u$ must vanish everywhere. These properties are shared by large classes of differential operators, such as for instance the stationary Schr\"odinger operator $H = - \Delta + V$ which plays a central role in quantum mechanics. At the basis of J. Von Neumann's axiomatic formulation there is the assumption of the absence of eigenvalues of $H$ embedded in the continuous spectrum.  However, a famous example of Wigner and Von Neumann (1954) showed that in the presence of tunneling effects such eigenvalues can in fact arise, see \cite{RS}. It thus became important to understand assumptions on $V$ that would rule out such possibility.

On the other hand, a famous theorem of F. Rellich \cite{Rel} states that for every $R>0$ there cannot exist eigenfunctions of the Laplacean in $L^2(\Rn\setminus B(0,R))$ corresponding to $\lambda >0$.  
This is where the (ucp) for the operator $H = - \Delta + V$, with $V\in L^\infty_{loc}(\Rn)$ comes into the play. Assuming for simplicity that $V$ be compactly supported, it is easy to see that
\[
\text{(ucp)}\ +\ \text{Rellich}\ \Longrightarrow\ \text{Absence of embedded eigenvalues}.
\]
In 1939 T. Carleman introduced a powerful method to prove the (ucp) or even the (sucp) for $H = - \Delta + V$ based on weighted a priori estimates of the type
\[
||e^{t\Phi(x)} u||_{L^q(\Rn)} \le C(n) ||e^{t \Phi(x)} \Delta u||_{L^p(\Rn)},
\]
where $1\le p\le q<\infty$, $t\in \R$ is a parameter which is allowed to go to $+\infty$ avoiding a discrete set, and $u\in C^\infty_0(\Rn\setminus\{0\})$.
Notice that when $\Phi\equiv 0$ and $\frac 1p - \frac 1q = \frac 2n$, the above estimate is just the Sobolev embedding theorem (for this theorem see e.g. \cite{St}). When $\Phi\not\equiv 0$ instead, this gap allows a potential $V\in L^{n/2}_{loc}(\Rn)$, which is sharp for (sucp), see the celebrated work of D. Jerison and C. Kenig \cite{JK85}. However, when $V\in L^\infty_{loc}(\Rn)$, then no gap is necessary. As a consequence, one can take $p = q = 2$  and use Hilbert spaces framework (integration by parts).

A different method to establish the (sucp) for more general elliptic operators of the type
\begin{equation}\label{lip}
Lu = \operatorname{div}(A(x)\nabla u) + <b,\nabla u> + V u,
\end{equation}
was developed by F.H. Lin and the author in \cite{GL1}, \cite{GL2}. Their approach is energy based and it establishes directly the (sucp), and in fact more, without considerable technical complexities. The central idea is a remarkable monotonicity property that was discovered by F. Almgren in \cite{A} in his approach to the regularity of mass minimizing currents.

\begin{theorem}[Almgren's monotonicity formula]\label{T:al}
Let $\Delta u = 0$ in $B(0,1)$. Then, the so-called \emph{frequency} of $u$ 
\[
r\ \to\ N(u,r) = \frac{ r D(u,r)}{H(u,r)} = \frac{r \int_{B_r} |\nabla u|^2}{\int_{S_r} u^2} 
\]
is monotonically increasing. Moreover, $N(u,r) \equiv \kappa$ if and only if $u$ is homogeneous of degree $\kappa$.
\end{theorem}  

The name \emph{frequency} comes from the fact that when $u = P_\kappa$, a harmonic polynomial having frequency $\kappa$, then $N(u,r)\equiv \kappa$.
One important consequence of the monotonicity of the frequency (in fact, only its boundedness suffices!) is the following \emph{doubling condition}: for every $0<r<1$ one has
\[
\int_{B_{2r}} u^2 dx \le C(n,||u||_{W^{1,2}(B_1)})  \int_{B_{r}} u^2 dx.
\]
It is well know, see \cite{GL1}, that:
\[
\operatorname{Doubling\ condition}\ \Longrightarrow\ \operatorname{(sucp)}.
\]

In the cited papers \cite{GL1} and \cite{GL2} Theorem \ref{T:al} was generalized to solutions of elliptic equations $Lu = 0$, with $L$ as in \eqref{lip} and the coefficients of the leading part are Lipschitz continuous. This is known to be the minimal smoothness for the unique continuation to be true. For counterexamples see the groundbreaking paper by Plis \cite{Pl} for nondivergence form equations, and the subsequent work of Miller \cite{Mi} for divergence form operators. 

In view of what has been said up to this point, it is natural to ask whether the (sucp) holds for nonlocal operators. Concerning this question in this section we will prove the following result.

\begin{theorem}[Strong unique continuation for $(-\Delta)^s$]\label{T:sucpfl}
Let $u\in \mathscr L_s(\Rn)\cap C(\Rn)$ be a weak solution of $(-\Delta)^s u = 0$ in a connected open set $\Om\subset \Rn$. If $u$ vanishes to infinite order at a point $x_0\in \Om$, then $u\equiv 0$ in $\Om$.
\end{theorem}

\begin{remark}\label{R:sucp}
The reader should bear in mind that there is a way of dealing with Theorem \ref{T:sucpfl} different from the one that we are going to present below. In fact, similarly to the local case, solutions of $(-\Delta)^s u = 0$ in an open set $\Om\subset \Rn$ are not only $C^\infty(\Om)$, see Theorems \ref{T:fh} and \ref{T:hypo}, but in fact real analytic! We will however not pursue this interesting aspect in these notes since, besides not being the real-analyticity an easy fact to establish, it would also confine our discussion to a special situation. 
\end{remark}  

Since there presently exists no direct Almgren type monotonicity formula, nor to the best of our knowledge there exists a direct Carleman estimate for solutions of $(-\Delta)^s u = 0$, in order to prove Theorem \ref{T:sucpfl} one has to proceed indirectly and resort to the extension problem \eqref{ext2} above. As we will see, this imposes a delicate detour since it is not a priori true that the information of the vanishing to infinite order transfers from the solution of $(-\Delta)^s u = 0$, to that of the extension problem.
In light of this it would be desirable (and also quite interesting) to know whether a direct nonlocal analogue of Theorem \ref{T:al} be possible. Perhaps the Rellich-Pohozaev identity in the cited work \cite{ROS2} might prove useful in this direction. Concerning nonlocal integral identities we also mention the work \cite{Gr3} which contains a generalization to pseudodifferential operators of the results in \cite{ROS2}.

In what follows we sketch a proof of Theorem \ref{T:sucpfl} based on monotonicity formulas. This is a special case of the work \cite{FF14}, in which the authors prove a general result for nonlocal semilinear equations combining the extension procedure with the approach in \cite{GL1}, \cite{GL2}. We mention that a different approach that combines the extension procedure with Carleman estimates was developed in \cite{Ru}. 

In the sequel we continue to use the notation of Section \ref{S:flso}, and will assume that $-1<a<1$ is given. 
We will need the following simple lemma which provides a trace inequality (and a reverse form of it) for functions in the Sobolev space $W^{1,2}(B_e(r),|y|^a dX)$.

\begin{lemma}\label{L:trace}
For $r>0$ let $\tilde U\in W^{1,2}(B_e(r),|y|^a dX)$. There exists a constant $C(n,a)>0$ such that
\begin{equation}\label{t1}
\int_{S_e(r)} \tilde U^2 |y|^a d\sigma  \le C(n,a) \left(\frac 1r \int_{B_e(r)} \tilde U^2 |y|^a dX +  r \int_{B_e(r)} |\nabla \tilde U|^2 |y|^a dX\right).
\end{equation}
and
\begin{equation}\label{t2}
\frac 1r \int_{B_e(r)} \tilde U^2 |y|^a dX \le C(n,a) \left(\int_{S_e(r)} \tilde U^2 |y|^a d\sigma +  r \int_{B_e(r)} |\nabla \tilde U|^2|y|^a dX\right).
\end{equation}
\end{lemma}

\begin{proof}
Since $\omega(X) = |y|^a$ is an $A_2$ weight of Muckenhoupt, from the general result in Theorem 6.1 in \cite{Chu} we know that $C^\infty(\overline B_e(r))$ is dense in $W^{1,2}(B_e(r),|y|^a dX)$. Therefore, in order to prove \eqref{t1} or \eqref{t2} it suffices to assume that $\tilde U\in C^\infty(\overline B_e(r))$. For any such function, if we indicate with $\nu$ the outer unit normal to $S_e(r)$, applying the divergence theorem we obtain
\begin{align*}
\int_{S_e(r)} \tilde U^2 |y|^a d\sigma & = \frac 1r \int_{S_e(r)}  <\tilde U^2 |y|^a X,\nu> d\sigma = \int_{B_e(r)} \operatorname{div}(\tilde U^2 |y|^a X) dX
\\
& = \frac{n+a+1}r \int_{B_e(r)} \tilde U^2 |y|^a dX + \frac 2r \int_{B_e(r)} \tilde U <\nabla \tilde U,X> |y|^a dX
\end{align*}
From this identity the desired conclusions \eqref{t1} and \eqref{t2} easily follow.

\end{proof}

We now assume that $\tilde U(x,y)\in W^{1,2}(B_e(R),|y|^a dX)$ is a weak solution of the equation $L_a \tilde U = 0$ in a ball $B_e(R)$, and we assume that $\tilde U$ is even in the variable $y$. For any $r\in (0,R)$ we consider the quantities 
\begin{align}
\tilde H(\tilde U,r) & = \int_{S_{e}(r)} \tilde U^2 |y|^a  dH_{n},
\label{tH}\\
\tilde D(\tilde U,r) & = (1-a)^{2a}\int_{B_e(r)}  |\nabla \tilde U|^2 |y|^a dx dy,\ \ \ 0<r<R,
 \label{tD}\\
 \tilde N(\tilde U,r) &=  \frac{r \tilde D(\tilde U,r)}{\tilde H(\tilde U,r)}.
\label{tN}
\end{align}

In Theorem 6.1 in \cite{CS07} Caffarelli and Silvestre proved the following important result.

\begin{theorem}[Monotonicity formula of Almgren type for $L_a$]\label{T:almgrenEO}
Suppose that for no $r\in (0,R_0)$ we have $\tilde H(\tilde U,r) =0$. Then, the function $r\to \tilde N(\tilde U,r)$ is nondecreasing on $(0,R_0)$. Furthermore, $\tilde N(\tilde U,\cdot) \equiv \tilde \kappa$ if and only if $\tilde U$ is homogeneous of degree $\tilde \kappa$ with respect to the standard Euclidean dilations $\tilde \delta_\lambda(x,y) = (\la x,\la y)$.
\end{theorem}

In \cite{GRO} it was shown that, at least in the range $0\le a<1$, Theorem \ref{T:almgrenEO} can be directly derived by the Almgren type monotonicity formula for solutions of $\Ba U = 0$ established in \cite{Ga}. In fact, with $a$ is such range consider the number $\alpha \ge 0$ defined by $\alpha = \frac{a}{1-a}$, and for $U(x,z)$ defined by the equation \eqref{tu}, consider the height function, the Dirichlet integral and the frequency of $U$, respectively defined  by:
\begin{align}
H(u,r) & = \int_{S_{\rho_\alpha}(r)} U^2\frac{\psi_\alpha}{|\nabla \ra|}\ dH_{n},\ \ \ 0<r<R, \label{D}\\
D(u,r) & = \int_{B_{\rho_\alpha}(r)}|\nabla_\alpha U|^2 dx dz
\\
N(U,r) & = \frac{rD(U,r)}{H(U,r)},
\end{align}
where $|\nabla_\alpha U|^2$ is the degenerate \emph{carr\'e du champ} introduced in \eqref{carre} above. One has the following transformation formulas contained in Lemmas 2.7 and 2.8 in \cite{GRO}.

\begin{prop}\label{P:hth}
For every $r>0$ one has
\begin{equation}\label{tfs}
\tilde H(\tilde U,r) =  r^a H(U,h(r)),\ \ \ \ \ \ \ \ \tilde D(\tilde U,r) = (1-a)^{a}  D(U,h(r)).
\end{equation}
\end{prop}
As a consequence of \eqref{tfs} we obtain the remarkable formula
\begin{equation}\label{tfN}
\tilde N(\tilde U,r) = (1-a) N(U,h(r)).
\end{equation}

We next state a result which is a special case of Theorem 4.2 in \cite{Ga}.

\begin{theorem}[Almgren type monotonicity formula for $\Ba$]\label{T:Ga}
Suppose that for no $r\in (0,R_0)$ we have $H(U,r) =0$. Then, the function $r\to N(U,r)$ is nondecreasing in $(0,R_0)$. Furthermore, $N(U,\cdot) \equiv \kappa$ if and only if $U$ is homogeneous of degree $\kappa$ with respect to the nonisotropic dilations \eqref{dilB}.
\end{theorem}

Having stated Theorem \ref{T:Ga}, we now make an interesting observation.
 
\begin{theorem}\label{T:equi}
The following is true when $0\le a<1$:
\[
\text{Theorem}\ \ref{T:Ga}\ \Longleftrightarrow\ \text{Theorem}\ \ref{T:almgrenEO}.
\]
\end{theorem}

\begin{proof}
It is enough to observe that the function $r\to h(r)$ in \eqref{h} is monotonically increasing. The desired conclusion thus follows immediately by combining \eqref{tfN} with Theorem \ref{T:Ga}, keeping in mind that $\alpha = \frac{a}{1-a}$, and therefore $a = \frac{\alpha}{1+\alpha}$.  

\end{proof}

Theorem \ref{T:equi} was not noted in \cite{CS07}, but it was subsequently observed in Remark 3.2 in \cite{CSS}. We also observe that the same conclusion continues to work also in the range $-1<a<0$, since using the regularization procedures employed in \cite{Ga} one can extend the validity of the results  there to the range $-\frac 12 < \alpha <0$.

Let us return to the central objective of this section, namely Theorem \ref{T:sucpfl}. We note that one important corollary of Theorem \ref{T:almgrenEO} is the following.

\begin{corollary}\label{C:kappa}
Let $\tilde U$ be as in the hypothesis of Theorem \ref{T:almgrenEO}. Then,
$\underset{r\to 0^+}{\lim} N(\tilde U,r) = \tilde \kappa$
exists finite.
\end{corollary}

Another crucial consequence is the following result.

\begin{theorem}[Non-degeneracy]\label{T:nondeg}
Under the assumptions in Theorem \ref{T:almgrenEO}, given $R\in (0,R_0)$,  one has for every $0<r<R$ 
\begin{equation}\label{nd}
\tilde H(\tilde U,r) \geq  \tilde H(\tilde U,R) \left(\frac{r}{R}\right)^{n +a + 2||\tilde N(\tilde U,\cdot)||_{L^\infty(0,R)}}.
\end{equation}
\end{theorem}

\begin{proof}
By Theorem \ref{T:almgrenEO} we know that the function $t\to \tilde N(\tilde U,t)$ is monotonically increasing on $(0,R)$. On the other hand, using the equation $L_a \tilde U = 0$ and integration by parts, it is not difficult to prove that
\[
\tilde H'(\tilde U,t) = \frac{n+a}{t}\tilde H(\tilde U,t) + 2 \tilde D(\tilde U,t).
\]   
Keeping the definition \eqref{tN} in mind, we can rewrite this equation in the following way
\[
\frac{d}{dt} \log \frac{\tilde H(\tilde U,t)}{t^{n+a}} = 2 \frac{\tilde N(\tilde U,t)}{t}.
\]
Integrating this formula on the interval $(r,R)$, and using the monotonicity of $r\to \tilde N(\tilde U,r)$ (in fact, just the boundedness of this function suffices in this argument), we obtain the desired conclusion \eqref{nd}.

\end{proof}

\begin{corollary}\label{C:int}
If we have $\tilde H(\tilde U,R_0) \neq 0$, then we must have $\tilde H(\tilde U,r) \neq 0$ for all $0<r <R_0$.  
\end{corollary}

\begin{proof}

We argue by contradiction and assume that there exist $0 <\overline r < R_0$  such that $\tilde H(\tilde U,\overline r) =0$. Define
\[
\rho = \sup  \{r\leq R_0\mid \tilde H(\tilde U,r) =0\}.
\]
Since by the hypothesis $\tilde H(\tilde U,R_0) \neq 0$, we must have $0 < \rho< R_0$. But then, we have $\tilde H(\tilde U,r)\not= 0$ for $r \in (\rho, R_0]$. Applying Theorem \ref{T:nondeg} we obtain for $r \in (\rho, R_0]$
\[
\tilde H(\tilde U,r) \geq  \tilde H(\tilde U,R_0) \left(\frac{r}{R_0}\right)^{n +a + 2||\tilde N(\tilde U,\cdot)||_{L^\infty(0,R_0)}} > 0.
\]
Letting $r \to  \rho^+$ this leads to a contradiction since $\tilde H(\tilde U,\rho)=0$. 
 
\end{proof}

We next define an important family of non-homogeneous rescalings. In a different context, they were introduced the first time in \cite{ACS} in the blowup analysis of the Signorini problem. In what follows, we indicate with $X = (x,y)$ the generic point in $\Rn\times \R$.

\begin{definition}\label{D:abups}
Let $R_0>0$ be such that $\tilde H(\tilde U,r)>0$ for every $0<r<R_0$. We define the \emph{Almgren rescalings} of a function $\tilde U$ as 
\begin{equation}\label{bl}
\tilde U_r (X)= \frac{r^{\frac{n+a}2}\tilde U(\tilde \delta_r(X))}{\sqrt{\tilde H(\tilde U,r)}} = \frac{r^{\frac{n+a}2}\tilde U(rX)}{\sqrt{\tilde H(\tilde U,r)}}.
\end{equation}
\end{definition}
Obviously, if $\tilde U$ is a solution of $L_a \tilde U = 0$ in $B_e(R_0)$, then  $L_a \tilde U_r = 0$ in $B_e(R_0/r)$.  An elementary, yet crucial property, of the Almgren rescalings which follows from \eqref{tH} and a change of variable is that
\begin{equation}\label{15}
\tilde H(\tilde U_r,1)= \int_{S_e(1)} \tilde U_r^2(X) |y|^{a} dH_n(X) = 1.
\end{equation}
Another basic property is the following identity 
\begin{equation}\label{si00}
\tilde N(\tilde U_r,\rho) = \tilde N(\tilde U,r\rho),\ \ \ \ \ \ \ \ \ \ \ \ r, \rho >0.
\end{equation}

The next lemma plays a key role in the proof of Theorem \ref{T:sucpfl}. Its proof hinges crucially on the monotonicity formula in Theorem \ref{T:almgrenEO}.

\begin{lemma}\label{L:abu}
Suppose $\tilde H(\tilde U,R_0) \neq 0$. With $\tilde U_r$ as in \eqref{bl}, there exists a subsequence  $r_j \to 0$, and a function  $\tilde U_0: \Rn \times \R \to \R$, such that  $\tilde U_{r_j}=\tilde U_j$ converges uniformly to $\tilde U_0$  and $\nabla \tilde U_j \to \nabla \tilde U_0$ weakly in $L^{2}(\Rn\times \R,|y|^a dX)$ on compact subsets of $\Rn \times \R$. Moreover, $\tilde U_0$ is a weak solution (even in the variable $y$) to 
\begin{equation}\label{hom}
\begin{cases}
\operatorname{div} (|y|^{a} \nabla \tilde U_0)= 0,
\\
\underset{y \to 0}{\lim} y^{a} \partial_y \tilde U_0 =0,
\end{cases}
\end{equation}
on every compact  subset of $\Rn \times \R$. Finally, $\tilde U_0$ is homogeneous of degree $\tilde \kappa = \underset{r\to 0^+}{\lim} N(\tilde U,r)$. 
\end{lemma}

\begin{proof}

Using \eqref{15} and \eqref{si00} we obtain for $0<r<R_0$
\begin{equation}\label{gradients}
\int_{B_e(1)} |\nabla \tilde U_r(X)|^2 |y|^a dX = \tilde N(\tilde U_r,1) = \tilde N(\tilde U,r) \le \tilde N(\tilde U,R_0) <\infty, 
\end{equation}
where in the second to the last inequality we have used the monotonicity of $r\to \tilde N(\tilde U,r)$ in Theorem \ref{T:almgrenEO}. We note that since we are assuming $\tilde H(\tilde U,R_0) \neq 0$, Corollary \ref{C:int} allows to infer that $\tilde H(\tilde U,r) \neq 0$ for every $0<r<R_0$, and thus we are in the hypothesis of Theorem \ref{T:almgrenEO}. 

At this point we combine the trace inequality \eqref{t2} with \eqref{15} and \eqref {gradients} to infer that for every $0<r<R_0$
\begin{equation}\label{155}
\int_{B_e(1)} \tilde U_r^2 |y|^{a} dX <\infty.
\end{equation}
Combining \eqref{155} with \eqref{gradients}, we conclude that for every $0<r<R_0$
\[
||\tilde U_r||_{W^{1,2}(B_e(1),|y|^a dX)} \le C<\infty,
\]
for some constant independent of $r$. This implies the existence of a sequence $r_j\searrow 0$ and of a function $\tilde U_0\in W^{1,2}(B_e(1),|y|^a dX)$ such that
\[
\tilde U_{r_j}\ \longrightarrow\ \tilde U_0
\]
weakly in $W^{1,2}(B_e(1),|y|^a dX)$. By the local H\"older continuity of solutions of $L_a \tilde U = 0$, we conclude that, possibly on a subsequence, which we still denote $\tilde U_{r_j}$, we have convergence of $\tilde U_{r_j}\ \longrightarrow\ \tilde U_0$ in $C^{0,\alpha}$ norm on compact subset of $\Rn\times \R$. The Caccioppoli inequality \eqref{cacc} then implies that for any $0<s<t$
\begin{equation}\label{cacc00}
\int_{B_e(s)} |\nabla \tilde U_{r_j} - \nabla \tilde U_0|^2 |y|^a dX \le \frac{C(n,a)}{t-s} \int_{S_e(t)} |\tilde U_{r_j} - \tilde U_0|^2 |y|^a dH_n\ \longrightarrow\ 0,
\end{equation}
as $j\to \infty$. The inequality \eqref{cacc00} and the $C^{0,\alpha}_{loc}$ estimates for $\tilde U$ imply that $\tilde U_0\in W^{1,2}_{loc}(\Rn\times \R,|y|^a dX)$, and that $\tilde U_0$ is a weak solution of \eqref{hom} on every compact  subset of $\Rn \times \R$. 
Moreover, \eqref{cacc00} also implies that for any $\rho>0$
\begin{equation}\label{ndD}
\tilde D(\tilde U_{r_j},\rho)\ \longrightarrow\ \tilde D(\tilde U_{0},\rho).
\end{equation}
Since from the uniform convergence of $\tilde U_{r_j}\ \longrightarrow\ \tilde U_0$ in $C^{0,\alpha}$ we have
\begin{equation}\label{ndH}
\tilde H(\tilde U_{r_j},\rho)\ \longrightarrow\ \tilde H(\tilde U_{0},\rho).
\end{equation}
Next, we claim that for every $0<\rho<1$ we must have
\begin{equation}\label{nduzero}
\tilde H(\tilde U_{0},\rho) >0.
\end{equation}
This follows from the estimate 
\eqref{nd} in Theorem \ref{T:nondeg} that gives for every $0<r<R_0$ and $0<\rho<1$ 
\begin{equation}\label{nd00}
\tilde H(\tilde U,r\rho) \geq  \tilde H(\tilde U,r) \rho^{n +a + 2||\tilde N(\tilde U,\cdot)||_{L^\infty(0,R_0)}}.
\end{equation}
On the other hand, an easy change of variable gives
\[
\tilde H(\tilde U_r,\rho) = \frac{\tilde H(\tilde U,r\rho)}{\tilde H(\tilde U,r)},
\]
and from this observation and \eqref{nd00} we find
\[
\tilde H(\tilde U_{r_j},\rho) \ge \rho^{n +a + 2||\tilde N(\tilde U,\cdot)||_{L^\infty(0,R_0)}}.
\]
Letting $j\to \infty$ and using \eqref{ndH} we conclude that \eqref{nduzero} holds. Now that we know \eqref{nduzero}, from \eqref{ndD} and \eqref{ndH}
we conclude that as $j\to \infty$
\[
\tilde N(\tilde U_{r_j},\rho)\ \longrightarrow\ \tilde N(\tilde U_{0},\rho).
\]
On the other hand, \eqref{si00} and Corollary \ref{C:kappa} give
\[
\tilde N(\tilde U_{r_j},\rho) = \tilde N(\tilde U,r_j \rho) \to \tilde \kappa,
\]
as $j\to \infty$. Therefore, for every $0<\rho<1$ we must have
\begin{equation}\label{kappa=}
\tilde N(\tilde U_{0},\rho) \equiv \tilde \kappa.
\end{equation}
 Once we know this, we finish by invoking the second part of Theorem \ref{T:almgrenEO} that allows to infer that $\tilde U_0$ must be homogeneous of degree $\tilde \kappa$ in $B_e(1)$.

\end{proof}

\begin{definition}\label{D:abu}
We call any function $\tilde U_0$ as in Lemma \ref{L:abu} an \emph{Almgren blowup} of $U$ at $x_0 = 0$.
\end{definition}

We are now ready to give the

\begin{proof}[Proof of Theorem \ref{T:sucpfl}]
Let $u$ be as in the statement of the theorem. Without loss of generality we can assume that the point $x_0$ at which $u$ vanishes to infinite order be $x_0 = 0$. We want to show that $u\equiv 0$ in $\Om$. We would like to show that there exists $R_0>0$ such that $u\equiv 0$ in $B_e(R_0)$. A standard connectedness argument would then imply that $u\equiv 0$ in $\Om$. Consider the extension problem \eqref{ext2} above with Dirichlet datum $u$, and denote by 
$U(x,y) = P_s(\cdot,y)\star u(x)$ its solution in $\Rn\times \R^+$. 
Arguing as in the proof of Theorem \ref{T:abg2} we know that $\tilde U(x,y) = U(x,|y|)$, is a weak solution of $L_a \tilde U = 0$ in $\Om\times (-d,d)$, where $a = 1-2s$.

Let $R_0 = \operatorname{dist}(x_0,\p \Om)$. If we show that $\tilde H(\tilde U,R_0)=0$, then by the Caccioppoli inequality \eqref{cacc} we would infer that $\int_{B_e(s)} |\nabla \tilde U(x',y)|^2 |y|^a dx' dy  = 0$ for every $0<s<R_0$ and thus $\tilde U$ would have to be constant in $B_e(R_0)$. But $\tilde H(\tilde U,R_0)=0$ would force the constant to be zero. This would imply in particular that $u(x) = \tilde U(x,0) = 0$ for every $|x|<R_0$, and we would be done. 
 
We thus assume that $\tilde H(\tilde U,R_0)>0$ and show that this leads to a contradiction. Now, this assumption and Corollary \ref{C:int} imply that $\tilde H(\tilde U,r)>0$ for every $0<r<R_0$. Therefore, by Theorem  \ref{T:almgrenEO} we infer that the function $r\to \tilde N(\tilde U,r)$ is nondecreasing on $(0,R_0)$. Let $\tilde U_0$ be an Almgren blowup of $\tilde U$ at $x_0 = 0$. If $\tilde \kappa = \underset{r\to 0^+}{\lim} N(\tilde U,r)$, we know from Lemma \ref{L:abu} that $\tilde U_0$ is a global solution of $L_a \tilde U_0 = 0$, even in $y$, and which is homogeneous of degree $\tilde \kappa$. At this point we make the important observation that $\tilde U_0\not\equiv 0$. This follows from its homogeneity and the non-degeneracy estimate \eqref{nduzero} above. 
We now make the following

\medskip

\noindent \textbf{Claim:} $\tilde U_0$ restricted to $\{y=0\}$ is not identically zero.

\medskip

The proof of this claim proceeds by contradiction and it is based on an argument in Step $1$ and $2$ in the proof of Proposition 2.2 in \cite{Ru}, and we refer the reader to that source.
What R\"uland shows is that, if the claim is not true, then $\tilde U_0$ must vanish to infinite order at $(0,0)$ in the thick space $\Rn\times \R$. However, such vanishing to infinite order at $(0,0)$ in the thick space $\Rn\times \R$ would contradict the homogeneity of $\tilde U_0$. Therefore, the claim must be true. 

From the claim we know that $\tilde U_0(\cdot,0)\not\equiv 0$. The final step is now proving that this fact contradicts the assumption that $u(x) = \tilde U(x,0)$ vanishes to infinite order at $x_0 = 0$. 

Since $\tilde U_0(\cdot,0)\not\equiv 0$, there exist $C,\overline r>0$ such that  
\[
||\tilde U_0||_{L^\infty(B_e(\overline r)\cap \{y=0\})} = C.
\]
From the uniform convergence of $\tilde U_{r_j}\ \longrightarrow\ \tilde U_0$
on $B_e(\overline r)\cap \{y=0\}$, we know that for sufficiently large $j\in \mathbb N$ we must have
\[
||\tilde U_{r_j}||_{L^\infty(B_e(\overline r)\cap \{y=0\})} \ge \frac{C}2.
\]
Using the definition \eqref{bl} of $\tilde U_{r_j}$, we obtain from the latter inequality and the fact that $\tilde U(x,0) = u(x)$,
\[
||u||_{L^\infty(B(r_j \overline r)\cap \{y=0\})} \ge \frac{C}2 r_j^{-\frac{n+a}{2}} 
\sqrt{\tilde H(\tilde U,r_j)}.
\]
Finally, applying Theorem \ref{T:nondeg} we find 
\begin{equation*}
\tilde H(\tilde U,r_j) \geq  \tilde H(\tilde U,R_0) \left(\frac{r_j}{R_0}\right)^{n +a + 2||\tilde N(\tilde U,\cdot)||_{L^\infty(0,R_0)}}.
\end{equation*}
We thus conclude
\[
||u||_{L^\infty(B(r_j \overline r)\cap \{y=0\})} \ge \frac{C}2 r_j^{-\frac{n+a}{2}} 
\sqrt{\tilde H(\tilde U,R_0)} \left(\frac{r_j}{R_0}\right)^{\frac{n +a + 2||\tilde N(\tilde U,\cdot)||_{L^\infty(0,R_0)}}{2}}.
\]
This contradicts the assumption that $u$ vanishes to infinite order at $x_0 = 0$, unless of course $\tilde H(\tilde U,R_0) = 0$. By the monotonicity of $r\to 
\tilde H(\tilde U,r)$ we conclude that it must be $\tilde U \equiv 0$ in $B_e(R_0)$, and therefore $u\equiv 0$ in $B(R_0)$. By connectedness we infer that it must be $u\equiv 0$ in $\Om$.

\end{proof}

In connection with Theorem \ref{T:sucpfl} above we mention the recent work \cite{BG} in which the authors establish a delicate theorem of strong unique continuation backward in time for solutions of the nonlocal equation
\begin{equation}\label{e0}
(\p_t - \Delta)^s u = V(x,t) u,\ \ \ \ \ \ \ \ \ \ \ 0<s<1.
\end{equation}
They assume that the potential $V:\Rnn\to \R$ is such that for some $K>0$ 
\begin{equation}\label{vasump}
\begin{cases}
||V||_{C^{1}(\Rnn)}    \le K,\ \ \ \ \ \text{if}\ 1/2\le s < 1,
\\
\\
||V||_{C^{2}(\Rnn)},\ \ \ \ \   ||<\nabla_x V, x>||_{L^\infty(\Rnn)} \le K,\ \ \ \ \  \text{if}\ 0< s < 1/2.
\end{cases}
\end{equation}

The main result in \cite{BG} is as follows.

\begin{theorem}[Space-time strong unique continuation property]\label{main}
Let $u\in  \operatorname{Dom}(H^{s})$ be a solution to \eqref{e0} with $V$ satisfying \eqref{vasump}. If $u$ vanishes to infinite order backward in time at some point $(x_0, t_0)$ in $\R^{n+1}$, then $u(\cdot,t) \equiv 0$ for all $t \leq t_0$.  
\end{theorem}

One of the central tools in the proof of Theorem \ref{main} is Theorem \ref{T:mon1} below. In it the function $U$ represents the solution to 
the following problem: 
\begin{equation}\label{ext}
\begin{cases}
y^{a} \partial_{t} U(X,t) = \operatorname{div}_{X}(y^{a} \nabla_{X}U)(X,t),
\\
U(x, 0,t)= u(x,t),
\\
\underset{y\to 0^+}{\lim} y^{a} \frac{\p U}{\p y}(x,y,t) = - V(x,t) u(x,t),
\end{cases}
\end{equation}
where $y^{a} \partial_{t} U(X,t) = \operatorname{div}_{X}(y^{a} \nabla_{X}U)(X,t)$ is the parabolic extension operator, introduced by Nystr\"om and Sande in \cite{NS} and independently by Stinga and Torrea in \cite{ST}, whereas $N(U,r) = I(U,r)/H(U,r)$ represents a suitable frequency function in Gaussian space, see Definition 6.2 in \cite{BG}. 

\begin{theorem}[Monotonicity of the adjusted frequency]\label{T:mon1}
Let $u\in  \operatorname{Dom}(H^{s})$ be a solution to \eqref{e0} with $V$ satisfying \eqref{vasump}. There exist universal constants $C, t_0>0$, depending only on $n, s$ and the number $K$ in \eqref{vasump}, such that with $r_0= \sqrt{t_0}$ and $a = 1-2s$, under the assumption that  $H(U,r) \neq 0$  for all $0< r \leq r_0$, then
\begin{equation}\label{mon}
r  \to \exp\left\{C\int_{0}^r t^{-a} dt\right\} \left(N(U,r)  +  C \int_{0}^{r} t^{-a} dt\right)
\end{equation}
is monotone increasing on $(0,r_0)$. Furthermore, when the potential $V\equiv 0$, then the constant $C$ in \eqref{mon} can be taken equal to zero and we have pure monotonicity of the function $r\to N(U,r)$. In such case, $N(U,r) \equiv \kappa$ for $0<r<R$ if and only if $U$ is homogeneous of degree $2\kappa$ in $\Sa_{R}^{+}$ with respect to the parabolic dilations $\delta_\la(X,t) = (\la X,\la^2 t)$.
\end{theorem}

Theorem \ref{T:mon1} generalizes to the case $s\not= 1/2$ a related monotonicity result that was obtained in \cite{DGPT} for the case $s=1/2$ to establish the optimal regularity of the solution of the Signorini problem for the heat equation. We also mention the recent paper \cite{Yu} in which the author uses a generalization of the monotonicity formula in Theorem \ref{T:almgrenEO} to establish the (sucp) for solutions of the nonlocal equation $(-L)^s u = 0$, where $L$ is a divergence form elliptic operator with Lipschitz continuous coefficients.

Concerning monotonicity formulas we cannot fail to mention the considerable developments that have taken place over the last decade in connection with free boundary problems for nonlocal operators, starting with the pioneering paper \cite{ACS}. In that paper, for the first time, the role of the Almgren monotonicity formula in Theorem \ref{T:al} was recognized as central to the analysis of both the optimal regularity of the solution and of the regular part of the free boundary in the nonlocal obstacle problem:
\emph{given a smooth function $\psi :\R^n\to \R$,  find a function $u:\Rn \to \R$ such that}
\begin{equation}\label{pb}
\left\{\begin{array}{l}
\min\bigl\{ u-\psi,\,(-\Delta)^su\bigr\}=0\quad \textrm{in}\ \R^n,\\
\lim_{|x|\to\infty}u(x)=0.
\end{array}\right.
\end{equation}
Via the extension procedure described above, this problem becomes equivalent to the following \emph{thin obstacle problem} (the name ``thin" comes from the fact that now the obstacle lives on the thin manifold $M = \Rn_x\times \{0\}$ in the thick space $\Rn_x\times \R_y$) for the extension operator $L_a$, with $a = 1-2s$ and for the function $\tilde U(x,y)$:
\begin{equation}\label{sb}
 \begin{cases}
L_a \tilde U = \operatorname{div}_{x,y}(|y|^a \nabla_{x,y} \tilde U) = 0\ \ \ \ \ \ \ \ \text{in}\ \R^{n+1}_+\cup \R^{n+1}_-,
\\
\tilde U(x,-y) = \tilde U(x,y),\ \ \ \ \ \ \ \ \ \  \ \ \ \ \ \ \ \ \ \ \text{for}\ x\in \Rn, y\in \R,
\\
\tilde U(x,0) \ge \psi(x),\ \ \ \ \ \ \ \ \ \ \ \ \ \ \ \ \ \ \ \ \ \ \ \ \ \ \text{for}\ x\in \Rn,
\\
- \underset{y\to 0^+}{\lim} y^a D_y\tilde U(x,y) \ge 0, \ \ \ \ \ \ \ \ \ \ \ \ \ \ \  \text{for}\ x\in \Rn,
\\
\underset{y\to 0^+}{\lim} y^a D_y \tilde U(x,y) = 0, \ \ \ \ \ \ \ \ \ \ \ \ \ \ \ \ \ \text{on the set where}\ \tilde u(x,0) > \psi(x).
\end{cases}
\end{equation}
We emphasize that when $s = 1/2$, we have $a = 1-2s = 0$, and thus the extension operator $L_a$ is simply the Laplacean in the variables $(x,y)$. In such case the problem \eqref{sb} is known as the famous problem in elasticity posed in the 50's by Antonio Signorini, an engineer and mathematician: \emph{what is the equilibrium configuration of a spherically shaped elastic body resting on a rigid frictionless plane}.  
Thus, the nonlocal obstacle problem \eqref{pb} for $(-\Delta)^{1/2}$ is equivalent to the Signorini problem for the operator $\Delta$ in $\R^{n+1}$. For a discussion of the latter one should see the beautiful book \cite{PSU}. A remarkable up-to-date account on obstacle problems for fractional operators is the recent survey paper \cite{DS}.

Returning to \eqref{pb}, in \cite{ACS} the authors considered the case of zero obstacle $\psi$ and critical exponent $s = 1/2$, and using Theorem \ref{T:al} they proved:
\begin{itemize}
\item[1)] the optimal regularity $C^{1,1/2}_{loc}$ of the solution from either side of the thin manifold (a different proof of such result had first been found in \cite{AC});
\item[2)] that the \emph{regular part} of the free boundary is locally a $C^{1,\alpha}$ hypersurface. 
\end{itemize}
 In \cite{CSS} the authors using an almost monotonicity formula that generalizes Theorem \ref{T:almgrenEO} above were able to extend the results in \cite{ACS} to the full range $0<s<1$, and to the case of general obstacle. In \cite{GP} some new  one-parameter families of monotonicity formulas of Weiss and Monneau type were discovered. With such formulas the authors were able to analyze, for the fractional exponent $s = 1/2$ in \eqref{pb} and for a general obstacle, the so-called \emph{singular part} of the free boundary. They classified singular points and proved the rectifiability of the singular part of the free boundary. Using some new monotonicity formulas their results have been recently generalized to the full range $0<s<1$ in \cite{GRO}. In the above cited work \cite{DGPT} the authors developed an extensive analysis of the obstacle problem 
\begin{equation}\label{pb00}
\left\{\begin{array}{l}
\min\bigl\{ u-\psi,\,(\p_t -\Delta)^{1/2} u\bigr\}=0\quad \textrm{in}\ \R^n,\\
\lim_{|(x,t)|\to\infty}u(x,t)=0,
\end{array}\right.
\end{equation} 
for the nonlocal heat equation corresponding to the value $s=1/2$ of the fractional exponent and for a general obstacle $\psi$. They establish the optimal regularity of the solution, the regularity of the regular part of the free boundary, they classify the singular set and prove its regularity. Some of the central tools in their analysis are some new Almgren type monotonicity formulas inspired to those found by Poon in \cite{Po} for the standard heat equation, as well as one-parameter parabolic monotonicity formulas of Weiss and Monneau type. 

Finally, the papers  \cite{GS}, \cite{PP}, \cite{GPS}, and \cite{GPPS}  contain various new monotonicity formulas of Almgren, Weiss and Monneau type which are applied to nonlocal obstacle problems such as \eqref{pb} above, but in which either $(-\Delta)^{1/2}$ is replaced by $(-L)^{1/2}$, where $L$ is a divergence form elliptic operator with Lipschitz continuous coefficients (for a previous result for $C^{1,\gamma}$ coefficients, see \cite{Gui}), or $(-\Delta)^s$ is replaced by $(-\Delta)^s + <b(x),\nabla >$. The authors establish for the relevant problems the optimal regularity of the solution, the $C^{1,\alpha}$ smoothness of the regular free boundary and the regularity of the singular part of the latter. We note that being able to treat variable coefficient operators in the obstacle problem  is crucial to understanding \eqref{pb} above when the separating manifold is non-flat. In such situation, the natural approach is to flatten the manifold. If the latter is, for instance, $C^{1,1}$, one is thus lead to the study of a Signorini problem for a variable coefficient operator with Lipschitz coefficients and flat thin manifold. A different approach to the Signorini problem for variable coefficient operators, based on Carleman estimates, was found in \cite{KRS16}, \cite{KRS17}. 


\section{Nonlocal Poisson kernel and mean-value formulas}\label{S:smean}

The most fundamental property of classical harmonic functions is Gauss' mean-value property: if $\Delta u = 0$ in an open set $\Om\subset \Rn$, then for every $x\in \Om$ and every $0<r<\operatorname{dist}(x,\p \Om)$, one has \eqref{har} above.
Classical potential theory, i.e., the study of subharmonic functions, can be entirely developed starting from the corresponding sub-mean value formula for subharmonic functions, see for instance \cite{He} and \cite{DP}. It is thus not surprising that a suitable analogue of \eqref{har} should play an equally important role in the potential theory of the fractional Laplacean. In this respect one should keep in mind that one way (admittedly, not the simplest one!) of obtaining the spherical mean-value formula in \eqref{har} is by choosing $x = 0$ in the Poisson representation formula
\begin{equation}\label{cpf}
u(x) = \int_{S_r} P_r(x,y) \vf(y) d\sigma(y),\ \ \ \ \ \ x\in B_r,
\end{equation}
where we have denoted with 
\[
P_r(x,y) = \frac{1}{\sigma_{n-1} r} \frac{r^2 - |x|^2}{|y-x|^n},\ \ \ \ x\in B_r,\ y\in S_r,
\]
the Poisson kernel for the ball $B_r$. We recall that the function $u$ in \eqref{cpf} provides the unique solution to the Dirichlet problem for the ball  $B_r$.
In his seminal paper \cite{R} using the fundamental solution $E_s(x)$ in \eqref{fs} of Theorem \ref{T:fs} above, and the \emph{nonlocal Kelvin transform} of a function $u$, defined by
\begin{equation}\label{kt}
\tilde u(x) = E_s(x) u\left(\frac{x}{|x|^{2}}\right),
\end{equation}
M. Riesz  constructed the nonlocal Poisson kernel \eqref{nlpk} below, and with it he was able to solve the Dirichlet problem
\begin{equation}\label{dpnl}
\begin{cases}
(-\Delta)^s u = 0\ \ \  \text{in}\ B_r, 
\\
u = \vf\ \ \ \ \ \ \ \  \text{in}\ \Rn\setminus B_r.
\end{cases}
\end{equation}
In the following definition we recall  the nonlocal counterpart of \eqref{cpf} discovered by M. Riesz, see formula (3) on p. 17 in \cite{R}, but also (1.6.11') and (1.6.2) on pages 122 and 112 in \cite{La}.

\begin{definition}\label{D:nlpk}
For every $0<s<1$ and $r>0$ we define the \emph{nonlocal interior Poisson kernel} for $B_r$ as
\begin{equation}\label{nlpk}
P^{(s)}_r(x,y) = c(n,s) \left(\frac{r^2 - |x|^2}{|y|^2 - r^2}\right)^s \frac{1}{|y-x|^n},\ \ \ \ \ \ \ |x|<r, |y|>r,
\end{equation}
where 
\begin{equation}\label{cns}
c(n,s) = \frac{\sin(\pi s) \G(\frac n2)}{\pi^{\frac n2 + 1}} = \frac{2 \sin(\pi s)}{\pi \sigma_{n-1}}.
\end{equation}
When $x = 0$ is the center of the ball $B_r$, then we use the notation 
\begin{equation}\label{ar}
A^{(s)}_r(y) = P^{(s)}_r(0,y) = \begin{cases}
c(n,s) \frac{r^{2s}}{(|y|^2 - r^2)^s |y|^n},\ \ \ \ \ \ \ \ \ |y|>r,
\\
0\ \ \ \ \ \ \ \ \ \ \ \ \ \ \ \ \ \ \ \ \ \ \ \ \ \ \ \ \ \ \ |y|\le r.
\end{cases}
\end{equation}
We call $A^{(s)}_r(y)$ the kernel of the \emph{nonlocal mean-value operator}
\begin{equation}\label{smvo}
\mathscr A^{(s)}_r u(x) = A^{(s)}_r \star u(x).
\end{equation}
\end{definition}
The mean-value operator defined by \eqref{smvo} is the nonlocal counterpart of  the spherical mean \eqref{nsm2}. In Proposition \ref{P:vague} below we show that as $s\nearrow 1$ then $\mathscr A^{(s)}_r u(x) \to \mathscr M_ru(x)$. 

Returning to \eqref{nlpk}, the following theorem of M. Riesz's provides the unique solution to the nonlocal Dirichlet problem \eqref{dpnl}.

\begin{theorem}\label{T:spe}
Let $\vf\in \mathscr L_s(\Rn)\cap C(\Rn)$. Consider the function $u$ in $\Rn$ defined by
\begin{equation}\label{spe}
u(x) =  \begin{cases}
\int_{\Rn\setminus B_r} P^{(s)}_r(x,y) \vf(y) dy,\ \ \ \ \ x\in B_r,
\\
\vf(x),\ \ \ \ \ \ \ \ \ \ \ \ \ \ \ \ x\in \Rn\setminus B_r.
\end{cases}
\end{equation}
Then, $u$ is the unique solution to the Dirichlet problem for the ball $B_r$.
\end{theorem}

Let us observe right-away that from \eqref{ar} we have $A^{(s)}_r \in L^1(\Rn)$.
Therefore, Young's convolution theorem (or Minkowski integral inequality) implies that 
\[
\mathscr A^{(s)}_r :  L^p(\Rn)\ \longrightarrow\ L^p(\Rn),\ \ \ \ \ \ \ \ 1\le p \le \infty,
\]
and that
\[
||\mathscr A^{(s)}_r u||_{L^p(\Rn)} \le ||A^{(s)}_r||_{L^1(\Rn)} ||u||_{L^p(\Rn)},\ \ \ \ \ \ \ u \in L^p(\Rn).
\]
The result that follows shows that $\mathscr A^{(s)}_r$ is a contraction in $L^p(\Rn)$.

\begin{lemma}\label{L:ai}
For every $0<s<1$ and $r>0$ one has
\[
||A^{(s)}_r||_{L^1(\Rn)} = \int_{\Rn} A^{(s)}_r(y) dy = 1.
\]
\end{lemma}

\begin{proof}
Using \eqref{ar} we find
\begin{align*}
\int_{\Rn} A^{(s)}_r(y) dy & = c(n,s) \int_{|y|>r} \frac{r^{2s}}{(|y|^2 - r^2)^s |y|^n} dy
\\
& = r^n c(n,s) \int_{|z|>1} \frac{dz}{(|z|^2 - 1)^s r^n |z|^n}  = c(n,s) \int_{|z|>1} \frac{dz}{(|z|^2 - 1)^s  |z|^n}
\\
& = c(n,s) \sigma_{n-1} \int_{1}^\infty \frac{1}{(\rho^2 - 1)^s} \frac{d\rho}{\rho}
\end{align*}
Now we make the substitution $\rho = \sec \vartheta$, for which $d\rho = \sec \vartheta \tan \vartheta d\vartheta$. With such substitution we find 
\begin{align*}
\int_{\Rn} A^{(s)}_r(y) dy =  & c(n,s) \sigma_{n-1} \int_{0}^{\frac \pi{2}} \frac{\sec \vartheta \tan \vartheta}{(\tan^2 \vartheta)^s} \frac{d\vartheta}{\sec \vartheta}
\\
& = c(n,s) \sigma_{n-1} \int_{0}^{\frac \pi{2}} (\sin \vartheta)^{1-2s} (\cos \vartheta)^{2s-1} d\vartheta.
\end{align*}

Applying \eqref{beta} above with $y = 1-s$ and $x = s$, we conclude
\begin{align*}
\int_{\Rn} A^{(s)}_r(y) dy = & c(n,s) \sigma_{n-1} \int_{0}^{\frac \pi{2}} (\sin \vartheta)^{1-2s} (\cos \vartheta)^{2s-1} d\vartheta 
\\ 
& = \frac{c(n,s) \sigma_{n-1}}{2} B(s,1-s)  = \frac{c(n,s) \sigma_{n-1}}{2} \G(s)\G(1-s)
\end{align*}
Finally, we use the formula \eqref{sine}
to obtain
\begin{align*}
\int_{\Rn} A^{(s)}_r(y) dy = & \frac{c(n,s) \sigma_{n-1}}{2}\frac{\pi}{\sin \pi s}.
\end{align*}
If we now recall the formula \eqref{sn1} it is easy to see that, by choosing $c(n,s)>0$ as in \eqref{cns},
we reach the desired conclusion that
\[
\int_{\Rn} A^{(s)}_r(y) dy = 1.
\]

\end{proof}

The next result shows that, as $s\to 1$, the nonlocal mean-value operator in \eqref{smvo} converges to the spherical mean-value operator \eqref{MA0} for the sphere.

\begin{prop}[Asymptotic behavior of $\mathscr A^{(s)}_r u(x)$ as $s\nearrow 1$]\label{P:vague}
For every function $u\in \mathscr S(\Rn)$ one has 
\[
\underset{s\to 1}{\lim} \mathscr A^{(s)}_r u(x) = \mathscr M_r u(x),
\]
for every $x\in \Rn$, where $\mathscr M_r u(x)$ is defined as in \eqref{MA0} above.
\end{prop}

\begin{proof}
It suffices to prove the result when $x=0$. By Lemma \ref{L:ai} we have
\begin{align*}
\mathscr A^{(s)}_r u(0) & = c(n,s) \int_{|y|>r} \frac{r^{2s}}{(|y|^2 - r^2)^s |y|^n} \left[u(y) - \mathscr M_r u(0)\right] dy + \mathscr M_r u(0)
\\
& = \mathscr M_r u(0) + \frac{2\sin(\pi s)}{\pi} r^{2s}  \int_r^\infty \frac{1}{(\rho^2 - r^2)^s \rho} \left[\mathscr M_\rho u(0) - \mathscr M_r u(0)\right] d\rho.
\end{align*}
Since $\sin(\pi s) \to 0$ as $s\to 1$, to finish the proof it suffices to show that there exists a number $C>0$ independent of $s\in (0,1)$ such that
\[
\left|\int_r^\infty \frac{1}{(\rho^2 - r^2)^s \rho} \left[\mathscr M_\rho u(0) - \mathscr M_r u(0)\right] d\rho\right| \le C.
\]
From (i) of Proposition \ref{P:epd} we find
\[
\left|\frac{\p \mathscr M_r u}{\p r}(0)\right| = \left|\frac{1}{\sigma_{n-1} r^{n-1}} \int_{S(0,r)} \frac{\p u}{\p \nu}(y) d\sigma(y)\right| \le ||\nabla u||_{L^\infty(\Rn)}.
\]
We thus have for a fixed $R>r$
\begin{align*}
& \int_r^\infty \frac{1}{(\rho^2 - r^2)^s \rho} \left[\mathscr M_\rho u(0) - \mathscr M_r u(0)\right] d\rho = \int_r^R \frac{1}{(\rho^2 - r^2)^s \rho} \left[\mathscr M_\rho u(0) - \mathscr M_r u(0)\right] d\rho
\\
& + \int_R^\infty \frac{1}{(\rho^2 - r^2)^s \rho} \left[\mathscr M_\rho u(0) - \mathscr M_r u(0)\right] d\rho = I(s) + II(s).
\end{align*}
Now, for any $\rho\in (r,R)$ we have from the above estimate
\begin{align*}
\left|\mathscr M_\rho u(0) - \mathscr M_r u(0)\right| = \left|\int_r^\rho \frac{d}{dt} \mathscr M_t u(0) dt\right| \le ||\nabla u||_{L^\infty(\Rn)} (\rho-r).
\end{align*}
This gives
\begin{align*}
| I(s)| & \le ||\nabla u||_{L^\infty(\Rn)} \int_r^R \frac{(\rho-r)^{1-s}}{(\rho + r)^s \rho}  d\rho \le ||\nabla u||_{L^\infty(\Rn)} 2^{-s} r^{-1-s} \int_r^R (\rho-r)^{1-s}  d\rho
\\
& \le  ||\nabla u||_{L^\infty(\Rn)} 2^{-s} r^{-1-s} \frac{(R-r)^{2-s}}{2-s} \le C, 
\end{align*}
with a $C>0$ independent of $s\to 1$. On the other hand, we have
\begin{align*}
| II(s)| \le 2 ||u||_{L^\infty(\Rn)} \int_R^\infty \frac{d\rho}{(\rho^2 - r^2)^s \rho}   \le C,
\end{align*}
where again $C>0$ can be taken independent of $s\to 1$.

\end{proof}

Proposition \ref{P:vague} has been very recently generalized in \cite{BuSq} by allowing in the definition of $\mathscr A^{(s)}_r u(x)$ a measure on the unit sphere $\mathbb S^{n-1}$ which is bounded both from above and  from below (away from zero).

In the next result, we use the space $\mathscr L_s(\Rn)$ introduced in Definition \ref{D:sobs}, whereas the notation $C^{2s+\varepsilon}_{loc}$ is that introduced in Definition \ref{D:holder}.

\begin{lemma}\label{L:prepriv}
Let $0<s<1$ and suppose that $u\in \mathscr L_s(\Rn)$. Assume furthermore that for some $0<\varepsilon < 1$ we have $u\in C_{loc}^{2s+\varepsilon}$ in a neighborhood of  $x\in \Rn$. Then,
\begin{equation}\label{=}
\underset{r\to 0^+}{\lim} \int_{|y|>r} \frac{2u(x) - u(x+y) - u(x-y)}{(|y|^2 - r^2)^s |y|^n} dy = \int_{\Rn} \frac{2u(x) - u(x+y) - u(x-y)}{|y|^{n+2s}} dy.
\end{equation}
\end{lemma}

\begin{proof}
 For $0<r<\frac{1}{\sqrt 2}$ we can write
\begin{align*}
& \int_{|y|>r} \frac{2u(x) - u(x+y) - u(x-y)}{(|y|^2 - r^2)^s |y|^n} dy = \int_{r<|y|\le 1} \frac{2u(x) - u(x+y) - u(x-y)}{(|y|^2 - r^2)^s |y|^n} dy
\\
& + \int_{1< |y|} \frac{2u(x) - u(x+y) - u(x-y)}{(|y|^2 - r^2)^s |y|^n} dy = I_1(r) + I_2(r).
\end{align*}
Now on the set where $|y|>1$ we have
\[
|y|^2> 1 > 2 r^2,
\]
and so 
\[
\frac{1}{|y|^2 - r^2} < \frac{2}{|y|^2}.
\]
This gives
\[
I_2(r) < 2^s \int_{1< |y|} \frac{2u(x) - u(x+y) - u(x-y)}{|y|^{n+2s}} dy < \infty,
\]
since by assumption $u\in \mathscr L_s(\Rn)$. By Lebesgue dominated convergence theorem, we conclude that
\[
\underset{r\to 0^+}{\lim} I_2(r) = \int_{1< |y|} \frac{2u(x) - u(x+y) - u(x-y)}{|y|^{n+2s}} dy.
\]
Next, we evaluate 
\[
I_1(r) = \int_{r<|y|\le 1} \frac{2u(x) - u(x+y) - u(x-y)}{(|y|^2 - r^2)^s |y|^n} dy.
\]
We distinguish two cases:
\begin{itemize}
\item[(i)] $0<s<1/2$;
\item[(ii)] $1/2\le s<1$.
\end{itemize}
In case (i) we know by Definition \ref{D:holder} that $u\in C^{0,2s+\varepsilon}$ at $x$. Thus, for $|y|\le 1$ we have
\begin{equation}\label{h1}
|2u(x) - u(x+y) - u(x-y)|\le C |y|^{2s+\varepsilon}.
\end{equation}
We then find
\begin{align*}
I_1(r) & = \int_{r<|y|\le 1} \frac{2u(x) - u(x+y) - u(x-y)}{|y|^{n+2s}} dy + \left(I_1(r) - \int_{r<|y|\le 1} \frac{2u(x) - u(x+y) - u(x-y)}{|y|^{n+2s}} dy\right)
\\
& = \int_{r<|y|\le 1} \frac{2u(x) - u(x+y) - u(x-y)}{|y|^{n+2s}} dy  + \tilde I_1(r).
\end{align*}
Now, \eqref{h1} and dominated convergence give
\[
\underset{r\to 0^+}{\lim} \int_{r<|y|\le 1} \frac{2u(x) - u(x+y) - u(x-y)}{|y|^{n+2s}} dy = \int_{|y|\le 1} \frac{2u(x) - u(x+y) - u(x-y)}{|y|^{n+2s}} dy.
\]
On the other hand, again by \eqref{h1} we have
\begin{align*}
|\tilde I_1(r)| & \le \int_{r<|y|\le 1} \left|2u(x) - u(x+y) - u(x-y)\right| \left|\frac{1}{(|y|^2 - r^2)^s |y|^n} - \frac{1}{|y|^{n+2s}}\right| dy
\\
& \le C  \int_{r<|y|\le 1} |y|^{2s+\varepsilon} \left|\frac{1}{(|y|^2 - r^2)^s |y|^n} - \frac{1}{|y|^{n+2s}}\right| dy 
\\
& = C'(n) \int_r^1 t^{2s+\varepsilon}  \left|\frac{1}{(t^2 - r^2)^s} - \frac{1}{t^{2s}}\right|  \frac{dt}{t}  = C'(n) r^\varepsilon \int_1^{\frac1r} \tau^{2s+\varepsilon}  \left|\frac{1}{(\tau^2 - 1)^s} - \frac{1}{\tau^{2s}}\right|  \frac{d\tau}{\tau} 
\end{align*}
Clearly, the integrand is summable near $\tau = 1$ since $0<s<1$. Near $\tau = \infty$ the integrand behaves like
\[
\frac{1}{\tau^{1-\varepsilon}} \left|\frac{1}{(1 - \tau^{-2})^s} - 1\right| \cong \frac{1}{\tau^{3-\varepsilon}}  
\]
which is in $L^1$ provided that $3-\varepsilon >1$, which is of course true. 
Therefore, 
\[
|\tilde I_1(r)| \le C'(n) r^\varepsilon \int_1^{\infty} \tau^{2s+\varepsilon}  \left|\frac{1}{(\tau^2 - 1)^s} - \frac{1}{\tau^{2s}}\right|  \frac{d\tau}{\tau} = C''(n,\e)  r^\varepsilon\ \longrightarrow\ 0,
\]
as $r\to 0^+$.
In conclusion, when $0<s<\frac 12$ we obtain
\[
\underset{r\to 0^+}{\lim} I_1(r) = \int_{|y|\le 1} \frac{2u(x) - u(x+y) - u(x-y)}{|y|^{n+2s}} dy.
\]
It follows that
\begin{align*}
& \underset{r\to 0^+}{\lim} \int_{|y|>r} \frac{2u(x) - u(x+y) - u(x-y)}{(|y|^2 - r^2)^s |y|^n} dy = \underset{r\to 0^+}{\lim} I_1(r) + \underset{r\to 0^+}{\lim} I_2(r)
\\
& = \int_{|y|\le 1} \frac{2u(x) - u(x+y) - u(x-y)}{|y|^{n+2s}} dy + \int_{1< |y|} \frac{2u(x) - u(x+y) - u(x-y)}{|y|^{n+2s}} dy.
\end{align*}
This proves \eqref{=} in the case (i).

Suppose now that case (ii) occurs. By Definition \ref{D:holder} the assumption  $u\in C_{loc}^{2s+\varepsilon}$ in a neighborhood of $x$ now reads $u\in C_{loc}^{1,2s+\varepsilon-1}$. Using Taylor's formula we now find
\[
2u(x) - u(x+y) - u(x-y)  = O(|y|^{2s+\varepsilon}),
\]   
and thus \eqref{h1} is valid again. We can thus argue as we did in the case (i) to reach the conclusion that \eqref{=} hold. This completes the proof.

\end{proof}

We next draw an interesting consequence of Lemma \ref{L:prepriv}. Recall Proposition \ref{P:genlap} above.
As it is well-known the operators defined by the right-hand sides of \eqref{BP0} are at the basis of the development of potential theory of non-smooth subharmonic functions. We next prove a result that generalizes Proposition \ref{P:genlap}
to the nonlocal setting.

\begin{prop}[The Blaschke-Privalov fractional Laplacean]\label{P:bps} Let $0<s<1$ and suppose that $u\in \mathscr L_s(\Rn)$ be in $C^{2s+\varepsilon}$ in a  neighborhood of $x\in \Rn$, for some $0<\varepsilon < 1$. One has
\begin{equation}\label{bp}
(-\Delta)^s u(x) = \gamma(n,s) \underset{r\to 0^+}{\lim} \frac{u(x) - A^{(s)}_r \star u(x)}{r^{2s}},
\end{equation}
where $\gamma(n,s) = \frac{s 2^{2s} \G(\frac n2 + s)}{\pi^{\frac n2} \G(1-s)}$ is the constant in \eqref{gnsfin} in Proposition \ref{P:gns}.
\end{prop}

\begin{proof}
By the definition \eqref{smvo} we have
\[
A^{(s)}_r \star u(x) = \int_{\Rn} A^{(s)}_r(y)  u(x-y) dy.
\]
Changing $y$ in $-y$, and using the fact that $A^{(s)}_r(-y)  = A^{(s)}_r(y) $, we also have
\[
A^{(s)}_r \star u(x) =  \int_{\Rn} A^{(s)}_r(y)  u(x+y) dy.
\]
On the other hand, Lemma \ref{L:ai} gives for $r>0$
\[
2 u(x) = \int_{\Rn} A^{(s)}_r(y)  2 u(x) dy.
\]
Therefore, we obtain
\begin{align*}
& u(x) - A^{(s)}_r \star u(x) = \frac 12 \int_{\Rn} A^{(s)}_r(y)  \left[2 u(x) - u(x+y) - u(x-y)\right] dy
\\
& = \frac{r^{2s}}{2} \int_{|y|>r} \frac{2u(x) - u(x+y) - u(x-y)}{(|y|^2 - r^2)^s |y|^n} dy. 
\end{align*}
This gives
\[
\frac{u(x) - A^{(s)}_r \star u(x)}{r^{2s}} = \frac 1{2} \int_{|y|>r} \frac{2u(x) - u(x+y) - u(x-y)}{(|y|^2 - r^2)^s |y|^n} dy
\]
From this identity and from \eqref{=} in Lemma \ref{L:prepriv}, under the given hypothesis on $u$, we obtain
\begin{align*}
& \underset{r\to 0^+}{\lim} \frac{u(x) - A^{(s)}_r \star u(x)}{r^{2s}} = \frac 1{2}  \underset{r\to 0^+}{\lim} \int_{|y|>r} \frac{2u(x) - u(x+y) - u(x-y)}{(|y|^2 - r^2)^s |y|^n} dy
\\
&  = \frac 1{2} \int_{\Rn} \frac{2u(x) - u(x+y) - u(x-y)}{|y|^{n+2s}} dy = \frac 1{2}  \frac{2}{\gamma(n,s)} (-\Delta)^s u(x)
\\
& = \frac{1}{\gamma(n,s)} (-\Delta)^s u(x),
\end{align*}
where in the last equality we have used \eqref{fls}. This establishes \eqref{bp}.

\end{proof}

We mention that the operator defined by the limit relation \eqref{bp} is known as the \emph{Dynkin operator} to people in probability, see \cite{Dy} and also the interesting paper \cite{Kwa}.

The classical Theorem of K\"oebe \ref{T:koebe} states that if a continuous function in an open set $\Om\subset \Rn$ satisfies either one of the mean-value properties \eqref{har} for every  $x\in \Om$ and $0<r<\operatorname{dist}(x,\p \Om)$, then in fact $u\in C^\infty(\Om)$ and $\Delta u = 0$. We next establish a K\"oebe type theorem for the fractional Laplacean.

\begin{prop}\label{P:fk}
Let $0<s<1$ and suppose that $u\in \mathscr L_s(\Rn)$ and that, for some $0<\varepsilon < 1$, we have $u\in C^{2s+\varepsilon}_{loc}$ in a  neighborhood of $x\in \Rn$. If for every $r>0$ sufficiently small we have $u(x) = A^{(s)}_r \star u(x)$, then we must have $(-\Delta)^s u(x) = 0$.
\end{prop}

\begin{proof}
By hypothesis we know that for all $r>0$ arbitrarily small we have
\[
u(x) =  A^{(s)}_r \star u(x) = \int_{\Rn} A^{(s)}_r(y) u(x-y) dy.
\]
Changing $y$ in $-y$, and using the fact that $A^{(s)}_r(-y) = A^{(s)}_r(y)$, we also have
\[
u(x) =  \int_{\Rn} A^{(s)}_r(y) u(x+y) dy.
\]
On the other hand, thanks to Lemma \ref{L:ai} we can write for every such $r>0$
\[
u(x) = \int_{\Rn} A^{(s)}_r(y) u(x) dy.
\]
We thus have for every $r>0$ sufficiently small
\[
0 = \int_{\Rn} A^{(s)}_r(y) \left[2u(x) - u(x+y) - u(x-y)\right] dy = r^{2s} \int_{|y|>r} \frac{2u(x) - u(x+y) - u(x-y)}{(|y|^2 - r^2)^s |y|^n} dy. 
\]
This formula gives
\[
\underset{r\to 0^+}{\lim} \int_{|y|>r} \frac{2u(x) - u(x+y) - u(x-y)}{(|y|^2 - r^2)^s |y|^n} dy = 0.
\]
On the other hand \eqref{=} in Lemma \ref{L:prepriv} gives
\begin{align*}
& \underset{r\to 0^+}{\lim} \int_{|y|>r} \frac{2u(x) - u(x+y) - u(x-y)}{(|y|^2 - r^2)^s |y|^n} dy = \int_{\Rn} \frac{2u(x) - u(x+y) - u(x-y)}{|y|^{n+2s}} dy 
\\
& = \frac{2}{\gamma(n,s)} (-\Delta)^s u(x).
\end{align*}
We conclude that $(-\Delta)^s u(x) = 0$.

\end{proof}

\begin{corollary}[Nonlocal K\"oebe theorem]\label{C:nlkoebe}
Let $\Om\subset \Rn$ be an open set, and suppose that $u\in \mathscr L_s(\Rn)\cap C_{loc}^{2s+\varepsilon}$ for some $\e>0$. If for every $x\in \Om$ and every $0<r<\operatorname{dist}(x,\p\Om)$ we have $u(x) = A^{(s)}_r \star u(x)$, then $(-\Delta)^s u = 0$ in $\Om$, and therefore by Theorem \ref{T:fh} we also have $u\in C^\infty(\Om)$. 
\end{corollary}


We next want to establish a converse to Corollary \ref{C:nlkoebe}, namely that if $(-\Delta)^s u = 0$ in an open set $\Om\subset \Rn$, then $u(x) = A^{(s)}_r \star u(x)$ for all $x\in \Om$ and for sufficiently small $r>0$. In the local case, under the assumption that $u\in C^2(\Om)$, this can be easily achieved by using (ii) of Proposition \ref{P:epd}. An alternative, less direct way of proving this in the local case is to argue as follows. Suppose without loss of generality that $x = 0\in \Om$, and consider the ball $B_r = B_r(0)\subset \overline B_r \subset \Om$. Denote by $P_r(x,y)$ and $G_r(x,y)$ respectively the Poisson kernel and the Green function for the ball $B_r$. The function 
\[
v(x) = - \int_{B_r} G_r(x,y) \Delta u(y) dy,
\]
solves the problem
\[
\begin{cases}
\Delta v = \Delta u\ \ \ \ \ \  \ \text{in}\ B_r,
\\
v = 0\ \ \ \ \ \ \ \ \ \ \ \  \text{on}\ S_r,
\end{cases}
\] 
whereas the function
\[
w(x) = \int_{S_r} P_r(x,y) u(y) d\sigma(y),
\]
solves the problem 
\[
\begin{cases}
\Delta w = 0\ \ \ \ \ \ \ \ \ \  \ \text{in}\ B_r,
\\
w = u\ \ \ \ \ \ \ \ \ \ \ \ \ \  \text{on}\ S_r.
\end{cases}
\] 
The maximum principle implies that $u = v + w$ in $B_r$. In particular, we have
$u(0) = v(0) + w(0)$. This gives 
\begin{align}\label{u(0)}
u(0) & = \int_{S_r} P_r(0,y) u(y) d\sigma(y) - \int_{B_r} G_r(0,y) \Delta u(y) dy
\\
& = \mathscr M_r u(0) + \int_{B_r} G_r(0,y) (-\Delta) u(y) dy.
\notag
\end{align}
From the identity \eqref{u(0)} it is thus clear that if $\Delta u = 0$ in $\Om$, then for every $x\in \Om$ we must have $u(x) = \mathscr M_r u(x)$ for all $r>0$ such that $\overline B_r\subset \Om$. These considerations can be repeated in the nonlocal case, if we use M. Riesz' Poisson kernel and Green function for the ball.  Before we turn to this we note that 
\[
G_r(x,y) = E(x,y) - \int_{S_r} P_r(y,\xi)E(x,\xi) d\sigma(\xi),
\]
and therefore, for $n\ge 3$, we have with $c_n = \frac{1}{(n-2)\sigma_{n-1}}$,
\begin{align*}
G_r(0,y) & = E(0,y) - \int_{S_r} P_r(y,\xi)E(0,\xi) d\sigma(\xi) = E(0,y) - \frac{c_n}{r^{n-2}} \int_{S_r} P_r(y,\xi) d\sigma(\xi)
\\
& = E(0,y) - \frac{c_n}{r^{n-2}},
\end{align*}
which, by translation invariance, gives the well-known formula
\[
u(x)  = \mathscr M_r u(x) + \int_{B_r(x)} \left\{\frac{c_n}{|y-x|^{n-2}} - \frac{c_n}{r^{n-2}}\right\}  (-\Delta) u(y) dy.
\]

\begin{prop}\label{P:conversemvp}
Let $\Om\subset \Rn$ be an open set, with $n\ge 2$, and suppose that $u\in \mathscr L_s(\Rn)\cap C_{loc}^{2s+\varepsilon}$ for some $\e>0$. Assume that $(-\Delta)^s u = 0$ in the open set $\Om\subset\Rn$. Then, for every $x\in \Om$ and every $r>0$ such that $\overline B_r \subset \Om$, we have
\[
u(x) = \mathscr A^{(s)}_r u(x) = A^{(s)}_r  \star u(x).
\]
\end{prop}

\begin{proof}
We assume without restriction that $x=0$, with $B_r \subset \overline B_r \subset \Om$. Let $u\in \mathscr L_s(\Rn)\cap C_{loc}^{2s+\varepsilon}$ for some $\e>0$ (not necessarily a solution of $(-\Delta)^s u = 0$ in $\Om$), and 
consider the function $v$ defined by
\begin{equation}\label{spee}
v(z) =  \begin{cases}
\int_{\Rn\setminus B_r} P^{(s)}_r(z,y) u(y) dy,\ \ \ \ \ z\in B_r,
\\
u(z),\ \ \ \ \ \ \ \ \ \ \ \ \ \ \ \ z\in \Rn\setminus B_r.
\end{cases}
\end{equation}
Then, by Theorem \ref{T:spe} $v$ is the unique solution to the Dirichlet problem for the ball $B_r$
\begin{equation}\label{dpnl0}
\begin{cases}
(-\Delta)^s v = 0\ \ \  \text{in}\ B_r, 
\\
v = u\ \ \ \ \ \ \ \  \text{in}\ \Rn\setminus B_r.
\end{cases}
\end{equation}
Consider the function $w = u-v$. It solves the problem
\begin{equation}\label{dpnl00}
\begin{cases}
(-\Delta)^s w = (-\Delta)^s u \ \ \  \text{in}\ B_r, 
\\
w = 0\ \ \ \ \ \ \ \  \text{in}\ \Rn\setminus B_r.
\end{cases}
\end{equation}
Then, (see e.g. Theorem 3.2 in \cite{Bu}) the function $w$ is represented by the formula 
\[
w(z) =  \int_{B_r} G_r^{(s)}(z,y) (-\Delta)^s u(y) dy,
\]
where for every $z \in B_r$ the function 
\[
G_r^{(s)}(z,y) = E_s(z,y) - h^{(s)}_z(y) > 0
\]
is the Green function for $(-\Delta)^s$ for the ball $B_r$, see \cite{R}, the appendix in \cite{La}, and also \cite{Bu}. We notice explicitly that the function $h^{(s)}_z$ represents the solution to the Dirichlet problem
\begin{equation}\label{dpnl000}
\begin{cases}
(-\Delta)^s h^{(s)}_z = 0 \ \ \  \text{in}\ B_r, 
\\
h^{(s)}_z = E_s(z,\cdot)\ \ \ \ \ \ \ \  \text{in}\ \Rn\setminus B_r,
\end{cases}
\end{equation}
and therefore it is given by 
\[
h^{(s)}_z(y) = \begin{cases} \int_{\Rn\setminus B_r} P^{(s)}_r(y,\xi) E_s(z,\xi) d\xi\ \ \ \ y\in B_r,
\\
E_s(z,y)\ \ \ \ \ \ \ \ \ \ \ \ \ \ \ \ \ \ \ \ \ \ \ \ \ \ \ \ \ y\in \Rn\setminus B_r.
\end{cases}
\]
We conclude that for every $z\in B_r$ the value of $u$ at $z$ can be written as 
\[
u(z) = v(z) + w(z).
\]
In particular, this gives
\[
u(0) = v(0) + w(0) = A^{(s)}_r \star u(0) + \int_{B_r} G_r^{(s)}(0,y) (-\Delta)^s u(y) dy.
\] 
It is clear from this formula that if $(-\Delta)^s u = 0$ in $\Om$, then $u(x) =  A^{(s)}_r \star u(x)$, provided that $r>0$ is such that $\overline B_r \subset \Om$.

\end{proof}


\section{The heat semigroup $\pt = e^{t(-\Delta)^s}$}\label{S:heat}

As it is well-known the heat operator $H = \p_t - \Delta$ plays a fundamental role in almost all areas of mathematics. Since the focus of this note is the fractional Laplacean, it is only natural that we also discuss the nonlocal heat operator $\p_t + (-\Delta)^s$ and the semigroup associated with it. The present section, as well as the following five ones are devoted to analyze some of the most elementary properties of such semigroup. 

A first comment is that this is not such a simple task since, unlike what happens for the classical heat equation, the relevant heat kernel is only explicitly known on the Fourier transform side, a fact that rules out the possibility of explicit elegant and simple computations which characterize the analysis of the classical case. With this state of affairs, even a fundamental fact such as the positivity of the heat kernel is not a priori obvious. 

Before we begin our discussion however, we mention that the first pioneering results about the nonlocal heat equation $\p_t + (-\Delta)^s$ go back to the work of Bochner, see \cite{B49} and \cite{Y}. More recently, several authors have analyzed properties such as estimates of the heat kernel, the weak Harnack inequality and the H\"older continuity of solutions, interior and at the boundary, for more general nonlocal parabolic equations, results of Fujita type, a Widder type theorem, Nash-type inequalities, eigenvalue estimates, nonlocal porous medium equation etc., see e.g. \cite{Ko95}, \cite{BLW05}, \cite{BJ07}, \cite{BM}, \cite{BGR},  \cite{CKK}, \cite{DQRV12}, \cite{FK}, \cite{BP13}, \cite{BPSV}, \cite{AMPP},  \cite{BSV17}, \cite{Fr}, \cite{GI}, \cite{V17}, \cite{V17'} but this list is by no means exhaustive. The regularity theory for very general nonlocal parabolic operators has been developed in the paper \cite{CCV11}.  Further regularity properties of solutions have been extensively studied in \cite{CR}, where the authors study fractional nonlinear parabolic equations, and in \cite{CF} in the context of the nonlocal obstacle problem
\[
\begin{cases}
\min\{u-\psi,-u_t +(-\Delta)^s u\} = 0,\ \ \ \ \ \ \text{in}\ \ \Rn\times [0,T],
\\
u(T) = \psi,
\end{cases}
\]
where the function $\psi$ represents the obstacle.

We begin by discussing the Cauchy problem: given $0<s<1$ and a function $\vf\in \mathscr S(\Rn)$, find a solution to the problem 
\begin{equation}\label{heat00}
\begin{cases}
H_{s} u = \frac{\p u}{\p t} + (-\Delta)^s u = 0\ \ \ \ \ \ \ \ \text{in}\ \R^{n+1}_+,
\\
\\
u(x,0) = \vf(x),\ \ \ \ \ \ \ \ \ \ \ \  x\in \Rn.
\end{cases}
\end{equation}

An observation that follows immediately from the definition of $H_s$ and from \eqref{di} is that the natural scaling for the fractional heat operator is given by the  nonisotropic dilations
\begin{equation}\label{hdil}
\tilde \delta_\la(x,t) = (\la x, \la^{2s} t).
\end{equation}
By this we mean that for every function $u(x,t)$ we have
\begin{equation}\label{hdil2}
H_s(\tilde \delta_\la u)(x,t) = \la^{2s} \tilde \delta_\la (H_s u)(x,t).
\end{equation}
Thus, $H_s$ is an operator of fractional ``order" $2s$ with respect to the dilations \eqref{hdil}.

As in the case when $s = 1$, we can solve \eqref{heat00} by formally taking a partial Fourier transform with respect to the space variable $x$. If we let
\[
\hat u(\xi,t) = \int_{\Rn} e^{-2\pi i <\xi,x>} u(x,t) dx,
\]
then \eqref{fls3} in Proposition \ref{P:slapft} gives
\begin{equation}\label{heat2}
\begin{cases}
\frac{\p \hat u}{\p t}(\xi,t) + (2\pi |\xi|)^{2s} \hat u(\xi,t) = 0\ \ \ \ \ \ \ \ \text{in}\ \R^{n+1}_+,
\\
\\
\hat u(\xi,0) = \hat \vf(\xi),\ \ \ \ \ \ \ \ \ \ \ \  \xi\in \Rn.
\end{cases}
\end{equation}
For every $\xi\in \Rn$ the solution to \eqref{heat2} is given by
\[
\hat u(\xi,t) = \hat \vf(\xi) e^{-t (2\pi |\xi|)^{2s}}.
\]
We now define a function $G_s(x,t)$ by the equation
\begin{equation}\label{G}
\F_{x\to \xi}(G_s(\cdot,t)) = e^{-t (2\pi |\xi|)^{2s}},
\end{equation}
or, equivalently, 
\begin{equation}\label{GG}
G^{(s)}(x,t) = \int_{\Rn} e^{- 2 \pi i <x,\xi>} e^{-t (2\pi |\xi|)^{2s}} d\xi.
\end{equation} 
With this definition it is clear that a solution to \eqref{heat2} is given by the formula
\begin{equation}\label{hsg}
P^{(s)}_t \vf(x) \overset{def}{=} u(x,t) = \int_{\Rn} G^{(s)}(x-y,t) \vf(y) dy.
\end{equation} 
We observe that since in view of \eqref{G} the function $\F_{x\to \xi}(G^{(s)}(\cdot,t))$ decays rapidly, by \eqref{derft} above we conclude that $G^{(s)}(\cdot,t)\in C^\infty(\Rn)$ (this in fact could also be proved directly from \eqref{GG}).

\begin{definition}\label{D:hsg}
The \emph{fractional heat semigroup} $\pt$ is the operator defined by \eqref{hsg}. When $s=1$ we indicate with $P_t$ the standard heat semigroup defined by $P_t u(x) = G(\cdot,t) \star u(x)$, where $G(x,t) = (4\pi t)^{-\frac n2} e^{-\frac{|x|^2}{4t}}$.
\end{definition} 

Using the Fourier transform it is immediate to verify that $\pt$ is in fact a semigroup, i.e., for every $t, \tau >0$ the following property holds
\begin{equation}\label{sg}
P^{(s)}_{t+\tau} = P^{(s)}_t \circ P^{(s)}_\tau.
\end{equation}

One immediate important property of $G^{(s)}$ is the following scale invariance
\begin{equation}\label{Gsi}
G^{(s)}(x,t) = t^{-\frac{n}{2s}} G^{(s)}(\frac{x}{t^{1/2s}},1).
\end{equation}
This can be verified by writing \eqref{GG} as follows
\[
G^{(s)}(x,t) = \F(\delta_{t^{1/2s}} e^{- (2\pi |\cdot|)^{2s}})(x) = t^{-\frac{n}{2s}} \F(e^{- (2\pi |\cdot|)^{2s}})(\frac{x}{t^{1/2s}}),
\]
where we have used \eqref{dil} above. Therefore, if we set 
\begin{equation}\label{time1}
\Phi_s(x) = \F(e^{- (2\pi |\cdot|)^{2s}})(x) = G^{(s)}(x,1),
\end{equation}
then \eqref{Gsi} can be recast in the following self-similar expression
\begin{equation}\label{heat3}
G^{(s)}(x,t) = t^{-\frac{n}{2s}} \Phi_s(\frac{x}{t^{1/2s}}).
\end{equation}

Notice that \eqref{heat3} implies that $G^{(s)}$ is homogeneous of degree $-n$ with respect to the dilations \eqref{hdil}, i.e., 
\begin{equation}\label{heat4}
G^{(s)}(\la x,\la^{2s} t) = \la^{-n} G^{(s)}(x,t).
\end{equation}
If we denote by $Z_s$ the infinitesimal generator of the dilations \eqref{hdil}, we thus have
\[
Z_s G^{(s)} =  <x,\nabla_x G^{(s)}> + 2s t \frac{\p G^{(s)}}{\p t} = - n G^{(s)}.
\]
Equivalently, if we introduce the \emph{nonlocal entropy} $\log G^{(s)}$, then we have
\begin{equation}\label{heat44}
\frac{\p (\log G^{(s)})}{\p t} = - \frac{n}{2s}\frac 1t  - \frac{1}{2s}  <\frac{x}{t},\nabla_x (\log G^{(s)})>. 
\end{equation}
The equation \eqref{heat44} is a first form of \emph{nonlocal Li-Yau inequality} for solutions of \eqref{heat00}, and we emphasize that it has been deduced exclusively from the scaling properties of the kernel $G^{(s)}$. We will return to it in Section \ref{S:LY}.

Before proceeding we note that, using Theorem \ref{T:Fourier-Bessel}, for $x\not= 0$ we obtain from \eqref{time1}
\begin{align}\label{phiFT}
\Phi_s(x) & = \frac{2\pi}{|x|^{\frac n2 - 1}} \int_0^\infty e^{-(2\pi r)^{2s}} r^{\frac n2} J_{\frac n2 -1}(2\pi |x| r) dr
\\
& = \frac{(2\pi)^{-\frac n2}}{|x|^n} \int_0^\infty e^{-(\frac{u}{|x|})^{2s}} u^{\frac n2} J_{\frac n2 -1}(u) du.
\notag
\end{align} 
The equation (5.3.5) on p. 103 in \cite{Le}, gives
\[
\frac{d}{du} \left[u^{\frac n2} J_{\frac n2}(u)\right] = u^{\frac n2} J_{\frac n2 -1}(u).
\]
Substituting this information in \eqref{phiFT}, and integrating by parts (using \eqref{bfbehzero} and \eqref{jnuinfty2}), we find
\begin{equation}\label{phiFT2}
\Phi_s(x) = \frac{2s (2\pi)^{-\frac n2}}{|x|^{n+2s}} \int_0^\infty e^{-(\frac{u}{|x|})^{2s}} u^{\frac n2 + 2s - 1} J_{\frac n2}(u) du.
\end{equation}
When $s=1/2$, the underlying process is a Poisson process and the integral in the right-hand side of \eqref{phiFT2} can be computed explicitly, see Proposition \ref{P:onehalf} below. However, when $0<s<1$ and $s\not= 1/2$ the analysis of such integral is more delicate, see Theorem \ref{T:polya} below and the comments that follow.

When $s=1$ one basic property of the classical heat semigroup intimately connected to the maximum principle is the strict positivity of its kernel $G(x,t) = (4\pi t)^{-\frac n2} e^{-\frac{|x|^2}{4t}}$ which is of course a trivial consequence of its explicit expression. Since there exists no analogue of such explicit formula for the kernel $G_s(x,t)$, its positivity is far from obvious. The next result establishes this fact. It was certainly obtained by Bochner based on the positivity of the subordination function, see Proposition 2 on p. 261 in \cite{Y} and Theorem \ref{T:sub} below, but it was perhaps known earlier to Paul Levy. We have adapted the beautiful proof that follows, which uses a typical Abelian-Tauberian argument, from the one-dimensional presentation in \cite{DGV}. 

\begin{prop}\label{P:posG}
For every $0<s<1$ and for every $(x,t)\in \R^{n+1}_+$, we have
\[
G^{(s)}(x,t) \ge  0.
\]
\end{prop} 

\begin{proof}
In view of \eqref{Gsi} it suffices to show that for one $T>0$
\begin{equation}\label{Gpos}
G^{(s)}(x,T) \ge 0,
\end{equation}
for every $x\in \Rn$. If this holds, in fact, then for every $t>0$ and every $x\in \Rn$ we have 
\[
G^{(s)}(x,t) = \left(\frac Tt\right)^{\frac n{2s}} G^{(s)}\left(\left(\frac{T}t\right)^{1/2s} x,T\right) \ge 0,
\]
and we are done. To prove \eqref{Gpos} consider the spherically symmetric function in $L^1(\Rn)$
\[
f = \frac{A}{|\cdot|^{n+2s}} \mathbf 1_{B(0,1)^c},
\]
where the constant $A>0$ is chosen so that $\hat f(0) = \int_{\Rn} f(x) dx = 1$. Using such normalization, we obtain for every $\xi\not= 0$
\begin{align*}
\hat f(\xi) & = 1 + \hat f(\xi) - \int_{\Rn} f(x) dx
\\
& = 1 - \int_{\Rn} [1 - e^{- 2 \pi i <\xi,x>}] f(x) dx = 1 - A \int_{|x|\ge 1} \frac{1 - \cos(2 \pi <\xi,x>)}{|x|^{n+2s}} dx
\\
& = 1 - A |\xi|^{2s} \int_{|y|\ge |\xi|} \frac{1- \cos(2 \pi <\xi/|\xi|,y>)}{|y|^{n+2s}} dy = 1 - A |\xi|^{2s} \int_{|y|\ge |\xi|} \frac{1- \cos(2 \pi y_n)}{|y|^{n+2s}} dy,
\end{align*}
where we have first changed the variable to $y = |\xi| x$, and then used the invariance of the function 
\[
\xi\ \longrightarrow\ \int_{|x|\ge |\xi|} \frac{1- \cos(2 \pi <\xi/|\xi|,x>)}{|x|^{n+2s}} dx
\]
 with respect to orthogonal transformations in $\Rn$.
If we denote by 
\[
\Psi(\xi) = \int_{|x|\ge |\xi|} \frac{1- \cos(2 \pi x_n)}{|x|^{n+2s}} dx,
\]
then by Lebesgue dominated convergence we see that $\Psi(\xi) \to \alpha >0$ as $\xi\to 0$, where 
\[
\alpha = \int_{\Rn} \frac{1- \cos(2 \pi x_n)}{|x|^{n+2s}} dx = (2\pi)^{2s} \int_{\Rn} \frac{1- \cos(z_n)}{|z|^{n+2s}} dz = \frac{(2\pi)^{2s}}{\gamma(n,s)},
\]
with $\gamma(n,s)$ as in the proof of \eqref{gns} in Proposition \ref{P:slapft}. It follows that we can write
\[
\hat f(\xi) = 1 - A \alpha |\xi|^{2s} (1 + \omega(\xi)),
\]
where $\omega(\xi) = O(|\xi|^{2(1-s)}) = o(1)$ as $\xi \to 0$. We will prove that \eqref{Gpos} holds with 
\[
T = \frac{A}{\gamma(n,s)}.
\]
 This will complete the proof.

With this objective in mind, for every positive integer $k$ consider the function
\[
f_k(x) = k^{\frac n{2s}} (f\star...\star f)(k^{1/2s} x),
\]
where the convolution is repeated $k$ times. By \eqref{dil} we have
\[
\hat f_k(\xi) = \F(f\star...\star f)(k^{-1/2s} \xi) = (\hat f(k^{-1/2s} \xi))^k = \left(1 - \frac{A \alpha |\xi|^{2s} (1 + \omega(k^{-1/2s} \xi))}{k}\right)^k.
\]
This shows the crucial fact that for every $\xi\in \Rn$
\begin{equation}\label{clim}
\underset{k\to \infty}{\lim} \hat f_k(\xi) = e^{- A \alpha |\xi|^{2s}}.
\end{equation}
Since for every $k\in \mathbb N$ we have 
\
\[
||\hat f_k||_{L^\infty(\Rn)} \le ||f_k||_{L^1(\Rn)} \le ||f||^k_{L^1(\Rn)}  = 1,
\]
we conclude that \eqref{clim} also holds in $\mathscr S'(\Rn)$. We thus have for every $\vf \in \mathscr S(\Rn)$
\[
<f_k,\hat \vf> = <\hat f_k,\vf> \longrightarrow <e^{- A \alpha |\cdot|^{2s}}, \vf> = <\F(e^{- A \alpha |\cdot|^{2s}}),\hat \vf>.
\]
This is equivalent to saying that in $\mathscr S'(\Rn)$
\begin{equation}\label{clim2}
\underset{k\to \infty}{\lim} f_k = \F(e^{- A \alpha |\cdot|^{2s}}).
\end{equation}
Since $f_k \ge 0$ for every $k\in \mathbb N$, we conclude from \eqref{G} that $\F(e^{- A \alpha |\cdot|^{2s}})(x) = G^{(s)}(x,T) \ge 0$ for every $x\in \Rn$, with $T = \alpha A(2\pi)^{-2s} = \frac{A}{\gamma(n,s)}$. This proves \eqref{Gpos}.

\end{proof}

\begin{prop}\label{P:sp}
For every $0<s<1$ and for every $(x,t)\in \R^{n+1}_+$, we have
\[
G^{(s)}(x,t) >  0.
\]
\end{prop}

\begin{proof}

From Proposition \ref{P:posG} we know that $G^{(s)}(x,t)\ge 0$ globally. Then, its strict positivity follows from the strong maximum principle, or the weak Harnack inequality, see Theorem 1.1 in \cite{FK}.

\end{proof} 

The case $s=1/2$ has a special interest, and since it is quite surprising we state it in a proposition. 

\begin{prop}\label{P:onehalf}
When $s = 1/2$ the heat kernel $G^{(1/2)}(x,t)$ for $\p_t + (-\Delta)^{1/2}$ is given by the Poisson kernel for the Laplacean for the half-space $\Rn\times \R_+$, i.e., 
\begin{equation}\label{h12}
G^{(1/2)}(x,t) = P(x,t) = \frac{\G(\frac{n+1}{2})}{\pi^{\frac{n+1}{2}}} \frac{t}{(t^2 + |x|^2)^{\frac{n+1}{2}}}.
\end{equation}
\end{prop}

\begin{proof}
The equation \eqref{h12} follows from the well-known formula for the Poisson kernel, see e.g. Proposition 5, Sec. 2 in Chap. 3 of \cite{St}, 
\[
\F_{x\to \xi}(e^{- 2\pi t |\cdot|}) = P_{1/2}(x,t),
\]
see also Remark \ref{R:12} above.

\end{proof}

In particular, \eqref{h12} says that the decay of the heat kernel for $\p_t + (-\Delta)^{1/2}$ is not exponential, but polynomial. The next result states that such behavior is shared by all nonlocal semigroups $P_t^{(s)}$, $0<s<1$.

\begin{prop}\label{P:decay00}
For every $0<s<1$ let $\Phi_s(x)$ be as in \eqref{time1}. Then, there exists a constant $\beta(n,s)>0$ such that for every $x\in \Rn$
\[
\frac{\beta(n,s)}{1+|x|^{n+2s}} \le \Phi_s(x) \le \frac{\beta^{-1}(n,s)}{1+|x|^{n+2s}}.
\]
\end{prop}

The proof of Proposition \ref{P:decay00} follows immediately from the following  result. 

\begin{theorem}\label{T:polya}
For every $0<s<1$ one has
\[
\underset{|x|\to \infty}{\lim}\ |x|^{n+2s} \Phi_s(x) = \gamma(n,s)>0,
\]
where $\gamma(n,s)$ is the constant in \eqref{gnsfin} in Proposition \ref{P:gns}.
\end{theorem}

We note that, using \eqref{phiFT2} one obtains
\begin{equation}\label{hks}
|x|^{n+2s} \Phi_s(x) = 2s(2\pi)^{-\frac n2} \int_0^\infty  e^{-\left(\frac u{|x|}\right)^{2s}} u^{\frac n2 + 2s -1} J_{\frac n2}(u) du.
\end{equation}

When $n=1$ Theorem \ref{T:polya} was first proved by G. P\'olya in \cite{Po}. It is not easy to find this reference, but a detailed presentation of P\'olya's proof can be found in Lemma 3.4 in \cite{CS15}. P\'olya's argument was generalized to any dimension by Blumenthal and Getoor, see Theorem 2.1 in their paper \cite{BG60}. However, the delicate part of the proof is not presented there and the authors refer to \cite{Po}. Integrals like that in the right-hand side of \eqref{hks} are studied in \cite{PT69}. A different proof of Theorem \ref{T:polya} based on the subordination formula \eqref{gssub} below was given by Bendikov in \cite{Be}. One should also see the paper \cite{Ko} by Kolokoltsov for the proof of a more general result. Several heat kernel estimates are contained in the recent paper \cite{BSV17}, see also \cite{V17}.

In the local case $s=1$ a basic property of the heat kernel $G(x,t)$ is that for $n\not= 2$ and for every $x\not= 0$ one has
\[
\int_0^\infty G(x,t) dt = \frac{|x|^{2-n}}{(n-2)\sigma_{n-1}}.
\]
We recall that the right-hand side is the fundamental solution of $-\Delta$ with pole at $x = 0$. 
The next result expresses a similar property of $G^{(s)}(x,t)$.

\begin{prop}\label{P:heatsat}
Let $n\ge 2$. Then, for every $x\not= 0$ one has
\begin{equation}\label{fsss}
\int_0^\infty G^{(s)}(x,t) dt = E_s(x),
\end{equation}
where $E_s(x)$ is the fundamental solution of $(-\Delta)^s$ with singularity at $x=0$ in Theorem \ref{T:fs}.
\end{prop}

There is more than one way of proving Proposition \ref{P:heatsat}. A quick one resorts to the following classical result for which we refer the reader to chapter $5$ in \cite{St}.

\begin{theorem}\label{T:FTalpha}
For any $0<\alpha<n$ we have in $\mathscr S'(\Rn)$
\[
\F(|\cdot|^{-\alpha}) = \pi^{\alpha - \frac{n}{2}}
\frac{\Gamma\left(\frac{n-\alpha}{2}\right)}{\Gamma\left(\frac{\alpha}{2}\right)}\
|\cdot|^{\alpha-n}.
\]
\end{theorem}

\begin{proof}[Proof of Proposition \ref{P:heatsat}]
We proceed formally since, similarly to Theorem \ref{T:FTalpha}, for a rigorous proof we should verify the following steps in $\mathscr S'(\Rn)$, and not in the pointwise sense. The interested reader can easily provide the missing details. We notice that, in view of \eqref{fs} and \eqref{ans}, proving \eqref{fsss} is equivalent to showing 
\[
\int_0^\infty \mathscr F(G^{(s)})(x,t) dt = \hat E_s(x)= \frac{\G(\frac{n}2 -s)}{2^{2s} \pi^{\frac n2} \G(s)} \F(|\cdot|^{2s-n})(x).
\]
In light of \eqref{G} and Theorem \ref{T:FTalpha} (which we can apply with $\alpha = 2s$ since the hypothesis $0<s<1$ and $n\ge 2$ automatically guarantee that $0<\alpha <n$), this is in turn equivalent to showing that 
\[
\int_0^\infty e^{-t (2\pi |x|)^{2s}} dt = \pi^{\frac{n}{2}- 2s}
\frac{\Gamma\left(s\right)}{\Gamma\left(\frac{n-2s}{2}\right)} \frac{\G(\frac{n}2 -s)}{2^{2s} \pi^{\frac n2} \G(s)} (2\pi)^{2s} (2\pi |x|)^{-2s} = (2\pi |x|)^{-2s}.
\]
That the integral in the left-hand side is equal to the right-hand side of the above chain of equalities follows from a standard change of variable in the integral.

\end{proof}
 
\begin{prop}\label{P:markov}
For every $t>0$ we have
\[
\pt 1(x) = \int_{\Rn} G^{(s)}(x,t) dx = 1.
\]
Thus the semigroup $\pt$ is stochastically complete.
\end{prop}

\begin{proof}
We have from \eqref{heat3}
\[
\int_{\Rn} G^{(s)}(x,t) dx = \int_{\Rn} \Phi_s(x) dx = \F^{-1}(\Phi_s)(0).
\]
Keeping in mind that $\Phi_s(\xi) = \F(e^{-(2\pi |\cdot|)^{2s}})(\xi)$, we see that
$\F^{-1}(\Phi_s)(\xi) = e^{-(2\pi |\xi|)^{2s}}$, and thus  $\F^{-1}(\Phi_s)(0) = 1$.

\end{proof}

Combining Propositions \ref{P:posG} and \ref{P:markov} we obtain the following maximum principle.

\begin{prop}\label{P:maxpr}
Let $\vf \in \mathscr S(\Rn)$. Then, for every $x\in \Rn$ and $t>0$ one has
\[
\underset{\Rn}{\inf}\ \vf \le \pt \vf(x) \le \underset{\Rn}{\sup}\ \vf.
\]
\end{prop}


\section{Bochner's subordination: from $P_t$ to $(-\Delta)^s$}\label{S:sub} 

In this section we take a momentary pause from the previous one, to discuss the important fact that, using the standard heat semigroup $P_t$, we can recover $(-\Delta)^s$. This result is based on Bochner's principle of subordination (for this see Chapter 4 in \cite{B49}) and the outcome of it is yet another expression of the fractional Laplacean, see formula \eqref{flheat} in Theorem \ref{T:flheat} below, that is alternative to the ones that we know so far, namely \eqref{fls}, \eqref{fl2}, \eqref{dtn} in Theorem \ref{T:cs} and \eqref{bp} in Proposition \ref{P:bps}. We will present two proofs of such result. We begin with a preliminary observation that connects the heat semigroup $P_t = e^{-t\Delta}$ to the spherical mean-value operator $\mathscr M_r u(x)$.

\begin{lemma}\label{L:save}
Let $u\in \mathscr S(\Rn)$. Then, for every $0<t<\infty$ one has
\[
P_t u(x) - u(x) = \sigma_{n-1} (4\pi t)^{-\frac n2} \int_0^\infty e^{-\frac{r^2}{4t}} r^{n-1} [\mathscr M_r u(x) - u(x)] dr.
\]
\end{lemma}

\begin{proof}
This is a simple consequence of Cavalieri's principle, the spherical symmetry of $G(x,t)$,  and of the fact that $P_t 1 = 1$. We have
\begin{align*}
P_t u(x) - u(x) & = \int_{\Rn} G(x-y,t)[u(y) - u(x)] dy 
\\
& = \int_0^\infty \int_{S(x,r)} G(x-y,t)[u(y) - u(x)] d\sigma(y) dr 
\\
& = (4\pi t)^{-\frac n2}  \int_0^\infty e^{-\frac{r^2}{4t}} \int_{S(x,r)} [u(y) - u(x)] d\sigma(y) dr
\\
& = \sigma_{n-1} (4\pi t)^{-\frac n2} \int_0^\infty e^{-\frac{r^2}{4t}} r^{n-1} [\mathscr M_r u(x) - u(x)] dr.
\end{align*}

\end{proof} 

We are now in a position to establish the main result about Bochner's subordination for the nonlocal operator $(-\Delta)^s$.

\begin{theorem}\label{T:flheat}
Let $0<s<1$. For any $u\in \operatorname{Dom}(-\Delta)$, hence in particular, for any $u\in \mathscr S(\Rn)$, one has
\begin{align}\label{flheat}
(-\Delta)^s u(x) & =   \frac{1}{\G(-s)} \int_0^\infty t^{-s-1} \left[P_t u(x) - u(x)\right] dt
\\
& = - \frac{s}{\G(1-s)} \int_0^\infty t^{-s-1} \left[P_t u(x) - u(x)\right] dt.
\notag
\end{align}
\end{theorem}

\begin{proof}[First proof]
Let $\alpha>0$. Using Lemma \ref{L:save}, we find
\begin{align*}
& \int_0^\infty t^{-\alpha-1} [P_t u(x) - u(x)] dt = \sigma_{n-1} (4\pi)^{-\frac n2}\int_0^\infty t^{-\alpha - \frac n2}  \int_0^\infty e^{-\frac{r^2}{4t}} r^{n-1} [\mathscr M_r u(x) - u(x)] dr \frac{dt}{t}
\\
& =  \sigma_{n-1} (4\pi)^{-\frac n2} \int_0^\infty \left(\int_0^\infty t^{-\alpha - \frac n2} e^{-\frac{r^2}{4t}} \frac{dt}{t}\right) r^{n-1} [\mathscr M_r u(x) - u(x)] dr, 
\end{align*}
assuming that we can exchange the order of integration.
Now,
\begin{align*}
& \int_0^\infty t^{-\alpha - \frac n2} e^{-\frac{r^2}{4t}} \frac{dt}{t} = 2^{2\alpha + n} \G(\frac n2 +\alpha) r^{-2\alpha - n}. 
\end{align*}
Substituting in the above formula we find
\begin{align*}
& \int_0^\infty t^{-\alpha-1} [P_t u(x) - u(x)] dt  = \sigma_{n-1} \pi^{-\frac n2} 2^{2\alpha} \G(\frac n2 +\alpha)  \int_0^\infty r^{-2\alpha -1} [\mathscr M_r u(x) - u(x)] dr.
\end{align*}
Comparing the right-hand side with that of the equation in Proposition \ref{P:flave}, it is now clear that, in order for the former to provide a multiple of $(-\Delta)^s u(x)$ we must have $\alpha = s$. With such choice we obtain
\begin{align*}
& - \int_0^\infty t^{-s-1} [P_t u(x) - u(x)] dt = - \sigma_{n-1} \pi^{-\frac n2} 2^{2s} \G(\frac n2 + s) \int_0^\infty r^{-2s -1} [\mathscr M_r u(x) - u(x)] dr.
\end{align*}
Since by Proposition \ref{P:flave} and Proposition \ref{P:gns} we find
\[
- \int_{\Rn} r^{-1-2s} \big[\mathscr M_r u(x) - u(x)] dr = \frac{(-\Delta)^s u(x)}{\sigma_{n-1}   \gamma(n,s)} = \frac{(-\Delta)^s u(x)}{\sigma_{n-1}} \frac{\pi^{\frac n2} \G(1-s)}{s 2^{2s} \G(\frac n2 + s)},
\] 
the desired conclusion \eqref{flheat} follows.

\end{proof}

\begin{proof}[Second proof]

The proof that the right-hand side of \eqref{flheat} is in fact equal to $(-\Delta)^s u(x)$ can also be accomplished using the Fourier transform. Thanks to \eqref{fls3} in Proposition \ref{P:slapft}, we see that proving \eqref{flheat} is equivalent to showing
\begin{align*}
(2\pi |\xi|)^{2s} \hat u(\xi) = - \frac{s}{\G(1-s)} \int_0^\infty t^{-s-1} \left(\widehat{P_t u}(\xi) - \hat u(\xi)\right) dt.
\end{align*}
Since $\widehat{P_t u}(\xi)  = \F(G(\cdot,t) \star u)(\xi) = \hat G(\xi,t) \hat u(\xi) = e^{-t (2\pi |\xi|)^2} \hat u(\xi)$, we thus see that \eqref{flheat} is equivalent to 
\[
(2\pi |\xi|)^{2s} \hat u(\xi) =  \frac{s}{\G(1-s)} \int_0^\infty t^{-s-1} \left(1 -  e^{-t (2\pi |\xi|)^2}\right) dt\   \hat u(\xi).
 \]
This identity holds if and only if it is true that
\[
(2\pi |\xi|)^{2s}  =  \frac{s}{\G(1-s)} \int_0^\infty t^{-s} \left(1 -  e^{-t (2\pi |\xi|)^2}\right) \frac{dt}{t} = (2\pi  |\xi|)^{2s}  \frac{s}{\G(1-s)} \int_0^\infty u^{-s-1} \left(1 -  e^{- u}\right) du.
\]
The validity of this equation immediately follows from \eqref{hsg0} above.
This completes the proof.

\end{proof}  

Theorem \ref{T:flheat} extends to a very general framework, and covers situations in which the Euclidean $-\Delta$ is replaced by an operator $-L$ that neither necessarily generates a semigroup, nor its domain is necessarily dense in the relevant Banach space. This result is due to A. V. Balakrishnan, see formula (2.1) in \cite{B}, or also (4) and (5) p. 260 in \cite{Y}.


\section{More subordination: from $P_t$ to $P^{(s)}_t$}\label{S:moresub} 

``\emph{The `stable laws' $\{e^{-t|\alpha|^{2p}}\}$, $0<p<1$, are each subordinate to the Gaussian law $\{e^{-t|\alpha|^{2}}\}$, and quite generally if  $\{e^{-t \psi(\alpha)}\}$ is any subdivisible process then so is $\{e^{-t \psi(\alpha)^p}\}$ for any $0<p<1$}".

This quote is from p. 93 in \cite{B49}. This section is devoted to further illustrating Bochner's beautiful subordination idea with a twofold purpose. On one hand, it leads to an explicit representation of the nonlocal semigroup $P^{(s)}_t$ in terms of the standard heat semigroup, see Theorem \ref{T:sub} below. On the other hand, as we have already mentioned in the closing of the previous section,  the principle of subordination allows for far-reaching generalizations. These developments were already envisioned by Bochner  himself, when he said on p. 95 of \cite{B49}: 

``\emph{Now, our `subordination' can also be introduced on spaces in general provided we shift the emphasis from Fourier transformation (which may not even be definable) to (generalizations of) the distributions $F(u;A)$...}" 

In what follows we will discuss material from \cite{B49}, \cite{P}, \cite{B} and p. 259-268 in \cite{Y}. We begin with a definition. 

\begin{definition}[Bochner's subordinator]\label{D:sub}
Let $0<s<1$. For every fixed $t>0$ we introduce the \emph{subordinator} function as the inverse Laplace transform of $z \to e^{-t z^s}$, $\Re z >0$,
\begin{equation}\label{ftau}
f_{s}(t;\tau) = 
\begin{cases}
\frac{1}{2\pi i} \int_{\e - i \infty}^{\e + i \infty} e^{z\tau - t z^s} dz\ \ \ \ \ \ \tau \ge 0,
\\
\\
0\ \ \ \ \ \ \ \ \ \ \ \ \ \ \ \ \ \ \ \ \ \ \ \ \ \ \ \ \ \ \ \tau <0.
\end{cases}
\end{equation}
\end{definition}

In \eqref{ftau} the parameter $\e>0$ is fixed and $z^s$ denotes the branch such that $\Re(z^s)>0$ when $\Re z>0$. In this way, $z^s$ is a one-valued holomorphic function in the $z$-plane cut along the negative real axis. It is clear that, thanks to Cauchy's integral formula, the value of the integral is independent of $\e>0$. If $\tau>0$ we have 
\[
f_{s}(t;\tau) = \frac{e^{\e \tau}}{2\pi} \int_{- \infty}^{\infty} e^{i y \tau} e^{-t (\e + i y)^s} dy.
\]
Thus, the convergence of the integral in \eqref{ftau} is guaranteed by the decay of the factor $e^{- t z^s}$, and we have $f_{s}(t;\tau)  \cong e^{\e \tau}$. The following simple, yet crucial formula, is key to the subordination principle. It expresses the fact that the Laplace transform of $f_s(t;\cdot)$ is the function $z\to e^{-t z^s}$.

\begin{lemma}\label{L:sub}
For every $t>0$ and $a>0$ one has
\[
e^{-t a^s} = \int_0^\infty  f_s(t;\tau) e^{-\tau a} d\tau.
\]
\end{lemma}

\begin{proof}
Consider the function
\[
g(z) = \frac{ e^{- t z^s}}{z-a}.
\]
It is a holomorphic function in $\{z\in \mathbb C\mid \Re z>0, z\not= a\}$, with a simple pole in $z=a$. Its residue is given by Res$(g,a) = e^{-t a^s}$. Having observed this, for any $a>0$ consider the line $\e + i y$, with $0< \e < a$. We have from \eqref{ftau}, after an exchange of the order of integration,
\begin{align*}
& \int_0^\infty e^{-\tau a} f_s(t;\tau) d\tau = \frac{1}{2\pi i} \int_{\e - i \infty}^{\e + i \infty} e^{- t z^s} \int_0^\infty e^{-(a-z)\tau} d\tau dz
\\
& = \frac{1}{2\pi i} \int_{\e - i \infty}^{\e + i \infty} e^{- t z^s} \left[\frac{e^{-(a-z)\tau}}{z-a}\right]_0^\infty dz
\\
& =  - \frac{1}{2\pi i} \int_{\e - i \infty}^{\e + i \infty} \frac{e^{- t z^s}}{z-a} dz = - \frac{1}{2\pi i} \int_{\e - i \infty}^{\e + i \infty} g(z) dz.
\end{align*}
We now apply Cauchy's residue theorem to the function $g(z)$ and to the curve $\G_R$, composed of a vertical piece $\e + iy$, with $|y|\le R$ and by a half-circle $C_R$ of points $z = \e + R e^{i\vt}$, for large $R$. From what observed above, we obtain
\[
\int_{\G_R} g(z) dz = 2 \pi i e^{-t a^s}.
\]
Since $\int_{C_R} g(z) dz \to 0$ as $R\to \infty$ we conclude that
\[
- \frac{1}{2\pi i} \int_{\e - i \infty}^{\e + i \infty} g(z) dz =  e^{-t a^s}.
\]
This completes the proof.

\end{proof}

A remarkable consequence of Lemma \ref{L:sub} is the following result that connects the fractional semigroup $\pt$ to the standard heat semigroup $P_t$.

\begin{theorem}\label{T:sub}
Let $0<s<1$. Then, for every $x\in \Rn$ and $t>0$ one has 
\begin{equation}\label{gssub}
G^{(s)}(x,t) = \int_0^\infty f_{s}(t;\tau) G(x,\tau) d\tau.
\end{equation}
As a consequence, for any $u\in \mathscr S(\Rn)$ one has
\begin{equation}\label{sgsub}
\pt u(x) = \int_0^\infty f_{s}(t;\tau) P_\tau u(x) d\tau.
\end{equation}
\end{theorem}

\begin{proof}
For every $t>0$ and $\xi\in \Rn$ we have using Fubini's theorem
\begin{align*}
& \F(\int_0^\infty f_{s}(t;\tau) P_\tau u(\cdot) d\tau)(\xi) = \int_0^\infty f_{s}(t;\tau) \F(P_\tau u(\cdot))(\xi) d\tau
\\
& = \int_0^\infty f_{s}(t;\tau) \F(G(\cdot,\tau)\star u)(\xi) d\tau = \int_0^\infty f_{s}(t;\tau) \F(G(\cdot,\tau))(\xi) \hat u(\xi) d\tau
\\
& = \hat u(\xi) \int_0^\infty f_{s}(t;\tau) e^{-\tau (2\pi |\xi|)^{2}} d\tau = \hat u(\xi) e^{-t (2\pi |\xi|)^{2s}} 
\\
& = \hat u(\xi) \F(G^{(s)}(\cdot,t))(\xi) = \F(\pt u)(\xi),
\end{align*}
where in the third to the last equality we have used Lemma \ref{L:sub} with $a = (2\pi |\xi|)^2$. This proves \eqref{sgsub}. The proof of \eqref{gssub} is done in a similar, but simpler, way.

\end{proof}
 
We can now use Theorem  \ref{T:sub} to draw two basic properties of the subordination function $f_{s}(t;\tau)$. We only prove one of them, \eqref{ft1} below, and refer to Proposition 3 on p.262 in \cite{Y} for a proof of \eqref{ft2}. 

\begin{prop}\label{P:sub1}
For any fixed $t, t'>0$ one has  
\begin{equation}\label{ft1}
\int_0^\infty f_{s}(t;\tau) d\tau= 1,
\end{equation}
and
\begin{equation}\label{ft2}
f_{s}(t+t';\tau) = \int_0^\infty f_{s}(t;\tau - \sigma) f_{s}(t';\sigma) d\sigma.
\end{equation}
\end{prop}

\begin{proof}
Using Proposition \ref{P:markov} and the identity \eqref{gssub}, we find
\begin{align*}
1 & = \int_{\Rn} G^{(s)}(x,t) dx = \int_{\Rn}  \int_0^\infty f_{s}(t;\tau) G(x,\tau) d\tau dx
\\
& = \int_0^\infty f_{s}(t;\tau)  \int_{\Rn} G(x,\tau) dx d\tau = \int_0^\infty f_{s}(t;\tau) d\tau,
\end{align*} 
which proves \eqref{ft1}.

\end{proof}


\section{A chain rule for $(-\Delta)^s$}\label{S:cr}

In \cite{CC} the authors proved a basic pointwise inequality for the fractional Laplacean. Such inequality, which can be seen as a form of nonlocal chain-rule,  plays a remarkable role in many problems from the applied sciences involving the nonlocal operator $(-\Delta)^s$, see for instance the beautiful papers \cite{CC2} and \cite{CV}, respectively on the two-dimensional quasi-geostrophic equation, and nonlinear evolution equations with fractional diffusion. 

We will present two accounts of the chain rule, the former from \cite{CC}, the latter from an interesting generalization given in \cite{CM}. 
Let us begin with a simple observation. If $u\in C^2(\Rn)$ and $\vf\in C^2(\R)$ the standard chain rule gives
\[
\Delta \vf(u) = \vf''(u) |\nabla u|^2 + \vf'(u) \Delta u.
\]
If we assume that $\vf$ is also convex, then $\vf''\ge 0$, and we obtain in a trivial way 
\[
(-\Delta) \vf(u) \le \vf'(u) (-\Delta) u.
\]
The next result generalizes to the nonlocal setting this observation.

\begin{theorem}[Chain rule for $(-\Delta)^s$]\label{T:cr}
Let $0<s\le 1$ and $\vf\in C^1(\R)$ be a convex function. Then, for any $u\in \mathscr S(\Rn)$ one has
\[
(-\Delta)^s \vf(u) \le \vf'(u(x)) (-\Delta)^s u.
\]
\end{theorem}

\begin{proof}[First proof]
Since the function $\vf\in C^1(\R)$ and is convex, we have for any $\tau, \sigma\in \R$
\[
\vf'(\sigma)(\tau - \sigma) \le \phi(\tau) - \phi(\sigma).
\]
This inequality easily gives for $u\in \mathscr S(\Rn)$  and for every $x, y\in \Rn$
\[
2 \vf(u(x)) - \vf(u(x+y)) - \vf(u(x-y)) \le \vf'(u(x))(2 u(x) - u(x+y) - u(x-y)).
\]
Dividing the latter inequality by $|y|^{n+2s}$ and integrating in $y\in \Rn$ we immediately obtain the desired conclusion keeping in mind the definition \eqref{fls} of $(-\Delta)^s$. 

\end{proof}

\begin{proof}[Second proof]
The second proof we present is taken from \cite{CM} and it has the advantage of carrying over to a situation where $\Rn$ is replaced by a compact $n$-dimensional manifold $M$, in which case the representation \eqref{fls} is no longer available. Consider the Cauchy problems
\begin{equation}\label{heat000}
\begin{cases}
\frac{\p U}{\p t} + (-\Delta)^s U = 0\ \ \ \ \ \ \ \ \text{in}\ \R^{n+1}_+,
\\
\\
U(x,0) = u(x),\ \ \ \ \ \ \ \ \ \ \ \  x\in \Rn.
\end{cases}
\end{equation}
and
\begin{equation}\label{heatbis}
\begin{cases}
\frac{\p V}{\p t} + (-\Delta)^s V = 0\ \ \ \ \ \ \ \ \text{in}\ \R^{n+1}_+,
\\
\\
V(x,0) = \vf(u)(x),\ \ \ \ \ \ \ \ \ \ \ \  x\in \Rn.
\end{cases}
\end{equation}

Their solutions are respectively given by $U(x,t) = P_t^{(s)} u(x)$ and $V(x,t) = P_t^{(s)} \vf(u)(x)$. Now, by Jensen inequality, Proposition \ref{P:posG} and \ref{P:markov}, we obtain
\[
\vf(U(x,t)) = \vf\left(\int_{\Rn} G^{(s)}(x-y,t) u(y) dy\right) \le \int_{\Rn} \vf(u)(y) G^{(s)}(x-y,t) dy = V(x,t).
\]
This shows that if, for a fixed $x\in \Rn$, we consider the function
\[
\Psi(t) = V(x,t) - \vf(U(x,t)),
\]
then we have $\Psi(t) \ge 0$ for every $t> 0$. Since  
we clearly have 
\[
\Psi(0) = V(x,0) - \vf(U(x,0)) = \vf(u(x)) - \vf(u(x)) = 0,
\] 
we conclude that it must be $\Psi'(0) \ge 0$. 
 
We now have
\begin{align*}
\Psi'(t) & = \frac{ \p P_t \vf(u)}{\p t}(x) - \vf'(P_t u(x)) \frac{ \p P_t u}{\p t}(x)
\\
& = - (-\Delta)^s P_t(\vf(u))(x) +  \vf'(P_t u(x)) (-\Delta)^s P_t u(x).
\end{align*}
This formula gives
\[
0 \le \Psi'(0) = - (-\Delta)^s \vf(u) (x) +  \vf'(u(x)) (-\Delta)^s u(x),
\]
which gives the desired conclusion.

 \end{proof}

 
\section{The Gamma calculus for $(-\Delta)^s$}\label{S:gamma}

In the applications of pde's to geometry there is a remarkable tool that allows to connect the heat semigroup to the geometry of the underlying manifold. This tool is the so-called Bakry-Emery Gamma calculus, for which we refer the reader to the beautiful recent book \cite{BGL}  and the references therein. 
At the heart of this calculus there is the so-called \emph{curvature-dimension inequality} CD$(\kappa,n)$ that we introduce in \eqref{cd} and Definition \ref{D:cd} below. It is a remarkable fact that, on a $n$-dimensional Riemannian manifold $\mathbb M$, such inequality on functions is in fact equivalent to the lower bound Ric$\ \ge \kappa$ on the Ricci tensor. More importantly, a stunning aspect of the gamma calculus is that the curvature-dimension inequality \underline{alone}, in combination with properties of the heat semigroup, suffices to develop a wide program that connects the geometry of $\mathbb M$ to various global properties of the manifold itself, such as: 
\begin{itemize}
\item global volume bounds for the geodesic balls, 
\item the Bonnet-Myers compactness theorem, 
\item global Poincar\'e inequalities, 
 \item parabolic scale invariant Harnack inequalities, 
\item Gaussian upper and lower bounds, 
\item Liouville theorems
\item De Giorgi-Nash-Moser estimates...and much more.
\end{itemize}

Traditionally, the majority of these fundamental results from Riemannian geometry rest on two pillars: 
\begin{itemize}
\item[(i)] the celebrated \emph{Bochner identity};
\item[(ii)] the \emph{Laplacean comparison theorem}.
\end{itemize}
Whereas (i) is a purely pointwise identity on functions (see \eqref{ric}
 below), (ii) relies on deeper aspects which are more genuinely Riemannian, such as the fact that the exponential map is locally a diffeomorphism and the theory of Jacobi fields. The gamma calculus allows to remove from the equation the Laplacean comparison theorem, and in a way it elevates the Bochner identity to a preeminent role in the development of the Li-Yau theory. This is especially important in situations where the above mentioned Riemannian tools are lacking. In this connection one should see the works \cite{BG1}, \cite{BaG1}, \cite{BG2}, \cite{BG3}, \cite{BG4}.

Inspired by the above discussion, in this section we propose that a gamma calculus be developed in the context of nonlocal operators such as $(-\Delta)^s$. We stress that although this would have an interest in its own even in the setting of flat $\Rn$, a nonlocal gamma calculus would also be instrumental to considerable  developments both in analysis and geometry. In what follows for the sake of simplicity we write $L = -(-\Delta)^s$, for a given $0<s<1$. 

\begin{definition}[Nonlocal \emph{carr\'e du champ}]\label{D:nlcdc}
Given $u, v\in \mathscr S(\Rn)$ we define
\begin{equation}\label{gamma}
\Gamma^{(s)}(u,v) = \frac 12 [L(uv) - u L v - v Lu].
\end{equation}
When $u = v$ in \eqref{gamma}, we simply write
\[
\Gamma^{(s)}(u) \overset{def}{=} \Gamma^{(s)}(u,u).
\] 
\end{definition} 
It is obvious that $\Gamma^{(s)}(u,v) = \Gamma^{(s)}(v,u)$. We have the following result.

\begin{lemma}\label{L:gamma}
For $u, v\in \mathscr S(\Rn)$ one has
\[
\Gamma^{(s)}(u,v)(x) =  \frac{\gamma(n,s)}{2} \int_{\Rn} \frac{(u(x) - u(y)) (v(x) - v(y))}{|x-y|^{n + 2s}} dy.
\]
\end{lemma} 

\begin{proof}
It is easier to adopt the alternative expression \eqref{fl2} of $(-\Delta)^s$. Keeping in mind our sign convention for $L$ and using \eqref{gamma} we find
\begin{align*}
\Gamma^{(s)}(u,v) & = - \frac{\gamma(n,s)}{2} \operatorname{PV} \int_{\Rn} \frac{u(x) v(x) - u(y) v(y) - u(x)(v(x) - v(y)) - v(x) (u(x) - u(y))}{|x-y|^{n + 2s}} dy
\\
& = - \frac{\gamma(n,s)}{2} \operatorname{PV} \int_{\Rn} \frac{u(x) v(y) + v(x) u(y) - u(x) v(x) - u(y) v(y)}{|x-y|^{n + 2s}} dy
\\
& =  \frac{\gamma(n,s)}{2} \int_{\Rn} \frac{(u(x) - u(y)) (v(x) - v(y))}{|x-y|^{n + 2s}} dy.
\end{align*}
Notice that we have dropped the principal value sign in front of the last integral since, thanks to a cancellation that has occurred, the numerator in the last integral is $O(|x-y|^2)$ near $x$, so that the integrand is in fact locally in $L^1$. 

\end{proof}

When $u = v$ we obtain from Lemma \ref{L:gamma} for every $x\in \Rn$
\begin{equation}\label{gamma2}
\Gamma^{(s)}(u)(x) = \frac{\gamma(n,s)}{2} \int_{\Rn} \frac{(u(x) - u(y))^2}{|x-y|^{n + 2s}} dy \ge 0.
\end{equation}

\begin{remark}\label{R:posgamma}
We note that the positivity of $\Gamma^{(s)}(u)$ could have also been deduced immediately from its definition
\begin{equation}\label{gu}
- \Gamma^{(s)}(u) = \frac 12 \left[(-\Delta)^s(u^2) - 2 u (-\Delta)^s u\right],
\end{equation}
and the chain rule in Theorem \ref{T:cr}. The latter gives in fact 
\begin{equation}\label{crsquare}
(-\Delta)^s(u^2) \le 2 u (-\Delta)^s u,
\end{equation}
which shows $- \Gamma^{(s)}(u) \le 0$.
\end{remark}

\begin{definition}[Nonlocal energy]\label{D:energy}
Given a function $u\in \mathscr S(\Rn)$  we define its $s-$\emph{energy} as follows
\[
\mathscr E_{(s)}(u) = \frac 12 \int_{\Rn} \Gamma^{(s)}(u)(x) dx = \frac{\gamma(n,s)}{4} \int_{\Rn}\int_{\Rn} \frac{(u(x) - u(y))^2}{|x-y|^{n + 2s}} dy dx.
\]
\end{definition}
The reader should note that the energy $\mathscr E_{(s)}(u)$ is precisely the one that enters in the definition of the fractional Sobolev space 
\begin{equation}\label{fss}
W^{s,2}(\Rn) = \{u\in L^2(\Rn)\mid \mathscr E_{(s)}(u)<\infty\}.
\end{equation}
This is a Hilbert space if we endow it with the norm
\[
||u||_{W^{s,2}(\Rn)} = \left(||u||^2_{L^2(\Rn)} + \mathscr E_{(s)}(u)\right)^{1/2},
\]
see \cite{Ad}, and also \cite{DPV}.
The space $W^{s,2}(\Rn)$ can also be characterized using the Fourier transform. Consider in fact the Sobolev space $H^{s,2}(\Rn)$ recalled in \eqref{sobs} above. Using the Fourier transform it is easy to show that $W^{2,s}(\Rn) \cong H^{s,2}(\Rn)$. In this connection we note that if $u\in \mathscr S(\Rn)$, then by applying  in this order Corollary \ref{C:semi}, Lemma \ref{L:ibp} and Plancherel's theorem, we find
\begin{align*}
& \int_{\Rn} u\ (-\Delta)^s u\ dx = \int_{\Rn} \left((-\Delta)^{s/2} u\right)^2 dx = \int_{\Rn} \left(\mathscr F\left((-\Delta)^{s/2} u\right)\right)^2 dx
\\
& = \int_{\Rn} \left(2\pi|\xi|\right)^{2s} |\hat u|^2 dx.
\end{align*}
We have already discussed related questions in Section \ref{S:riesz} above. 

Returning to Definition \ref{D:energy}, we have the following result that shows that $(-\Delta)^s$ is the first variation of the energy $\mathscr E_{(s)}$, and thus such operator also has a nice variational structure.

\begin{prop}\label{P:el}
The fractional Laplacean is the Euler-Lagrange equation of the functional $u\to \mathscr E_{(s)}(u)$. Given $u\in \mathscr S(\Rn)$, we have in fact for every $\vf\in \mathscr S(\Rn)$
\begin{equation}\label{el}
\frac{d}{dt} \mathscr E_{(s)}(u+t \vf)\big|_{t=0}  =  \int_{\Rn}  (-\Delta)^s u(x) \vf(x) dx.
\end{equation}
This shows that $u$ is a critical point of $\mathscr E_{(s)}$ if and only if $(-\Delta)^s u = 0$.
\end{prop}

\begin{proof}
Given $u, \vf \in \mathscr S(\Rn)$ we consider the function $t\to \mathscr E_{(s)}(u+t \vf)$, and take its derivative at $t = 0$. After some elementary computations we obtain
\begin{align*}
\frac{d}{dt} \mathscr E_{(s)}(u+t \vf)\big|_{t=0}  & = \frac{\gamma(n,s)}{2} \int_{\Rn} PV \int_{\Rn} \frac{u(x) \vf(x) - u(y) \vf(y)}{|x-y|^{n + 2s}} dy dx
\\
& = \frac{\gamma(n,s)}{2} \int_{\Rn} PV \int_{\Rn} \frac{u(x)(\vf(x) - \vf(y)) + (u(x) - u(y)) \vf(y)}{|x-y|^{n + 2s}} dy dx
\\
& = \frac{\gamma(n,s)}{2} \left\{\int_{\Rn} u(x) PV \int_{\Rn} \frac{\vf(x) - \vf(y)}{|x-y|^{n + 2s}} dy dx + \int_{\Rn} \vf(y)  PV \int_{\Rn} \frac{u(x) - u(y)}{|x-y|^{n + 2s}} dx dy\right\}
\\
& = \frac 12 \int_{\Rn} u(x) (-\Delta)^s \vf(x) dx + \frac 12 \int_{\Rn} \vf(x) (-\Delta)^s u(x) dx 
\\
& = \int_{\Rn} \vf(x) (-\Delta)^s u(x) dx, 
\end{align*}
where in the second to the last equality we have used \eqref{fl2}, while in the last equality we have used Lemma \ref{L:ibp}. 

\end{proof}

In connection with the variational structure of $(-\Delta)^s$ we mention that in the existing literature the notion of weak solution of the problem \eqref{nhdp} above is formulated by saying the $u\in H^{s,2}(\Rn)$, $u=0$ a.e. in $\Rn\setminus \Om$, and 
\[
\int_{\Rn} (-\Delta)^{s/2} u (-\Delta)^{s/2} \vf dx = \int_\Om f \vf dx,
\]
for every $\vf\in H^{s,2}(\Rn)$, such that $\vf=0$ a.e. in $\Rn\setminus \Om$. It is worth observing here that such notion is equivalent to requesting that
\[
\int_{\Rn} \Gamma^{(s)}(u,\vf)(x) dx =  \int_\Om f \vf dx,
\]  
for all $\vf$ as above.

Using Theorem \ref{T:flheat} we can obtain an alternative expression of $\Gamma^{(s)}(u)$ based on the heat semigroup. This observation is important since it allows to introduce a notion of nonlocal energy in non-Euclidean situations. We leave the proof of the following Proposition \ref{P:gammaheat} to the interested reader. 

\begin{prop}\label{P:gammaheat}
Let $u\in \mathscr S(\Rn)$, then
\[
\Gamma^{(s)}(u)(x) =  \frac{s}{2 \G(1-s)} \int_0^\infty t^{-s-1} \left(P_t u^2(x) - 2 u(x) P_t u(x) + u^2(x)\right) dt.
\]
\end{prop}

Notice that since by Jensen's inequality, or simply Cauchy-Schwarz, we have $P_t u^2(x) \ge (P_t u(x))^2$, we obtain
\[
\Gamma^{(s)}(u)(x)\ge \frac{s}{2\G(1-s)} \int_0^\infty t^{-s-1} \left(P_t u(x) -  u(x) \right)^2 dt \ge 0,
\]
which confirms the positivity of $\Gamma^{(s)}(u)$, see \eqref{gamma2} and Remark \ref{R:posgamma} above.

Suppose we are on a $n$-dimensional Riemannian manifold $M$, with Laplacean $L$. In the local case $s=1$, let us simply denote $\G^{(1)}(u)$ by $\G(u)$. We easily obtain from \eqref{gu} that $\Gamma(u) = |\nabla u|^2$. In such situation, the celebrated Bochner's identity gives
\begin{equation}\label{ric}
L(\Gamma(u)) = 2 ||\nabla^2 u||^2 + 2 \Gamma(u,Lu) + 2 \operatorname{Ric}(\nabla u,\nabla u),
\end{equation}
where Ric indicates the Ricci tensor on $M$, and we have denoted by $\nabla^2$ the Hessian on $M$. The central idea of the Bakry-Emery's gamma calculus is to reverse the recipe and use \eqref{ric} as an analytical definition of Ricci tensor. Such intuition is implemented through Definition \ref{D:cd} below, see \cite{BGL}. First, we introduce the functional
\begin{equation}\label{gamma20}
\G_2(u) = \frac 12 \left[L(\Gamma(u)) - 2 \Gamma(u,Lu)\right].
\end{equation}
With this definition it is clear that we can rewrite \eqref{ric} as
\begin{equation}\label{ric3}
\G_2(u) =  ||\nabla^2 u||^2 +  \operatorname{Ric}(\nabla u,\nabla u).
\end{equation}
Suppose now that $M$ satisfies the Ricci lower bound Ric$\ge \kappa$. Then, $\operatorname{Ric}(\nabla u,\nabla u) \ge \kappa |\nabla|^2 = \kappa \G(u)$, and we obtain from \eqref{ric3}
\begin{equation}\label{ric4}
\G_2(u) \ge  ||\nabla^2 u||^2  + \kappa \G(u),
\end{equation}
for every $f\in C^\infty(M)$. On the other hand, Newton's inequality gives
\[
||\nabla^2 u||^2 \ge \frac 1n (Lu)^2,
\]
and we thus find from \eqref{ric4} that for every $f\in C^\infty(M)$ one has
\begin{equation}\label{cd}
\G_2(u) \ge \frac 1n (Lu)^2 + \kappa \Gamma(u).
\end{equation}

\begin{definition}[Bakry-Emery]\label{D:cd}
The manifold $M$ and the operator $L$ are said to satisfy the \emph{curvature-dimension inequality} \emph{CD}$(\kappa,n)$ if \eqref{cd} holds true for every $u\in C^\infty(M)$. 
\end{definition} 

We have shown above that
\[
\operatorname{Ric}\ \ge \kappa\ \Longrightarrow\ \operatorname{CD}(\kappa,n).
\]
It is a remarkable fact that the inequality CD$(\kappa,n)$, which is a condition on functions, does in fact completely characterize Ricci lower bounds on $M$, in the sense that
\[
\operatorname{Ric}\ \ge \kappa\ \Longleftrightarrow\ \operatorname{CD}(\kappa,n).
\]
For the implication $\Longleftarrow$ the reader should see  Proposition 6.2 in \cite{Bak}.

In view of the above discussion, and since as we have shown in \eqref{gamma2} above there exists a natural nonlocal \emph{carr\'e du champ}, we next introduce a nonlocal counterpart of the form $\G_2$.

\begin{definition}[Nonlocal $\Gamma_2^{(s)}$]\label{D:ricci}
Given $u,v\in \mathscr S(\Rn)$ we define
\begin{equation}\label{gamma22}
\Gamma^{(s)}_2(u,v) = \frac 12 \left[L \G_s(u,v) - \Gamma_s(u,Lv) - \G_s(v,Lu)\right].
\end{equation}
\end{definition}
 
Again, it is obvious that $\Gamma^{(s)}_2(u,v) = \Gamma^{(s)}_2(v,u)$. As for $\Gamma^{(s)}$ we set
\begin{equation}\label{gamma222}
\Gamma^{(s)}_2(u) \overset{def}{=} \Gamma^{(s)}_2(u,u) = \frac 12 \left[L \Gamma^{(s)}(u) - 2 \Gamma^{(s)}(u,Lu)\right].
\end{equation}

\medskip

\noindent \textbf{Open problem:} Is there a number $d>0$ such that
\begin{equation}\label{be}
\G^{(s)}_{2}(u) \ge \frac 1d \left((-\Delta)^s u\right)^2\ ?
\end{equation}
Is $d = n$?
If true, the inequality \eqref{be} would be quite relevant in adapting to the nonlocal setting the Bakry-Emery gamma calculus, and for instance obtain Riemannian results in the spirit of \cite{BG1}, or the extensions to some sub-Riemannian spaces as in \cite{BaG1}, \cite{BG2}, \cite{BG3} and \cite{BG4}. Understanding the question \eqref{be} is inextricably connected to understanding 
\[
(-\Delta)^s \Gamma^{(s)}(u) = ...?
\]
i.e., a nonlocal analogue of the celebrated identity of Bochner \eqref{ric} above. We plan to come back to these questions in future works.


\section{Are there nonlocal Li-Yau inequalities?}\label{S:LY}

The celebrated Li-Yau inequality states that if $M$ is a $n$-dimensional boundariless complete Riemannian manifold with nonnegative Ricci tensor, then for any positive solution $f(x,t)$ of the heat equation on $M\times (0,\infty)$, with $u = \log f$ one has
\begin{equation}\label{ly}
|\nabla u|^2 - u_t \le \frac{n}{2t}.
\end{equation}
To be precise, \eqref{ly} is only one case of the more general inequality of Li and Yau, see \cite{LY}. One fundamental consequence of \eqref{ly} is the following scale invariant \emph{Harnack inequality}: under the above hypothesis on $M$, let $f>0$ be a solution of the heat equation on  $M\times (0,\infty)$. Then, for every $x,y\in M$ and any $0<\tau<t<\infty$ one has
\begin{equation}\label{ly2}
f(x,\tau) \le f(y,t) \left(\frac t\tau\right)^{\frac n2} \exp\left(\frac{d(x,y)^2}{4(t-\tau)}\right).
\end{equation}

This section is devoted to setting forth some interesting conjectures concerning the heat kernel $G_s(x,t)$ defined in \eqref{GG}. But before we do that, we would like to provide some motivation. The first observation is that the standard heat kernel in flat $\Rn$ satisfies \eqref{ly} above with equality.

\begin{lemma}\label{L:LYG}
Let $G(x,t) = (4\pi t)^{-\frac n2} e^{-\frac{|x|^2}{4t}}$ be the Gauss-Weierstrass kernel, and define the \emph{entropy} $E(x,t) = \log G(x,t)$. Then, for every $(x,t)\in \R^{n+1}_+$ one has
\begin{equation}\label{EE}
|\nabla_x E(x,t)|^2 - E_t(x,t) = \frac{n}{2t}.
\end{equation}
\end{lemma}

\begin{proof}
The proof is a simple computation based on the observation that
\[
E(x,t) = - \frac n2 \log(4\pi t) - \frac{|x|^2}{4t}.
\]
This gives
\[
\nabla_x E(x,t) = - \frac{x}{2t},\ \ \ \ \ \ \ E_t(x,t) = - \frac{n}{2t} + \frac{|x|^2}{4t^2},
\]
and the result immediately follows.

\end{proof}

\begin{remark}\label{R:ly}
It is interesting to observe that Lemma \ref{L:LYG} can also be derived by the self-similar equation \eqref{heat4} which, for $s=1$, we rewrite
\begin{equation}\label{EEE}
-  <\frac{x}{2t},\nabla_x E(x,t)>  - E_t(x,t) =  \frac{n}{2t}. 
\end{equation}
Since, as one easily recognizes, 
\begin{equation}\label{selfs}
-   <\frac{x}{2t},\nabla_x E(x,t)> = |\nabla_x E(x,t)|^2,
\end{equation}
we immediately obtain \eqref{EE} from the self-similar equation \eqref{EEE}. 
\end{remark}

A remarkable consequence of Lemma \ref{L:LYG} is the following inequality of Li-Yau type for the heat semigroup in $\Rn$.

\begin{theorem}\label{T:LYH}
Let $\vf \in C(\Rn)\cap L^\infty(\Rn)$, $\vf \ge 0$, and consider the function $f(x,t) = P_t \vf(x)$. If $u(x,t) = \log f(x,t)$, then
\[
|\nabla_x u(x,t)|^2 - u_t(x,t) \le \frac{n}{2t}.
\]
\end{theorem}

\begin{proof}
We begin by observing that the sought for conclusion can be reformulated in the following way
\begin{equation}\label{ly1}
\frac{|\nabla_x f(x,t)|^2}{f(x,t)} \le f_t(x,t) + \frac{n}{2t} f(x,t).
\end{equation} 
Similarly, the conclusion in Lemma \ref{L:LYG} can be written as
\begin{equation}\label{ly22}
\frac{|\nabla_x G(x,t)|^2}{G(x,t)} = G_t(x,t) + \frac{n}{2t} G(x,t).
\end{equation} 
Since 
\[
f(x,t) = P_t \vf(x) = \int_{\Rn} G(x-y,t) \vf(y) dy,
\]
denoting $D_i = \frac{\p}{\p x_i}$, we now have
\begin{align*}
D_i f(x,t) & = \int_{\Rn} D_i G(x-y,t) \vf(y) dy =  \int_{\Rn} \frac{D_i G(x-y,t)}{G(x-y,t)^{1/2}}  \vf(y)^{1/2} G(x-y,t)^{1/2} \vf(y)^{1/2} dy
\\
& \le \left(\int_{\Rn} \frac{D_i G(x-y,t)^2}{G(x-y,t)}  \vf(y) dy\right)^{1/2} \left(\int_{\Rn} G(x-y,t)  \vf(y) dy\right)^{1/2}. 
\end{align*}
This gives
\begin{align*}
|\nabla_x f(x,t)|^2 & \le \sum_{i=1}^n \int_{\Rn} \frac{D_i G(x-y,t)^2}{G(x-y,t)}  \vf(y) dy \int_{\Rn} G(x-y,t)  \vf(y) dy
\\
& = f(x,t) \int_{\Rn} \frac{|\nabla_x G(x-y,t)|^2}{G(x-y,t)}  \vf(y) dy
\\
& = f(x,t) \int_{\Rn} \left(G_t(x-y,t) + \frac{n}{2t} G(x-y,t)\right) \vf(y) dy,
\end{align*}
where in the last equality we have used \eqref{ly22}.
From this estimate we infer
\begin{align*}
\frac{|\nabla_x f(x,t)|^2}{f(x,t)} & \le \int_{\Rn} G_t(x-y,t) \vf(y) dy + \frac{n}{2t} \int_{\Rn} G(x-y,t) \vf(y) dy
\\
& = f_t(x,t) + \frac{n}{2t} f(x,t),
\end{align*} 
which is the desired inequality \eqref{ly1}.

 \end{proof}
 
 We are now ready to prove the following fundamental result that was first independently obtained by B. Pini \cite{Pi} and J. Hadamard \cite{H} in the plane $\R_x\times \R_t$. It extends to the heat equation (with some fundamental differences) the celebrated inequality first proved for the Laplacean by Axel Harnack, see \cite{Har}.
  
\begin{theorem}[The Pini-Hadamard scale invariant Harnack inequality]
\label{T:hH}
Let $\vf \in C(\Rn)\cap L^\infty(\Rn)$, $\vf \ge 0$, and consider the function $f(x,t) = P_t \vf(x)$. For every $x, y\in \Rn$, and every $0<s<t<\infty$ one has
\[
f(x,s) \le f(y,t) \left(\frac ts\right)^{\frac n2} \exp\left(\frac{|x-y|^2}{4(t-s)}\right).
\]
\end{theorem}

\begin{proof}
Consider the function
\[
\psi(\tau) = \log f(\gamma(\tau)),\ \ \ \ \ \ \ 0\le \tau \le 1,
\]
where $\gamma(\tau)$ is a straight line (geodesic) starting at the ``upper" point  $(y,t)$ and ending at the ``lower" point $(x,s)$, i.e., 
\[
\gamma(\tau) = (y + \tau(x-y),t+ \tau(s-t)), \ \ \ \ \ \ \ \ \ 0 \le \tau \le 1.
\]
We clearly have
\begin{align*}
\log \frac{f(x,s)}{f(y,t)} & = \log f(\gamma(1)) - \log f(\gamma(0)) = \psi(1) - \psi(0) = \int_0^1 \psi'(\tau) d\tau
\\
& = \int_0^1 <\frac{\nabla_x f(\gamma(\tau))}{f(\gamma(\tau))},x-y> d\tau - (t-s) \int_0^1 \frac{f_t(\gamma(\tau))}{f(\gamma(\tau))} d\tau
\\
& \le |x-y| \int_0^1 \frac{|\nabla_x f(\gamma(\tau))|}{f(\gamma(\tau))} d\tau + (t-s) \int_0^1 \frac{n}{2(t+\tau(s-t))} d\tau
\\
& - (t-s) \int_0^1 \frac{|\nabla_x f(\gamma(\tau))|^2}{f(\gamma(\tau))^2} d\tau,
\end{align*}
where in the last line we have used the Li-Yau inequality for $f$ in Theorem  \ref{T:LYH}, see also \eqref{ly1}. An easy argument now gives
\[
\int_0^1 \frac{n}{2(t+\tau(s-t))} d\tau = \frac{n}{2(t-s)} \int_s^t \frac{dr}{r} = \frac{1}{t-s} \log\left(\frac ts\right)^{\frac n2}.
\]
On the other hand, for every $\e>0$ we have
\begin{align*}
& |x-y| \int_0^1 \frac{|\nabla_x f(\gamma(\tau))|}{f(\gamma(\tau))} d\tau \le 
|x-y| \left(\int_0^1 \frac{|\nabla_x f(\gamma(\tau))|^2}{f(\gamma(\tau))^2} d\tau\right)^{1/2}
\\
& \le \frac{\e}{2} \int_0^1 \frac{|\nabla_x f(\gamma(\tau))|^2}{f(\gamma(\tau))^2} d\tau + \frac{1}{2\e} |x-y|^2.
\end{align*}
We thus find for every $\e>0$ 
\begin{align*}
\log \frac{f(x,s)}{f(y,t)} & \le \frac{\e}{2} \int_0^1 \frac{|\nabla_x f(\gamma(\tau))|^2}{f(\gamma(\tau))^2} d\tau  - (t-s) \int_0^1 \frac{|\nabla_x f(\gamma(\tau))|^2}{f(\gamma(\tau))^2} d\tau
\\
& +  \frac{1}{2\e} |x-y|^2 + \log\left(\frac ts\right)^{\frac n2}. 
\end{align*}
Choosing $\e = 2(t-s)$ in the latter inequality, we obtain
\[
\log \frac{f(x,s)}{f(y,t)} \le \log\left(\frac ts\right)^{\frac n2} + \frac{|x-y|^2}{4(t-s)}.
\]
Exponentiating, we reach the desired conclusion. 

\end{proof}

In view of Lemma \ref{L:LYG}, Theorem \ref{T:LYH} and Theorem \ref{T:hH} it is natural to wonder whether such results have nonlocal analogues. As we have indicated in Section \ref{S:gamma}, besides having an interest in its own right, such question is also relevant to the development of a nonlocal Li-Yau theory. We are thus led to formulating the following: 

\medskip

\noindent \textbf{Conjecture 1:} With $\Gamma^{(s)}$ defined as in \eqref{gamma2} above, is it true that for every $x\in \Rn$ and $t>0$ one has
\begin{equation}\label{con}
\frac{\Gamma^{(s)}(G^{(s)})(x,t)}{G^{(s)}(x,t)^2} - \frac{\frac{\p G^{(s)}}{\p t}(x,t)}{G^{(s)}(x,t)} \le \frac{n}{2s} \frac 1t \ ?
\end{equation}
Equivalently, we can write this conjecture in the following way
\begin{equation}\label{coneq}
\Gamma^{(s)}(G^{(s)})(x,t) - \frac{\p G^{(s)}}{\p t}(x,t) G^{(s)}(x,t) \le \frac{n}{2s} \frac 1t G^{(s)}(x,t)^2\ ?
\end{equation}
 
\medskip

Recall that we have observed in \eqref{heat44} above that
\[
- \frac{\p G^{(s)}}{\p t} =  \frac{n}{2s}\frac 1t G^{(s)}(x,t) + \frac{1}{2s}  <\frac{x}{t},\nabla_x  G^{(s)}>. 
\]
It follows that \eqref{coneq} is true if and only if: 

\medskip

\noindent \textbf{Conjecture 2:}
Is it true that
\begin{equation}\label{con2}
\Gamma^{(s)}(G^{(s)}(\cdot,t))(x) + \frac{1}{2s}  <\frac{x}{t},\nabla_x G^{(s)}(x,t)> G^{(s)}(x,t)\ \le\ 0\ ?
\end{equation}

\medskip

We observe explicitly that \eqref{con2} represents the nonlocal counterpart of the local identity \eqref{selfs} noted above. We also note that from the formula \eqref{gssub} in Theorem \ref{T:sub} above we have for every $x\in \Rn$ and $t>0$ 
\begin{align}\label{gssub0}
<\frac{x}{2}, \nabla_x G^{(s)}(x,t)> & = \int_0^\infty f_{s}(t;\tau) <\frac x2,\nabla_x G(x,\tau)> d\tau 
\\
& = - \pi |x|^2 \int_0^\infty f_{s}(t;\tau) \tilde G(x,\tau) d\tau,
\notag
\end{align}
where we have denoted by $\tilde G(x,\tau) = (4\pi\tau)^{-\frac{n+2}{2}} e^{-\frac{|x|^2}{4\tau}}$.


\section{A Li-Yau inequality for Bessel operators}\label{S:bessel}

The discussion of the extension problem \eqref{ext2'} in Section \ref{S:pk} has evidenced the key role of the Bessel operator
\begin{equation}\label{Ba}
\mathscr B_a = \frac{\p^2 }{\p y^2} + \frac{a}{y} \frac{\p }{\p y},\ \ \ \ \ \ \ \ a =1-2s, 
\end{equation}
in the analysis of the fractional Laplacean $(-\Delta)^s$. But $\mathscr B_a$ plays an equally important role in the study of other nonlocal operators such as, for instance, the fractional heat operator $(\p_t - \Delta)^s$, see \cite{NS}, \cite{ST} and also \cite{BG}. Because of the ubiquitous presence of \eqref{Ba} in the fractional world, and since this topic is perhaps more frequented by workers in probability than analysts and geometers, in this section we provide a purely analytical construction of the fundamental solution of the heat semigroup associated with $\mathscr B_a$, see Proposition \ref{P:fsbessel} below. After that result, we recall a  proposition from the work in progress \cite{BaG2} which states that the Neumann heat semigroup associated with $\mathscr B_a$ satisfies a \emph{curvature-dimension inequality}. These interesting facts plays a role in establishing a nonlocal Harnack inequality in the (general) geometric framework of \cite{BaG1}. 

If with $-1<a<1$ we define the parameter $\nu$ by the equation
\[
a = 2\nu+1,
\]
then we can write \eqref{Ba} in the following way
\begin{equation}\label{Ba2}
\mathscr B_\nu = \frac{\p^2 }{\p y^2} + \frac{2\nu+1}{y} \frac{\p }{\p y},\ \ \ \ \ \ -1<\nu<0.
\end{equation}
We intend to study the Cauchy problem for the heat equation associated with \eqref{Ba2},
\begin{equation}\label{cp}
\begin{cases}
\frac{\p^2 f}{\p y^2} + \frac{2\nu+1}{y} \frac{\p f}{\p y} = \frac{\p f}{\p t},\ \ \ \ \ \ \ \ \ y>0, t>0,
\\
f(y,0) = \vf(y).
\end{cases}
\end{equation}
We remark that the case $a=0$ (which in the extension problem corresponds to the critical exponent $s = 1/2$) corresponds to the value $\nu = - 1/2$.

To solve \eqref{cp} we introduce the \emph{modified Hankel transform} of a function $f$
\begin{equation}\label{ht}
\mathcal H_\nu(f)(x) = \int_0^\infty f(y) G_\nu(xy) y^{2\nu+1} dy,
\end{equation}
see \cite{MS}, 
where we have let
\begin{equation}\label{gnu}
G_\nu(z) = z^{-\nu} J_\nu(z).
\end{equation} 
We recall, see e.g. 5.3.5 on p. 103 in \cite{Le}, that
\begin{equation}\label{rec}
G'_\nu(z) = - z G_{\nu + 1}(z).
\end{equation} 
Since for $z\in \mathbb C$ such that $|\arg z|<\pi$ we have
\begin{equation}\label{bfbehzero0}
\begin{cases}
G_\nu(z) \to \frac{2^{-\nu}}{\G(\nu+1)},\ \ \ \ \ \ \ \ \text{as}\ z\to 0,
\\
|G_\nu(z)| \le \frac{C}{|z|^{\nu+\frac 12}},\ \ \ \ \ \ \ \text{as}\ |z|\to \infty,
\end{cases}
\end{equation}
the integral defining \eqref{ht} is finite if $f\in \mathcal C_\nu(0,\infty)$, where  
\[
\mathcal C_\nu(0,\infty) = \{f\in C(0,\infty)\mid \forall R>0\ \text{one has}\ \int_0^R |f(y)| y^{2\nu+1} dy < \infty,
\ \ \int_R^\infty |f(y)| y^{\nu + \frac 12} dy < \infty\}.
\]
Note that $f\in \mathcal C_\nu(0,\infty)$ implies, in particular, that
\[
\underset{y\to 0^+}{\liminf}\ y^{2\nu+2} |f(y)| = 0,\ \ \ \ \ \underset{y\to \infty}{\liminf}\ y^{\nu+\frac 32} |f(y)| = 0.
\]
We will work with functions $f$ in such class. We will also need the following class
\[
\mathcal C^1_\nu(0,\infty) = \{f\in C^1(0,\infty)\mid f, \frac{1}{y} f' \in \mathcal C_\nu(0,\infty)\}.
\]
We notice that membership in $\mathcal C^1_\nu(0,\infty)$ imposes, in particular, the weak \emph{Neumann condition} 
\begin{equation}\label{weakn}
\underset{y\to 0^+}{\liminf}\ y^{2\nu+1} |f'(y)| = 0.
\end{equation}

\begin{lemma}\label{L:h1}
Let $f\in C^1_\nu(0,\infty)$. Then,
\[
\mathcal H_\nu(\frac{1}{y} \frac{d f}{d y})(x) = - \mathcal H_{\nu-1}(f)(x).
\]
\end{lemma}

\begin{proof}
Under the given assumptions on $f$ we can integrate by parts in the following integral and omit the boundary contributions since they vanish. Using \eqref{rec}, we find
\begin{align*}
\mathcal H_\nu(\frac{1}{y} \frac{d f}{d y})(x) & = \int_0^\infty f'(y) G_\nu(xy) y^{2\nu} dy = - \int_0^\infty f(y) x G'_\nu(xy) y^{2\nu} dy
\\
& - 2\nu \int_0^\infty f(y) G_\nu(xy) y^{2\nu - 1} dy
\\
& = \int_0^\infty f(y) \left[(xy)^2 G_{\nu+1}(xy) - 2 \nu G_\nu(xy)\right] y^{2\nu - 1} dy. 
\end{align*}
If we now use formula 5.3.6 on p. 103 in \cite{Le}
\[
J_{\nu+1}(z) =  \frac{2\nu}{z} J_\nu(z)  - J_{\nu-1}(z),
\]
we find
\begin{equation}\label{rec2}
z^2 G_{\nu+1}(z) = 2 \nu G_\nu(z) - G_{\nu-1}(z).
\end{equation}
Using this identity in the above integral we finally obtain
\[
\mathcal H_\nu(\frac{1}{y} \frac{d f}{d y})(x)  = - \int_0^\infty f(y) G_\nu(xy) y^{2\nu - 1} dy = - \mathcal H_{\nu-1}(f)(x).
\]

\end{proof}

\begin{lemma}\label{L:h2}
Suppose now that $f\in C^2(0,\infty)$ and that $f, f''\in C_\nu(0,\infty)$. Then,
\[
\mathcal H_\nu(f'')(x) = (2\nu +1) \mathcal H_{\nu-1}(f)(x) - x^2 \mathcal H_\nu(f)(x).
\]
\end{lemma}

\begin{proof}
Again, we can integrate by parts omitting the boundary terms in the integral defining $\mathcal H_\nu(f'')(x)$. This gives
\begin{align*}
\mathcal H_\nu(f'')(x) & = \int_0^\infty f(y) \left[(xy)^2 G''_\nu(xy) + 2(2\nu+1) xy G_\nu'(xy) + 2\nu(2\nu+1) G_\nu(xy)\right]y^{2\nu -1} dy
\end{align*}
Next we use the fact that $J_\nu$ satisfies Bessel's differential equation \eqref{besseleq},
\[
z^2 J_\nu''(z) + z J_\nu'(z) + (z^2 - \nu^2) J_\nu(z) = 0,
\]
to deduce that
\begin{equation}\label{deG}
z^2 G_\nu''(z) + (2\nu+1) z G_\nu'(z) + z^2  G_\nu(z) = 0.
\end{equation}
Using \eqref{deG} we find, after some simplification,
\begin{align*}
& (xy)^2 G''_\nu(xy) + 2(2\nu+1) xy G_\nu'(xy) + 2\nu(2\nu+1) G_\nu(xy) 
\\
& = - (2\nu+1) (xy)^2 G_{\nu+1}(xy) + 2\nu(2\nu+1) G_\nu(xy) - (xy)^2  G_\nu(xy),
\end{align*}
where in the last equality we have used \eqref{rec}. Next, we use \eqref{rec2} to conclude
\begin{align*}
& (xy)^2 G''_\nu(xy) + 2(2\nu+1) xy G_\nu'(xy) + 2\nu(2\nu+1) G_\nu(xy) 
\\
& = (2\nu+1) G_{\nu-1}(xy) - (xy)^2 G_\nu(xy).
\end{align*}
Substituting in the above integral we finally obtain
\begin{align*}
\mathcal H_\nu(f'')(x) & = \int_0^\infty f(y) \left[(2\nu+1) G_{\nu-1}(xy) - (xy)^2 G_\nu(xy)\right]y^{2\nu -1} dy
\\
& = (2\nu+1) \mathcal H_{\nu-1}(f)(x) - x^2 \mathcal H_\nu(f)(x).
\end{align*}
This completes the proof.

\end{proof}

With Lemmas \ref{L:h1} and \ref{L:h2} in hands, we return to the Cauchy problem with the purpose of finding a representation formula of the solution. We have the following result. For a different probabilistic approach see \cite{BS}, but one should keep in mind that the probabilist's generator is $\frac 12 \mathscr B_a$, with $\mathscr B_a$ as in \eqref{Ba2}. 

\begin{prop}\label{P:fsbessel}
Let $\nu>-1$ and consider the heat semigroup $P_t^\nu = e^{-t\mathscr B_\nu}$ on $(\R_+,y^{2\nu+1} dy)$  with generator $\mathscr B_\nu$ as in \eqref{Ba2}. Then, the Neumann heat kernel associated with $e^{-t\mathscr B_\nu}$ is given by
\begin{align}\label{heatsg}
p^N_\nu(x,y,t) & = p^N_\nu(y,x,t)  = (2t)^{-(\nu+1)} \left(\frac{xy}{2t}\right)^{-\nu} I_\nu\left(\frac{xy}{2t}\right) e^{-\frac{x^2+y^2}{4t}},
\end{align}
where we have denoted by $I_\nu(z)$ the modified Bessel function of the first kind defined by \eqref{Inu}.
This means that for any given function $\vf\in \mathcal C^1_\nu(0,\infty)$ the solution of the Cauchy problem \eqref{cp} is given by
\begin{equation}\label{Ncp}
P_t^\nu \vf(x)  = \int_0^\infty \vf(y) p^N_\nu(x,y,t) y^{2\nu+1} dy.
\end{equation}
\end{prop}

\begin{proof}

We assume that $f$ be a solution to \eqref{cp}. We formally apply to \eqref{cp} the Hankel transform $\mathcal H_\nu$ with respect to the variable $y$. I.e., we let 
 \[
 \mathcal H_\nu(f)(x,t) =  \int_0^\infty f(y,t) G_\nu(xy) y^{2\nu+1} dy.
 \]
Remarkably, if we use Lemma \ref{L:h2}
 and Lemma \ref{L:h1}, the term $(2\nu+1) \mathcal H_{\nu-1}(f)(x)$ magically drops, and the Cauchy problem \eqref{cp} is converted into the following one
\begin{equation}\label{cp2}
\begin{cases}
\frac{\p\mathcal H_\nu(f)}{\p t}(x,t) =  - x^2 \mathcal H_\nu(f)(x,t),\ \ \ \ \ \ \ \ \ \ \ x>0, t>0,
\\
\mathcal H_\nu(f)(x,0) = \mathcal H_\nu(\vf)(x),
\end{cases}
\end{equation}
whose unique solution is 
\begin{equation}\label{cph}
\mathcal H_\nu(f)(x,t) = \mathcal H_\nu(\vf)(x) e^{-tx^2}.
\end{equation}
We next apply formally $\mathcal H_\nu$ to \eqref{cph}, and use Hankel's inversion formula 
\[
\mathcal H_\nu(\mathcal H_\nu(f)) = f
\]
to find
\begin{align*}
f(x,t) & = \int_0^\infty (xz)^{-\nu} e^{-t z^2}\left(\int_0^\infty (zy)^{-\nu} J_\nu(zy) \vf(y) y^{2\nu+1} dy\right) J_\nu(xz) z^{2\nu+1} dz
\\
& = x^{-\nu} \int_0^\infty \vf(y) y^{\nu+1} \left(\int_0^\infty z e^{-t z^2} J_\nu(yz) J_\nu(xz) dz\right) dy.
\end{align*}
At this point we appeal to formula 3. on p. 223 in \cite{PBM} that gives
\begin{equation}\label{pru}
\int_0^\infty z e^{-t z^2} J_\nu(yz) J_\nu(xz) dz = \frac{1}{2t} I_\nu\left(\frac{xy}{2t}\right) e^{-\frac{x^2+y^2}{4t}},
\end{equation}
provided that $\Re \nu > - 1$. Since $\nu\in \R$ and $\nu>-1$, we can use \eqref{pru} and finally obtain
\begin{equation}\label{heat}
f(x,t) = \int_0^\infty \vf(y) p^N_\nu(x,y,t) y^{2\nu+1} dy,
\end{equation}
where
\begin{align}\label{heatsg2}
p^N_\nu(x,y,t) & = (2t)^{-(\nu+1)} \left(\frac{xy}{2t}\right)^{-\nu} I_\nu\left(\frac{xy}{2t}\right) e^{-\frac{x^2+y^2}{4t}}.
\notag
\end{align}
This establishes \eqref{heatsg}, \eqref{Ncp}, thus completing the proof.
 
\end{proof}

\begin{remark}\label{R:gaussian}
We note explicitly that for every $y>0, t>0$ one has
\begin{equation}\label{bfs0}
p^N_\nu(0,y,t)  = \frac{1}{2^{2\nu+1} \G(\nu+1)} t^{-(\nu+1)} e^{-\frac{y^2}{4t}}.
\end{equation}
This can be seen by the following power series representation
\begin{equation}\label{inu}
z^{-\nu} I_\nu(z) = 2^{-\nu} \sum_{k=0}^\infty \frac{(z/2)^{2k}}{\G(k+1) \G(k+\nu+1)},
\end{equation}
valid for $|\arg z|<\pi$. From \eqref{inu} we immediately recognize that, similarly to \eqref{bfbehzero}, we have
\begin{equation}\label{bfbehzero2}
z^{-\nu} I_\nu(z)\cong\frac{2^{-\nu}}{\Gamma(\nu+1)},\quad\text{as }z\to 0, \ \ |\arg z|<\pi.
\end{equation}
The desired conclusion \eqref{bfs0} immediately follows from \eqref{heatsg}
and \eqref{bfbehzero2}. Note that when $\nu = - \frac 12$ we obtain from \eqref{bfs0}
\[
p^N_{-\frac 12}(0,y,t)  = 2 (4\pi t)^{-1/2} e^{-\frac{y^2}{4t}}.
\]
\end{remark}


Our next result shows that the Neumann heat kernel associated with the Bessel operator $\mathscr B_a$ satisfies an inequality of Li-Yau type reminiscent of that in Lemma \ref{L:LYG} above (notice however that, unlike the classical heat kernel, we presently have an inequality, not an equality).  

\begin{prop}[Li-Yau inequality]\label{P:lybessel}
Let $- \frac 12 \le \nu < 0$, and denote by $E_\nu(x,y,t) = \log p^N_\nu(x,y,t)$, where $p^N_\nu(x,y,t)$ is as in \eqref{heatsg} in Proposition \ref{P:fsbessel} above. Then, the following Li-Yau type inequality holds
\begin{equation}\label{LYbessel}
(D_y \log E_\nu)^2 - D_t \log E_\nu \le \frac{\nu+1}{t}.
\end{equation}
\end{prop}

This result is derived from Proposition \ref{P:fsbessel} and we omit the relevant details. Similarly to Theorems \ref{T:LYH} and \ref{T:hH}, Proposition \ref{P:lybessel} leads to a related Li-Yau inequality and to a Harnack inequality 
for positive solutions of the heat equation $\p_t - \mathscr B_\nu$, where $\mathscr B_\nu$ is given in \eqref{Ba2} above. All this however is a small part of a bigger puzzle from \cite{BaG2}, and we refer the reader to that forthcoming work.


\section{The fractional $p$-Laplacean}\label{S:fpl}

We cannot close this fractional note without a brief discussion of a nonlinear nonlocal operator which has been attracting a great deal of attention over the past few years, and whose analysis poses remarkable challenges and open questions. 

We have seen in Proposition \ref{P:el} that the fractional Laplacean has a variational structure, in the sense that it also arises as the Euler-Lagrange equation of the energy functional $\mathscr E_{(s)}(u)$ in Definition \ref{D:energy}. In the local case $s=1$ the corresponding energy $\mathscr E(u) = \frac 12 \int_{\Rn} |\nabla u|^2 dx$ is only one in the infinite scale of exponents 
\[
\mathscr E_p(u) = \frac 1p \int_{\Rn} |\nabla u|^p dx, \ \ \ \ \ \ \ \ \ \ 1<p<\infty,
\]
 (we leave out the end-point cases $p=1$ and $p=\infty$ since their discussion would deserve a book in its own). It is well-known that the Euler-Lagrange equation of the functional $u\to \mathscr E_p(u)$ is the so-called $p$-\emph{Laplace equation}
 \[
 \Delta_p u = \operatorname{div}(|\nabla u|^{p-2} \nabla u) = 0,\ \ \ \ \ \ \ \ \ \ 1<p<\infty.
 \]
This operator is of course nonlinear and degenerate elliptic and, despite some similarities with its linear ancestor, the Laplacean, its analysis is much harder and not yet completely understood. The most fundamental open problem in dimension $n\ge 3$ remains to present day the \emph{unique continuation property}: it is disheartening that we do not know whether a nontrivial solution of $\Delta_p u = 0$ in a connected open set can have a zero of infinite order, or vanish in an open subset. When $n=2$ the strong unique continuation does hold as a consequence of the results of Bojarski and Iwaniec \cite{BI87} ($p>2$), Alessandrini \cite{Al87} ($1<p<\infty$), and Manfredi \cite{Ma88} ($1<p<\infty$).

It has long been known, however, that weak solutions of $\Delta_p u = 0$ have at best a locally H\"older continuous gradient. For instance, the function $u(x) = |x|^{p/(p-1)}$ satisfies the equation $\Delta_p u = c(n,p)$, and clearly we have $u\in C^{1,\alpha}_{loc}$, but $u\not\in C^2$, at least when $p>2$. The fundamental $C^{1,\alpha}$ regularity result was first proved in the late 60's by N. Ural'tseva when $p\ge 2$, and subsequently independently generalized to all $1<p<\infty$ (and to more general quasilinear equations) by J. Lewis \cite{Le83}, Di Benedetto \cite{DB83} and Tolksdorff \cite{To84}.

Because of its variational structure the $p$-Laplacean presents itself in connection with the case $p\not= 2$ of the Sobolev embedding theorem 
\[
W^{1,p}(\Rn)\ \hookrightarrow\ L^q(\Rn),\ \ \ \ \ \ \ \ \ \ \ \ \ \ \frac 1p - \frac 1q = \frac 1n.
\]  
But $\Delta_p$ plays an important role also in the applied sciences, for instance in the study of non-Newtonian fluids.

In view of what has been said so far it seems natural to consider for any $0<s<1$ the following fractional energy
\begin{equation}\label{nle}
\mathscr E_{(s),p}(u) = \frac 1p \int_{\Rn} \int_{\Rn} \frac{|u(x) -u(y)|^{p}}{|x-y|^{n+ps}} dx dy,\ \ \ \ \ \ \ \ \ \ \ 1<p<\infty.
\end{equation}
When $\Rn$ is replaced by the boundary of a bounded open set $\Om\subset \Rn$, such energy was introduced independently by Gagliardo \cite{gagliardo} and Slobodeckji \cite{Slo} in connection with the characterization of the traces on $\p \Om$ of functions in the Sobolev space $W^{1,p}(\Om)$. For a general geometric approach to the characterization of traces we refer the reader to the Memoir of the AMS \cite{DGN}. 

For any $0<s<1$ the \emph{fractional $p$-Laplace operator} $(-\Delta_p)^s$ is defined as the Euler-Lagrange equation of the energy functional $u\to \mathscr E_{(s),p}(u)$. It is an easy exercise to show that a function $u$ in the fractional Sobolev space 
\[
W^{s,p}(\Rn) = \{u\in L^p(\Rn)\mid \mathscr E_{(s),p}(u) <\infty\},
\]
endowed with the natural norm 
\[
||u||_{W^{s,p}(\Rn)} = \left(||u||^p_{L^p(\Rn)} + \mathscr E_{(s),p}(u)\right)^{1/p},
\]
is a weak solution of $(-\Delta_p)^s u = 0$ if for any $\vf\in W^{s,p}(\Rn)$ having compact support one has
\begin{equation}\label{fpl}
\int_{\Rn} \int_{\Rn} \frac{|u(x) -u(y)|^{p-2}(u(x) - u(y))(\vf(x) - \vf(y))}{|x-y|^{n+ps}} dx dy = 0.
\end{equation}
The equation $(-\Delta_p)^s u = 0$ was first independently introduced in the papers \cite{AMRT} and \cite{IN}.

There presently exists a large literature on the fractional $p$-Laplacean. Unfortunately, in this brief section we cannot go into a detailed discussion of all the interesting work that has been done. We will only quote some papers, referring the interested reader to those sources and the references therein. The Perron method for $(-\Delta_p)^s$ has been studied in \cite{LL}. The analogue of Serrin's 1964 $C^{0,\alpha}_{loc}$ regularity for weak solutions of $(-\Delta_p)^s u = 0$ is known, and it has been proved in \cite{DKP}. The same authors established the Harnack inequality in \cite{DKP0}. The H\"older continuity up to the boundary for the Dirichlet problem in $C^{1,1}$ domains was proved in \cite{IMS}. One should also see the preprint \cite{cozzi} which contains related results for minimizers of nonlocal functionals of the calculus of variations. Regularity estimates for solutions with measure data were established in \cite{KMS}.

There are of course many basic open questions, and the reader could derive some of them from the discussion of the nonlocal linear case $p=2$ in this note. But at present the most fundamental open problem concerning the operator $(-\Delta_p)^s$ is the nonlocal counterpart of the above cited $C^{1,\alpha}$ regularity theorem for the local case. In this connection, an interesting new contribution has been recently given in \cite{BL}, where the authors establish the nonlocal counterpart of a famous theorem of K. Uhlenbeck stating that when $p\ge 2$ weak solutions in $W^{1,p}_{loc}$ of the $p$-Laplacean system have in fact
\[
|\nabla u|^{\frac{p-2}{2}} \nabla u \in W^{1,2}_{loc}.
\]   
We note in passing that this fact implies that 
\[
\nabla u\in W^{s,p}_{loc},\ \ \ \ \ 0<s<\frac{2}p.
\]
The main result in \cite{BL}, which is Theorem 1.5,  provides a nonlocal analogue of this latter conclusion. Since the precise statement is somewhat involved, we refer the reader to their paper. We also mention the recent preprint \cite{BLS17} in which the authors establish an improved H\"older regularity result for solutions to the equation $(-\Delta_p)^s u = f$.

However, the optimal interior regularity of weak solutions of the equation $(-\Delta_p)^s u = 0$ presently remains \emph{terra incognita}.

\end{document}